\documentclass[12pt, a4paper]{article}

\usepackage{amsmath}
\usepackage{amssymb}
\usepackage{amsthm}
\usepackage{enumitem}
\usepackage{extarrows}
\usepackage{adjustbox}
\usepackage{tikz-cd}
\usepackage[arrow,matrix]{xy}
\usepackage[nottoc,notlot,notlof]{tocbibind}
\usepackage[hyperfootnotes=false]{hyperref}

\theoremstyle{plain}
\newtheorem{prop}{Proposition}[subsection]
\newtheorem{conj}[prop]{Conjecture}
\newtheorem{lem}[prop]{Lemma}
\newtheorem{thm}[prop]{Theorem}
\newtheorem{cor}[prop]{Corollary}
\theoremstyle{definition}
\newtheorem{definit}[prop]{Definition}
\newtheorem{ex}[prop]{Example}
\newtheorem{rem}[prop]{Remark}
\newtheorem{hyp}[prop]{Hypothesis}

\theoremstyle{plain}
\newtheorem{prop0}{Proposition}[section]
\newtheorem{conj0}[prop0]{Conjecture}
\newtheorem{lem0}[prop0]{Lemma}
\newtheorem{thm0}[prop0]{Theorem}
\newtheorem{cor0}[prop0]{Corollary}
\theoremstyle{definition}

\newcommand{\Z}{\mathbb{Z}}
\newcommand{\Q}{\mathbb{Q}}
\newcommand{\Zp}{\Z_p}
\newcommand{\Qp}{\Q_p}
\newcommand{\Fp}{\mathbb{F}_p}
\newcommand{\Fpbar}{\overline{\Fp}}

\newcommand{\bbA}{{\mathbb A}}
\newcommand{\bbG}{{\mathbb G}}
\newcommand{\bbF}{{\mathbb F}}

\newcommand{\rhobar}{\overline{\rho}}

\newcommand{\cB}{\mathcal B}
\newcommand{\cC}{\mathcal C}
\newcommand{\cG}{\mathcal G}
\newcommand{\cH}{\mathcal H}
\newcommand{\cL}{\mathcal L}
\newcommand{\cM}{\mathcal M}
\newcommand{\cN}{\mathcal N}
\newcommand{\cO}{\mathcal O}
\newcommand{\cR}{\mathcal R}
\newcommand{\cS}{\mathcal S}
\newcommand{\cE}{\mathcal E}
\newcommand{\cF}{\mathcal F}

\newcommand{\fa}{\mathfrak a}
\newcommand{\fb}{\mathfrak b}
\newcommand{\fg}{\mathfrak g}

\newcommand{\fl}{\mathfrak l}
\newcommand{\fm}{\mathfrak m}
\newcommand{\fn}{\mathfrak n}
\newcommand{\fp}{\mathfrak p}
\newcommand{\fR}{\mathfrak R}
\newcommand{\ft}{\mathfrak t}
\newcommand{\fz}{\mathfrak z}

\newcommand{\ol}{\overline}
\newcommand{\gl}{\mathfrak{gl}}
\newcommand{\GL}{\mathrm{GL}}

\newcommand{\id}{\mathrm{id}}
\newcommand{\sm}{\mathrm{sm}}
\newcommand{\alg}{\mathrm{alg}}
\newcommand{\val}{\mathrm{val}}
\newcommand{\unr}{\mathrm{unr}}
\newcommand{\res}{\mathrm{res}}
\newcommand{\soc}{\mathrm{soc}}
\newcommand{\rig}{\mathrm{rig}}
\newcommand{\cris}{\mathrm{cris}}
\newcommand{\diag}{\mathrm{diag}}
\newcommand{\cosoc}{\mathrm{cosoc}}
\newcommand{\Ad}{\mathrm{Ad}}
\newcommand{\Hom}{\mathrm{Hom}}
\newcommand{\Ext}{\mathrm{Ext}}
\newcommand{\Ker}{\mathrm{ker}}
\newcommand{\Fil}{\mathrm{Fil}}
\newcommand{\Gal}{\mathrm{Gal}}
\newcommand{\Ind}{\mathrm{Ind}}
\newcommand{\End}{\mathrm{End}}
\newcommand{\Res}{\mathrm{Res}}
\newcommand{\Ima}{\mathrm{image}}
\newcommand{\Spec}{\mathrm{Spec}}
\newcommand{\dR}{\mathrm{dR}}
\newcommand{\pdR}{\mathrm{pdR}}
\newcommand{\Frob}{\mathrm{Frob}}
\newcommand{\loc}{\mathrm{loc}}
\newcommand{\an}{\mathrm{an}}
\newcommand{\Spf}{\mathrm{Spf}}
\newcommand{\tri}{\mathrm{tri}}
\newcommand{\SL}{\mathrm{SL}}
\newcommand{\wt}{\mathrm{wt}}

\newcommand{\smat}[1]{\left( \begin{smallmatrix} #1 \end{smallmatrix} \right)}
\newcommand{\hooklongrightarrow}{\lhook\joinrel\longrightarrow}
\newcommand{\twoheadlongrightarrow}{\relbar\joinrel\twoheadrightarrow}

\title{Hodge filtration and crystalline representations of \texorpdfstring{$\GL_n$}{GLn}}

\author{Christophe Breuil\footnote{CNRS, B\^atiment 307, Facult\'e d'Orsay, Universit\'e Paris-Saclay, 91405 Orsay Cedex, France}\\
\and
Yiwen Ding\footnote{B.I.C.M.R., Peking University, No.~5 Yiheyuan Road Haidian District, Beijing, P.R.~China 100871}}

\date{}

\begin{document}

\maketitle

\vspace{-0.1cm}

\begin{abstract}
Let $p$ be a prime number, $n$ an integer $\geq 2$ and $\rho$ an $n$-dimensional automorphic $p$-adic Galois representation (for a compact unitary group) such that $r:=\rho\vert_{\Gal(\overline\Qp/\Qp)}$ is crystalline. Under a mild assumption on the Frobenius eigenvalues of $D:=D_{\cris}(r)$ and under the usual Taylor-Wiles conditions, we show that the locally analytic representation of $\GL_n(\Qp)$ associated to $\rho$ in the corresponding Hecke eigenspace of the completed $H^0$ contains an explicit finite length subrepresentation which determines and only depends on $r$. This generalizes previous results of the second author which assumed that the Hodge filtration on $D$ was as generic as possible. Our approach provides a much more explicit link to this Hodge filtration (in all cases), which allows to study the internal structure of this finite length locally analytic subrepresentation.
\end{abstract}

\tableofcontents

\section{Introduction}\label{sec:intro}

We fix a prime number $p$ and an integer $n\geq 2$. Let $L$ be a number field, $v$ a finite place of $L$ and $(S_{U^vU_v})_{U_v\subset G(L_v)}$ a tower of Shimura varieties over $L$ with fixed prime-to-$v$ level $U^v$, where $L_v$ is the completion of $L$ at $v$. Let $d$ be the common dimension of the $S_{U^pU_p}$ and assume $G(L_v)=\GL_n(L_v)$. Let $\rho_\pi$ be an irreducible Galois representation associated to some automorphic representation $\pi$ of $G({\mathbb A}_L)$ and assume that $U^v$ is small enough so that the Hecke eigenspace 
\begin{equation}\label{eq:classical}
\varinjlim_{U_v}H^d_{\text{Betti}}(S_{U^vU_v}, \overline\Qp)[\pi]
\end{equation}
associated to $\pi$ is non-zero. It is expected, and now established in many cases, that (\ref{eq:classical}) is a direct sum of finitely many copies of $\pi_v$, the $v$-factor of the automorphic representation $\pi$, which is a smooth irreducible representation of $\GL_n(L_v)$ corresponding to (the $F$-semi-simplification of) the Weil-Deligne representation $\text{WD}(\rho_{\pi,v})$ associated to $\rho_{\pi,v}:=\rho_{\pi}\vert_{\Gal(\overline{L_v}/L_v)}$ by a suitable twist of the local Langlands correspondence for $\GL_n(L_v)$. Here, to define $\text{WD}(\rho_{\pi,v})$ one distinsguishes two cases. If $v\!\nmid \!p$ then $\text{WD}(\rho_{\pi,v})$ was defined by Deligne a long time ago in \cite{De73}. If $v\!\mid \!p$, then $\text{WD}(\rho_{\pi,v})$ was defined by Fontaine a little less time ago in \cite{Fo94}. One key difference is that, while in the former case $\text{WD}(\rho_{\pi,v})$ (hence $\pi_v$) and $\rho_{\pi,v}$ contain essentially the same data, in the latter case an important piece of data is lost when going from $\rho_{\pi,v}$ to $\text{WD}(\rho_{\pi,v})$ or equivalently $\pi_v$: the Hodge filtration on $D_{\text{dR}}(\rho_{\pi,v})=(B_{\text{dR}}\otimes_{\Qp}\rho_{\pi,v})^{\Gal(\overline{L_v}/L_v)}$ which, roughly speaking, is the true $p$-adic part of $\rho_{\pi,v}$. Where did that Hodge filtration go on the $\GL_n(L_v)$-side?\bigskip

From now on we assume $v\!\mid \!p$. In practice, it is convenient to replace $\overline\Qp$ in (\ref{eq:classical}) by a sufficiently large finite extension $E$ of $\Qp$ (depending on $\pi$). To get the missing Hodge filtration, it is expected that one should replace (\ref{eq:classical}) by the Hecke eigenspace 
\begin{equation}\label{eq:complete}
\widetilde H^d(S_{U^v}, E)[\pi]
\end{equation}
of the (so-called) completed cohomology group $\widetilde H^d(S_{U^v}, E)$ as defined in \cite{Em01}. This is a $p$-adic Banach space over $E$ endowed with a continuous action of $\GL_n(L_v)$. When $d=0$ or $d=1$, which are in practice the main cases where the $\GL_n(L_v)$-representation (\ref{eq:complete}) has been seriously studied so far, $\widetilde H^d(S_{U^v}, E)$ is just the $p$-adic completion of $\varinjlim H^d_{\text{Betti}}(S_{U^vU_v}, E)$ with respect to its invariant lattice $\varinjlim H^d_{\text{Betti}}(S_{U^vU_v}, \cO_E)$. In these cases, the above expectation is a theorem in various situations and under various assumptions on $G$, $L$ or $\rho_\pi$, notably (the following list is \emph{not} exhaustive): $G=\GL_2$ and $L=\Q$ (\cite{Br10}, \cite{Co10}, \cite{Em10}, \cite{CDP14}, \cite{DLB17}, \cite{Pa25}), $G$ is a quaternion algebra, $L$ is totally real and $\rho_{\pi,v}$ is semi-stable non-crystalline (\cite{Di13}), $G$ is a compact unitary group in $3$ variables, $L$ is totally real, $L_v=\Qp$ and $\rho_{\pi,v}$ is semi-stable non-crystalline with Hodge-Tate weights $(2,1,0)$ (\cite{BD20}), $G$ is a compact unitary group, $L$ is totally real, $L_v=\Qp$ and $\rho_{\pi,v}$ is crystalline with a very generic Hodge filtration on $D_{\text{dR}}(\rho_{\pi,v})$ (\cite{Di25}), $G$ is a unitary similitude group in $2$ variables, $L$ is CM and $\rho_{\pi,v}$ is potentially semi-stable non-crystalline of parallel Hodge-Tate weights $(1,0)$ (\cite{QS25}). Except when $(G,L)=(\GL_2,\Q)$ (where a lot is known), the continuous $G(L_v)$-representation (\ref{eq:complete}) remains a mystery, and the strategy in the above cases is to show that it contains an explicit finite length locally $\Qp$-analytic $G(L_v)$-subrepresentation which ``contains'' the Hodge filtration on $D_{\text{dR}}(\rho_{\pi,v})$.\bigskip

One breakthrough of \cite{Di25} is that it has no restriction on $n$ and it concerns the crystalline case (which is somehow the first case one wants to treat). However it assumes that the Hodge filtration on $D_{\cris}(\rho_{\pi,v})=(B_{\cris}\otimes_{\Qp}\rho_{\pi,v})^{\Gal(\overline \Qp/\Qp)}$ is as generic as possible (precisely:~all refinements are non-critical), which is a quite strong assumption. The aim of the present work is twofold:
\begin{enumerate}[label=(\roman*)]
\item
remove this genericity hypothesis on the Hodge filtration;
\item
give a more explicit construction of the (not so explicit) finite length locally $\Qp$-analytic subrepresentation in \cite{Di25} which allows a much better understanding of its internal structure, in particular of its link to the Hodge filtration.
\end{enumerate}
Recall from \cite{BHS19} that allowing non-generic Hodge filtrations causes the appearance of new constituents in the $\GL_n(L_v)$-socle of (\ref{eq:complete}) called companion constituents. Hence the present work can also be seen as a sequel to \cite{BHS19} as we use this socle and go beyond it (though eventually we only need companion constituents associated to simple reflections). We now give with more details the results of this work.\bigskip

We fix $G$ a unitary group in $n$ variables over a totally real number field $F^+$ such that $G$ splits over an imaginary quadratic extension $F$ of $F^+$ and is compact at all infinite places of $F^+$ (in particular we have $d=0$). We assume that all $p$-adic places of $F^+$ split in $F$. We now write $\wp$ for the above place $v\!\mid \!p$ and do not assume anything on the field $F_{\wp}^+$. We choose a place $\widetilde\wp$ of $F$ above $\wp$ which determines an isomorphism $G(F^+_{\wp})\buildrel\sim\over\rightarrow \GL_n(F_{\wp}^+)\cong \GL_n(F_{\widetilde\wp})$. In that situation $\widetilde H^0(S_{U^{\wp}}, E)$ is the Banach space $\widehat S(U^{\wp},E)$ of continous functions $G(F^+)\backslash G(\mathbb{A}_{F^+}^{\infty})/U^\wp\longrightarrow E$ with continuous action of $\GL_n(F_{\wp}^+)$ by right translation. For technical reasons, instead of $\widehat S(U^{\wp},E)$ for $U^{\wp}$ sufficiently small, it is more convenient to set $U^{\wp}:=U^p\prod_{v\mid p,v\ne\wp}\!\GL_n(\cO_{F_v^+})$ with $U^p$ sufficiently small and use the $p$-adic Banach space
\[\widehat{S}_{\tau}(U^{\wp},E):=\big(\widehat{S}(U^p,E) \otimes_{E} (\otimes_{v\mid p,v\ne\wp}\sigma(\tau_v)^\vee)\big)^{\prod_{v\mid p,v\ne\wp}\!\GL_n(\cO_{F_v^+})}\]
where $\sigma(\tau_v)$ is a fixed $\GL_n(\cO_{F_v^+})$-type at $v$, see (\ref{Sxitau}) (in the paper we also fix distinct arbitrary Hodge-Tate weights at $v\mid p$, $v\ne\wp$, see \emph{loc.~cit.}, we ignore this in the introduction).\bigskip

We fix $\pi$ an automorphic representation of $G(\mathbb{A}_{F^+})$ such that $\widehat{S}_{\tau}(U^{\wp},E)[\pi]\ne 0$, $\rho_\pi$ is absolutely irreducible and $\rho_{\pi,\widetilde \wp}$ is crystalline. We write $K:=F_{\wp}^+=F_{\widetilde\wp}$, $f$ the degree over $\Fp$ of the residue field of $K$, $r:=\rho_{\pi,\widetilde \wp}$ $(=\rho_{\pi}\vert_{\Gal(\overline{F_{\widetilde\wp}}/F_{\widetilde\wp})})$ and $D:=D_{\cris}(r)$ which is a filtered $\varphi$-module. For each embedding $\sigma:K\hookrightarrow E$ (assuming $E$ large enough), the extension of scalars $\otimes_{K,\sigma}E$ gives a filtered $\varphi^f$-module $D_\sigma$ of dimension $n$ over $E$. We let $\{\varphi_0,\dots,\varphi_{n-1}\}$ (arbitrary numbering) be the eigenvalues in $E$ of the Frobenius $\varphi^f$ on $D_\sigma$ (increasing $E$ if necessary), which do not depend on $\sigma$. We require the following mild genericity assumption on the $\varphi_j$: $\varphi_j\varphi_k^{-1}\notin \{1, p^f\}\ \forall\ j\ne k$. Let $\widehat{S}_{\tau}(U^{\wp},E)[\pi]^{\Qp\text{-}\an}$ be the locally $\Qp$-analytic vectors of the continuous representation $\widehat{S}_{\tau}(U^{\wp},E)[\pi]$. One of the main results of this work is:

\begin{thm0}[Corollary \ref{cor:main}]\label{thm:mainintro}
Assume the Taylor-Wiles assumptions (see Hypothesis \ref{TayWil0}). The isomorphism class of the locally $\Qp$-analytic representation $\widehat{S}_{\tau}(U^{\wp},E)[\pi]^{\Qp\text{-}\an}$, hence also the isomorphism class of the continuous representation $\widehat{S}_{\tau}(U^{\wp},E)[\pi]$, determine the isomorphism classes of all the filtered $\varphi^f$-modules $D_{\sigma}$ for all embeddings $\sigma$. In particular if $K=\Qp$ the representations $\widehat{S}_{\tau}(U^{\wp},E)[\pi]^{\Qp\text{-}\an}$ and $\widehat{S}_{\tau}(U^{\wp},E)[\pi]$ determine the $\Gal(\overline\Qp/\Qp)$-representation $r=\rho_{\pi,\widetilde{\wp}}$.
\end{thm0}

It is very likely that the locally $\Qp$-analytic representation $\widehat{S}_{\tau}(U^{\wp},E)[\pi]^{\Qp\text{-}\an}$ completely determines the \emph{full} filtered $\varphi$-module $D$ (with its filtration on $D\otimes_{\Qp}K$) when $K\ne \Qp$, and thus the $\Gal(\overline K/K)$-representa\-tion $r$, but this seems much harder and is currently not known even for $n=2$. The proof of Theorem \ref{thm:mainintro} consists in finding an explicit finite length subrepresentation of $\widehat{S}_{\tau}(U^{\wp},E)[\pi]^{\Qp\text{-}\an}$ which we can relate to the Hodge filtration on the $D_{\sigma}$. We describe this subrepresentation and its properties below when $K=\Qp$ (for simplicity).\bigskip

The locally $\Qp$-algebraic vectors of $\widehat{S}_{\tau}(U^{\wp},E)[\pi]^{\Qp\text{-}\an}$ are of the form $(\pi_{\alg}(D)\otimes_E\varepsilon^{n-1})^{\oplus m}$ where $\pi_{\alg}(D)$ is the irreducible locally $\Qp$-algebraic representation of $\GL_n(K)$ over $E$ associated to the $\varphi$-module $D$ and its Hodge-Tate weights by the local Langlands correspondence (see (\ref{eq:alg})), $m$ is an integer $\geq 1$, $\varepsilon$ the $p$-adic cyclotomic character and $\otimes_E\varepsilon^{n-1}$ is short for $\otimes_E\varepsilon^{n-1}\circ{\det}$ ($\varepsilon$ is seen as a character of $K^\times$ via local class field theory). The representation $(\pi_{\alg}(D)\otimes_E\varepsilon^{n-1})^{\oplus m}$ does not see the Hodge filtration on $D_{\dR}(D)$. It is natural to look for the latter in $\widehat{S}_{\tau}(U^{\wp},E)[\pi]^{\Qp\text{-}\an}$ which is known to be strictly larger than $(\pi_{\alg}(D)\otimes_E\varepsilon^{n-1})^{\oplus m}$ (see for instance \cite[Thm.~1.1]{BH20} or \cite[Thm.~1.4]{BHS19}). Unfortunately the results of \emph{loc.~cit.}~still do not produce a subrepresentation of $\widehat{S}_{\tau}(U^{\wp},E)[\pi]^{\Qp\text{-}\an}$ which determines the Hodge filtration (except when $\GL_n(K)=\GL_2(\Qp)$, which is a quite special case).\bigskip

During \ many \ years \ the \ first \ author \ looked \ for \ possible \ extra \ constituents \ in \ $\widehat{S}_{\tau}(U^{\wp},E)[\pi]^{\Qp\text{-}\an}$ likely to determine the missing Hodge filtration, without success. In \cite{HHS25}, the authors made the remarkable discovery that, when $n=3$ and $r$ is split reducible, $\widehat{S}_{\tau}(U^{\wp},E)[\pi]^{\Qp\text{-}\an}$ contains a copy of $\pi_{\alg}(D)\otimes_E\varepsilon^{n-1}$ which is \emph{not} in its socle. Very recently, one of us discovered in \cite{Di25} that, at least when all refinements on $D$ are non-critical for all embeddings $K\hookrightarrow E$ (recall a refinement is an ordering on the set $\{\varphi_0,\dots,\varphi_{n-1}\}$), the representation $\widehat{S}_{\tau}(U^{\wp},E)[\pi]^{\Qp\text{-}\an}$ contains in its third layer a number of copies of $\pi_{\alg}(D)\otimes_E\varepsilon^{n-1}$ which is exponentially growing with $n$. One may wonder if these new locally algebraic constituents carry any information on the Hodge filtration. It turns out they do: the resulting subrepresentation of $\widehat{S}_{\tau}(U^{\wp},E)[\pi]^{\Qp\text{-}\an}$ has the form $(\pi(D)\otimes_E\varepsilon^{n-1})^{\oplus m}$ where $\pi(D)$ determines the Hodge filtration on all $D_{\sigma}$ (see \emph{loc.~cit.}). One aim of this work is to extend the definition of $\pi(D)$ to any $D$. We explain this in the case $K=\Qp$, which we assume from now on.\bigskip

To make more natural the definition of $\pi(D)$, we first need some heuristic background, which already underlies \cite{Di25} and for which we assume $m=1$. If one is optimistic, one could hope for a natural isomorphism of finite dimensional $E$-vector spaces (as for $\GL_2(\Qp)$)
\begin{equation}\label{eq:intro}
\Ext^1_{\Gal(\overline\Qp/\Qp)}(r,r)\buildrel{\substack{?\\ \sim}}\over\longrightarrow \Ext^1_{\GL_n(\Qp)}\big(\widehat{S}_{\tau}(U^{\wp},E)[\pi]^{\Qp\text{-}\an},\widehat{S}_{\tau}(U^{\wp},E)[\pi]^{\Qp\text{-}\an}\big).
\end{equation}
Via \ (\ref{eq:intro}) \ we \ let \ $\Ext^1_0(r,r)\subset \Ext^1_{\Gal(\overline\Qp/\Qp)}(r,r)$ \ be \ the \ kernel \ of \ the \ natural \ map \ to $\Ext^1_{\GL_n(\Qp)}(\pi_{\alg}(D)\otimes_E\varepsilon^{n-1},\widehat{S}_{\tau}(U^{\wp},E)[\pi]^{\Qp\text{-}\an})$ and define
\[\ol{\Ext}^1_{\Gal(\overline\Qp/\Qp)}(r,r):={\Ext}^1_{\Gal(\overline\Qp/\Qp)}(r,r)/\Ext^1_0(r,r).\]
Then (\ref{eq:intro}) would induce an embedding
\begin{equation}\label{eq:introbis}
\ol\Ext^1_{\Gal(\overline\Qp/\Qp)}(r,r)\buildrel{?}\over\hooklongrightarrow \Ext^1_{\GL_n(\Qp)}\big(\pi_{\alg}(D)\otimes_E\varepsilon^{n-1},\widehat{S}_{\tau}(U^{\wp},E)[\pi]^{\Qp\text{-}\an}\big).
\end{equation}
With even more optimism, we could expect that (\ref{eq:introbis}) is actually an isomorphism. Now assume we know an explicit subrepresentation $\pi_R(D)\otimes_E\varepsilon^{n-1}\subset \widehat{S}_{\tau}(U^{\wp},E)[\pi]^{\Qp\text{-}\an}$ strictly containing $\pi_{\alg}(D)\otimes_E\varepsilon^{n-1}$. Then the natural morphism given by functoriality (+ twisting by $\varepsilon^{n-1}$)
\[\Ext^1_{\GL_n(\Qp)}\big(\pi_{\alg}(D),\pi_R(D)\big)\longrightarrow \Ext^1_{\GL_n(\Qp)}\big(\pi_{\alg}(D)\otimes_E\varepsilon^{n-1},\widehat{S}_{\tau}(U^{\wp},E)[\pi]^{\Qp\text{-}\an}\big)\]
would induce a morphism $t_D:\Ext^1_{\GL_n(\Qp)}(\pi_{\alg}(D),\pi_R(D))\longrightarrow \ol\Ext^1_{\Gal(\overline\Qp/\Qp)}(r,r)$ which itself would induce a $\GL_n(\Qp)$-equivariant embedding (for formal reasons)
\begin{equation}\label{eq:embeddingintro}
\big(\pi_R(D)\!\begin{xy} (30,0)*+{}="a"; (40,0)*+{}="b"; {\ar@{-}"a";"b"}\end{xy}\!(\pi_{\alg}(D)\otimes_E\Ker(t_D))\big)\!\otimes_E\!\varepsilon^{n-1}\hooklongrightarrow \widehat{S}_{\tau}(U^{\wp},E)[\pi]^{\Qp\text{-}\an}
\end{equation}
where $\pi(D):=\pi_R(D)\!\begin{xy} (30,0)*+{}="a"; (39,0)*+{}="b"; {\ar@{-}"a";"b"}\end{xy}\!(\pi_{\alg}(D)\otimes_E\Ker(t_D))$ is the tautological extension associated to the subspace $\Ker(t_D)\subseteq \Ext^1_{\GL_n(\Qp)}(\pi_{\alg}(D),\pi_R(D))$. Note that, if $\Ext^1_{g}(r,r)\subset \Ext^1_{\Gal(\overline\Qp/\Qp)}(r,r)$ is the subspace of de Rham (equivalently here crystalline) extensions, one could hope that (\ref{eq:intro}) also induces an isomorphism between $\Ext^1_{g}(r,r)$ and the subspace of extensions such that their image in $\Ext^1_{\GL_n(\Qp)}(\pi_{\alg}(D)\otimes_E\varepsilon^{n-1},\widehat{S}_{\tau}(U^{\wp},E)[\pi]^{\Qp\text{-}\an})$ lies in the subspace of locally algebraic extensions $\Ext^1_{\alg}(\pi_{\alg}(D)\otimes_E\varepsilon^{n-1},\pi_{\alg}(D)\otimes_E\varepsilon^{n-1})$. As taking $D_{\cris}$ gives a natural map $\Ext^1_{g}(r,r)\longrightarrow \Ext^1_\varphi(D,D)\cong \Ext^1_{\alg}(\pi_{\alg}(D)\otimes_E\varepsilon^{n-1},\pi_{\alg}(D)\otimes_E\varepsilon^{n-1})$ where $\Ext^1_\varphi(D,D)$ means extensions as $\varphi$-modules, we see that $\Ext^1_0(r,r)$ would also be the kernel of this map. This is how we define $\Ext^1_0(r,r)$ in the text, see (\ref{eq:surligne}).\bigskip

In real life, we do not have \emph{a priori} an isomorphism as in (\ref{eq:introbis}). But for a specific explicit $\pi_R(D)$, we do construct by hand a morphism $t_D$ as above such that $\Ker(t_D)$ is \emph{big}. The fact that such a $\pi_R(D)$ with a big $\Ker(t_D)$ exists was discovered by the second author in \cite{Di25} when $D$ has no critical refinement and was unexpected (for instance $\Ker(t_D)=0$ when $n=2$). Moreover, the resulting $\pi(D)$ turns out to ``contain'' the Hodge filtration on $D$. We also do not have for free the embedding (\ref{eq:embeddingintro}) (again because we do not have (\ref{eq:introbis})), but we can still prove by ad hoc methods that at least a certain direct summand $\pi(D)^\flat\otimes_E\varepsilon^{n-1}$ of $\pi(D)\otimes_E\varepsilon^{n-1}$ embeds into $\widehat{S}_{\tau}(U^{\wp},E)[\pi]^{\Qp\text{-}\an}$ (when the refinements on $D$ are not too critical we have $\pi(D)^\flat\cong \pi(D)$). Fortunately this direct summand still determines the Hodge filtration. Moreover, using this direct summand together with results of Z.~Wu, we can prove that we do have an isomorphism (\ref{eq:introbis}) \emph{a posteriori}. Note that in the proofs it is much more flexible to work with the $(\varphi,\Gamma)$-module over the Robba ring $\cM(D):=D_{\rig}(r)$ associated to $r$ in \cite{CC98}, \cite{Be081} instead of $r$ itself (recall $\Ext^1_{\Gal(\overline\Qp/\Qp)}(r,r)\buildrel\sim\over\rightarrow \Ext^1_{(\varphi,\Gamma)}(\cM(D),\cM(D))$).\bigskip

Let $R:=\{s_1,\dots,s_{n-1}\}$ be the set of simple reflections of $\GL_n$ (this is the ``$R$'' of $\pi_R(D)$), our representation $\pi_R(D)$~is:
\begin{equation*}
\pi_R(D):= \pi_{\alg}(D)\begin{xy} (0,0)*+{}="a"; (12,0)*+{}="b"; {\ar@{-}"a";"b"}\end{xy}\!\Big(\bigoplus_{I}C(I, s_{\vert I\vert})\Big)
\end{equation*}
where $I$ runs through the subsets of $\{\varphi_0,\dots,\varphi_{n-1}\}$ of cardinality in $\{1,\dots,n-1\}$, $C(I,s_{\vert I\vert})$ is the socle of an explicit locally analytic principal series (see (\ref{eq:dotaction})) and where the subextension $\pi_{\alg}(D)\begin{xy} (0,0)*+{}="a"; (8,0)*+{}="b"; {\ar@{-}"a";"b"}\end{xy}\!C(I, s_{\vert I\vert})$ is split if and only if $C(I,s_{\vert I\vert})$ is a companion constituent (since $\dim_E\Ext^1_{\GL_n(\Qp)}(C(I, s_{\vert I\vert}),\pi_{\alg}(D))=1$ this uniquely determines $\pi_R(D)$). Note that $\pi_R(D)\otimes_E\varepsilon^{n-1}\subset \widehat{S}_{\tau}(U^{\wp},E)[\pi]^{\Qp\text{-}\an}$ by \cite[Thm.~1.1]{BH20} with \cite[Thm.~1.4]{BHS19}. We then prove:

\begin{thm0}[Theorem \ref{thm:independant}]\label{thm:tDintro}
There is a canonical surjection of finite dimensional $E$-vector spaces
\[t_D:\Ext^1_{\GL_n(\Qp)}\big(\pi_{\alg}(D),\pi_R(D)\big)\twoheadlongrightarrow \ol\Ext^1_{\Gal(\overline\Qp/\Qp)}(r,r)\]
such that $\dim_E\Ker(t_D)=2^n-1-\frac{n(n+1)}{2}$.
\end{thm0}

Note that $\Ker(t_D)=0$ if and only if $n=2$. Theorem \ref{thm:tDintro} is proved in \cite[Thm.~1.3]{Di25} when all refinements on $D$ are non-critical, but its proof heavily uses this non-criticality assumption. In this work we prove Theorem \ref{thm:tDintro} in two steps which both do not require non-criticality (and provide a different proof even when $D$ has no critical refinements):
\begin{enumerate}[label=(\roman*)]
\item
We prove that there is a surjection depending on a choice of $\log(p)\in E$ (Proposition \ref{prop:map} with Proposition \ref{prop:epsilon})
\begin{equation}\label{eq:filintro}
t_D:\Ext^1_{\GL_n(\Qp)}(\pi_{\alg}(D),\pi_R(D))\twoheadlongrightarrow \Ext^1_{\varphi}(D,D) \bigoplus \Hom_{\Fil}(D,D)
\end{equation}
where $\Hom_{\Fil}(D,D)$ is the endomorphisms of $E$-vector spaces which respect the Hodge filtration and where the kernel of (\ref{eq:filintro}) has dimension $2^n-1-\frac{n(n+1)}{2}$.
\item
We prove that there is an isomorphism depending on a choice of $\log(p)\in E$ (Corollary \ref{cor:splitsigma})
\begin{equation}\label{eq:gammaintro}
\Ext^1_{\varphi}(D,D) \bigoplus \Hom_{\Fil}(D,D) \buildrel\sim\over\longrightarrow \Ext^1_{\Gal(\overline\Qp/\Qp)}(r,r)
\end{equation}
such that the composition (still denoted) $t_D:=$(\ref{eq:gammaintro})$\circ$(\ref{eq:filintro}) does not depend on any choice (Theorem \ref{thm:independant}) and coincides with \cite[Thm.~1.3]{Di25} when all refinements on $D$ are non-critical (Corollary \ref{cor:Di25}). This is the map $t_D$ of Theorem \ref{thm:tDintro}.
\end{enumerate}
Let us give a few details on the surjection (\ref{eq:filintro}), which is new and important because this is where we link the Hodge filtration to the $\GL_n(\Qp)$-side. For each $i\in \{1,\dots,n-1\}$ we fix an isomorphism of $E$-vector spaces (in the spirit of \cite[(1.1)]{BD23}, see Remark \ref{rem:BD23}):
\begin{equation}\label{eq:isointro}
\Ext^1_{\GL_n(\Qp)}\Big(\bigoplus_{\vert I\vert=i}C(I, s_{i}), \pi_{\alg}(D)\Big)\buildrel\sim\over\longrightarrow \bigwedge\nolimits_E^{\!n-i}\!D
\end{equation}
sending $\Ext^1_{\GL_n(\Qp)}(C(I, s_{\vert I\vert}),\pi_{\alg}(D))$ to $\bigwedge\nolimits_E^{\!n-i}$ of the $(n-i)$-dimensional subspace of $D$ of $\varphi$-eigenvectors with eigenvalue $\notin I$ (see (\ref{eq:epsiloni})). For $i\in \{0,\dots,n-1\}$ let $\Fil_i^{\max}D\subseteq \bigwedge\nolimits_E^{\!n-i}\!D$ be the first (one-dimensional) step of the filtration on $\bigwedge\nolimits_E^{\!n-i}\!D$ induced by the Hodge filtration on $D$, that we see as a subspace of $\Ext^1_{\GL_n(\Qp)}(\bigoplus_{\vert I\vert=i}C(I, s_{\vert I\vert}),\pi_{\alg}(D))$ via (\ref{eq:isointro}) when $i>0$. Then $\pi_R(D)$ is isomorphic to the tautological extension of $\bigoplus_{i=1}^{n-1}\!\big(\big(\bigoplus_{\vert I\vert=i}C(I, s_{i})\big)\otimes_E\Fil_i^{\max}D\big)$ by $\pi_{\alg}(D)$. The map (\ref{eq:filintro}) then factors as follows (writing $\Ext^1$ for $\Ext^1_{\GL_n(\Qp)}$):
{\scriptsize
\begin{eqnarray*}
\Ext^1(\pi_{\alg}(D),\pi_R(D))&\buildrel\sim\over\longrightarrow &\Ext^1(\pi_{\alg}(D),\pi_{\alg}(D))\bigoplus \Big(\bigoplus_{i=1}^{n-1}\Ext^1\Big(\pi_{\alg}(D),\Big(\bigoplus_{\vert I\vert=i}C(I, s_{i,\sigma})\Big)\otimes_E\Fil_i^{\max}D\Big)\Big)\\
&\buildrel\sim\over\longrightarrow & \Ext^1_{\varphi}(D,D) \bigoplus \Big(\bigoplus_{i=0}^{n-1}\Hom_E\big(\bigwedge\nolimits_E^{\!n-i}\!D, \Fil_i^{\max}D\big)\Big)\\
&\twoheadlongrightarrow & \Ext^1_{\varphi}(D,D) \ \bigoplus \ \Hom_{\Fil}(D,D)
\end{eqnarray*}}
\!\!where the first isomorphism depends on a choice of $\log(p)$, the second uses (\ref{eq:isointro}) and a natural duality (see (\ref{eq:isodual})) and the last surjection is linear algebra (see the proof of Proposition \ref{prop:map}, in particular Step $3$). Up to isomorphism the map (\ref{eq:filintro}) does not depend on the choices of the isomorphisms (\ref{eq:isointro}) (see Proposition \ref{prop:epsilon}).\bigskip 

The proof of (\ref{eq:gammaintro}) is entirely on the Galois side. To prove that (\ref{eq:gammaintro})$\circ$(\ref{eq:filintro}) in (ii) does not depend on the choice of $\log(p)$ is a bit subtle and essentially relies on the important Lemma \ref{lem:nuI}.\bigskip

Though we stick to the crystalline case in this work, we expect the surjection $t_D$ to exist (and to have a description analogous to (\ref{eq:filintro})) without assuming $r$ crystalline once one has a suitable $\pi_R(D)$ (remembering $r$ is always de Rham with distinct Hodge-Tate weights).\bigskip

Recall $\pi(D)$ is the tautological extension $\pi_R(D)\!\begin{xy} (30,0)*+{}="a"; (39,0)*+{}="b"; {\ar@{-}"a";"b"}\end{xy}\!(\pi_{\alg}(D)\otimes_E\Ker(t_D))$. For $S$ a subset of $R$, we let $\pi(D)(S)\subseteq \pi(D)$ be the maximal subrepresentation which does not contain any $C(I,s_{\vert I\vert})$ for $s_{\vert I\vert}\notin S$ in its Jordan-H\"older constituents. Defining $t_D$ as in (\ref{eq:filintro}) is very useful for proving the following theorem:

\begin{thm0}\label{thm:structureintro}\ 
\begin{enumerate}[label=(\roman*)]
\item
The representation $\pi(D)$ has a central character and an infinitesimal character (Corollary \ref{lem:inf}).
\item
The representation $\pi(D)(S)$ has the form
\[\pi_{\alg}(D)\!\begin{xy} (0,0)*+{}="a"; (10,0)*+{}="b"; {\ar@{-}"a";"b"}\end{xy}\!\!\Big(\!\bigoplus_{I\mathrm{\ \!s.t.\!\ }s_{\vert I\vert}\in S} \!\!\!\big(C(I, s_{\vert I\vert})\big)\Big) \!\begin{xy} (30,0)*+{}="a"; (40,0)*+{}="b"; {\ar@{-}"a";"b"}\end{xy}\big(\pi_{\alg}(D)^{\oplus(\sum_{s_i\in S}\binom{n}{i}) + 1 - \dim(r_{P_{S^c}})}\big)\]
(with possibly split extensions as subquotients) where $r_{P_{S^c}}$ is the full radical of the standard parabolic subgroup of $\GL_n$ associated to the simple roots \emph{not} in $S$ (see (\ref{eq:formS})).
\item
Let $0=\Fil^{-h_{n-1}+1}(D)\subsetneq \Fil^{-h_{n-1}}(D)\subsetneq \cdots \subsetneq \Fil^{-h_{1}}(D) \subsetneq \Fil^{-h_{0}}(D)=D$ be the Hodge filtration on $D$. The isomorphism class of $\pi(D)(S)$ determines and only depends on the Hodge-Tate weights $\{h_{j},0\leq j\leq n-1\}$ and the isomorphism class of the filtered $\varphi$-module $D$ endowed with the partial filtration $(\Fil^{-h_{i}}(D),i\text{ such that }s_i\in S)$. In particular $\pi(D)$ determines (and only depends on) the $\Gal(\overline\Qp/\Qp)$-representation $r$ (Theorem \ref{thm:fil}).
\end{enumerate}
\end{thm0}

We refer to (ii) of Remark \ref{rem:hodge} for an explicit representation theoretic way to ``see'' the Hodge filtration of $D$ on $\pi(D)$. In fact \emph{loc.~cit.}~is just a sample, there are other similar ways to see the Hodge filtration which seems ``widespread'' in $\pi(D)$ and ``overdetermined'' by $\pi(D)$. Note that (\ref{eq:filintro}), (an analogue of) (\ref{eq:gammaintro}) and Theorem \ref{thm:structureintro} are proven for arbitrary $K$ replacing the filtered $\varphi$-module $D$ by the filtered $\varphi^f$-module $D_\sigma$ for an arbitrary embedding $\sigma:K\hookrightarrow E$.\bigskip

A refinement $\fR=(\varphi_{j_1},\dots,\varphi_{j_n})$ is said to be compatible with a subset $I\subset \{\varphi_0,\dots,\varphi_{n-1}\}$ if $I=\{\varphi_{j_1},\dots,\varphi_{j_{\vert I\vert}}\}$. To $\fR$ one can associate a permutation $w_{\fR}\in S_n$, and $\fR$ is non-critical if and only if $w_{\fR}=w_0:=$ the longest permutation in $S_n$ (see for instance \cite[\S~3.6]{BHS19}). We say that a subset $I$ is \emph{very critical} if there exists a refinement $\fR$ compatible with $I$ such that $s_{\vert I\vert}$ appears with multiplicity at least $2$ in \emph{all} reduced expressions of $w_{\fR}w_0$. In that case we can prove that the same actually holds for all refinements compatible with $I$ (see Definition \ref{def:critical}). We then prove that $\pi(D)$ has the form (see Proposition \ref{prop:flat})
\[\pi(D) \ \ \cong \ \ \pi(D)^\flat \ \bigoplus \ \Big(\bigoplus_{I\text{ very critical}}\big(C(I, s_{\vert I\vert})\!\begin{xy} (30,0)*+{}="a"; (39,0)*+{}="b"; {\ar@{-}"a";"b"}\end{xy}\!\pi_{\alg}(D)\big)\Big)\]
(for a certain direct summand $\pi(D)^\flat$) where all extensions on the right are non-split. Since $\dim_E\Ext^1_{\GL_n(\Qp)}(\pi_{\alg}(D),C(I, s_{\vert I\vert}))=1$, the isomorphism classes of $\pi(D)$ and $\pi(D)^\flat $ determine each other, in particular $\pi(D)^\flat$ still determines $r$ by (iii) of Theorem \ref{thm:structureintro}. Note that $\pi(D)\cong \pi(D)^\flat$ if there are no very critical $I$, which always holds when $n\leq 3$.\bigskip

We then conjecture:

\begin{conj0}\label{conj:mainintro}
The injection $(\pi_{\alg}(D)\otimes_E\varepsilon^{n-1})^{\oplus m}\hookrightarrow \widehat{S}_{\tau}(U^{\wp},E)[\pi]^{\Qp\text{-}\an}$ extends to a $\GL_n(\Qp)$-equivariant injection: 
\[(\pi(D) \otimes_E \varepsilon^{n-1})^{\oplus m}\hooklongrightarrow \widehat{S}_{\tau}(U^{\wp},E)[\pi]^{\Qp\text{-}\an}.\]
\end{conj0}

See Conjecture \ref{conj:main} for a more precise statement. The following is our main result towards Conjecture \ref{conj:mainintro}.

\begin{thm0}[Theorem \ref{T: lg} and Theorem \ref{T: lg2}]\label{thm:mainintrobis}
Assume the Taylor-Wiles assumptions (see Hypothesis \ref{TayWil0}). There exist integers $m_I\geq m$ for $I$ very critical and a possibly split extension
\begin{equation}\label{eq:reprintro}
(\pi(D)^{\flat})^{\oplus m}\!\begin{xy} (30,0)*+{}="a"; (42,0)*+{}="b"; {\ar@{-}"a";"b"}\end{xy}\!\Big(\bigoplus_{I\text{ very critical}}\!\!\big(C(I, s_{\vert I\vert})\!\begin{xy} (30,0)*+{}="a"; (38,0)*+{}="b"; {\ar@{-}"a";"b"}\end{xy}\!\pi_{\alg}(D)\big)^{\oplus m_I}\Big)\end{equation}
(still with non-split extensions on the right) satisfying the following properties:
\begin{enumerate}[label=(\roman*)]
\item
the representation (\ref{eq:reprintro}) contains as a subrepresentation
\[(\pi(D)^{\flat})^{\oplus m}\ \bigoplus \ \Big(\bigoplus_{I\text{ very critical}}\!\!C(I, s_{\vert I\vert})^{\oplus m_I}\Big);\]
\item
there is a $\GL_n(\Qp)$-equivariant injection
\begin{multline}\label{eq:injintro}
\bigg((\pi(D)^{\flat})^{\oplus m}\!\begin{xy} (30,0)*+{}="a"; (42,0)*+{}="b"; {\ar@{-}"a";"b"}\end{xy}\!\Big(\bigoplus_{I\text{ very critical}}\!\!\big(C(I, s_{\vert I\vert})\!\begin{xy} (30,0)*+{}="a"; (38,0)*+{}="b"; {\ar@{-}"a";"b"}\end{xy}\!\pi_{\alg}(D)\big)^{\oplus m_I}\Big)\bigg)\otimes_E\varepsilon^{n-1}\\
\hooklongrightarrow \widehat{S}_{\tau}(U^{\wp},E)[\pi]^{\Qp\text{-}\an}
\end{multline}
extending $(\pi_{\alg}(D)\otimes_E\varepsilon^{n-1})^{\oplus m}\hookrightarrow \widehat{S}_{\tau}(U^{\wp},E)[\pi]^{\Qp\text{-}\an}$ and such that
\[\Hom_{\GL_n(\Qp)}\big(\pi_{\alg}(D)\otimes_E \varepsilon^{n-1}, \widehat{S}_{\tau}(U^{\wp},E)[\pi]^{\Qp\text{-}\an}/Y\big)=0\]
where $Y$ denotes the image of (\ref{eq:injintro}).
\end{enumerate}
\end{thm0} 

Again, Theorem \ref{thm:mainintrobis} is proved in the text for any extension $K$ (not just $K=\Qp$). Although strictly speaking this is not implied by Conjecture \ref{conj:mainintro}, we expect all $m_I$ in Theorem \ref{thm:mainintrobis}~to~be~$m$ (see Conjecture \ref{conj:bis}). But proving that the middle extension in (\ref{eq:reprintro}) is split and that all $m_I=m$ (which would give $\pi(D)^{\oplus m}$) seems hard, even for $n=4$. We could only gather indirect evidence via the Bezrukavnikov functor of \cite[\S~7.2]{HHS25}, see the~end~of~\S~\ref{sec:loc-glob2}.\bigskip

Let $D'$ be a filtered $\varphi$-module with distinct Hodge-Tate weights and Frobenius eigenvalues satisfying the same genericity assumption as $D$. We expect that, if there is an injection $\pi(D')^\flat\hookrightarrow \widehat{S}_{\tau}(U^{\wp},E)[\pi]^{\Qp\text{-}\an}$, then $D'=D$, but we cannot prove it. However we can prove it for certain $D'$ (Proposition \ref{prop:nasty}). Theorem \ref{thm:mainintro} then easily follows from this (with \cite[Thm.~1.4]{BHS19}) and from the embedding $(\pi(D)^{\flat})^{\oplus m}\hookrightarrow \widehat{S}_{\tau}(U^{\wp},E)[\pi]^{\Qp\text{-}\an}$ induced by (\ref{eq:injintro}) (see Corollary \ref{cor:main}).\bigskip

We now give some details on the (long) proof of Theorem \ref{thm:mainintrobis} (in the case $K=\Qp$). We prove it in two separate steps: first we prove an injection $(\pi(D)^\flat \otimes_E \varepsilon^{n-1})^{\oplus m}\hookrightarrow \widehat{S}_{\tau}(U^{\wp},E)[\pi]^{\Qp\text{-}\an}$, then we use it to prove an injection as in (\ref{eq:injintro}).\bigskip

Let us start with $\pi(D)^\flat$. Here, the strategy is the same as in \cite{Di25} but now we have to deal with (not too) critical refinements. Let $R_r$ be the local complete $E$-algebra pro-representing framed deformations of $r$ over artinian $E$-algebras and $\fm_{R_r}$ its maximal ideal. There exists a local Artinian $E$-subalgebra $A_D$ of $R_r/\fm_{R_r}^2$ of maximal ideal $\fm_{A_D}$ such that 
\begin{equation*}
(\fm_{R_r}/\fm_{R_r}^2)^{\vee}\cong \Ext^1_{\Gal(\overline \Qp/\Qp)}(r,r) \twoheadlongrightarrow (\fm_{A_D})^{\vee}\cong \ol\Ext^1_{\Gal(\overline \Qp/\Qp)}(r,r).
\end{equation*}
Let $\widetilde{\pi}_R(D)$ be the tautological extension of $\pi_{\alg}(D) \otimes_E \Ext^1_{\GL_n(\Qp)}(\pi_{\alg}(D), \pi_R(D))$ by $\pi_R(D)$. Replacing $\pi_R(D)$ by $\pi_{\flat}(D):=\pi_R(D)\cap \pi(D)^\flat$ (intersection inside $\pi(D)$), define in a similar way $\widetilde{\pi}_\flat(D)$. Then it is easy to check that $\widetilde{\pi}_\flat(D)$ is a direct summand of $\widetilde{\pi}_R(D)$, and using the map $t_D$ of Theorem \ref{thm:tDintro} we can define a natural $\GL_n(\Qp)$-equivariant action of $A_D$ on $\widetilde{\pi}_R(D)$ preserving $\widetilde{\pi}_\flat(D)$, see (\ref{E: action}). It is formal to check that the subrepresentation $\widetilde{\pi}_R(D)[\fm_{A_D}]$ of elements cancelled by $\fm_{A_D}$ is $\pi(D)$, and likewise $\widetilde{\pi}_\flat(D)[\fm_{A_D}]\cong \pi(D)^{\flat}$.\bigskip

Using \cite[\S~2]{CEGGPS16} as slightly enhanced in \cite[Thm.~3.5]{BHS171}, recall one can patch the localization $\widehat{S}_{\tau}(U^{\wp},E)_{\overline\rho_\pi}$ into a continuous $R_\infty(\tau)$-admissible $\GL_n(\Qp)$-representation $\Pi_{\infty}(\tau)$ (where $R_\infty(\tau)$ is the patched deformation ring of type $\tau_v$ at $v\!\mid \!p,v\ne\wp$) such that
\[\Pi_{\infty}(\tau)^{R_{\infty}(\tau)\text{-}\an}[\pi]\cong \widehat{S}_{\tau}(U^{\wp},E)[\pi]^{\Qp\text{-}\an}\]
where $\Pi_{\infty}(\tau)^{R_{\infty}(\tau)\text{-}\an}$ is the subspace of $\Pi_{\infty}(\tau)$ of locally $R_\infty(\tau)$-analytic vectors in the sense of \cite[\S~3.1]{BHS171}. It is not difficult to define an ideal $\fa_{\pi}$ of $R_{\infty}(\tau)[1/p]$ such that $\Pi_{\infty}(\tau)^{R_{\infty}(\tau)\text{-}\an}[\pi]\subset \Pi_{\infty}(\tau)^{R_{\infty}(\tau)\text{-}\an}[\fa_{\pi}]$ and $A_D[1/p] \xrightarrow{\sim} R_{\infty}(\tau)[{1}/{p}]/\fa_{\pi}$, see (\ref{eq:api}). In particular $\Pi_{\infty}(\tau)^{R_{\infty}(\tau)\text{-}\an}[\fa_{\pi}]$ is equipped with a $\GL_n(\Qp)$-equivariant action of $A_D$ induced from the action of $R_{\infty}(\tau)$. To prove Conjecture \ref{conj:mainintro}, it would be enough to prove that there is a $\GL_n(\Qp) \times A_D$-equivariant injection
\begin{equation*}
(\widetilde{\pi}_R(D) \otimes_E \varepsilon^{n-1})^{\oplus m} \hooklongrightarrow \Pi_{\infty}(\tau)^{R_{\infty}(\tau)\text{-}\an}[\fa_{\pi}]
\end{equation*}
and then take the subspaces killed by $\fm_{A_D}$ on both sides. Though we think that such an injection exists, we do not know how to prove it (essentially because we do not know how to deal with very critical $I$). But we do have:

\begin{prop0}[Proposition \ref{prop:T: lg}]\label{prop:intro}
Assume the Taylor-Wiles assumptions (see Hypothesis \ref{TayWil0}). The injection $(\pi_{\alg}(D)\otimes_E\varepsilon^{n-1})^{\oplus m}\hookrightarrow \Pi_{\infty}(\tau)^{R_{\infty}(\tau)\text{-}\an}[\pi]$ extends to a $\GL_n(\Qp) \times A_D$-equivariant injection
\begin{equation*}
(\widetilde{\pi}_\flat(D) \otimes_E \varepsilon^{n-1})^{\oplus m} \hooklongrightarrow \Pi_{\infty}(\tau)^{R_{\infty}(\tau)\text{-}\an}[\fa_{\pi}],
\end{equation*}
hence to a $\GL_n(\Qp)$-equivariant injection: 
\[(\pi(D)^\flat \otimes_E \varepsilon^{n-1})^{\oplus m}\hooklongrightarrow \Pi_{\infty}(\tau)^{R_{\infty}(\tau)\text{-}\an}[{\pi}]\cong \widehat{S}_{\tau}(U^{\wp},E)[\pi]^{\Qp\text{-}\an}.\]
\end{prop0}

Let us give the steps for the proof of Proposition \ref{prop:intro}. For $I\subset \{\varphi_0,\dots,\varphi_{n-1}\}$ (of cardinality in $\{1,\dots,n-1\}$) let $\pi_I(D):= \pi_{\alg}(D)\begin{xy} (0,0)*+{}="a"; (8,0)*+{}="b"; {\ar@{-}"a";"b"}\end{xy}\!C(I, s_{\vert I\vert})$ (unique non-split extension) if $C(I, s_{\vert I\vert})$ is not a companion constituent, $\pi_I(D):= \pi_{\alg}(D)\oplus C(I, s_{\vert I\vert})$ if $C(I, s_{\vert I\vert})$ is a companion constituent (see (\ref{eq:VipiI}), (\ref{eq:splitI})). Define $\widetilde{\pi}_I(D)$ as the tautological extension of $\pi_{\alg}(D) \otimes_E \Ext^1_{\GL_n(\Qp)}(\pi_{\alg}(D), \pi_I(D))$ by $\pi_I(D)$. Then $\widetilde{\pi}_\flat(D)$ is an amalgamated sum of the $\widetilde{\pi}_I(D)$ for those $I$ which are not very critical (see (\ref{eq:decopflat}) for a precise description), and each $\widetilde{\pi}_I(D)$ is preserved by the action of $A_D$ inside $\widetilde{\pi}_\flat(D)$. Hence it is enough to prove that $(\pi_{\alg}(D)\otimes_E\varepsilon^{n-1})^{\oplus m}\hookrightarrow \Pi_{\infty}(\tau)^{R_{\infty}(\tau)\text{-}\an}[\pi]$ extends to a $\GL_n(\Qp) \times A_D$-equivariant injection for each such $I$
\begin{equation}\label{eq:Iintro}
(\widetilde{\pi}_I(D) \otimes_E \varepsilon^{n-1})^{\oplus m} \hooklongrightarrow \Pi_{\infty}(\tau)^{R_{\infty}(\tau)\text{-}\an}[\fa_{\pi}]
\end{equation}
and then amalgamate.\bigskip

Showing (\ref{eq:Iintro}) lies at the heart of the proof of Proposition \ref{prop:intro}. Fix $I$ which is not very critical, $i:=\vert I\vert$ and $\fR$ a refinement compatible with $I$. To $\fR$, we first associate a point $x_{\fR}$ (see (\ref{EptxR})) on a parabolic eigenvariety $\cE_{\infty}(\tau)_{i}$ (see below (\ref{eq:sections}) where it is denoted $\cE_{\infty}(\xi,\tau)_{\sigma,i}$ as in the text we fix arbitrary distinct Hodge-Tate weights at $p$-adic places not $\wp$ and do not assume $K=\Qp$). Here the parabolic subgroup is the maximal standard parabolic $P_i$ of $\GL_n$ containing all simple roots except $e_i-e_{i+1}$. The advantage of using such a {parabolic} eigenvariety is that (i) the point $x_{\fR}$ is smooth on $\cE_{\infty}(\tau)_{i}$ (Corollary \ref{C: smooth}) and (ii) the adjunction formulae we use (see (\ref{eq:adnc}), (\ref{Eadjun1})) involve much less constituents.\bigskip

Let us explain (\ref{eq:Iintro}) when $C(I, s_i)$ is a companion constituent (the case where $C(I, s_i)$ is not a companion constituent being somewhat similar to the non-critical case in \cite[\S~4.1]{Di25}). Consider the following subspace of the $E$-vector space of additive characters $\Hom(T(\Qp),E)$ (where $T$ is the diagonal torus of $\GL_n$):
\[\Hom_{0}(T(\Qp),E):=\Hom_{\sm}(T(\Qp),E) \!\!\!\bigoplus_{\Hom_{\sm}(\GL_n(\Qp),E)} \!\!\!\Hom(\GL_n(\Qp),E)\ \subset \ \Hom(T(\Qp),E)\]
where ``sm'' means locally constant. We can identify $\Hom_{0}(T(\Qp),E)$ with a subspace of $\Ext^1_{T(\Qp)}(\delta_{\fR},\delta_{\fR})\cong \Hom(T(\Qp),E)$ where $\delta_{\fR}$ is the locally algebraic character of $T(\Qp)$ associated to $x_{\fR}$ (see (\ref{EptxR})), and define $\widetilde{\delta}_{\fR,0}$ as the tautological extension of $\delta_{\fR} \otimes_E \Hom_{0}(T(\Qp),E)$ by $\delta_{\fR}$. We then we prove a $T(\Qp)\times A_D$-equivariant injection
\[\widetilde{\delta}_{\fR,0}^{\oplus m}\oplus \delta_{\fR}^{\oplus m}\hooklongrightarrow J_B\big(\Pi_{\infty}(\tau)^{R_{\infty}(\tau)\text{-}\an}[\fa_{\pi}]\big)\]
where $J_B$ is Emerton's locally analytic Jacquet functor with respect to the upper Borel of $\GL_n$. Using the adjunction formula of \cite[Thm.~4.3]{Br15} with the fact $I$ is not very critical and the description of $\widetilde{\pi}_I(D)$ in that case (see (\ref{Eisomuniv2})), we then deduce (\ref{eq:Iintro}) (see (\ref{Ecrisi}) which crucially uses Lemma \ref{lem:splitX}).\bigskip

Finally it remains to explain how we prove the injection (\ref{eq:injintro}). We do not know the multiplicity $m_I$ of the companion constituent $C(I, s_{\vert I\vert})$ inside $\widehat{S}_{\tau}(U^{\wp},E)[\pi]^{\Qp\text{-}\an}$ when $I$ is very critical, but a close examination of the (delicate) induction in the proof of \cite[Thm.~5.3.3]{BHS19} shows that $m_I\geq m$. As this is not stated in \emph{loc.~cit.}~we prove it in Appendix \ref{sec:multiplicity} (where we also prove that $C(I, s_{\vert I\vert})$ has multiplicity exactly $m$ when $C(I, s_{\vert I\vert})$ is a companion constituent but $I$ is not very critical, see Proposition \ref{prop:multiplicity}). With the last statement of Proposition \ref{prop:intro}, we obtain a $\GL_n(\Qp)$-equivariant injection
\begin{equation}\label{eq:injintrobis}
(\pi(D)^\flat \otimes_E \varepsilon^{n-1})^{\oplus m}\ \bigoplus \ \Big(\!\bigoplus_{I\text{ very critical}}\!\!C(I, s_{\vert I\vert})^{\oplus m_I}\Big)
\hooklongrightarrow \widehat{S}_{\tau}(U^{\wp},E)[\pi]^{\Qp\text{-}\an}.
\end{equation}
It formally follows from (the proof of) Proposition \ref{prop:intro} that we also have
\[\Hom_{\GL_n(\Qp)}\big(\pi_{\alg}(D)\otimes_E \varepsilon^{n-1}, \widehat{S}_{\tau}(U^{\wp},E)[\pi]^{\Qp\text{-}\an}/(\pi(D)^\flat \otimes_E \varepsilon^{n-1})^{\oplus m}\big)=0\]
which by d\'evissage implies an injection
\begin{multline}\label{eq:injintroter}
\Ext^1_{\GL_n(\Qp)}\big(\pi_{\alg}(D)\otimes_E \varepsilon^{n-1},(\pi(D)^{\flat} \otimes_E \varepsilon^{n-1})^{\oplus m}\big)\\\hooklongrightarrow \Ext^1_{\GL_n(\Qp)}\big(\pi_{\alg}(D) \otimes_E \varepsilon^{n-1}, \widehat{S}_{\tau}(U^{\wp},E)[{\pi}]^{\Qp\text{-}\an}\big).
\end{multline}
On the one hand the left hand side of (\ref{eq:injintroter}) is easily checked to have dimension $m\big(n+\frac{n(n+1)}{2}\big)$ (Corollary \ref{cor:isoflat}), on the other hand Z.~Wu proves in Appendix \ref{Wu} that the right hand side has dimension smaller or equal than $m\big(n+\frac{n(n+1)}{2}\big)$ (Theorem \ref{theoreminequalityext1}). Hence (\ref{eq:injintroter}) is an isomorphism (and, assuming $m=1$, the injection (\ref{eq:introbis}) is really an isomorphism!). Using (\ref{eq:injintrobis}) with the isomorphism (\ref{eq:injintroter}) and $\dim_E\Ext^1_{\GL_n(\Qp)}(C(I, s_{\vert I\vert}),\pi_{\alg}(D))=1$, we easily deduce (\ref{eq:injintro}) and the last statement of Theorem \ref{thm:mainintrobis} (see Theorem \ref{T: lg2}).\bigskip

Every section and subsection has a few introductory lines explaining its contents, and we have tried to provide full details in the proofs (which explains the length of this text). We end up this introduction with the main notation (some of which have already been used).\bigskip

In the whole text we fix an integer $n\geq 2$, a finite extension $K$ of $\Qp$ with maximal unramified subextension $K_0$, and a finite extension $E$ of $\Qp$ such that $\vert \Sigma\vert=[K:\Qp]$ where $\Sigma:=\{K\hookrightarrow E\}$. We let $f:=[K_0:\Qp]$ and $\cO_K\subset K$, $\cO_E\subset E$ the rings of integers of $K$, $E$ respectively. For $\alpha\in E^\times$ we let $\unr(\alpha):K^\times \rightarrow E^\times$ be the unique unramified character which sends any uniformizer of $K$ to $\alpha$ and we write $\vert \cdot\vert_K:=\unr(p^{-f})$. We let $\val:K^\times\rightarrow \mathbb{Q}\hookrightarrow E$ the $p$-adic valuation normalized by $\val(p)=1$ and $\log:\cO_K^\times\rightarrow K$ the $p$-adic logarithm. We recall that any choice of $\log(p)\in E$ allows $\sigma\circ\log:\cO_K^\times\rightarrow E$ to be extended to the whole $K^\times$ (for $\sigma\in \Sigma$). We normalize the reciprocity map of local class field theory by sending uniformizers of $K$ to (lifts of) the geometric Frobenius. We denote by $\varepsilon:\Gal(\overline \Q/\Q)\twoheadrightarrow \Gal(\Q^{\rm ab}/\Q)\rightarrow \Zp^\times\hookrightarrow E^\times$ the $p$-adic cyclotomic character and still write $\varepsilon$ for its restriction to any subgroup of $\Gal(\overline \Q/\Q)$, for instance $\Gal(\overline K/K)$. We again write $\varepsilon:K^\times\rightarrow E^\times$ for its precomposition with the reciprocity map $K^\times\rightarrow \Gal(K^{\rm ab}/K)$.\bigskip

We let $\cR_{K}$ be the Robba ring for $K$ (see for instance \cite[\S~I.2]{Be081} except that we prefer the notation $\cR_{K}$ to ${\mathbf B}_{{\rm rig},K}^\dagger$). For an artinian $E$-algebra $A$, for instance $A=E$ or $A=E[\epsilon]/\epsilon^2 =$ the dual numbers, we let $\cR_{K,A}:=\cR_{K} \otimes_{\Qp} A=\cR_{K,E} \otimes_{E} A$.\bigskip

If $G$ is a $p$-adic Lie group over $K$ (\cite[\S~13]{Sc11}) and $\sigma\in \Sigma$ we denote by $\Hom_{\sm}(G,E)\subseteq \Hom_{\sigma}(G,E)\subseteq \Hom(G,E)$ the (respectively) locally constant, locally $\sigma$-analytic and locally $\Qp$-analytic group homomorphisms from $G$ to $E$ with its \emph{additive} structure. Note that they are $E$-vector spaces. For instance $\dim_E\Hom(K^\times,E)=1+[K:\Qp]$ with a basis given by $(\val, \tau\circ \log\ \rm{for}\ \tau\in \Sigma)$ where log is extended to $K^\times$ by any choice of $\log(p)$, and $\dim_E\Hom_\sigma(K^\times,E)=2$ with a basis given by $(\val, \sigma\circ \log)$.\bigskip

We refer to \cite{ST03} for the background on the abelian categories of admissible locally $\sigma$-analytic representations and admissible locally $\Qp$-analytic representations of locally $K$-analytic groups. We denote by $\Ext^1_{\GL_n(K),\sigma}$, $\Ext^1_{\GL_n(K)}$ the respective (Yoneda) extension groups in these categories, and by $\Ext^1_{\GL_n(K),\sigma,Z}$, $\Ext^1_{\GL_n(K),Z}$ the subgroups with a central character. When the representations are locally $\Qp$-algebraic, we write $\Ext^1_{\alg}$ for the group of locally $\Qp$-algebraic extensions. For smooth, locally $\sigma$-analytic and locally $\Qp$-analytic parabolic inductions, and their properties, we use without comment the work of Orlik-Strauch \cite{OS15}. If $R_1$ and $R_2$ are two representations of a topological group which are (topologically) of finite length, we denote by $\!\begin{xy} (0,0)*+{R_1}="a"; (13,0)*+{R_2}="b"; {\ar@{-}"a";"b"}\end{xy}\!$ an arbitrary (possibly split) extension of $R_2$ by $R_1$.\bigskip

We let $T$ be the diagonal torus of $\GL_n$, $B$ the Borel of upper triangular matrices, $N$ its unipotent radical and $B^-$ the opposite Borel. For $P$ a standard parabolic subgroup of $\GL_n$ we let $N_P$ be its unipotent radical, $P^-$ the opposite parabolic and $L_P$ the Levi subgroup. For $i\in \{1,\dots,n-1\}$ we let $P_i\subset {\GL_n}$ be the maximal standard parabolic subgroup associated to all the simple roots of $\GL_n$ except $e_i-e_{i+1}$. For a connected reductive algebraic group $H$ over $K$, we let $H_{\Sigma}:=(\Res_{K/\Qp} H)\times_{\Spec \Qp}\Spec E$, $H_{\sigma}:=H \times_{\Spec K, \sigma} \Spec E$ for $\sigma\in \Sigma$, and recall that $H_{\Sigma}\cong \prod_{\sigma\in \Sigma} H_\sigma$. For a lie algebra $\fl$ over $K$ we let $\fl_{\Sigma}:=\fl \otimes_{\Qp} E\buildrel\sim\over\rightarrow \bigoplus_{\sigma\in \Sigma} (\fl \otimes_{K,\sigma} E)\text{ \ and \ }\fl_{\sigma}:=\fl \otimes_{K,\sigma} E$. We let $\fg:=\gl_n(K)$, $\mathfrak{b}$ the subalgebra of upper triangular matrices, $\fn$ its nilpotent radical and $\ft$ the diagonal matrices. We let $\mathfrak{r}_{P_i}$ (resp.~$\mathfrak{n}_{P_i}$) be the full (resp.~{nilpotent}) radical of the Lie algebra of ${P_i}$ over $K$, $\fl_{P_i}:=$ the Lie algebra of $L_{P_i}$ over $K$ and $\fz_{P_i}$ the center of $\fl_{P_i}$. We denote by $R:=\{s_{1},\dots,s_{n-1}\}$ the set of simple reflections of ${\GL_n}$, $\leq$ the Bruhat order on $S_n$ relative to $R$ and $\lg$ the length on $S_n$ relative to $R$.\bigskip

If $V$ is a topological $E$-vector space we denote by $V^\vee$ its continuous dual. If $V$ has no specified topology, we tacitly endow it with the discrete topology and in that case $V^\vee$ is its linear dual. For a left $T(K)$-module $V$, we let $t\in T(K)$ act on $V^\vee$ by $(tf)(-):=f(t(-))$ where $f\in \Hom_E(V,E)$ (in particular if $\dim_EV=1$ we have $V^\vee\cong V$ as $T(K)$-modules).\bigskip

\noindent
\textbf{Acknowledgements}: We thank Roman Bezrukavnikov, Valentin Hernandez, Benjamin Schraen and Zhixiang Wu for discussions, and Stefano Morra for his comments. We also thank R.~Bezrukavnikov for his note \cite{Be25}, V.~Hernandez and B.~Schraen for their computational work on the Bezrukavnikov functor, and Z.~Wu for providing Appendix \ref{Wu}. The second author is partially supported by the National Natural Science Foundation of China under agreement No.~NSFC-12231001 and No.~NSFC-12321001, and by the New Cornerstone Foundation.

\section{The locally analytic representations \texorpdfstring{$\pi(D)$}{piD} and \texorpdfstring{$\pi(D)^\flat$}{piDflat}}\label{sec:local}

To any filtered $\varphi$-module $D$ with distinct Hodge-Tate weights for each $\sigma\in \Sigma$ and with a mild assumption on the eigenvalues of $\varphi$ we associate a locally $\Qp$-analytic representation $\pi(D)$ and a subrepresentation $\pi(D)^\flat\subseteq \pi(D)$, and we show that both determine and only depend on the collection of \emph{filtered} $\varphi^f$-modules $D\otimes_{K_0\otimes E, \sigma\vert_{K_0}\otimes \id}E$ for $\sigma\in \Sigma$. This section is purely local.

\subsection{Preliminary material}\label{sec:prel}

We give important definitions and results that will be used in the next sections.\bigskip

We fix a regular filtered $\varphi$-module $(D, \varphi, \Fil^\bullet(D_K))$ with $D$ free of rank $n$ over $K_0\otimes_{\Qp}E$ and $D_K:=K\otimes_{K_0}D$. We write $D=\prod_{\sigma\in \Sigma}D_\sigma$ where
\[D_\sigma := D_K \otimes_{K\otimes_{\Qp}E,\sigma \otimes \id}E.\]
We obviously have for $\sigma\in \Sigma$
\begin{equation}\label{eq:K0}
D\otimes_{K_0\otimes E, \sigma\vert_{K_0}\otimes \id}E\buildrel\sim\over\longrightarrow D_\sigma
\end{equation}
so that we can endow $D_\sigma$ with the $E$-linear automorphism $\varphi^f$ (which still acts on the left hand side, contrary to $\varphi$). The $\varphi^f$-module $D_\sigma$ does not depend on $\sigma$ up to isomorphism and we denote by $\{\varphi_j\in E,\ 0\leq j \leq n-1\}$ its eigenvalues (for an arbitrary but fixed numbering). We assume that they satisfy
\begin{equation}\label{eq:phi}
\varphi_j\varphi_k^{-1}\notin \{1, p^f\}\ \ \forall\ j\ne k.
\end{equation}
The (decreasing exhaustive) filtration $(\Fil^h(D_K))_{h\in \Z}$ on $D_K$ can also be written 
\[\Fil^h(D_K) = \prod_{\sigma\in \Sigma} \Fil^h(D_\sigma)\]
where $(\Fil^h(D_\sigma))_{h\in \Z}$ is a (decreasing exhaustive) filtration on $D_\sigma$. We recall that regular above means that, for each $\sigma\in \Sigma$, $(\Fil^h(D_\sigma))_{h\in \Z}$ is a full flag on the $n$-dimensional $E$-vector space $D_\sigma$. We denote by $h_{0,\sigma}>h_{1,\sigma}>\cdots>h_{n-1,\sigma}$ the integers in $\Z$ such that
\[\Fil^{-h_{j,\sigma}+1}(D_\sigma)\subsetneq \Fil^{-h_{j,\sigma}}(D_\sigma) \ \forall\ 0\leq j\leq n-1,\]
so we have
\begin{equation}\label{eq:fil}
0=\Fil^{-h_{n-1,\sigma}+1}(D_\sigma)\subsetneq \Fil^{-h_{n-1,\sigma}}(D_\sigma)\subsetneq \cdots \subsetneq \Fil^{-h_{1,\sigma}}(D_\sigma) \subsetneq \Fil^{-h_{0,\sigma}}(D_\sigma)=D_\sigma
\end{equation}
and
\begin{equation}\label{eq:dim}
\dim_E\Fil^{-h_{j,\sigma}}(D_\sigma)=n-j.
\end{equation}
The minus sign comes from the fact that, when $D=D_{\rm cris}(\rho):=({\rm B}_{\rm cris}\otimes_{\Qp}\rho)^{\Gal(\overline K/K)}$ for $\rho$ a crystalline representation of $\Gal(\overline K/K)$ over $E$, the integers $h_{j,\sigma}$ are the Hodge-Tate weights of $\rho$ ``in the $\sigma$-direction''. Hence for each $\sigma\in \Sigma$ we have a filtered $\varphi^f$-module $(D_\sigma, \varphi^f, \Fil^\bullet(D_\sigma))$. Let us recall the following elementary lemma:

\begin{lem}\label{lem:sad}
The isomorphism class of the filtered $\varphi$-module $(D, \varphi, \Fil^\bullet(D_K))$ is determined by the isomorphism classes of all the filtered $\varphi^f$-modules $(D_\sigma, \varphi^f, \Fil^\bullet(D_\sigma))$ for $\sigma\in \Sigma$ if and only if $K=\Qp$.
\end{lem}
\begin{proof}
When $K\ne \Qp$, the $E$-vector space $\Hom_\varphi(D,D)$ (:= endomorphisms which commute with $\varphi$) is strictly smaller than the $E$-vector space $\prod_{\sigma\in \Sigma} \Hom_\varphi(D_\sigma,D_\sigma)$, hence scaling in $D$ is strictly more restrictive than scaling in each $D_\sigma$. This easily implies the lemma (we leave the details to the reader).
\end{proof}\bigskip

For $\sigma\in \Sigma$ we define an action of the Weil group ${\rm Weil}(\overline K/K)$ on $D_\sigma$ by making $w\in {\rm Weil}(\overline K/K)$ act by $\varphi^{-\alpha(w)}$ where $\alpha(w)\in f\Z$ is the unique integer such that the image of $w$ in $\Gal(\Fpbar/\Fp)$ is the $\alpha(w)$-th power of the absolute Frobenius $x\mapsto x^p$. The resulting Weil representation does not depend on $\sigma\in \Sigma$ and we let $\pi_p$ be the corresponding smooth representation of $\GL_n(K)$ over $E$ by the local Langlands correspondence normalized as in \cite[\S~4]{BS07}. Concretely it is the following smooth unramified principal series:
\begin{equation}\label{eq:pip}
\pi_p\cong \bigg(\Ind_{B^-(K)}^{\GL_n(K)}\big(\unr(\varphi_{0})\vert \cdot\vert_K^{1-n}\boxtimes \unr(\varphi_{1})\vert \cdot\vert_K^{2-n}\boxtimes \cdots \boxtimes \unr(\varphi_{n-1})\big)\bigg)^{\sm}
\end{equation}
and we recall that, thanks to (\ref{eq:phi}), the representation (\ref{eq:pip}) is irreducible and does not depend up to canonical isomorphism on the ordering of the eigenvalues of $\varphi^f$ (see e.g.~\cite[\S~VII.3.4]{Re10}).\bigskip

For $\sigma\in \Sigma$ and $j\in \{0, \dots, n-1\}$ we let $\lambda_{j,\sigma}:=h_{j,\sigma}-(n-1-j)$ (note that $\lambda_{0,\sigma}\geq \lambda_{1,\sigma}\geq\cdots\geq \lambda_{n-1,\sigma}$) and we write $\lambda_\sigma:T(K)\rightarrow E^\times$ for the character:
\begin{equation}\label{eq:t}
\begin{pmatrix}t_0 && \\ & \ddots & \\ && t_{n-1}\end{pmatrix}\in T(K) \longmapsto \prod_{j=0}^{n-1}\sigma(t_j)^{\lambda_{j,\sigma}}.
\end{equation}
We denote by $L(\lambda_\sigma)$ the irreducible $\sigma$-algebraic finite dimensional representation of $\GL_n(K)$ over $E$ of highest weight $\lambda_\sigma$ with respect to the upper Borel $B(K)$. Here $\sigma$-algebraic means that $K$ is seen in $E$ via the embedding $\sigma$. We then define the (irreducible) locally $\sigma$-algebraic representation of $\GL_n(K)$ over $E$:
\begin{equation}\label{eq:algsigma}
\pi_{\alg}(D_\sigma):=\pi_p \otimes_E L(\lambda_\sigma)
\end{equation}
and (for later use) the (irreducible) locally $\Qp$-algebraic representation of $\GL_n(K)$ over $E$:
\begin{equation}\label{eq:alg}
\pi_{\alg}(D):=\pi_p \otimes_E (\otimes_{\sigma}L(\lambda_\sigma)).
\end{equation}

For $\sigma\in \Sigma$ and $i\in \{1,\dots,n-1\}$ we write $s_{i,\sigma}\!\cdot \!\lambda_\sigma :T(K)\rightarrow E^\times$ for the character:
\begin{equation}\label{eq:dotaction}
\begin{pmatrix}t_0 && \\ & \ddots & \\ && t_{n-1}\end{pmatrix}\in T(K) \buildrel {s_{i,\sigma}\cdot\lambda_\sigma} \over \longmapsto \bigg(\prod_{j\ne i-1,i}\sigma(t_j)^{\lambda_{j,\sigma}}\bigg)\sigma(t_{i-1})^{\lambda_{i,\sigma}-1}\sigma(t_{i})^{\lambda_{i-1,\sigma}+1}.
\end{equation}
Here $\{s_{1,\sigma},\dots,s_{n-1,\sigma}\}$ is the set of simple reflections of ${\GL_n}\times_{K,\sigma}E$ and $s_{i,\sigma}\!\cdot \!\lambda_\sigma$ is the dot action on the weight $\lambda_\sigma$ with respect to $B\times_{K,\sigma}E$. For $i\in \{1,\dots,n-1\}$. Recall that a refinement is an ordering $(\varphi_{j_1},\dots,\varphi_{j_n})$ of the set of eigenvalues $\{\varphi_j,\ 0\leq j \leq n-1\}$. For a fixed refinement $(\varphi_{j_1},\dots,\varphi_{j_n})$ we consider the following locally $\Qp$-analytic representation of $\GL_n(K)$ over $E$:
\begin{equation}\label{eq:CI}
\soc_{\GL_n(K)}\Big(\Ind_{B^-(K)}^{\GL_n(K)}\big(\unr(\varphi_{j_1})\vert \cdot\vert_K^{1-n}\boxtimes \unr(\varphi_{j_2})\vert \cdot\vert_K^{2-n}\boxtimes \cdots \boxtimes \unr(\varphi_{j_n})\big)s_{i,\sigma}\!\cdot \!\lambda_{\sigma}\Big)^{\Qp\text{-}\an}.
\end{equation}

\begin{prop}
Let $\sigma\in \Sigma$.
\begin{enumerate}[label=(\roman*)]
\item
For $i\in \{1,\dots,n-1\}$ the representation (\ref{eq:CI}) is irreducible admissible and up to isomorphism only depends on the set $\{\varphi_{j_1},\dots,\varphi_{j_i}\}$ (or equivalently on the set $\{\varphi_{j_{i+1}},\dots,\varphi_{j_n}\}$) and not on the full refinement $(\varphi_{j_1},\dots,\varphi_{j_n})$.
\item
For $i,i'\in \{1,\dots,n-1\}$, two different sets $\{\varphi_{j_1},\dots,\varphi_{j_i}\}$, $\{\varphi_{j'_1},\dots,\varphi_{j'_{i'}}\}$ give two non-isomorphic representations in (\ref{eq:CI}).
\end{enumerate}
\end{prop}
\begin{proof}
(i) is a special case of \cite[Lemma 5.5(i)]{BH20} (with admissibility following from) while (ii) is a special case of \cite[Lemma 5.5(ii)]{BH20}.
\end{proof}

For $\sigma\in \Sigma$, $i\in \{1,\dots,n-1\}$ and $I\subset \{\varphi_j,\ 0\leq j \leq n-1\}$ of cardinality $i$ we denote by
\[C(I, s_{i,\sigma})\]
the irreducible locally $\Qp$-analytic representation in (\ref{eq:CI}). Since (\ref{eq:CI}) is irreducible, $C(I, s_{i,\sigma})$ is also the socle of $(\Ind_{B^-(K)}^{\GL_n(K)}(\unr(\varphi_{j_1})\vert \cdot\vert_K^{1-n}\boxtimes \cdots \boxtimes \unr(\varphi_{j_n}))s_{i,\sigma}\!\cdot \!\lambda_{\sigma})^{\sigma\text{-}\an}$. In particular $C(I, s_{i,\sigma})$ is locally $\sigma$-analytic. Note that $i$ is determined by $I$ (as $i=\vert I\vert$), but it is convenient to keep the simple reflection $s_{i,\sigma}$ in the notation. An obvious count gives that (for $\sigma$ fixed) there are $2^n-2$ distinct representations $C(I,s_{i,\sigma})$.

\begin{definit}\label{def:compatible}
Let $I\subset \{\varphi_j,\ 0\leq j \leq n-1\}$ of cardinality $i\in \{1,\dots,n-1\}$. We say that a refinement $(\varphi_{j_1},\dots,\varphi_{j_n})$ is \emph{compatible} with $I$ if $I=\{\varphi_{j_1},\dots,\varphi_{j_i}\}$ (as a set).
\end{definit}

\begin{rem}\label{rem:intertwining}
\hspace{2em}
\begin{enumerate}[label=(\roman*)]
\item
Let $I\subset \{\varphi_j,\ 0\leq j \leq n-1\}$ of cardinality $i\in \{1,\dots,n-1\}$. For two different refinements $(\varphi_{j_1},\dots,\varphi_{j_n})$, $(\varphi_{j'_1},\dots,\varphi_{j'_n})$ compatible with $I$ there is a \emph{canonical} isomorphism between the corresponding two representations (\ref{eq:CI}) using \cite[Prop.~4.9(b)]{OS15} combined with the canonical intertwining operators between smooth principal series, see \cite[\S~VII.3.4]{Re10}. Because this isomorphism is canonical, we need not worry about the choice of compatible refinements in this work.
\item
It follows from \cite[Prop.~4.9(b)]{OS15} that the representation
\[\soc_{\GL_n(K)}\Big(\Ind_{B^-(K)}^{\GL_n(K)}\big(\unr(\varphi_{j_1})\vert \cdot\vert_K^{1-n}\boxtimes \unr(\varphi_{j_2})\vert \cdot\vert_K^{2-n}\boxtimes \cdots \boxtimes \unr(\varphi_{j_n})\big)\lambda_{\sigma}\Big)^{\Qp\text{-}\an}\]
is \ the \ locally \ $\sigma$-algebraic \ representation \ $\pi_{\alg}(\!D_\sigma\!)$ \ in \ (\ref{eq:algsigma}) \ for \ any \ refinement $(\varphi_{j_1},\dots,\varphi_{j_n})$.
\end{enumerate}
\end{rem}

\begin{lem}\label{lem:nonsplit}
Let $\sigma\in \Sigma$, $i\in \{1,\dots,n-1\}$ and $I\subset \{\varphi_j,\ 0\leq j \leq n-1\}$ of cardinality $i$. We have equalities
\[\dim_E\Ext^1_{\GL_n(K)}(\pi_{\alg}(D_\sigma), C(I, s_{i,\sigma}))=\dim_E\Ext^1_{\GL_n(K)}(C(I, s_{i,\sigma}), \pi_{\alg}(D_\sigma))=1\]
and (canonical) isomorphisms
\begin{eqnarray*}
\Ext^1_{\GL_n(K),\sigma}(\pi_{\alg}(D_\sigma), C(I, s_{i,\sigma}))&\buildrel\sim\over\longrightarrow &\Ext^1_{\GL_n(K)}(\pi_{\alg}(D_\sigma), C(I, s_{i,\sigma}))\\
\Ext^1_{\GL_n(K),\sigma}(C(I, s_{i,\sigma}), \pi_{\alg}(D_\sigma))&\buildrel\sim\over\longrightarrow &\Ext^1_{\GL_n(K)}(C(I, s_{i,\sigma}), \pi_{\alg}(D_\sigma)).
\end{eqnarray*}
\end{lem}
\begin{proof}
The first statement can be deduced from \cite[Prop.~5.1.14]{BQ24} with \cite[Lemma 3.2.4(ii)]{BQ24}. For the second it is enough by the first to prove in each case the existence of a non-split locally $\sigma$-analytic extension. Let $L^-(-\lambda_\sigma)$, $L^-(-s_{i,\sigma}\!\cdot \!\lambda_{\sigma})$ the simple modules of highest weights $-\lambda_\sigma$, $-s_{i,\sigma}\!\cdot \!\lambda_{\sigma}$ (respectively) in the category ${\cO}^{{\mathfrak b}_\sigma^-}_{\alg}$ of \cite[\S~2.5]{OS15}. Then this follows for instance from \cite[Cor.~3.2.11]{Or20} applied with $M$ the unique non-split extension of $L^-(-s_{i,\sigma}\!\cdot \!\lambda_{\sigma})$ by $L^-(-\lambda_\sigma)$ (resp.~of $L^-(-\lambda_\sigma)$ by $L^-(-s_{i,\sigma}\!\cdot \!\lambda_{\sigma})$), see \cite[Lemma 3.2.4(ii)]{BQ24}.
\end{proof}

\begin{lem}\label{lem:isoextalg}
Let $\sigma\in \Sigma$. To each refinement $(\varphi_{j_1},\dots,\varphi_{j_n})$ one can associate isomorphisms of finite-dimensional $E$-vector spaces
\begin{equation}\label{eq:amalg}
\begin{gathered}
\begin{array}{lll}
\Hom_{\sm}(T(K),E)\bigoplus_{\Hom_{\sm}(K^\times,E)}\Hom(K^\times,E)\!&\!\buildrel\sim\over\longrightarrow \!&\!\Ext^1_{\GL_n(K)}(\pi_{\alg}(D_\sigma),\pi_{\alg}(D_\sigma))\\
\Hom_{\sm}(T(K),E)\bigoplus_{\Hom_{\sm}(K^\times,E)}\Hom_\sigma(K^\times,E)\!&\!\buildrel\sim\over\longrightarrow \!&\!\Ext^1_{\GL_n(K),\sigma}(\pi_{\alg}(D_\sigma),\pi_{\alg}(D_\sigma))
\end{array}
\end{gathered}
\end{equation}
which induce an isomorphism $\Hom_{\sm}(T(K),E)\buildrel\sim\over\longrightarrow \Ext^1_{\alg}(\pi_{\alg}(D_\sigma),\pi_{\alg}(D_\sigma))$.
Moreover in the first case of (\ref{eq:amalg}) the dimension is $n+[K:\Qp]$ and in the second $n+1$.
\end{lem}
\begin{proof}
The first isomorphism in (\ref{eq:amalg}) follows from \cite[Prop.~3.3(1)]{Di25} (where a refinement there is a permutation $w\in S_n$). The second isomorphism is then easily deduced from it. Let us at least define the maps. Recall that the $E$-vector space $\Hom(K^\times\!,E)$ (resp.~$\Hom_\sigma(K^\times\!,E)$) has dimension $1+[K:\Qp]$ (resp.~$2$), see \S~\ref{sec:intro}. We define a canonical injection
\begin{equation}\label{eq:extpialg}
\Hom(K^\times,E)\hookrightarrow \Ext^1_{\GL_n(K)}(\pi_{\alg}(D_\sigma),\pi_{\alg}(D_\sigma)),\ \psi\mapsto \pi_{\alg}(D_\sigma)\!\otimes_E (1+(\psi\circ{\det})\epsilon)
\end{equation}
where $1+(\psi\circ{\det})\epsilon$ is the character on the dual numbers $E[\epsilon]/(\epsilon^2)$
\[\GL_n(K)\buildrel{\det}\over\twoheadrightarrow K^\times\rightarrow (E[\epsilon]/(\epsilon^2))^\times\hookrightarrow E[\epsilon]/(\epsilon^2),\ g\mapsto 1+\psi({\det}(g))\epsilon.\]
Note that the image of $\Hom_\sigma(K^\times,E)$ via (\ref{eq:extpialg}) clearly falls in $\Ext^1_{\GL_n(K),\sigma}(\pi_{\alg}(D_\sigma),\pi_{\alg}(D_\sigma))$. Recall also that the $E$-vector space $\Hom_{\sm}(T(K),E)$ has dimension $n$. Let $\phi:T(K)\rightarrow E^\times$ be the character
\begin{equation}\label{eq:refinement}
\phi:=\unr(\varphi_{j_1})\vert \cdot\vert_K^{1-n}\boxtimes \unr(\varphi_{j_2})\vert \cdot\vert_K^{2-n}\boxtimes \cdots \boxtimes \unr(\varphi_{j_n}),
\end{equation}
we define another canonical injection
\begin{eqnarray}\label{eqn:extpism}
\nonumber \Hom_{\sm}(T(K),E)&\hookrightarrow &\Ext^1_{\GL_n(K),\sigma}(\pi_{\alg}(D_\sigma),\pi_{\alg}(D_\sigma))\\
\psi&\longmapsto& \big(\Ind_{B^-(K)}^{\GL_n(K)}(\phi\otimes_E(1+\psi\epsilon))\big)^{\sm}\otimes_E L(\lambda_\sigma)
\end{eqnarray}
where $1+\psi\epsilon$ is seen as an $(E[\epsilon]/(\epsilon^2))^\times$-valued character of $B^-(K)$ via $B^-(K)\twoheadrightarrow T(K)$. The two injections (\ref{eq:extpialg}) and (\ref{eqn:extpism}) are easily checked to coincide on $\Hom_{\sm}(K^\times,E)$, which gives locally $\Qp$-algebraic, or equivalently (here) locally $\sigma$-algebraic, extensions.
\end{proof}

\begin{rem}
One can also easily derive from Lemma \ref{lem:isoextalg} and from \cite[(3.6)]{Di25} an isomorphism of $(n-1)$-dimensional $E$-vector spaces which depends on the choice of a refinement:
\[\Hom_{\sm}(T(K),E)_0\buildrel\sim\over\rightarrow \Ext^1_{\GL_n(K),\sigma,Z}(\pi_{\alg}(D_\sigma), \pi_{\alg}(D_\sigma))
\buildrel\sim\over\rightarrow \Ext^1_{\GL_n(K),Z}(\pi_{\alg}(D_\sigma), \pi_{\alg}(D_\sigma))\]
where $\Hom_{\sm}(K^\times,E)_0:=\Ker(\Hom_{\sm}(T(K),E)\rightarrow \Hom_{\sm}(K^\times,E))$ via $K^\times\hookrightarrow T(K)$, $a\mapsto \diag(a)$.
\end{rem}\bigskip

See \S~\ref{sec:prel} for the definition of the maximal parabolic subgroup $P_i\subset B$, $i\in \{1,\dots,n-1\}$.

\begin{prop}\label{prop:isoext}
Let $\sigma\in \Sigma$, $i\in \{1,\dots,n-1\}$, $I\subset \{\varphi_j,\ 0\leq j \leq n-1\}$ of cardinality $i$ and $(\varphi_{j_1},\dots,\varphi_{j_n})$ a refinement compatible with $I$. Let $V_I$ be a locally $\sigma$-analytic representation of $\GL_n(K)$ over $E$ which is isomorphic to a non-split extension of $C(I, s_{i,\sigma})$ by $\pi_{\alg}(D_\sigma)$ (see Lemma \ref{lem:nonsplit}) and fix an injection $\iota:\pi_{\alg}(D_\sigma)\hookrightarrow V_I$. Then associated to $(V_I,\iota)$ and the above refinement there is a canonical isomorphism of $(n+2)$-dimensional $E$-vector spaces
\begin{equation}\label{eq:amalg2}
\Hom_{\sm}(T(K),E)\bigoplus_{\Hom_{\sm}(L_{P_i}(K),E)}\Hom_\sigma(L_{P_i}(K),E)\buildrel\sim\over\longrightarrow \Ext^1_{\GL_n(K),\sigma}(\pi_{\alg}(D_\sigma),V_I)
\end{equation}
which extends the second isomorphism in (\ref{eq:amalg}) (for the fixed refinement) via $\Hom_\sigma(K^\times,E)\buildrel{\det}\over\hookrightarrow \Hom_\sigma(L_{P_i}(K),E)$ and $\Ext^1_{\GL_n(K),\sigma}(\pi_{\alg}(D_\sigma),\pi_{\alg}(D_\sigma))\buildrel\iota\over \hookrightarrow \Ext^1_{\GL_n(K),\sigma}(\pi_{\alg}(D_\sigma),V_I)$. Moreover the restriction of (\ref{eq:amalg2}) to $\Hom_\sigma(L_{P_i}(K),E)$ only depends on $(V_I,\iota)$.
\end{prop}
\begin{proof}
To simplify the notation we write $\pi_{\alg}$ instead of $\pi_{\alg}(D_\sigma)$. We first define the injection $\Hom_{\sm}(T(K),E)\hookrightarrow \Ext^1_{\GL_n(K),\sigma}(\pi_{\alg},V_I)$ as the composition
\begin{equation*}
\Hom_{\sm}(T(K),E)\buildrel{(\ref{eqn:extpism})}\over\hookrightarrow \Ext^1_{\GL_n(K),\sigma}(\pi_{\alg},\pi_{\alg})\buildrel\iota\over\hookrightarrow \Ext^1_{\GL_n(K),\sigma}(\pi_{\alg},V_I)
\end{equation*}
noting that the second map is indeed injective since $\Hom_{\GL_n(K)}(\pi_{\alg},V_I/\pi_{\alg})=0$. For the rest of the proof, we proceed in three steps.\bigskip

\textbf{Step $1$}: We define locally $\sigma$-analytic representations $R_I$ and $\pi_{i,{\alg}}$.\\
Let $\phi$ as in (\ref{eq:refinement}) and define the (admissible) locally $\sigma$-analytic parabolic induction
\begin{equation}\label{eq:RI}
R_I:=\bigg(\Ind_{P_i^-(K)}^{\GL_n(K)}\Big(\big(\Ind_{L_{P_i}(K)\cap B^-(K)}^{L_{P_i}(K)}\phi\big)^{\sm}\!\otimes_E L_i(\lambda_\sigma)\Big)\bigg)^{\sigma\text{-}\an}
\end{equation}
where $L_i(\lambda_\sigma)$ is the irreducible $\sigma$-algebraic finite dimensional representation of $L_{P_i}(K)$ over $E$ of highest weight $\lambda_\sigma$ with respect to $L_{P_i}(K)\cap B(K)$. Define
\[\pi_{i,{\alg}}:=\big(\Ind_{L_{P_i}(K)\cap B^-(K)}^{L_{P_i}(K)}\phi\big)^{\sm}\!\otimes_E L_i(\lambda_\sigma),\]
(which is locally $\sigma$-algebraic) then as in (ii) of Remark \ref{rem:intertwining}
\[\pi_{i,{\alg}}=\soc_{L_{P_i}(K)}\big(\Ind_{L_{P_i}(K)\cap B^-(K)}^{L_{P_i}(K)}\phi\lambda_\sigma\big)^{\sigma\text{-}\an}\]
and by \emph{loc.~cit.}~and (\ref{eq:RI}) we have $\pi_{\alg}\cong \soc_{\GL_n(K)}R_I$. With the notation of \cite{OS15} we have
\begin{equation}\label{eq:RIOS}
R_I\cong {\mathcal F}_{B^-}^{\GL_n}\big(U({\mathfrak g}_\sigma)\otimes_{U({\mathfrak p}_{i,\sigma}^-)}L_i^-(-\lambda_\sigma),\phi\big)\end{equation}
where $L_i^-(-\lambda_\sigma)$ is the simple module of highest weight $-\lambda_\sigma$ in the category ${\cO}^{{\mathfrak p}_{i,\sigma}^-}_{\alg}$ of \emph{loc.~cit.}\ Note that by considerations analogous to (i) of Remark \ref{rem:intertwining}) changing the refinement by another refinement \emph{compatible with $I$} modifies $\pi_{i,{\alg}}$ and $R_I$ by representations which are canonically isomorphic to them.\bigskip

\textbf{Step $2$}: We define a canonical injection $\Hom_{\sigma}(L_{P_i}(K),E)\hookrightarrow \Ext^1_{\GL_n(K),\sigma}(\pi_{\alg}, V_I)$ associated to $(V_I,\iota)$.\\
Fix an injection $\iota_1:\pi_{\alg}\hookrightarrow R_I$ (unique up to scalar in $E^\times$) with $R_I$ as in (\ref{eq:RI}). Since the unique extension of $L^-(-\lambda_\sigma)$ by $L^-(-s_{i,\sigma}\!\cdot \!\lambda_{\sigma})$ occurs as a quotient of $U({\mathfrak g}_\sigma)\otimes_{U({\mathfrak p}_{i,\sigma}^-)}L_i^-(-\lambda_\sigma)$ (see the proof of Lemma \ref{lem:nonsplit} and \cite[Thm.\ 9.4(b)(c)]{Hu08}), it easily follows from (\ref{eq:RIOS}) with the dimension $1$ assertion of Lemma \ref{lem:nonsplit} that there is a unique injection $\iota_2:V_I\hookrightarrow R_I$ such that $\iota_2\circ \iota=\iota_1$. Moreover, we have a canonical injection obtained as the composition of the following two injections
\begin{equation}\label{eq:comp0}
\begin{array}{cclcl}
\Hom_{\sigma}(L_{P_i}(K),E)\!\!&\!\!\hookrightarrow \!\!&\!\!\Ext^1_{P_i^-(K)}(\pi_{i,{\alg}}, \pi_{i,{\alg}}) \!\!&\!\!\hookrightarrow \!\!& \!\!\Ext^1_{\GL_n(K),\sigma}(R_I, R_I)\\
\psi \!\!&\!\!\longmapsto\!\!&\pi_{i,{\alg}}\!\otimes_E(1+\psi\epsilon)\!\!&\!\!\longmapsto \!\!&\!\!\big(\Ind_{P_i^-(K)}^{\GL_n(K)}\pi_{i,{\alg}}\!\otimes_E(1+\psi\epsilon)\big)^{\sigma\text{-}\an}
\end{array}
\end{equation}
where $1+\psi\epsilon$ is the character $P_i^-(K)\twoheadrightarrow L_{P_i}(K)\longrightarrow (E[\epsilon]/(\epsilon^2))^\times\hookrightarrow E[\epsilon]/(\epsilon^2)$ on the dual numbers analogous to (\ref{eq:extpialg}). The second map in (\ref{eq:comp0}) is induced by the parabolic induction and is injective by (the proof of) \cite[Lemma 0.3]{Em07} applied with $P=P_i$. The composition
\begin{equation}\label{eq:comp}
\Hom_{\sigma}(L_{P_i}(K),E) \buildrel (\ref{eq:comp0})\over \hookrightarrow \Ext^1_{\GL_n(K),\sigma}(R_I, R_I)\longrightarrow \Ext^1_{\GL_n(K),\sigma}(\pi_{\alg}, R_I)
\end{equation}
where the second map is induced by the pull-back $\iota_1:\pi_{\alg}\hookrightarrow R_I$ is still injective using \cite[0.13]{Em06} combined with \cite[Lemma 0.3]{Em07}. A d\'evissage using \cite[Lemma 2.26(2)]{Di191} (together with \cite[Cor.~5.2]{Hu08}) shows that the push-forward $\iota_2:V_I\hookrightarrow R_I$ induces an isomorphism $\Ext^1_{\GL_n(K),\sigma}(\pi_{\alg}, V_I)\buildrel\sim\over\longrightarrow \Ext^1_{\GL_n(K),\sigma}(\pi_{\alg}, R_I)$, hence precomposing its inverse with (\ref{eq:comp}) gives an injective map
\begin{equation}\label{eq:comp2}
\Hom_{\sigma}(L_{P_i}(K),E) \hookrightarrow \Ext^1_{\GL_n(K),\sigma}(\pi_{\alg}, V_I).
\end{equation}
Replacing $\iota_1$ by $\lambda\iota_1$ for $\lambda\in E^\times$ and changing $\iota_2$ accordingly, it is an easy exercise left to the reader to check that the map (\ref{eq:comp2}) only depends on the injection $\iota$ and the fixed refinement. But in fact modifying the latter by another refinement which is compatible with $I$ and using the compatibility of intertwinings operators with (smooth) parabolic induction (\cite[Prop.~VII.3.5(ii)]{Re10}), we see that (\ref{eq:comp2}) does not depend on the fixed refinement (compatible with $I$). Finally, when restricted to $\Hom_\sigma(K^\times,E)$ via the determinant $L_{P_i}(K)\twoheadrightarrow K^\times$, one also checks that (\ref{eq:comp2}) lands in $\Ext^1_{\GL_n(K),\sigma}(\pi_{\alg},\pi_{\alg})$ (it is given by $\psi\mapsto \pi_{\alg}\otimes_E(1+\psi\epsilon)$ where $1+\psi\epsilon$ is seen as an $(E[\epsilon]/(\epsilon^2))^\times$-valued character of $\GL_n(K)$ via the determinant) and is compatible with (\ref{eq:extpialg}).\bigskip

\textbf{Step $3$}: We prove the statement of the lemma.\\
When $\psi\in \Hom_{\sm}(L_{P_i}(K),E)$, by an argument similar to the one in the proof of \cite[Prop.~3.3(1)]{Di25} the injection (\ref{eq:comp0}) factors through
\begin{multline*}
\big(\Ind_{P_i^-(K)}^{\GL_n(K)}\big((\Ind_{L_{P_i}(K)\cap B^-(K)}^{L_{P_i}(K)}\phi)^{\sm}\otimes_E(1+\psi\epsilon)\big)\big)^{\sm}\otimes_E L(\lambda_\sigma)\\
\cong \big(\Ind_{B^-(K)}^{\GL_n(K)}(\phi\otimes_E(1+\psi\epsilon))\big)^{\sm}\otimes_E L(\lambda_\sigma)
\end{multline*}
where $1+\psi\epsilon$ on the right hand side is seen as a character of $B^-(K)$ via $B^-(K)\hookrightarrow P_i^-(K)\twoheadrightarrow L_{P_i}(K)$. It follows that the injection (\ref{eq:comp2}) induces an injection
\[\Hom_{\sm}(L_{P_i}(K),E) \hookrightarrow \Ext^1_{\GL_n(K),\sigma}(\pi_{\alg}, \pi_{\alg})\]
(via $\Ext^1_{\GL_n(K),\sigma}(\pi_{\alg},\pi_{\alg})\buildrel\iota\over\hookrightarrow \Ext^1_{\GL_n(K),\sigma}(\pi_{\alg},V_I)$) which is compatible with the injection $\Hom_{\sm}(T(K),E)\hookrightarrow \Ext^1_{\GL_n(K),\sigma}(\pi_{\alg}, \pi_{\alg})$ in (\ref{eq:amalg}) (via the injection $\Hom_{\sm}(L_{P_i}(K),\!E) \hookrightarrow \Hom_{\sm}(T(K),\!E)$) for any refinement compatible with $I$, see (\ref{eqn:extpism}). Together with (\ref{eq:comp0}), we deduce a morphism as in (\ref{eq:amalg2}) which only depends on $(V,\iota)$ and $(\varphi_{j_1},\dots,\varphi_{j_n})$. Now, it follows from the second isomorphism in (\ref{eq:amalg}) with the end of Step $2$ (and the fact that $\Hom_\sigma(K^\times,E)\buildrel{\det}\over\hookrightarrow \Hom_\sigma(L_{P_i}(K),E)$ is not surjective) that the image of (\ref{eq:comp0}) is \emph{not} contained in $\Ext^1_{\GL_n(K),\sigma}(\pi_{\alg},\pi_{\alg})$. Using Lemma \ref{lem:nonsplit} this implies the surjectivity of (\ref{eq:amalg2}). We also deduce a short exact sequence
\begin{multline*}
0\longrightarrow \Ext^1_{\GL_n(K),\sigma}(\pi_{\alg},\pi_{\alg})\longrightarrow \Ext^1_{\GL_n(K),\sigma}(\pi_{\alg},\!V_I)\longrightarrow \Ext^1_{\GL_n(K),\sigma}(\pi_{\alg},\!V_I/\pi_{\alg})\longrightarrow 0
\end{multline*}
and $\dim_E \Ext^1_{\GL_n(K),\sigma}(\pi_{\alg},V_I)=n+2$ (using the last assertion of Lemma \ref{lem:isoextalg}). As the left hand side of (\ref{eq:amalg2}) is easily checked to also have dimension $n+2$, we finally obtain an isomorphism as in \emph{loc.~cit.}
\end{proof}

\begin{rem}
It is a consequence of (\ref{eq:amalg}), (\ref{eq:amalg2}) and the last statement of Proposition \ref{prop:isoext} that to $(V_I,\iota)$ as in \emph{loc.~cit.}~there is associated a canonical isomorphism of $1$-dimensional $E$-vector spaces
\begin{equation}\label{eq:amalg3}
\Hom_\sigma(L_{P_i}(\cO_K),E)/\Hom_\sigma(\cO_K^\times,E)\buildrel\sim\over\longrightarrow \Ext^1_{\GL_n(K),\sigma}(\pi_{\alg}(D_\sigma),V_I/\pi_{\alg}(D_\sigma))
\end{equation}
where $\Hom_\sigma(\cO_K^\times,E)$ embeds into $\Hom_\sigma(L_{P_i}(\cO_K),E)$ via the determinant $L_{P_i}(\cO_K)\twoheadrightarrow \cO_K^\times$.
\end{rem}\bigskip

The $\GL_{n-i}$ in the next statement is the second factor of $L_{P_i}=\smat{\GL_i & 0\\ 0 & \GL_{n-i}}$.

\begin{prop}\label{prop:pairing}
For $\sigma\in \Sigma$, $i\in \{1,\dots,n-1\}$ and $I\subset \{\varphi_j,\ 0\leq j \leq n-1\}$ of cardinality $i$ we have a perfect pairing of $1$-dimensional $E$-vector spaces:
\begin{multline*}
\Ext^1_{\GL_n(K),\sigma}(\pi_{\alg}(D_\sigma), C(I, s_{i,\sigma}))\times \Ext^1_{\GL_n(K),\sigma}(C(I, s_{i,\sigma}), \pi_{\alg}(D_\sigma))\\
\buildrel\sim\over\longrightarrow \Hom_{\sigma}(\GL_{n-i}(\cO_K),E).
\end{multline*}
\end{prop}
\begin{proof}
To simplify the notation we write $\pi_{\alg}$ instead of $\pi_{\alg}(D_\sigma)$. Consider the $\GL_n(K)$-representation $C(I, s_{i,\sigma})\otimes_E\Ext^1_{\GL_n(K),\sigma}(C(I, s_{i,\sigma}), \pi_{\alg})$ with trivial action of $\GL_n(K)$ on the $1$-dimensional factor $\Ext^1_{\GL_n(K),\sigma}(C(I, s_{i,\sigma}), \pi_{\alg})$. It is formal (and left to the reader) to check that there is a canonical isomorphism of $1$-dimensional $E$-vector spaces
\begin{multline}\label{mult:formal}
\Ext^1_{\GL_n(K),\sigma}(C(I, s_{i,\sigma}), \pi_{\alg})\otimes_E \Ext^1_{\GL_n(K),\sigma}(C(I, s_{i,\sigma}), \pi_{\alg})^\vee\\
\buildrel\sim\over\longrightarrow \Ext^1_{\GL_n(K),\sigma}\big(C(I, s_{i,\sigma})\otimes_E\Ext^1_{\GL_n(K),\sigma}(C(I, s_{i,\sigma}), \pi_{\alg}),\pi_{\alg}\big).
\end{multline}
We denote by $v_I$ the image of the canonical vector of the left hand side. We choose a representative $V_I$ of $v_I$, which is thus isomorphic to a non-split extension of $C(I, s_{i,\sigma})$ by $\pi_{\alg}$. By definition $V_I$ also comes with an injection $\iota:\pi_{\alg}\hookrightarrow V_I$ and an isomorphism
\begin{equation}\label{eq:kappa}
\kappa:V_I/\pi_{\alg}\buildrel\sim\over\longrightarrow C(I, s_{i,\sigma})\otimes_E\Ext^1_{\GL_n(K),\sigma}(C(I, s_{i,\sigma}), \pi_{\alg}).
\end{equation}
The canonical surjection $L_{P_i}\twoheadrightarrow \GL_{n-i}$ onto the second factor induces a canonical injection
\[\Hom_{\sigma}(\GL_{n-i}(\cO_K),E)\hookrightarrow \Hom_{\sigma}(L_{P_i}(\cO_K),E)\]
which composed with the surjection
\[\Hom_{\sigma}(L_{P_i}(\cO_K),E)\twoheadrightarrow \Hom_{\sigma}(L_{P_i}(\cO_K),E)/\Hom_{\sigma}(\cO_K^\times,E)\]
gives a canonical isomorphism of $1$-dimensional $E$-vector spaces
\begin{equation}\label{eq:isoglipi}
\Hom_{\sigma}(\GL_{n-i}(\cO_K),E)\buildrel\sim\over\longrightarrow \Hom_{\sigma}(L_{P_i}(\cO_K),E)/\Hom_{\sigma}(\cO_K^\times,E).
\end{equation}
Combined with (\ref{eq:amalg3}) applied to $V_I$ as above and using $\kappa$, we deduce canonical isomorphisms
\begin{eqnarray}\label{eqn:iso}
\nonumber \Hom_{\sigma}(\GL_{n-i}(\cO_K),E)\!\!&\!\!\buildrel\sim\over\longrightarrow \!\!&\!\!\Ext^1_{\GL_n(K),\sigma}\big(\pi_{\alg}, C(I, s_{i,\sigma})\!\otimes_E\!\Ext^1_{\GL_n(K),\sigma}(C(I, s_{i,\sigma}), \pi_{\alg})\big)\\
\!\!&\!\!\buildrel\sim\over\longleftarrow \!\!&\!\! \Ext^1_{\GL_n(K),\sigma}(\pi_{\alg}, C(I, s_{i,\sigma}))\!\otimes_E\!\Ext^1_{\GL_n(K),\sigma}(C(I, s_{i,\sigma}), \pi_{\alg})
\end{eqnarray}
where the second isomorphism is analogous to (\ref{mult:formal}) and purely formal. It is also formal to check that the isomorphism (\ref{eqn:iso}) does not depend on the chosen representative $V_I$ of $v_I$. This gives the canonical perfect pairing of the statement.
\end{proof}

\begin{rem}
A similar proof to that of Proposition \ref{prop:pairing}) also gives a perfect pairing with $\Hom_{\sigma}(\GL_{i}(\cO_K),E)$ instead of $\Hom_{\sigma}(\GL_{n-i}(\cO_K),E)$, that is, with the first factor of $L_{P_i}=\smat{\GL_i & 0\\ 0 & \GL_{n-i}}$.
\end{rem}

For $\sigma\in \Sigma$ we have a canonical isomorphism of $1$-dimensional $E$-vector spaces
\begin{equation}\label{eq:log}
E \buildrel\sim\over\longrightarrow \Hom_{\sigma}(\GL_{n-i}(\cO_K),E),\ 1\mapsto \sigma\circ \log\circ \det.
\end{equation}
To simplify the notation we simply write $\log\in \Hom_{\sigma}(\GL_{n-i}(\cO_K),E)$ for the image of $1$ in (\ref{eq:log}). For later use, we also write $\val := \val\circ \det \in \Hom_{\sm}(\GL_{n-i}(\cO_K),E)$. It follows from Proposition \ref{prop:pairing} that for each $\sigma\in \Sigma$, $i\in \{1,\dots,n-1\}$ and $I$ of cardinality $i$ we have a canonical isomorphism of $1$-dimensional $E$-vector spaces:
\begin{equation}\label{eq:isodual}
\Ext^1_{\GL_n(K),\sigma}(\pi_{\alg}(D_\sigma), C(I, s_{i,\sigma}))\buildrel\sim\over\longrightarrow \Ext^1_{\GL_n(K),\sigma}(C(I, s_{i,\sigma}), \pi_{\alg}(D_\sigma))^\vee.
\end{equation}
One can also reformulate (\ref{eq:isodual}) in the following way, which will be useful in \S~\ref{sec:def}: the isomorphism $\kappa$ in (\ref{eq:kappa}) induces an isomorphism
\begin{multline}\label{mult:kappa2}
\Ext^1_{\GL_n(K),\sigma}(\pi_{\alg}(D_\sigma),V_I/\pi_{\alg}(D_\sigma))\\
\buildrel\sim\over\longrightarrow \Ext^1_{\GL_n(K),\sigma}(C(I, s_{i,\sigma}), \pi_{\alg}(D_\sigma))^\vee\otimes_E\Ext^1_{\GL_n(K),\sigma}(C(I, s_{i,\sigma}), \pi_{\alg}(D_\sigma))
\end{multline}
such that the image of $\log\in \Hom_{\sigma}(\GL_{n-i}(\cO_K),E)$ in $\Ext^1_{\GL_n(K),\sigma}(\pi_{\alg}(D_\sigma),V_I/\pi_{\alg}(D_\sigma))$ by (\ref{eq:amalg3}) and (\ref{eq:isoglipi}) is sent by (\ref{mult:kappa2}) to the canonical vector of the right hand side of (\ref{mult:kappa2}).

\subsection{The map \texorpdfstring{$t_{D_\sigma}$}{tDsigma} and the representations \texorpdfstring{$\pi(D_\sigma)$}{piDsigma}, \texorpdfstring{$\pi(D)$}{piD}}\label{sec:def}

We define a crucial $E$-linear map $t_{D_\sigma}$ (Proposition \ref{prop:map}) and use it to define the locally $\sigma$-analytic representation $\pi(D_\sigma)$ of $\GL_n(K)$ over $E$ (Definition \ref{def:pi(d)}). We then prove that $\pi(D_\sigma)$ only depends on the isomorphism class of the filtered $\varphi^f$-module $D_\sigma$ (Corollary \ref{cor:ind}). We finally define the locally $\Qp$-analytic representation $\pi(D)$.\bigskip

We keep all the notation of \S~\ref{sec:prel}. In this section (except at the very end) we moreover fix an embedding $\sigma\in \Sigma$. In order to construct $\pi(D_\sigma)$, we first need to fix some choices. We then prove that, up to isomorphism, $\pi(D_\sigma)$ does not depend on these choices.\bigskip

We fix a basis $(e_0,e_1,\dots,e_{n-1})$ of $\varphi^f$-eigenvectors of $D_\sigma$ such that $\varphi^f(e_j)=\varphi_je_j$ for $0\leq j\leq n-1$, the choice of which won't matter (we should denote $e_j$ by $e_{j,\sigma}$ but there will be no ambiguity since $\sigma$ is fixed). For $I\subset \{\varphi_j,\ 0\leq j \leq n-1\}$ of cardinality $i\in \{1,\dots,n-1\}$ we set
\begin{equation}\label{eq:eI}
e_I:=\wedge_{\varphi_j\in I}e_j\in \bigwedge\nolimits_E^{\!i}\!D_\sigma.
\end{equation}
In fact the vector $e_I$ is only defined up to sign, but we will only use the vector space $Ee_I$ in the sequel. We then fix for each $I\subset \{\varphi_j,\ 0\leq j \leq n-1\}$ of cardinality $i\in \{1,\dots,n-1\}$ an isomorphism of $1$-dimensional $E$-vector spaces (using Lemma \ref{lem:nonsplit})
\begin{equation}\label{eq:epsilonI}
\varepsilon_I:\Ext^1_{\GL_n(K),\sigma}(C(I, s_{i,\sigma}), \pi_{\alg}(D_\sigma))\buildrel\sim\over\longrightarrow Ee_{I^c}
\end{equation}
where $I^c$ is the complement of $I$ in $\{1,\dots,n-1\}$. We define for each $i\in \{1,\dots,n-1\}$:
\begin{equation}\label{eq:epsiloni}
\varepsilon_i:=\bigoplus_{\vert I\vert=i}\varepsilon_I:\Ext^1_{\GL_n(K),\sigma}\Big(\bigoplus_{\vert I\vert=i}C(I, s_{i,\sigma}), \pi_{\alg}(D_\sigma)\Big)\buildrel\sim\over\longrightarrow \bigwedge\nolimits_E^{\!n-i}\!D_\sigma.
\end{equation}

\begin{rem}\label{rem:BD23}
The isomorphisms $\varepsilon_I$ and $\varepsilon_i$ do not come out of nowhere since, at least when $K=\Qp$, they can be made functorial, see \cite[Thm.~5.16(ii)]{BD23} with \cite[(5.32)]{BD23}. However, we do not need here the (quite delicate) functor of \emph{loc.~cit.}~(which anyway remains to be defined when $K\ne \Qp$), as it turns out that it is just enough for our purpose to \emph{choose} arbitrary isomorphisms of $E$-vector spaces $\varepsilon_I$ for each $I$.
\end{rem}

For $i\in \{1,\dots,n-1\}$ we define the following $1$-dimensional $E$-vector subspace of $\bigwedge\nolimits_E^{\!n-i}\!D_\sigma$
\begin{equation}\label{eq:filmax}
\begin{gathered}
\begin{array}{lll}
\Fil_i^{\max}D_\sigma &:=&\Fil^{-h_{n-1,\sigma}}(D_\sigma)\wedge \Fil^{-h_{n-2,\sigma}}(D_\sigma)\wedge\cdots\wedge \Fil^{-h_{i,\sigma}}(D_\sigma)\\
&\buildrel\sim\over\rightarrow &\bigwedge\nolimits_E^{\!n-i}\Fil^{-h_{i,\sigma}}(D_\sigma)\ \ \subseteq \ \ \bigwedge\nolimits_E^{\!n-i}\!D_\sigma
\end{array}
\end{gathered}
\end{equation}
where the isomorphism follows from (\ref{eq:fil}) and (\ref{eq:dim}). The following statement will be used later (its easy proof is left to the reader).

\begin{lem}\label{lem:triveq}
Let $I\subset \{\varphi_j,\ 0\leq j \leq n-1\}$ of cardinality $i\in \{1,\dots,n-1\}$. The coefficient of $e_{I^c}$ in $\Fil_i^{\max}D_\sigma$ is non-zero if and only if $\Fil^{-h_{i,\sigma}}(D_{\sigma}) \cap (\bigoplus_{\varphi_j\in I} E e_j)=0$.
\end{lem}

We now define several locally $\sigma$-analytic representations of $\GL_n(K)$ over $E$. For $i\in \{1,\dots,n-1\}$ consider the morphisms of (finite dimensional) $E$-vector spaces (writing $\pi_{\alg}$ for $\pi_{\alg}(D_\sigma)$)
\begin{multline}\label{mult:Vi}
\Ext^1_{\GL_n(K),\sigma}\big(\bigoplus_{\vert I\vert=i}C(I, s_{i,\sigma}), \pi_{\alg}\big)\otimes_E\Ext^1_{\GL_n(K),\sigma}\big(\bigoplus_{\vert I\vert=i}C(I, s_{i,\sigma}), \pi_{\alg}\big)^\vee\\
\buildrel\sim\over\longrightarrow \Ext^1_{\GL_n(K),\sigma}\Big(\big(\bigoplus_{\vert I\vert=i}C(I, s_{i,\sigma})\big)\otimes_E\Ext^1_{\GL_n(K),\sigma}\big(\bigoplus_{\vert I\vert=i}C(I, s_{i,\sigma}), \pi_{\alg}\big), \pi_{\alg}\Big)\\
\longrightarrow \Ext^1_{\GL_n(K),\sigma}\Big(\big(\bigoplus_{\vert I\vert=i}C(I, s_{i,\sigma})\big)\otimes_E\Fil_i^{\max}D_\sigma, \pi_{\alg}\Big)
\end{multline}
where the first isomorphism is canonical and formal (as in the proof of Proposition \ref{prop:pairing}, the action of $\GL_n(K)$ being trivial on the factor $\Ext^1_{\GL_n(K),\sigma}(\oplus_{\vert I\vert=i}C(I, s_{i,\sigma}), \pi_{\alg}(D_\sigma))$ and where the second morphism is the push-forward induced by the composition (see (\ref{eq:epsiloni}) for $\varepsilon_i$)
\begin{equation}\label{eq:filembed}
\Fil_i^{\max}D_\sigma\hookrightarrow \bigwedge\nolimits_E^{\!n-i}\!D_\sigma \buildrel\varepsilon_i^{-1}\over\longrightarrow \Ext^1_{\GL_n(K),\sigma}\Big(\bigoplus_{\vert I\vert=i}C(I, s_{i,\sigma}), \pi_{\alg}(D_\sigma)\Big).
\end{equation}
We denote by $\pi_{s_i}(D_\sigma)$ a representative of the image of the canonical vector of the left hand side of (\ref{mult:Vi}) by the composition (\ref{mult:Vi}) and by $\iota_i:\pi_{\alg}(D_\sigma)\hookrightarrow \pi_{s_i}(D_\sigma)$ the corresponding injection. For $I\subset \{\varphi_j,\ 0\leq j \leq n-1\}$ of cardinality $i\in \{1,\dots,n-1\}$ we denote by $\pi_I(D_\sigma)$ the pull-back of $\pi_{s_i}(D_\sigma)$ along the canonical injection
\[C(I, s_{i,\sigma})\otimes_E\Fil_i^{\max}D_\sigma\hookrightarrow \Big(\bigoplus_{\vert I\vert=i}C(I, s_{i,\sigma})\Big)\otimes_E\Fil_i^{\max}D_\sigma.\]
The representation $\pi_I(D_\sigma)$ also comes with an injection $\iota_I:\pi_{\alg}(D_\sigma)\hookrightarrow \pi_I(D_\sigma)$ (and a surjection $\pi_I(D_\sigma)\twoheadrightarrow C(I, s_{i,\sigma})\otimes_E\Fil_i^{\max}D_\sigma$), the composition $\pi_{\alg}(D_\sigma)\hookrightarrow \pi_I(D_\sigma)\hookrightarrow \pi_{s_i}(D_\sigma)$ being $\iota_{i}$, and we have a canonical isomorphism
\begin{equation}\label{eq:amalgi}
\bigoplus_{\vert I\vert=i, \pi_{\alg}(D_\sigma)}\!\!\!\!\pi_I(D_\sigma)\buildrel\sim\over\longrightarrow \pi_{s_i}(D_\sigma)
\end{equation}
where $\pi_{\alg}(D_\sigma)$ embeds into $\pi_I(D_\sigma)$ via $\iota_I$.\bigskip

Let us unravel (\ref{eq:amalgi}). Denote by $V_I(D_\sigma)$ the pullback of the representation $V_I$ below (\ref{mult:formal}) induced by the composition
\begin{equation}\label{eq:pullI}
\Fil_i^{\max}D_\sigma \buildrel{(\ref{eq:filembed})}\over\hookrightarrow \Ext^1_{\GL_n(K),\sigma}\big(\!\bigoplus_{\vert I\vert=i}C(I, s_{i,\sigma}), \pi_{\alg}(D_\sigma)\big)\twoheadrightarrow \Ext^1_{\GL_n(K),\sigma}(C(I, s_{i,\sigma}), \pi_{\alg}(D_\sigma))
\end{equation}
where the surjection is the canonical projection sending all $\Ext^1_{\GL_n(K),\sigma}(C(J, s_{i,\sigma}), \pi_{\alg})$ to $0$ for $J\ne I$. Then one easily checks comparing (\ref{mult:formal}) and (\ref{mult:Vi}) that there is a canonical isomorphism
\begin{equation}\label{eq:isoVpi}
V_I(D_\sigma)\buildrel\sim\over\longrightarrow \pi_I(D_\sigma)
\end{equation}
which is the identity on $\pi_{\alg}(D_\sigma)$ and on the quotient $C(I, s_{i,\sigma})\otimes_E\Fil_i^{\max}D_\sigma$. It follows from the structure of $V_I(D_\sigma)$ and from (\ref{eq:isoVpi}) that we have a canonical isomorphism when the coefficient of $e_{I^c}$ in the line $\Fil_i^{\max}D_\sigma$ of $\bigwedge\nolimits_E^{\!n-i}\!D_\sigma$ is non-zero
\begin{equation}\label{eq:VipiI}
V_I \buildrel\sim\over\longrightarrow \pi_I(D_\sigma)
\end{equation}
and a canonical isomorphism when this coefficient is $0$ 
\begin{equation}\label{eq:splitI}
\pi_{\alg}(D_\sigma)\bigoplus (C(I, s_{i,\sigma})\otimes_E\Fil_i^{\max}D_\sigma)\buildrel\sim\over\longrightarrow \pi_I(D_\sigma)
\end{equation}
(these conditions do not depend on any choice for the vector $e_{I^c}$). Going back to (\ref{eq:amalgi}) we deduce a canonical isomorphism
\begin{multline}\label{mult:decomp}
\bigg(\pi_{\alg}(D_\sigma)\begin{xy} (0,0)*+{}="a"; (12,0)*+{}="b"; {\ar@{-}"a";"b"}\end{xy}\!\Big(\!\!\bigoplus_{\substack{\vert I\vert=i\\ \mathrm{non-split}}} \!\!\!\!\!C(I, s_{i,\sigma})\otimes_E\Fil_i^{\max}D_\sigma\Big)\bigg) \ \bigoplus \ \Big( \bigoplus_{\substack{\vert I\vert=i\\ \mathrm{split}}} C(I, s_{i,\sigma})\otimes_E\Fil_i^{\max}D_\sigma\Big)\\
\buildrel\sim\over\longrightarrow \pi_{s_i}(D_\sigma)
\end{multline}
where the subsets $I$ in the first (resp.~second) direct summand are those such that the coefficient of $e_{I^c}$ in $\Fil_i^{\max}D_\sigma$ is non-zero (resp.~is $0$). Note that, although the second direct summand in (\ref{mult:decomp}) can be $0$ (for instance when the filtration $\Fil^\bullet(D_\sigma)$ is in a generic position), the first is always strictly larger than $\pi_{\alg}(D_\sigma)$ (because $\Fil_i^{\max}D_\sigma$ is non-zero) and each extension in subobject there is non-split. Finally for any non-empty subset $S\subseteq R$ of the set $R$ of simple reflections of ${\GL_n}$ we define the amalgamated sum
\begin{equation}\label{eq:amalg4}
\pi_S(D_\sigma):= \bigoplus_{s_{i}\in S, \pi_{\alg}(D_\sigma)}\pi_{s_i}(D_\sigma)
\end{equation}
where the sum is over $i\in \{1,\dots,n-1\}$ and $\pi_{\alg}(D_\sigma)$ embeds into $\pi_{s_i}(D_\sigma)$ via $\iota_i$ (thus $\pi_{s_i}(D_\sigma)=\pi_{\{s_i\}}(D_\sigma)$).

\begin{rem}
When $K=\Qp$ the representation $\pi_{s_i}(D_\sigma)$ is the representation denoted $({\mathcal F}_\alpha\circ {\mathcal E}_\alpha)^{-1}(\Fil^{\max}_\alpha)$ with $\alpha=e_i-e_{i+1}$ in \cite[Thm.~5.16]{BD23}.
\end{rem}\bigskip

Whereas all $\Ext^1_{\GL_n(K),\sigma}$ in (\ref{eq:epsiloni}), (\ref{mult:Vi}), (\ref{eq:filembed}) could be replaced by $\Ext^1_{\GL_n(K)}$ by Lemma \ref{lem:nonsplit}, in the following crucial proposition we do need $\Ext^1_{\GL_n(K),\sigma}$.

\begin{prop}\label{prop:map}
There is a surjection of finite dimensional $E$-vector spaces which only depends on the $(\varepsilon_I)_I$ in (\ref{eq:epsilonI}) and on a choice of $\log(p)\in E$:
\[t_{D_\sigma}:\Ext^1_{\GL_n(K),\sigma}(\pi_{\alg}(D_\sigma),\pi_R(D_\sigma))\twoheadlongrightarrow \Ext^1_{\varphi^f}(D_\sigma,D_\sigma) \bigoplus \Hom_{\Fil}(D_\sigma,D_\sigma)\]
where $\Ext^1_{\varphi^f}$ means the extensions as $\varphi^f$-modules and $\Hom_{\Fil}$ means the endomorphisms of $E$-vector spaces which respect the filtration $\Fil^\bullet(D_\sigma)$.
\end{prop}
\begin{proof}
To simplify the notation we write $\pi_{\alg}$, $\pi_I$, $\pi_{s_i}$, $\pi_R$ instead of $\pi_{\alg}(D_\sigma)$, $\pi_I(D_\sigma)$, $\pi_{s_i}(D_\sigma)$, $\pi_R(D_\sigma)$ respectively, and $\Ext^1_{\sigma}$ instead of $\Ext^1_{\GL_n(K),\sigma}$. Note that the maps induced by the injections $\pi_{\alg}\hookrightarrow \pi_I\hookrightarrow \pi_{s_i}\hookrightarrow \pi_R$ (when $i=\vert I\vert$):
\begin{equation}\label{eq:allinj}
\Ext^1_{\sigma}(\pi_{\alg},\pi_{\alg})\longrightarrow\Ext^1_{\sigma}(\pi_{\alg},\pi_I)\longrightarrow \Ext^1_{\sigma}(\pi_{\alg},\pi_{s_i})\longrightarrow \Ext^1_{\sigma}(\pi_{\alg},\pi_R)
\end{equation}
are all injective (their kernel is $0$ since $\Hom_{\GL_n(K)}(\pi_{\alg}, C(I,s_{i,\sigma}))=0$ for all $I$). Hence it is enough to define $t_{D_\sigma}$ in restriction to each $\Ext^1_{\sigma}(\pi_{\alg},\pi_{s_i})$ in such a way that its restriction to $\Ext^1_{\sigma}(\pi_{\alg},\pi_{\alg})$ via $\Ext^1_{\sigma}(\pi_{\alg},\pi_{\alg})\buildrel{\iota_i}\over\hookrightarrow\Ext^1_{\sigma}(\pi_{\alg},\pi_{s_i})$ does not depend on $i$.\bigskip

\textbf{Step $1$}: We define some splittings.\\
Let $i\in \{1,\dots,n-1\}$. The surjection onto the second factor $L_{P_i}(K)\twoheadrightarrow \GL_{n-i}(K)$ gives a canonical injection
\begin{equation}\label{eq:ison-i}
\Hom_{\sigma}(\GL_{n-i}(K),E)\hookrightarrow \Hom_{\sigma}(L_{P_i}(K),E).
\end{equation}
The choice of $\log(p)$ defines a section $\Hom_{\sigma}(\cO_K^\times,E)\hookrightarrow \Hom_{\sigma}(K^\times,E)$ to the restriction map $\Hom_{\sigma}(K^\times,E)\twoheadrightarrow \Hom_{\sigma}(\cO_K^\times,E)$ by sending $\sigma\circ\log:\cO_K^\times \rightarrow E$ to its unique extension to $K^\times$ sending $p$ to $\log(p)$. We obtain a corresponding section $\Hom_{\sigma}(\GL_{n-i}(\cO_K),E)\hookrightarrow \Hom_{\sigma}(\GL_{n-i}(K),E)$ to the restriction $\Hom_{\sigma}(\GL_{n-i}(K),E)\twoheadrightarrow \Hom_{\sigma}(\GL_{n-i}(\cO_K),E)$. One then easily deduces an isomorphism only depending on $\log(p)$
\begin{multline}\label{mult:isoLpi}
\Big(\Hom_{\sm}(L_{P_i}(K),E)\!\!\!\!\!\bigoplus_{\Hom_{\sm}(K^\times,E)}\!\!\!\!\!\Hom_\sigma(K^\times,E)\Big)\ \ \ \bigoplus \ \ \ \Hom_{\sigma}(\GL_{n-i}(\cO_K),E)\\
\buildrel\sim\over\longrightarrow \Hom_{\sigma}(L_{P_i}(K),E)
\end{multline}
such that its restriction to the first direct summand is induced by the canonical injection $\Hom_{\sm}(L_{P_i}(K),E)\hookrightarrow \Hom_{\sigma}(L_{P_i}(K),E)$ and by the determinant $L_{P_i}(K)\twoheadrightarrow K^\times$, and its restriction to the second direct summand is the composition
\[\Hom_{\sigma}(\GL_{n-i}(\cO_K),E)\hookrightarrow\Hom_{\sigma}(\GL_{n-i}(K),E)\hookrightarrow \Hom_{\sigma}(L_{P_i}(K),E).\]
Combining (\ref{eq:amalg3}) (applied to $(V_I,\iota)=(\pi_I,\iota_I)$ with $\vert I\vert=i$) and (\ref{eq:isoglipi}) we recall that we have a canonical isomorphism when $\pi_I$ is non-split
\begin{equation}\label{eq:isoindi}
\Hom_{\sigma}(\GL_{n-i}(\cO_K),E)\buildrel\sim\over\longrightarrow \Ext^1_{\sigma}(\pi_{\alg},\pi_I/\pi_{\alg}).
\end{equation}
Let us now choose a refinement $(\varphi_{j_1},\dots,\varphi_{j_n})$ compatible with $I$ (Definition \ref{def:compatible}). By Proposition \ref{prop:isoext} combined with (\ref{mult:isoLpi}) and (\ref{eq:isoindi}) when $\pi_I$ is non-split, or by the second isomorphism in (\ref{eq:amalg}) combined with (\ref{eq:splitI}) when $\pi_I$ is split, we deduce an isomorphism
\begin{equation}\label{eq:split3}
\Big(\Hom_{\sm}(T(K),E)\!\!\!\!\!\bigoplus_{\Hom_{\sm}(K^\times,E)}\!\!\!\!\!\Hom_\sigma(K^\times,E)\Big) \ \bigoplus \ \Ext^1_{\sigma}(\pi_{\alg},\pi_I/\pi_{\alg})\buildrel\sim\over\longrightarrow \Ext^1_{\sigma}(\pi_{\alg},\pi_I)
\end{equation}
such that its restriction to the first direct summand is the second isomorphism in (\ref{eq:amalg}) (composed with $\Ext^1_{\sigma}(\pi_{\alg},\pi_{\alg})\buildrel{\iota_I}\over\hookrightarrow\Ext^1_{\sigma}(\pi_{\alg},\pi_I)$), hence only depends on $(\varphi_{j_1},\dots,\varphi_{j_n})$, and its restriction to the second direct summand does not depend on $(\varphi_{j_1},\dots,\varphi_{j_n})$ (but depends on $\varepsilon_i$ via the definition of $\pi_I$) and depends on $\log(p)$ via (\ref{mult:isoLpi}) if and only if $\pi_I$ is non-split.\bigskip

In fact, for \emph{any} refinement (not necessarily compatible with $I$) we have the composition (which depends on that refinement):
\begin{equation}\label{eq:restr}
\Hom_{\sm}(T(K),E)\!\!\!\!\!\bigoplus_{\Hom_{\sm}(K^\times,E)}\!\!\!\!\!\Hom_\sigma(K^\times,E)\buildrel{(\ref{eq:amalg})}\over\cong \Ext^1_{\sigma}(\pi_{\alg},\pi_{\alg})\buildrel{\iota_I}\over\hookrightarrow\Ext^1_{\sigma}(\pi_{\alg},\pi_I).
\end{equation}
Hence choosing an arbitrary refinement $(\varphi_{j_1},\dots,\varphi_{j_n})$ we (easily) deduce from (\ref{eq:split3}), (\ref{eq:restr}) and (\ref{eq:amalgi}) an isomorphism 
\begin{multline}\label{mult:split4}
\Big(\Hom_{\sm}(T(K),E)\!\!\!\!\!\bigoplus_{\Hom_{\sm}(K^\times,E)}\!\!\!\!\!\Hom_\sigma(K^\times,E)\Big)\ \bigoplus \ \big(\bigoplus_{\vert I\vert=i}\Ext^1_{\sigma}(\pi_{\alg},\pi_I/\pi_{\alg})\big)\buildrel\sim\over\longrightarrow \\
\Ext^1_{\sigma}(\pi_{\alg},\pi_{s_i})
\end{multline}
such that its restriction to $\bigoplus_{\vert I\vert=i}\Ext^1_{\sigma}(\pi_{\alg},\pi_I/\pi_{\alg})$ only depends on $\varepsilon_i$ and $\log(p)$.\bigskip

By the discussion before Step $1$ and (\ref{mult:split4}) it is enough to define $t_{D_\sigma}$ in restriction to $\Ext^1_{\sigma}(\pi_{\alg},\pi_{\alg})$ and in restriction to $\bigoplus_{\vert I\vert = i}\Ext^1_{\sigma}(\pi_{\alg},\pi_I/\pi_{\alg})\buildrel\sim\over\rightarrow \Ext^1_{\sigma}(\pi_{\alg},\pi_{s_i}/\pi_{\alg})$, and to prove that $t_{D_\sigma}$ does not depend on any refinement.\bigskip

\textbf{Step $2$}: We define $t_{D_\sigma}$ in restriction to $\Ext^1_{\sigma}(\pi_{\alg},\pi_{\alg})$.\\
Recall from Step $1$ that the choice of $\log(p)$ gives an isomorphism
\[\Hom_{\sm}(K^\times,E)\bigoplus \Hom_{\sigma}(\cO_K^\times,E)\buildrel\sim\over\longrightarrow \Hom_{\sigma}(K^\times,E)\]
and hence an isomorphism
\begin{equation}\label{eq:smoothsplit}
\Hom_{\sm}(T(K),E)\bigoplus \Hom_{\sigma}(\cO_K^\times,E) \buildrel\sim\over\longrightarrow \Hom_{\sm}(T(K),E)\!\!\!\!\!\bigoplus_{\Hom_{\sm}(K^\times,E)}\!\!\!\!\!\Hom_\sigma(K^\times,E).
\end{equation}
By (\ref{eq:smoothsplit}) and the second isomorphism in (\ref{eq:amalg}) we only need to define $t_D$ in restriction to $\Hom_{\sm}(T(K),E)$ and to $\Hom_{\sigma}(\cO_K^\times,E)$.\bigskip

We start with the second. We have a canonical injection
\[\Ext^1_{\varphi^f}(D_\sigma,D_\sigma) \bigoplus E \hookrightarrow \Ext^1_{\varphi^f}(D_\sigma,D_\sigma) \bigoplus \Hom_{\Fil}(D_\sigma,D_\sigma)\]
which is the identity on $\Ext^1_{\varphi^f}(D_\sigma,D_\sigma) $ and sends $\lambda\in E$ to the multiplication by $\lambda$ on $D_\sigma$, which is obviously in $\Hom_{\Fil}(D_\sigma,D_\sigma)$. We then define
\begin{equation}\label{eq:scalar}
t_{D_\sigma}\vert_{ \Hom_{\sigma}(\cO_K^\times,E)}: \Hom_{\sigma}(\cO_K^\times,E)\buildrel\sim\over\longrightarrow 0 \oplus E\hookrightarrow \Ext^1_{\varphi^f}(D_\sigma,D_\sigma) \oplus E,\ \lambda(\sigma\circ \log)\longmapsto 0 + \lambda.
\end{equation}
Fixing a refinement $(\varphi_{j_1},\dots,\varphi_{j_n})$ we now define
\begin{equation}\label{eq:extphi}
t_{D_\sigma}\vert_{\Hom_{\sm}(T(K),E)}:\Hom_{\sm}(T(K),E)\buildrel\sim\over\longrightarrow \Ext^1_{\varphi^f}(D_\sigma,D_\sigma)
\end{equation}
by sending $\psi=(\psi_1,\dots,\psi_n)\in \Hom_{\sm}(T(K),E)$ (with $\psi_\ell:K^\times\rightarrow E$) to
\begin{equation}\label{eq:phiepsilon}
\Big(\bigoplus_{\ell=1}^n E[\epsilon]/(\epsilon^2)e_{j_\ell},\ \varphi^f(e_{j_\ell}):=\varphi_{j_\ell}(1+ \psi_\ell(\varpi_K)\epsilon )e_{j_\ell}\Big)\in \Ext^1_{\varphi^f}(D_\sigma,D_\sigma)
\end{equation}
where $\varpi_K$ is any uniformizer of $K$ (remembering that the smoothness of the additive character $\psi_\ell$ implies $\psi_\ell\vert_{\cO_K^\times}=0$). The map $t_{D_\sigma}\vert_{\Hom_{\sm}(T(K),E)}$ depends on a choice of refinement, however so does the injection $\Hom_{\sm}(T(K),E)\hookrightarrow \Ext^1_{\GL_n(K),\sigma}(\pi_{\alg},\pi_{\alg})$ in (\ref{eqn:extpism}), and one readily checks that, via the second isomorphism in (\ref{eq:amalg}) (and (\ref{eq:smoothsplit}), (\ref{eq:scalar})), the resulting map
\begin{multline}\label{mult:step2}
t_{D_\sigma}\vert_{\Ext^1_{\sigma}(\pi_{\alg},\pi_{\alg})}\!:\!\Ext^1_{\sigma}(\pi_{\alg},\pi_{\alg})\!\buildrel\sim\over\longrightarrow \!\Ext^1_{\varphi^f}(D_\sigma,D_\sigma) \bigoplus E\\
(\hookrightarrow \Ext^1_{\varphi^f}(D_\sigma,D_\sigma) \bigoplus \Hom_{\Fil}(D_\sigma,D_\sigma))
\end{multline}
only depends on the choice of $\log(p)$ and is an isomorphism.\bigskip

\textbf{Step $3$}: We define $t_{D_\sigma}$ in restriction to $\Ext^1_{\GL_n(K),\sigma}(\pi_{\alg},\pi_{s_i}/\pi_{\alg})$.\\
Recall that the definition of $\pi_{s_i}$ right after (\ref{eq:filembed}) comes with an isomorphism $\pi_{s_i}/\pi_{\alg}\buildrel\sim\over\longrightarrow (\bigoplus_{\vert I\vert=i}C(I, s_{i,\sigma}))\otimes_E\Fil_i^{\max}D_\sigma$. We then deduce isomorphisms
\begin{eqnarray}\label{eqn:isoI}
\nonumber\Ext^1_{\sigma}(\pi_{\alg},\pi_{s_i}/\pi_{\alg})&\buildrel\sim\over\longrightarrow &\Ext^1_{\sigma}\Big(\pi_{\alg},\big(\bigoplus_{\vert I\vert=i} C(I, s_{i,\sigma})\big)\otimes_E\Fil_i^{\max}D_\sigma\Big)\\
\nonumber&\buildrel\sim\over\longleftarrow &\Ext^1_{\sigma}\big(\pi_{\alg},\bigoplus_{\vert I\vert=i} C(I, s_{i,\sigma})\big)\otimes_E\Fil_i^{\max}D_\sigma\\
\nonumber&\buildrel{\stackrel{(\ref{eq:isodual})}{\sim}}\over\longrightarrow & \Ext^1_{\sigma}\big(\bigoplus_{\vert I\vert=i} C(I, s_{i,\sigma}),\pi_{\alg}\big)^\vee\otimes_E\Fil_i^{\max}D_\sigma\\
\nonumber&\buildrel{\stackrel{(\ref{eq:epsiloni})}{\sim}}\over\longleftarrow & \big(\bigwedge\nolimits_E^{\!n-i}\!D_\sigma\big)^\vee\otimes_E\Fil_i^{\max}D_\sigma\\
&\buildrel{\stackrel{(\ref{eq:filmax})}{\sim}}\over\longrightarrow & \Hom_E\big(\bigwedge\nolimits_E^{\!n-i}\!D_\sigma, \bigwedge\nolimits_E^{\!n-i}\Fil^{-h_{i,\sigma}}(D_\sigma)\big)
\end{eqnarray}
(where the second isomorphism is formal). The restriction map induces a canonical surjection
\begin{multline}\label{mult:wedge}
\Hom_E\big(\bigwedge\nolimits_E^{\!n-i}\!D_\sigma, \bigwedge\nolimits_E^{\!n-i}\Fil^{-h_{i,\sigma}}(D_\sigma)\big)\\
\twoheadrightarrow \Hom_E\big((\bigwedge\nolimits_E^{\!n-i-1}\Fil^{-h_{i,\sigma}}(D_\sigma)) \wedge D_\sigma, \bigwedge\nolimits_E^{\!n-i}\Fil^{-h_{i,\sigma}}(D_\sigma)\big)
\end{multline}
where $(\bigwedge\nolimits_E^{\!n-i-1}\Fil^{-h_{i,\sigma}}(D_\sigma)) \wedge D_\sigma$ denotes the image of $(\bigwedge\nolimits_E^{\!n-i-1}\Fil^{-h_{i,\sigma}}(D_\sigma)) \otimes_E D_\sigma$ in $\bigwedge\nolimits_E^{\!n-i}\!D_\sigma$. By (i) of Lemma \ref{lem:surj} below we have a canonical isomorphism
\begin{multline}\label{mult:wedge2}
\left\{f\in \Hom_E(D_\sigma, \Fil^{-h_{i,\sigma}}(D_\sigma)),\ f\vert_{\Fil^{-h_{i,\sigma}}(D_\sigma)}\ {\mathrm{scalar}}\right\}\\\buildrel\sim\over\longrightarrow \Hom_E\big((\bigwedge\nolimits_E^{\!n-i-1}\Fil^{-h_{i,\sigma}}(D_\sigma)) \wedge D_\sigma, \bigwedge\nolimits_E^{\!n-i}\Fil^{-h_{i,\sigma}}(D_\sigma)\big).
\end{multline}
Composing the isomorphism (\ref{eqn:isoI}) with the surjection (\ref{mult:wedge}) and the inverse of the isomorphism (\ref{mult:wedge2}) we obtain a surjection only depending on $\varepsilon_i$
\begin{multline}\label{mult:surj}
t_{D_\sigma}\vert_{\Ext^1_{\GL_n(K),\sigma}(\pi_{\alg},\pi_{s_i}/\pi_{\alg})}:\Ext^1_{\GL_n(K),\sigma}(\pi_{\alg},\pi_{s_i}/\pi_{\alg}) \\
\twoheadrightarrow \left\{f\in \Hom_E(D_\sigma, \Fil^{-h_{i,\sigma}}(D_\sigma)),\ f\vert_{\Fil^{-h_{i,\sigma}}(D_\sigma)}\ {\mathrm{scalar}}\right\}\hookrightarrow \Hom_{\Fil}(D_\sigma,D_\sigma).
\end{multline}
Finally the surjectivity of $t_{D_\sigma}$ follows from the surjectivity of (\ref{mult:wedge}) and from (ii) of Lemma \ref{lem:surj} below \ (noting \ that \ the \ scalar \ endomorphisms \ of \ $D_\sigma$ \ can \ also \ be \ described \ as \ $\{f\in \Hom_E(D_\sigma, \Fil^{-h_{0,\sigma}}(D_\sigma)),\ f\vert_{\Fil^{-h_{0,\sigma}}(D_\sigma)}\ {\mathrm{scalar}}\}$.
\end{proof}

The proof of Proposition \ref{prop:map} uses the following elementary lemma:

\begin{lem}\label{lem:surj}
\hspace{2em}
\begin{enumerate}[label=(\roman*)]
\item
For $i\in \{1,\dots,n-1\}$ we have a canonical isomorphism as in (\ref{mult:wedge2}) given by $f\longmapsto (x\wedge d\mapsto x\wedge f(d))$ for $x\in \bigwedge\nolimits_E^{\!n-i-1}\Fil^{-h_{i,\sigma}}(D_\sigma)$ and $d\in D_\sigma$.
\item
For $i\in \{0,\dots,n-1\}$ the inclusions
\[\left\{f\in \Hom_E(D_\sigma, \Fil^{-h_{i,\sigma}}(D_\sigma)),\ f\vert_{\Fil^{-h_{i,\sigma}}(D_\sigma)}\ {\mathrm{scalar}}\right\}\hookrightarrow \Hom_{\Fil}(D_\sigma,D_\sigma)\]
induce a surjection
\begin{equation*}
\bigoplus_{i=0}^{n-1}\{f\in \Hom_E(D_\sigma, \Fil^{-h_{i,\sigma}}(D_\sigma)),\ f\vert_{\Fil^{-h_{i,\sigma}}(D_\sigma)}\ {\mathrm{scalar}}\}\twoheadrightarrow \Hom_{\Fil}(D_\sigma,D_\sigma).
\end{equation*}
\end{enumerate}
\end{lem}
\begin{proof}
The proof of (ii) being straightforward, we only prove (i). Note first that the map is well-defined since, if $x\wedge d=0$, one easily checks that this implies $d\in \Fil^{-h_{i,\sigma}}(D_\sigma)$, and thus we also have $x\wedge f(d)=0$ as $f\vert_{\Fil^{-h_{i,\sigma}}(D_\sigma)}$ is scalar. A quick calculation with (\ref{eq:dim}) shows that both sides of (\ref{mult:wedge2}) have dimension $1+i(n-i)$ over $E$, hence it is enough to prove injectivity. But if $x\wedge f(d)=0$ for all $x\in \bigwedge\nolimits_E^{\!n-i-1}\Fil^{-h_{i,\sigma}}(D_\sigma)$, $d\in D_\sigma$, this implies that $f(d)$ belongs to all $(n-i-1)$-dimensional vector subspaces of $\Fil^{-h_{i,\sigma}}(D_\sigma)$, which obviously implies $f(d)=0$ (for all $d$).
\end{proof}

For later use, we recall that, fixing an arbitrary refinement, by (\ref{mult:split4}) and (\ref{eq:amalg4}) (for $S=R$) we have an isomorphism (depending on this refinement, on the isomorphisms $(\varepsilon_I)_I$ in (\ref{eq:epsilonI}) via the definitions of the representations $\pi_I(D_\sigma)$, $\pi_R(D_\sigma)$, and on $\log(p)$):
\begin{multline}\label{mult:fix}
\big(\Hom_{\sm}(T(K),E)\!\!\!\!\!\!\bigoplus_{\Hom_{\sm}(K^\times,E)}\!\!\!\!\!\!\Hom_\sigma(K^\times\!\!,E)\big) \bigoplus \Big(\!\bigoplus_I\Ext^1_{\sigma}\big(\pi_{\alg}(D_\sigma),\pi_I(D_\sigma)/\pi_{\alg}(D_\sigma)\big)\!\Big)\\
\buildrel\sim\over\longrightarrow \Ext^1_{\sigma}(\pi_{\alg}(D_\sigma),\pi_R(D_\sigma)).
\end{multline}

We can now define $\pi(D_\sigma)$.

\begin{definit}\label{def:pi(d)}
We define $\pi(D_\sigma)$ as the representation of $\GL_n(K)$ over $E$ associated to the image in $\Ext^1_{\GL_n(K),\sigma}(\pi_{\alg}(D_\sigma)\otimes_E\Ker(t_{D_\sigma}), \pi_R(D_\sigma))$ of the canonical vector of the source by the composition
\begin{multline*}
\Ext^1_{\GL_n(K),\sigma}(\pi_{\alg}(D_\sigma),\pi_R(D_\sigma))\otimes_E\Ext^1_{\GL_n(K),\sigma}(\pi_{\alg}(D_\sigma),\pi_R(D_\sigma))^\vee\\
\buildrel\sim\over\longrightarrow \Ext^1_{\GL_n(K),\sigma}\big(\pi_{\alg}(D_\sigma)\otimes_E\Ext^1_{\GL_n(K),\sigma}(\pi_{\alg}(D_\sigma),\pi_R(D_\sigma)), \pi_R(D_\sigma)\big)\\
\longrightarrow \Ext^1_{\GL_n(K),\sigma}\big(\pi_{\alg}(D_\sigma)\otimes_E\Ker(t_{D_\sigma}), \pi_R(D_\sigma)\big)
\end{multline*}
where the first isomorphism is formal and the second morphism is the pull-back induced by $\Ker(t_{D_\sigma})\hookrightarrow \Ext^1_{\GL_n(K),\sigma}(\pi_{\alg}(D_\sigma),\pi_R(D_\sigma))$.
\end{definit}

In the sequel we also call $\pi(D_\sigma)$ the \emph{tautological extension} of $\pi_{\alg}(D_\sigma)\otimes_E\Ker(t_{D_\sigma})$ by $\pi_R(D_\sigma)$.\bigskip

Though the surjection $t_{D_\sigma}$ in Proposition \ref{prop:map} depends on choices, we now prove that the isomorphism class of the representation $\pi(D_\sigma)$ in Definition \ref{def:pi(d)} does not. We first prove it does not depend on the isomorphisms (\ref{eq:epsilonI}).

\begin{prop}\label{prop:epsilon}
Up to isomorphism the map $t_{D_\sigma}$ does not depend on the $(\varepsilon_I)_I$ in (\ref{eq:epsilonI}) for $I\subset \{\varphi_j,\ 0\leq j \leq n-1\}$. In particular the representation $\pi(D_\sigma)$ of $\GL_n(K)$ over $E$ does not depend on the choice of the $\varepsilon_I$.
\end{prop}
\begin{proof}
As in the proof of Proposition \ref{prop:map} we write $\pi_{\alg}$, $\pi_I$, $\pi_R$ instead of $\pi_{\alg}(D_\sigma)$, $\pi_I(D_\sigma)$, $\pi_R(D_\sigma)$ and $\Ext^1_{\sigma}$ instead of $\Ext^1_{\GL_n(K),\sigma}$. We fix $\log(p)\in E$ and two choices $(\varepsilon_I)_I$ and $(\varepsilon'_I)_I$.\bigskip

Let $\pi_I$, $\pi_R$ (resp.~$\pi_I'$, $\pi_R'$) associated to $(\varepsilon_I)_I$ (resp.~$(\varepsilon'_I)_I$) in (\ref{eq:amalgi}), (\ref{eq:amalg4}) (for $S=R$), and $t_{D_\sigma}$ (resp.~$t'_{D_\sigma}$) the corresponding map associated to $\{(\varepsilon_I)_I, \log(p)\}$ (resp.~$\{(\varepsilon'_I)_I,\log(p)\}$) in Proposition \ref{prop:map}. We prove that there is a $\GL_n(K)$-equivariant isomorphism $\Psi:\pi_R\buildrel\sim\over\longrightarrow \pi_R'$ which induces a commutative diagram of $E$-vector spaces
\begin{equation}\label{eq:small}
\begin{gathered}
\xymatrix{
\Ext^1_{\sigma}\big(\pi_{\alg},\pi_R\big)\ar^{\stackrel{\Psi}{\sim}}[r]\ar@{->>}^{\!t_{D_\sigma}}[d] & \Ext^1_{\sigma}\big(\pi_{\alg},\pi_R'\big)\ar@{->>}^{t'_{D_\sigma}}[ld]\\
\Ext^1_{\varphi^f}(D_\sigma,D_\sigma) \bigoplus\Hom_{\Fil}(D_\sigma,D_\sigma)&
}
\end{gathered}
\end{equation}
(this is what we mean by ``up to isomorphism the map $t_{D_\sigma}$ does not depend on the $(\varepsilon_I)_I$''). By the construction of $t_{D_\sigma}$ in Step $2$ and Step $3$ of the proof of Proposition \ref{prop:map} and by (\ref{eq:amalgi}), (\ref{eq:amalg4}), (\ref{mult:fix}), to construct $\Psi$ (and obtain (\ref{eq:small})) it is enough to construct for each $I$ an isomorphism $\Psi_I:\pi_I\buildrel\sim\over\longrightarrow \pi_I'$ such that we have commutative diagrams
\begin{equation}\label{eq:smallI}
\begin{gathered}
\xymatrix{
\pi_I\ar^{\stackrel{\Psi_I}{\sim}}[r] & \pi_I'\\
\pi_{\alg}\ar@{^{(}->}^{\iota_I}[u] \ar@{^{(}->}^{\iota'_I}[ru]\\
}
\end{gathered}
\mathrm{\ and\ }
\begin{gathered}
\xymatrix{
\Ext^1_{\sigma}(\pi_{\alg},\pi_I/\pi_{\alg})\ar^{\stackrel{\overline{\Psi_I}}{\sim}}[r]\ar@{^{(}->}^{(\ref{eqn:isoI})}[d] & \Ext^1_{\sigma}(\pi_{\alg},\pi_I'/\pi_{\alg})\ar@{^{(}->}^{(\ref{eqn:isoI})'}[ld]\\
\Hom_E\big(\bigwedge\nolimits_E^{\!n-i}\!D_\sigma, \Fil_i^{\max}D_\sigma\big)&
}
\end{gathered}
\end{equation}
where $\overline{\Psi_I}:\pi_I/\pi_{\alg}\buildrel\sim\over\longrightarrow \pi_I'/\pi_{\alg}$ is induced by $\Psi_I$ (and where $i=\vert I\vert$). There is a unique $c_I\in E^\times$ such that $\varepsilon_I'=c_I\varepsilon_I$ in (\ref{eq:epsilonI}), thus we have a commutative diagram
\begin{equation}\label{eq:smallfil}
\begin{gathered}
\xymatrix{
\Fil_i^{\max}D_\sigma\ar^{\stackrel{c_I}{\sim}}[r]\ar^{\!(\ref{eq:pullI})}[d] & \Fil_i^{\max}D_\sigma\ar^{(\ref{eq:pullI})'}[ld]\\
\Ext^1_{\sigma}(C(I, s_{i,\sigma}), \pi_{\alg})\ .&
}
\end{gathered}
\end{equation}
Then (\ref{eq:smallfil}) induces an isomorphism $V_I(D_\sigma)\buildrel\sim\over\longrightarrow V_I(D_\sigma)'$ which is the identity in restriction to $\pi_{\alg}$ and the multiplication by $c_I$ on the quotient $C(I, s_{i,\sigma})\otimes_E\Fil_i^{\max}D_\sigma$ (see above (\ref{eq:pullI}) for $V_I(D_\sigma)$). By (\ref{eq:isoVpi}) this gives $\Psi_I:\pi_I\buildrel\sim\over\longrightarrow \pi_I'$ and the first diagram in (\ref{eq:smallI}). Note that $\overline{\Psi_I}$ is the multiplication by $c_I$ on $C(I, s_{i,\sigma})\otimes_E\Fil_i^{\max}D_\sigma$. The second diagram then follows since the image of the morphism (\ref{eqn:isoI}) in (\ref{eq:smallfil}) lies in $\Hom_E(Ee_{I^c},\Fil_i^{\max}D_\sigma)$ and since the dual of $\varepsilon_J$, which is used in the one but last map in (\ref{eqn:isoI}) (see also the right part of the diagram (\ref{eq:big}) below), ``compensates'' the scalar $c_I$.
\end{proof}

We now prove that the isomorphism class of $\pi(D_\sigma)$ does not depend on $\log(p)$. As this is more subtle, we need some preparation.

\begin{lem}\label{lem:changelog}
Let $I\subset \{\varphi_j,\ 0\leq j \leq n-1\}$ of cardinality $i\in \{1,\dots,n-1\}$ and let $c_I\in \Ext^1_{\GL_n(K),\sigma}(\pi_{\alg}(D_\sigma),\pi_I(D_\sigma)/\pi_{\alg}(D_\sigma))$. Assume $\pi_I$ is non-split and write $c_I=\lambda_I(c_I) \log$ where $\lambda_I(c_I)\in E$ and $\log \in \Ext^1_{\GL_n(K),\sigma}(\pi_{\alg}(D_\sigma),\pi_I(D_\sigma)/\pi_{\alg}(D_\sigma))$ is the image of $\log\in \Hom_E(\GL_{n-i}(\cO_K),E)$ under (\ref{eq:isoindi}) (see below (\ref{eq:log}) for {\rm log}). Changing $\log(p)$ into $\log(p)'$ replaces $c_I$ in (\ref{eq:split3}) for any refinement compatible with $I$ by
\[\lambda_I(c_I)(\log(p)-\log(p)')\val \ + \ c_I\]
where $\val\in \Hom_{\sm}(\GL_{n-i}(K),E)\buildrel{\ref{eq:ison-i}}\over\hookrightarrow \Hom_{\sm}(L_{P_i}(K),E)\hookrightarrow \Hom_{\sm}(T(K),E)$ (and if $\pi_I$ is split then $c_I$ does not change).
\end{lem}
\begin{proof}
This follows from the isomorphism (\ref{eq:VipiI}) and from the canonical commutative diagram
\[\xymatrix{
\Ext^1_{\GL_n(K),\sigma}(\pi_{\alg}(D_\sigma),V_I)\ar@{->>}[r] & \Ext^1_{\GL_n(K),\sigma}\big(\pi_{\alg}(D_\sigma),V_I/\pi_{\alg}(D_\sigma)\big)\\
\Hom_{\sigma}(\GL_{n-i}(K),E)\ar@{^{(}->}^{(\ref{eq:amalg2})}[u]\ar@{->>}[r] & \Hom_{\sigma}(\GL_{n-i}(\cO_K),E)\ar^{(\ref{eq:isoindi})\ \wr\!}[u]
}\]
where $\Hom_{\sigma}(\GL_{n-i}(K),E)$ is seen in $\Hom_{\sigma}(L_{P_i}(K),E)$ (in (\ref{eq:amalg2})) as in (\ref{eq:ison-i}).
\end{proof}

We go on with a crucial lemma which will be used several times in the sequel.\bigskip

For $c\in \bigoplus_I \Ext^1_{\GL_n(K),\sigma}(\pi_{\alg}(D_\sigma),\pi_I(D_\sigma)/\pi_{\alg}(D_\sigma))$ recall that we have defined $t_{D_\sigma}(c)\in \Hom_{\Fil}(D_\sigma,D_\sigma)$ in Step $3$ of the proof of Proposition \ref{prop:map}. Hence for each $j\in \{0,\dots,n-1\}$ there is a unique $\lambda_j(c)\in E$ such that
\[t_{D_\sigma}(c)(e_j)-\lambda_j(c)e_j \in \bigoplus_{j'\ne j}Ee_{j'}.\]
Moreover $\lambda_j(c)$ obviously does not depend on the choice of the basis $(e_0,\dots,e_{n-1})$ at the beginning of \S~\ref{sec:def}. The following technical but important lemma will be used several times.

\begin{lem}\label{lem:nuI}
Let $I\subset \{\varphi_j,\ 0\leq j \leq n-1\}$ of cardinality $i\in \{1,\dots,n-1\}$ and let $c=c_I\in \Ext^1_{\GL_n(K),\sigma}(\pi_{\alg}(D_\sigma),\pi_I(D_\sigma)/\pi_{\alg}(D_\sigma))$.
\begin{enumerate}[label=(\roman*)]
\item
If $j$ is such that $\varphi_j\in I$ then $t_{D_\sigma}(c_I)(e_j)=0$ and thus $\lambda_j(c_I)=0$.
\item
If the coefficient of $e_{I^c}$ in $\Fil_i^{\max}D_\sigma$ is $0$ and if $j$ is such that $\varphi_j\notin I$ then we have $t_{D_\sigma}(c_I)(e_j) \in \bigoplus_{\varphi_{j'}\in I}Ee_{j'}$. In particular if the coefficient of $e_{I^c}$ in $\Fil_i^{\max}D_\sigma$ is $0$ then $\lambda_j(c_I)=0$ for all $j\in \{0,\dots,n-1\}$.
\item
If the coefficient of $e_{I^c}$ in $\Fil_i^{\max}D_\sigma$ is non-zero and if $j$ is such that $\varphi_j\notin I$ then $\lambda_j(c_I)$ does not depend on such $j$ and is the unique scalar $\lambda_I(c_I)\in E$ such that $c_I=\lambda_I(c_I) \log$ where $\log \!\in \!\Ext^1_{\GL_n(K),\sigma}(\pi_{\alg}(D_\sigma),\pi_I(D_\sigma)/\pi_{\alg}(D_\sigma))$ is the image of $\log\!\in \!\Hom_E(\GL_{n-i}(\cO_K),E)$ (see below (\ref{eq:log})) under (\ref{eq:isoindi}). Moreover in that case we have
\[t_{D_\sigma}(c_I)(e_j)-\lambda_I(c_I)e_j \in \bigoplus_{\varphi_{j'}\in I}Ee_{j'}\ \ (\text{for }j\text{ such that }\varphi_j\notin I).\]
\end{enumerate}
\end{lem}
\begin{proof}
As usual we write $\pi_{\alg}$, $\pi_I$ instead of $\pi_{\alg}(D_\sigma)$, $\pi_I(D_\sigma)$ and $\Ext^1_{\sigma}$ instead of $\Ext^1_{\GL_n(K),\sigma}$.\bigskip

We prove (i). The image of the restriction of the morphism (\ref{eqn:isoI}) to $\Ext^1_{\sigma}(\pi_{\alg},\pi_I/\pi_{\alg})$ lies by construction in the subspace $\Hom_E(Ee_{I^c},\Fil_i^{\max}D_\sigma)$ of $\Hom_E(\bigwedge\nolimits_E^{\!n-i}\!D_\sigma, \Fil_i^{\max}D_\sigma)$. In particular, by the definition of the morphism (\ref{mult:wedge2}) in (i) of Lemma \ref{lem:surj}, we have $x\wedge t_{D_\sigma}(c_I)(d)=0$ in $\bigwedge\nolimits_E^{\!n-i}\Fil^{-h_{i,\sigma}}(D_\sigma)= \Fil_i^{\max}D_\sigma$ for any $x\in \bigwedge\nolimits_E^{\!n-i-1}\Fil^{-h_{i,\sigma}}(D_\sigma)$ and $d\in D_\sigma$ such that $x\wedge d\in \bigoplus_{J\ne I^c}Ee_J$ in $\bigwedge\nolimits_E^{\!n-i}\!D_\sigma$. Since $\varphi_j\in I$, we see that we always have $x\wedge e_j\in \bigoplus_{J\ne I^c}Ee_J$ and thus 
\[x\wedge t_{D_\sigma}(c_I)(e_j)=0\ \mathrm{in}\ \bigwedge\nolimits_E^{\!n-i}\Fil^{-h_{i,\sigma}}(D_\sigma)\ \mathrm{for\ any}\ x\in \bigwedge\nolimits_E^{\!n-i-1}\Fil^{-h_{i,\sigma}}(D_\sigma).\]
This implies that $t_{D_\sigma}(c_I)(e_j)$ lies in any $(n-i-1)$-dimensional vector subspace of $\Fil^{-h_{i,\sigma}}(D_\sigma)$, which implies $t_{D_\sigma}(c_I)(e_j)=0$.\bigskip

We prove (ii). The last assertion follows from the first and from (i). We prove the first assertion. Define
\begin{equation}\label{eq:fil^j}
\Fil^{-h_{i,\sigma}}(D_\sigma)^{(j)}:=\Fil^{-h_{i,\sigma}}(D_\sigma) \cap \big(\bigoplus_{j'\ne j}Ee_{j'}\big)
\end{equation}
which has dimension $\geq \dim_E\Fil^{-h_{i,\sigma}}(D_\sigma) - 1= n-i-1$ (see (\ref{eq:dim})) and let $x$ be a non-zero vector in $\bigwedge\nolimits_E^{\!n-i-1}\Fil^{-h_{i,\sigma}}(D_\sigma)^{(j)}$ (which is a non-zero vector space). If the coefficient of $e_{I^c}$ in $x\wedge e_j\in \bigwedge\nolimits_E^{\!n-i}\!D_\sigma$ is $0$ then as in the second sentence of the proof of (i) we have $x\wedge t_{D_\sigma}(c_I)(e_j)=0$ in $\bigwedge\nolimits_E^{\!n-i}\Fil^{-h_{i,\sigma}}(D_\sigma)$ and thus $t_{D_\sigma}(c_I)(e_j)\in \Fil^{-h_{i,\sigma}}(D_\sigma)^{(j)}$, which implies $\lambda_j(c_I)=0$. If the coefficient of $e_{I^c}$ in $x\wedge e_j$ is non-zero, then $x\wedge t_{D_\sigma}(c_I)(e_j)$ is non-zero in $\bigwedge\nolimits_E^{\!n-i}\Fil^{-h_{i,\sigma}}(D_\sigma)$. But if the coefficient $\lambda_j(c_I)$ of $e_j$ in $t_{D_\sigma}(c_I)(e_j)$ is also non-zero, then necessarily the coefficient of $e_{I^c}$ in $x\wedge t_{D_\sigma}(c_I)(e_j)$ is non-zero, contradicting the assumption on $\Fil_i^{\max}D_\sigma$. Hence we must again have $\lambda_j(c_I)=0$. Thus we have
\begin{equation}\label{eq:tDbis}
t_{D_{\sigma}}(c_I)(e_j)\in \bigoplus_{j'\ne j}Ee_{j'}.
\end{equation}
Let $j$ such that $\varphi_j\notin I$ and assume there exists $j'\ne j$ such that $\varphi_{j'}\notin I$ and the coefficient of $e_{j'}$ in (\ref{eq:tDbis}) is non-zero. As above let $x'$ a non-zero vector in $\bigwedge\nolimits_E^{\!n-i-1}\!\Fil^{-h_{i,\sigma}}(D_\sigma)^{(j')}\subseteq \bigwedge\nolimits_E^{\!n-i-1}\Fil^{-h_{i,\sigma}}(D_\sigma)$. Since the coefficient of $e_{I^c}$ in $x'\wedge e_j$ is $0$ (as $e_{j'}$ is missing), we have $x'\wedge t_{D_{\sigma}}(c_I)(e_j)=0$ in $\bigwedge\nolimits_E^{\!n-i}\!D_\sigma$ (again the second sentence of the proof of (i)). However, as $e_{j'}$ appears in $t_{D_{\sigma}}(c_I)(e_j)$ by assumption but $x'\wedge e_{j'}\ne 0$ (since $e_{j'}$ is missing in $x'$), we necessarily have $x'\wedge t_{D_{\sigma}}(c_I)(e_j)\ne 0$, a contradiction. This finishes the proof of (ii).\bigskip

We prove (iii). We denote by
\begin{equation}\label{eq:prI}
\mathrm{pr}_I:\Fil_i^{\max}D_\sigma\buildrel\sim\over \longrightarrow Ee_{I^c}
\end{equation}
the composition $\Fil_i^{\max}D_\sigma\hookrightarrow \bigwedge\nolimits_E^{\!n-i}\!D_\sigma \twoheadrightarrow Ee_{I^c}$ where the surjection is the canonical projection sending all $Ee_{J^c}\subset \bigwedge\nolimits_E^{\!n-i}\!D_\sigma$ to $0$ for $J\ne I$. We then have the following commutative diagram of $1$-dimensional $E$-vector spaces where, for each (non-obvious) arrow, we indicate the corresponding reference and where we write $C(I)$ instead of $C(I, s_{i,\sigma})$, $\Fil_i^{\max}$ instead of $\Fil_i^{\max}D_\sigma$ and $\otimes$ instead of $\otimes_E$:
\begin{equation}\label{eq:big}
\begin{gathered}
\xymatrix{
\Hom_{\sigma}(\GL_{n-i}(\cO_K),E)\ar^{\!\wr\ (\ref{eq:isoindi})}[d]\ar^{\ \ \stackrel{(\ref{eq:isoindi})}{\sim}}[r] & \Ext^1_{\sigma}(\pi_{\alg},\pi_I/\pi_{\alg})\ar^{\sim}[rd]&\\
\Ext^1_{\sigma}(\pi_{\alg},V_I/\pi_{\alg})\ar^{\ \ \ \stackrel{(\ref{eq:VipiI})}{\sim}}[ru]\ar^{\!\!\!\!\!\!\!\!\!\!\stackrel{(\ref{eq:epsilonI})}{\sim}}[r] & \Ext^1_{\sigma}(\pi_{\alg},C(I)\!\otimes \!Ee_{I^c})\ar^{\!\wr\ (\ref{eq:isodual})}[d] & \Ext^1_{\sigma}(\pi_{\alg},C(I)\!\otimes\!\Fil_i^{\max})\ar^{\!\!\stackrel{\sim}{(\ref{eq:prI})}}[l]\ar^{\!\wr\ (\ref{eq:isodual})}[d]\\
& \Ext^1_{\sigma}(C(I),\pi_{\alg})^\vee\!\otimes\! Ee_{I^c}& \Ext^1_{\sigma}(C(I),\pi_{\alg})^{\vee}\!\otimes\!\Fil_i^{\max} \ar^{\!\!\stackrel{\sim}{(\ref{eq:prI})}}[l]\\
& \Hom_E(Ee_{I^c},Ee_{I^c})\ar^{(\ref{eq:epsilonI})^{\!\vee}\!\otimes\id\ \wr\!}[u] & \Hom_E(Ee_{I^c},\Fil_i^{\max}).\ar^{\!\!\!\!\stackrel{\sim}{(\ref{eq:prI})}}[l]\ar^{(\ref{eq:epsilonI})^{\!\vee}\!\otimes\id\ \wr\!}[u]
}
\end{gathered}
\end{equation}
Moreover it follows from (the discussion below) (\ref{mult:kappa2}) that, in the diagram (\ref{eq:big}), the image of $\log\in \Hom_{\sigma}(\GL_{n-i}(\cO_K),E)$ in $\Hom_E(Ee_{I^c},Ee_{I^c})$ is the identity (note that the choices for $\varepsilon_I$ and $\varepsilon_I^\vee$ cancel each other). Note also that the right part of (\ref{eq:big}) is the restriction of (\ref{eqn:isoI}) to $\Ext^1_{\sigma}(\pi_{\alg},\pi_I/\pi_{\alg})$. Denote by
\begin{equation}\label{eq:eIcstar}
e_{I^c}^\star\in \Hom_E(Ee_{I^c},\Fil_i^{\max}D_\sigma)\subset \Hom_E\big(\bigwedge\nolimits_E^{\!n-i}\!D_\sigma, \Fil_i^{\max}D_\sigma\big)
\end{equation}
the inverse image of $\id \in \Hom_E(Ee_{I^c},Ee_{I^c})$ under (\ref{eq:prI}). In particular we have
\begin{equation}\label{eq:eIc}
e_{I^c}^\star(e_J)=0\ \mathrm{for}\ J\ne I^c\ \ \mathrm{and}\ \ e_{I^c}^\star(e_{I^c})-e_{I^c}\in \bigoplus_{J\ne I^c}Ee_J\subset \bigwedge\nolimits_E^{\!n-i}\!D_\sigma.
\end{equation}
Also the image of $e_{I^c}^\star$ by the composition (\ref{mult:wedge2})${}^{-1}\circ$(\ref{mult:wedge}) is $t_{D_{\sigma}}(\log)\in \Hom_E(D_\sigma,\Fil^{-h_{i,\sigma}}(D_\sigma))$ where here $\log\in \Ext^1_{\sigma}(\pi_{\alg},\pi_I/\pi_{\alg})$ is as in the statement of (iii). Hence we have to prove
\begin{equation}\label{eq:nuI}
t_{D_{\sigma}}(\log)(e_j)-e_j\in \bigoplus_{\varphi_{j'}\in I}Ee_{j'}\ \mathrm{when}\ \varphi_j\notin I.
\end{equation}
We first claim that when $\varphi_j\notin I$ we have
\begin{equation*}
\dim_E\Fil^{-h_{i,\sigma}}(D_\sigma)^{(j)}=n-i-1
\end{equation*}
where $\Fil^{-h_{i,\sigma}}(D_\sigma)^{(j)}$ is as in (\ref{eq:fil^j}). Indeed, otherwise we would have $\Fil^{-h_{i,\sigma}}(D_\sigma) \subset \bigoplus_{j'\ne j}Ee_{j'}$ and thus the coefficient of $e_{I^c}=\pm e_j\wedge (\wedge_{\varphi_j'\in I^c\setminus\{\varphi_j\}}e_{j'})$ in $\Fil_i^{\max}D_\sigma=\bigwedge\nolimits_E^{\!n-i}\Fil^{-h_{i,\sigma}}(D_\sigma)$ would be $0$, contradicting the assumption. Let $\Fil_i^{\max,(j)}\!D_\sigma:=\bigwedge\nolimits_E^{\!n-i-1}\Fil^{-h_{i,\sigma}}(D_\sigma)^{(j)}$ which is a line in $\bigwedge\nolimits_E^{\!n-i-1}\Fil^{-h_{i,\sigma}}(D_\sigma)$. Since the coefficient of $e_{I^c}$ in $\bigwedge\nolimits_E^{\!n-i}\Fil^{-h_{i,\sigma}}(D_\sigma)$ is non-zero, it follows that the coefficient of $e_{I^c\setminus\{\varphi_j\}}$ in $\Fil_i^{\max,(j)}\!D_\sigma$ is also non-zero and hence that for any non-zero $x\in \Fil_i^{\max,(j)}\!D_\sigma$ we have inside $\bigwedge\nolimits_E^{\!n-i}\!D_\sigma$
\begin{equation}\label{eq:non0}
x\wedge e_j\ \notin\ \bigoplus_{J\ne I^c}Ee_J. 
\end{equation}
By the definition of $t_{D_{\sigma}}(\log)$ in (i) of Lemma \ref{lem:surj}, for any $x\in \Fil_i^{\max,(j)}\!D_\sigma$ we have $e_{I^c}^\star(x \wedge e_j)=x\wedge t_{D_{\sigma}}(\log)(e_j)$ in $\Fil_i^{\max}D_\sigma$. By (\ref{eq:eIc}) for any $x\in \Fil_i^{\max,(j)}\!D_\sigma$ we also have
\[e_{I^c}^\star(x \wedge e_j)-x \wedge e_j\in \bigoplus_{J\ne I^c}Ee_J,\]
hence we obtain for any $x\in \Fil_i^{\max,(j)}\!D_\sigma$
\[x\wedge(t_{D_{\sigma}}(\log)(e_j)-e_j)\in \bigoplus_{J\ne I^c}Ee_J.\]
By (\ref{eq:non0}) (and since $x\wedge e_{j'}\in \bigoplus_{J\ne I^c}Ee_J$ for any $j'\ne j$), this already forces
\begin{equation}\label{eq:tD}
t_{D_{\sigma}}(\log)(e_j)-e_j\in \bigoplus_{j'\ne j}Ee_{j'}\text{ \ when \ }\varphi_j\notin I. 
\end{equation}
Let $j$ such that $\varphi_j\notin I$ and assume there exists $j'\ne j$ such that $\varphi_{j'}\notin I$ and the coefficient of $e_{j'}$ in (\ref{eq:tD}) is non-zero. By (\ref{eq:dim}) we have for $j''$ such that $\varphi_{j''}\notin I$ 
\[\dim_E \Big(\Fil^{-h_{i,\sigma}}(D_{\sigma}) \cap \Big(E e_{j''} \bigoplus \big(\bigoplus_{\varphi_k\in I} E e_k\big)\Big)\Big)\geq 1\]
and by Lemma \ref{lem:triveq} we have $\Fil^{-h_{i,\sigma}}(D_{\sigma}) \cap (\bigoplus_{\varphi_k\in I} E e_k)=0$. Hence for each $j''$ such that $\varphi_{j''}\notin I$ there exists a non-zero $f_{j''}$ such that
\begin{equation}\label{eq:hj}
f_{j''}=e_{j''}+\sum_{\varphi_k\in I} a_k e_k \in\Fil^{-h_{i,\sigma}}(D_{\sigma})
\end{equation}
and the elements $\{f_{j''},\ \varphi_{j''}\notin I\}$ form a basis $\Fil^{-h_{i,\sigma}}(D_{\sigma})$ (as they are obviously linearly independent). As $j\neq j'$ (and $\varphi_{j'}\notin I$), we see that the coefficient of $e_{I^c}$ must be $0$ in
\[\big(\wedge_{\substack{\varphi_{j''}\notin I \\ j''\neq j'}} f_{j''}\big)\wedge e_j\in \Big(\bigwedge\nolimits_E^{\!n-i-1}\Fil^{-h_{i,\sigma}}(D_{\sigma})\Big) \wedge D_\sigma \subset \bigwedge\nolimits_E^{\!n-i}D_{\sigma}.\]
By the second sentence in the proof of (i) it follows that
\[\big(\wedge_{\substack{\varphi_{j''}\notin I \\ j''\neq j'}} f_{j''}\big)\wedge t_{D_{\sigma}}(\log)(e_j) = 0\text{ in }\bigwedge\nolimits_E^{\!n-i}\!D_\sigma.\]
But, as $e_{j'}$ never occurs in (\ref{eq:hj}) when $j''\ne j'$, by (\ref{eq:tD}) and the assumption on $j'$ the coefficient of $e_{I^c}$ in $\big(\wedge_{\substack{\varphi_{j''}\notin I \\ j''\neq j'}} f_{j''}\big)\wedge t_{D_{\sigma}}(\log)(e_j)$ must be non-zero, a contradiction. This proves (\ref{eq:nuI}).
\end{proof}

\begin{rem}\label{rem:forlater}
\hspace{2em}
\begin{enumerate}[label=(\roman*)]
\item
With the notation of Lemma \ref{lem:nuI} let $\Psi_I$ be the image of $c_I$ in $\Hom_E(\bigwedge\nolimits_E^{\!n-i}\!D_\sigma, \Fil_i^{\max}D_\sigma)$ by (\ref{eqn:isoI}). If the coefficient of $e_{I^c}$ in $\Fil_i^{\max}D_\sigma$ is $0$ we have $F_{\psi_I}(\Fil_i^{\max}D_\sigma)=0$. If the coefficient of $e_{I^c}$ in $\Fil_i^{\max}D_\sigma$ is non-zero, by the paragraph below (\ref{eq:big}) (and the sentence below (\ref{eq:log})) we have $F_{\psi_I}=\lambda_I(c_I)e_{I^c}^\star$ with $e_{I^c}^\star$ as in (\ref{eq:eIcstar}), i.e.~$F_{\psi_I}$ is the unique morphism $\bigwedge\nolimits_E^{\!n-i}\!D_\sigma\rightarrow \Fil_i^{\max}D_\sigma$ sending $e_J$ to $0$ if $J\ne I^c$ and $e_{I^c}$ to $\lambda_I(c_I)\lambda_{I^c}^{-1}v_i$ where $v_i$ is any non-zero vector in $\Fil_i^{\max}D_\sigma$ and $\lambda_{I^c}\in E^\times$ is the coefficient of $e_{I^c}$ in $v_i$. In particular we have $F_{\psi_I}(v_i)=\lambda_{I^c}(\lambda_I(c_I)\lambda_{I^c}^{-1}v_i)=\lambda_I(c_I)v_i$ for any $v_i\in \Fil_i^{\max}D_\sigma$.
\item
Arguing \ as \ in \ the \ proof \ of \ (i) \ of \ Lemma \ \ref{lem:nuI} \ we \ see \ that \ we \ have $t_{D_{\sigma}}(\Ext^1_{\GL_n(K),\sigma}(\pi_{\alg}(D_\sigma),\pi_I(D_\sigma)/\pi_{\alg}(D_\sigma)))=0$ if and only if the coefficient of $e_{I^c }$ is $0$ in any vector of $(\wedge_E^{\!n-i-1}\Fil^{-h_{i,\sigma}}(D_\sigma))\wedge D_\sigma$ if and only if for any $\varphi_j\notin I$ the coefficient of $e_{I^c \setminus \{\varphi_j\}}$ is $0$ in any vector of $\wedge_E^{\!n-i-1}\Fil^{-h_{i,\sigma}}(D_\sigma)$.
\end{enumerate}
\end{rem}
 
We are now ready to prove that $\pi(D_\sigma)$ does not depend on $\log(p)$.

\begin{prop}\label{prop:log(p)}
Up to isomorphism the representation $\pi(D_\sigma)$ of $\GL_n(K)$ over $E$ does not depend on the choice of $\log(p)\in E$.
\end{prop}
\begin{proof}
We write again $\pi_{\alg}$, $\pi_I$, $\pi_R$ instead of $\pi_{\alg}(D_\sigma)$, $\pi_I(D_\sigma)$, $\pi_R(D_\sigma)$ and $\Ext^1_{\sigma}$ instead of $\Ext^1_{\GL_n(K),\sigma}$. We fix isomorphisms $(\varepsilon_I)_I$ as in (\ref{eq:epsilonI}) and prove the stronger result that the $E$-vector subspace $\Ker(t_{D_\sigma})$ of $\Ext^1_{\sigma}(\pi_{\alg},\pi_R)$ does not depend on the choice of $\log(p)$.\bigskip

We fix the refinement $(\varphi_0,\varphi_1,\dots,\varphi_{n-1})$. By (\ref{mult:fix}) (for this fixed refinement) and by (\ref{eq:smoothsplit}) an element $c\in \Ext^1_{\sigma}(\pi_{\alg},\pi_R)$ can be written $c=c_{\sm} + c_{Z} + c_{\Fil}$ where $c_{\sm}\in \Hom_{\sm}(T(K),E)$, $c_{Z}\in \Hom_{\sigma}(\cO_K^\times,E) $ and $c_{\Fil}\in \bigoplus_I\Ext^1_{\sigma}(\pi_{\alg},\pi_I/\pi_{\alg})$. If $t_{D_\sigma}(c)=0$ in $\Ext^1_{\varphi^f}(D_\sigma,D_\sigma) \oplus \Hom_{\Fil}(D_\sigma,D_\sigma)$, it follows from Step $2$ and Step $3$ of the proof of Proposition \ref{prop:map} that we must have $t_{D_\sigma}(c_{\sm})=t_{D_\sigma}(c_Z+c_{\Fil})=0$. But $t_{D_\sigma}(c_Z)$ is a scalar endomorphism of $D_\sigma$ by (\ref{eq:scalar}) while it follows from (\ref{mult:surj}) that $t_{D_\sigma}(c_{\Fil})\in \Hom_{\Fil}(D_\sigma,D_\sigma)$ can never be a non-zero scalar endomorphism. Hence we must have $t_{D_\sigma}(c_Z)=0$, which implies $c_Z=0$ by (\ref{eq:scalar}), and $t_{D_\sigma}(c_{\Fil})=0$. Note that $t_{D_\sigma}(c_{\sm})=0$ also implies $c_{\sm}=0$ by the isomorphism before (\ref{eq:phiepsilon}). Hence we have $c=c_{\Fil}$.\bigskip

Write \ $c=\sum_I c_I$ \ with \ $c_I\in \Ext^1_{\sigma}(\pi_{\alg},\pi_I/\pi_{\alg})$, \ we \ thus \ have \ $t_{D_\sigma}(c_{\Fil})(e_j)=\sum_It_{D_\sigma}(c_I)(e_j)=0$ for $j\in \{0,\dots,n-1\}$. From (i), (ii), (iii) of Lemma \ref{lem:nuI} we deduce that we have in particular for each $j\in \{0,\dots,n-1\}$:
\begin{equation}\label{eq:crucial}
\sum_{\substack{\varphi_j\notin I\\I\ \mathrm{non-split}}}\!\!\!\!\!\!\lambda_I(c_I)=0
\end{equation}
where $I$ non-split means that $\pi_I$ is non-split (equivalently that the coefficient of $e_{I^c}$ in $\Fil_i^{\max}D_\sigma$ is non-zero). Now we apply Lemma \ref{lem:changelog}, noting that when our fixed refinement $(\varphi_0,\dots,\varphi_{n-1})$ is not compatible with a subset $I$, one needs to permute the (diagonal) coordinates $t_0,\dots,t_{n-1}$ of $T(K)$ to apply \emph{loc.~cit.} In the end, we see that replacing $\log(p)$ by $\log(p)'$ replaces $c$ by $\delta + c$ on the left hand side of (\ref{mult:fix}) (for the refinement $(\varphi_0,\dots,\varphi_{n-1})$) where $\delta\in \Hom_{\sm}(T(K),E)$ is the character
\begin{equation}\label{eq:subtil}
\begin{pmatrix}t_0 && \\ & \ddots & \\ && t_{n-1}\end{pmatrix}\in T(K) \longmapsto \sum_{j=0}^{n-1}\Bigg(\bigg(\!\!\!\sum_{\substack{\varphi_j\notin I\\I\ \mathrm{non-split}}}\!\!\!\!\!\!\lambda_I(c_I)\bigg)(\log(p)-\log(p)')\val(t_j)\Bigg)
\end{equation}
(note that the condition $\varphi_j\notin I$ comes here from the second factor $\GL_{n-i}(K)$ of $L_{P_i}(K)$ in Lemma \ref{lem:changelog}). 
By (\ref{eq:crucial}) we have $\delta=0$, which shows that the subspace $\Ker(t_{D_\sigma})$ of (\ref{mult:fix}) does not depend on $\log(p)$.
\end{proof}

\begin{rem}
Note that, contrary to the first statement of Proposition \ref{prop:epsilon}, the map $t_{D_\sigma}$ does depend on the choice of $\log(p)$.
\end{rem}

\begin{cor}\label{cor:ind}
The isomorphism class of the locally $\sigma$-analytic representation $\pi(D_\sigma)$ of $\GL_n(K)$ over $E$ of Definition \ref{def:pi(d)} does not depend on any choice.
\end{cor}
\begin{proof}
Let $\pi(D_\sigma)$ associated to $\{(\varepsilon_I)_I, \log(p)\}$ and $\pi(D_\sigma)'$ associated to $\{(\varepsilon_I')_I, \log(p)'\}$. Let also $\pi(D_\sigma)''$ associated to $\{(\varepsilon'_I)_I, \log(p)\}$. By Proposition \ref{prop:epsilon} $\pi(D_\sigma)$ is isomorphic to $\pi(D_\sigma)''$, and by Proposition \ref{prop:log(p)} $\pi(D_\sigma)''$ is isomorphic to $\pi(D_\sigma)'$.
\end{proof}

Finally we end up this section with the definition of the following locally $\Qp$-analytic representation of $\GL_n(K)$ over $E$:
\begin{equation}\label{eq:pi(D)}
\pi(D):= \bigoplus_{\sigma, \ \!\pi_{\alg}(D)}\big(\pi(D_\sigma)\otimes_E (\otimes_{\tau\ne \sigma}L(\lambda_\tau))\big)
\end{equation}
where the amalgamated sum is over $\sigma\in \Sigma$ and where $\pi_{\alg}(D)$ (see (\ref{eq:alg})) embeds into $\pi(D_\sigma)\otimes_E (\otimes_{\tau\ne \sigma}L(\lambda_\tau))$ via the composition $\pi_{\alg}(D_\sigma)\hookrightarrow \pi_R(D_\sigma)\hookrightarrow\pi(D_\sigma)$ (deduced from (\ref{eq:amalg4}) for $S=R$ and Definition \ref{def:pi(d)}) tensored by $\otimes_{\tau\ne \sigma}L(\lambda_\tau)$.

\subsection{Some properties of \texorpdfstring{$\pi(D_\sigma)$}{piDsigma} and \texorpdfstring{$\pi(D)$}{piD}}\label{sec:property}

We prove several properties of the representations $\pi(D_\sigma)$ and $\pi(D)$, in particular we prove that $\pi(D_\sigma)$ determines the isomorphism class of the filtered $\varphi^f$-module $D_\sigma$ (Theorem \ref{thm:fil}).\bigskip

We keep the notation of \S\S~\ref{sec:prel}, \ref{sec:def} and denote by $\sigma\in \Sigma$ an arbitrary embedding. By Lemma \ref{lem:nonsplit} and (\ref{mult:fix}) we have $\dim_E\Ext^1_{\GL_n(K),\sigma}(\pi_{\alg}(D_\sigma),\pi_R(D_\sigma))=2^n+n-1$. Since $\dim_E\Ext^1_{\varphi^f}(D_\sigma,D_\sigma)=n$, $\dim_E \Hom_{\Fil}(D_\sigma,D_\sigma)=n(n+1)/2$ and the map $t_{D_\sigma}$ in Proposition \ref{prop:map} is surjective, we deduce
\[\dim_E\Ker(t_{D_\sigma}) = 2^n - 1 - \frac{n(n+1)}{2}.\]
Hence, from Definition \ref{def:pi(d)}, we see that the representation $\pi(D_\sigma)$ has the following form:
\begin{equation}\label{eq:form}
\pi_{\alg}(D_\sigma)\begin{xy} (0,0)*+{}="a"; (15,0)*+{}="b"; {\ar@{-}"a";"b"}\end{xy}\Big(\bigoplus_{I} \big(C(I, s_{\vert I\vert,\sigma})\otimes_E\Fil_{\vert I\vert}^{\max}D_\sigma\big)\Big) \begin{xy} (30,0)*+{}="a"; (45,0)*+{}="b"; {\ar@{-}"a";"b"}\end{xy} \big(\pi_{\alg}(D_\sigma)^{\oplus 2^n-1-\frac{n(n+1)}{2}}\big)
\end{equation}
from which one deduces an analogous form for $\pi(D)$ by (\ref{eq:pi(D)}).\bigskip

From now on $S$ denotes a (possibly empty) subset of the set $R$ of simple reflections of $\GL_n$. When $S\ne \emptyset$ recall $\pi_S(D_\sigma)$ is defined in (\ref{eq:amalg4}) and when $S=\emptyset$ we set $\pi_{\emptyset}(D_\sigma):=\pi_{\alg}(D_\sigma)$. We let $P_S\subseteq \GL_n$ be the parabolic subgroup (over $K$) containing $B$ with corresponding simple roots $\{\alpha, s_\alpha\in S\}$ and $r_{P_S}$ the \emph{full} radical subgroup of $P_S$ (hence to compare with the notation $P_i$ before Proposition \ref{prop:isoext} we have $P_i=P_{R\setminus\{s_i\}}$). Recall that the injection $\pi_S(D_\sigma)\hookrightarrow \pi_R(D_\sigma)$ deduced from (\ref{eq:amalg4}) induces an injection $\Ext^1_{\GL_n(K),\sigma}(\pi_{\alg}(D_\sigma),\pi_S(D_\sigma))\hookrightarrow \Ext^1_{\GL_n(K),\sigma}(\pi_{\alg}(D_\sigma),\pi_R(D_\sigma))$ analogous to the injections in (\ref{eq:allinj}), hence we can consider the restriction of the map $t_{D_\sigma}$ to $\Ext^1_{\GL_n(K),\sigma}(\pi_{\alg}(D_\sigma),\pi_S(D_\sigma))$. In particular, replacing everywhere $\pi_R(D_\sigma)$ by $\pi_S(D_\sigma)$ in Definition \ref{def:pi(d)} we denote by
\[\pi(D_\sigma)(S)\]
the representation of $\GL_n(K)$ over $E$ associated to the image in
\[\Ext^1_{\GL_n(K),\sigma}\Big(\pi_{\alg}(D_\sigma)\otimes_E\Ker\big(t_{D_\sigma}\vert_{\Ext^1_{\GL_n(K),\sigma}(\pi_{\alg}(D_\sigma),\pi_S(D_\sigma))}\big), \pi_S(D_\sigma)\Big)\]
of the canonical vector of $\Ext^1_{\GL_n(K),\sigma}(\pi_{\alg}(D_\sigma),\pi_S(D_\sigma))\otimes_E\Ext^1_{\GL_n(K),\sigma}(\pi_{\alg}(D_\sigma),\pi_S(D_\sigma))^\vee$. So we have $\pi(D_\sigma)(R)=\pi(D_\sigma)$ and (by Step $2$ in the proof of Proposition \ref{prop:map}) $\pi(D_\sigma)(\emptyset)=\pi_{\alg}(D_\sigma)$. We also denote by $\widetilde\pi_{S}(D_\sigma)$ the representation of $\GL_n(K)$ over $E$ associated to the image of the canonical vector of the source by the map
\begin{multline*}
\Ext^1_{\GL_n(K),\sigma}(\pi_{\alg}(D_\sigma),\pi_S(D_\sigma))\otimes_E\Ext^1_{\GL_n(K),\sigma}(\pi_{\alg}(D_\sigma),\pi_S(D_\sigma))^\vee\\
\buildrel\sim\over\longrightarrow \Ext^1_{\GL_n(K),\sigma}\big(\pi_{\alg}(D_\sigma)\otimes_E\Ext^1_{\GL_n(K),\sigma}(\pi_{\alg}(D_\sigma),\pi_S(D_\sigma)), \pi_S(D_\sigma)\big).
\end{multline*}
By construction $\pi(D_\sigma)(S)$ is the pull-back of $\widetilde\pi_{S}(D_\sigma)$ along the canonical injection
\[\pi_{\alg}(D_\sigma)\otimes_E\Ker(t_{D_\sigma}\vert_{\Ext^1_{\GL_n(K),\sigma}(\pi_{\alg}(D_\sigma),\pi_S(D_\sigma))})\hookrightarrow \pi_{\alg}(D_\sigma)\otimes_E\Ext^1_{\GL_n(K),\sigma}(\pi_{\alg}(D_\sigma),\pi_S(D_\sigma)).\]

\begin{lem}\label{lem:max}
The representation $\pi(D_\sigma)(S)$ is isomorphic to the maximal subrepresentation of $\pi(D_\sigma)$ which does not contain any $C(I,s_{i,\sigma})$ for $s_i\notin S$ in its Jordan-H\"older constituents.
\end{lem}
\begin{proof}
One can check that there is a commutative diagram of short exact sequences
\[\xymatrix{
0\ar[r] & \pi_S(D_\sigma)\ar@{^{(}->}[d]\ar[r] & \widetilde\pi_{S}(D_\sigma)\ar@{^{(}->}[d]\ar[r] & \pi_{\alg}(D_\sigma)\otimes_E \Ext^1_{\GL_n(K),\sigma}(\pi_{\alg}(D_\sigma),\pi_S(D_\sigma))\ar@{^{(}->}[d]\ar[r] & 0\\
0\ar[r] & \pi_R(D_\sigma)\ar[r] & \widetilde\pi_{R}(D_\sigma)\ar[r] & \pi_{\alg}(D_\sigma)\otimes_E \Ext^1_{\GL_n(K),\sigma}(\pi_{\alg}(D_\sigma),\pi_R(D_\sigma))\ar[r] & 0
}\]
where the vertical injection on the right is $\id \otimes$(injection induced by $\pi_S(D_\sigma)\hookrightarrow \pi_R(D_\sigma)$). The pull-back of the top (resp.~bottom) line induced by $\Ker(t_{D_\sigma}\vert_{\Ext^1_{\GL_n(K),\sigma}(\pi_{\alg}(D_\sigma),\pi_S(D_\sigma))})\hookrightarrow \Ext^1_{\GL_n(K),\sigma}(\pi_{\alg}(D_\sigma),\pi_S(D_\sigma))$ (resp.~$\Ker(t_{D_\sigma})\hookrightarrow \Ext^1_{\GL_n(K),\sigma}(\pi_{\alg}(D_\sigma),\pi_R(D_\sigma))$) is $\pi(D_\sigma)(S)$ (resp.~$\pi(D_\sigma)(R)=\pi(D_\sigma)$). In particular we have:
\begin{equation*}
\pi(D_\sigma)(S)\buildrel\sim\over\longrightarrow \widetilde\pi_{S}(D_\sigma)\cap \pi(D_\sigma)
\end{equation*}
where the intersection on the right hand side is inside $\widetilde\pi_{R}(D_\sigma)$. This is precisely the maximal subrepresentation of $\pi(D_\sigma)$ which does not contain the $C(I,s_{i,\sigma})$ for $i\notin S$.
\end{proof}

Note that Lemma \ref{lem:max} and Corollary \ref{cor:ind} imply that the isomorphism class of the representation $\pi(D_\sigma)(S)$ does not depend on any choice (this can also be checked directly as for $\pi(D_\sigma)$).

\begin{lem}\label{lem:multS}
We have
\[\dim_E\Ker\big(t_{D_\sigma}\vert_{\Ext^1_{\GL_n(K),\sigma}(\pi_{\alg}(D_\sigma),\pi_S(D_\sigma))}\big)=\Big(\sum_{s_i\in S}\binom{n}{i}\Big) + 1 - \dim(r_{P_{S^c}}).\]
\end{lem}
\begin{proof}
When $S=\emptyset$ the statement holds since both sides give $0$, so we can assume $S\ne \emptyset$. By Lemma \ref{lem:nonsplit} and the analogue of (\ref{mult:fix}) for $\pi_S(D_\sigma)$ we have
\begin{equation}\label{eq:dimS}
\dim_E\Ext^1_{\GL_n(K),\sigma}(\pi_{\alg}(D_\sigma),\pi_S(D_\sigma))=n+1 + \sum_{s_i\in S}\binom{n}{i}.
\end{equation}
By the analogue of (\ref{mult:fix}) for $\pi_S(D_\sigma)$ together with (\ref{eq:amalg}), (\ref{mult:step2}) and (\ref{mult:surj}) we have a surjection
\begin{multline}\label{mult:surjS}
t_{D_\sigma}\vert_{\Ext^1_{\GL_n(K),\sigma}(\pi_{\alg}(D_\sigma),\pi_S(D_\sigma))}:\Ext^1_{\GL_n(K),\sigma}(\pi_{\alg}(D_\sigma),\pi_S(D_\sigma))\twoheadrightarrow \big(\Ext^1_{\varphi^f}(D_\sigma,D_\sigma) \oplus E\big) \bigoplus\\
\sum_{s_i\in S}\left\{f\in \Hom_E(D_\sigma, \Fil^{-h_{i,\sigma}}(D_\sigma)),\ f\vert_{\Fil^{-h_{i,\sigma}}(D_\sigma)}\ {\mathrm{scalar}}\right\}
\end{multline}
where the sum on the right hand side is inside $\Hom_{\Fil}(D_\sigma,D_\sigma)$. By the proof of (i) of Lemma \ref{lem:surj} and a straightforward computation we have
\[\dim_E\left\{f\in \Hom_E(D_\sigma, \Fil^{-h_{i,\sigma}}(D_\sigma)),\ f\vert_{\Fil^{-h_{i,\sigma}}(D_\sigma)}\ {\mathrm{scalar}}\right\}=1+i(n-i)=\dim(r_{P_{R\setminus\{s_i\}}})-1\]
from which it is easy to deduce
\[\dim_E\sum_{s_i\in S}\left\{f\in \Hom_E(D_\sigma, \Fil^{-h_{i,\sigma}}(D_\sigma)),\ f\vert_{\Fil^{-h_{i,\sigma}}(D_\sigma)}\ {\mathrm{scalar}}\right\} = \dim(r_{P_{S^c}}) - 1.\]
By (\ref{eq:dimS}) and (\ref{mult:surjS}) the statement follows.
\end{proof}

From Lemma \ref{lem:multS} and the definition of $\pi(D_\sigma)(S)$, just as in (\ref{eq:form}) we deduce that the subrepresentation $\pi(D_\sigma)(S)$ of $\pi(D_\sigma)$ has the following form
{\small
\begin{equation}\label{eq:formS}
\pi_{\alg}(D_\sigma)\!\begin{xy} (0,0)*+{}="a"; (8,0)*+{}="b"; {\ar@{-}"a";"b"}\end{xy}\!\!\Big(\!\bigoplus_{I\mathrm{\ \!s.t.\!\ }s_{\vert I\vert}\in S} \!\!\!\big(C(I, s_{\vert I\vert,\sigma})\otimes_E\Fil_{\vert I\vert}^{\max}D_\sigma\!\big)\Big) \!\begin{xy} (30,0)*+{}="a"; (38,0)*+{}="b"; {\ar@{-}"a";"b"}\end{xy}\big(\pi_{\alg}(D_\sigma)^{\oplus(\sum_{s_i\in S}\binom{n}{i}) + 1 - \dim(r_{P_{S^c}})}\big)\!.
\end{equation}}

\begin{ex}
\hspace{2em}
\begin{enumerate}[label=(\roman*)]
\item
When $S=\{s_j\}$ ($j\in \{1,\dots,n-1\}$) we have $(\sum_{s_i\in S}\binom{n}{i}) + 1 - \dim(r_{P_{S^c}})=\binom{n}{j}-1-j(n-j)$ in (\ref{eq:formS}), hence in that case $\pi(D_\sigma)(S)$ has the form
\begin{multline*}
\pi_{\alg}(D_\sigma)\!\begin{xy} (0,0)*+{}="a"; (10,0)*+{}="b"; {\ar@{-}"a";"b"}\end{xy}\!\Big(\bigoplus_{\vert I\vert =j}\big(C(I, s_{j,\sigma})\otimes_E\Fil_j^{\max}D_\sigma\big) \Big)\!\begin{xy} (30,0)*+{}="a"; (40,0)*+{}="b"; {\ar@{-}"a";"b"}\end{xy}\!\big(\pi_{\alg}(D_\sigma)^{\oplus\binom{n}{j} -1-j(n-j)}\big)\\
\cong \pi_{s_j}(D_\sigma)\!\begin{xy} (0,0)*+{}="a"; (10,0)*+{}="b"; {\ar@{-}"a";"b"}\end{xy}\!\big(\pi_{\alg}(D_\sigma)^{\oplus\binom{n}{j} -1-j(n-j)}\big)
\end{multline*}
where $\pi_{s_j}(D_\sigma)$ is in (\ref{mult:decomp}). Note that $\binom{n}{j}-1-j(n-j)>0$ if and only if $j\notin \{1,n-1\}$. 
\item
When $S=\{s_1,s_{n-1}\}$ and $n\geq 3$ we have $(\sum_{s_i\in S}\binom{n}{i}) + 1 - \dim(r_{P_{S^c}})=1$ in (\ref{eq:formS}), hence in that case $\pi(D_\sigma)(S)$ has the form
\[\begin{xy} (0,0)*+{\big(\pi_{s_1}(D_\sigma)\bigoplus_{\pi_{\alg}(D_\sigma)}\pi_{s_{n-1}}(D_\sigma)\big)}="a"; (45,0)*+{\pi_{\alg}(D_\sigma).}="c"; {\ar@{-}"a";"c"}\end{xy}\]
Together with (i) and Lemma \ref{lem:max} this implies that $(\sum_{s_i\in S}\binom{n}{i}) + 1 - \dim(r_{P_{S^c}})=0$ if and only if $S=\{s_1\}$ or $S=\{s_{n-1}\}$ (which can also be checked directly).
\end{enumerate}
\end{ex}

In the two propositions below we fix a \emph{non-empty} subset $S$ of $R$ and $i\in \{1,\dots,n-1\}$ such that $s_i\in S$. We have the following (surjective) composition
\begin{eqnarray}\label{eq:comp3}
\nonumber\Ext^1_{\GL_n(K),\sigma}(\pi_{\alg}(D_\sigma),\pi_S(D_\sigma))&\twoheadrightarrow &\Ext^1_{\GL_n(K),\sigma}(\pi_{\alg}(D_\sigma),\pi_S(D_\sigma)/\pi_{\alg}(D_\sigma))\\
\nonumber&\buildrel{\stackrel{(\ref{eq:amalg4})}{\sim}}\over\longrightarrow &\bigoplus_{s_j\in S}\Ext^1_{\GL_n(K),\sigma}(\pi_{\alg}(D_\sigma),\pi_{s_j}(D_\sigma)/\pi_{\alg}(D_\sigma))\\
\nonumber&\twoheadrightarrow &\Ext^1_{\GL_n(K),\sigma}(\pi_{\alg}(D_\sigma),\pi_{s_i}(D_\sigma)/\pi_{\alg}(D_\sigma))\\&\buildrel{\stackrel{(\ref{eqn:isoI})}{\sim}}\over\longrightarrow &\Hom_E\big(\bigwedge\nolimits_E^{\!n-i}\!D_\sigma, \Fil_i^{\max}D_\sigma\big)
\end{eqnarray}
where the surjectivity of the first (canonical) map follows from the analogue of (\ref{mult:fix}) for $\pi_S(D_\sigma)$ and \ \ where \ \ the \ \ third \ \ map \ \ is \ \ the \ \ canonical \ \ projection \ \ sending \ \ all $\Ext^1_{\GL_n(K),\sigma}(\pi_{\alg}(D_\sigma),\pi_{s_j}(D_\sigma)/\pi_{\alg}(D_\sigma))$ to $0$ for $j\ne i$. Recall that the last isomorphism in (\ref{eq:comp3}) depends on the choice of isomorphisms $(\varepsilon_I)_I$ as in (\ref{eq:epsilonI}), which we tacitly fix all along.

\begin{prop}\label{prop:inf}
With \ the \ above \ notation, \ the \ image \ under \ (\ref{eq:comp3}) \ of \ the \ subspace $\Ext^1_{\GL_n(K),\sigma,\mathrm{inf}}(\pi_{\alg}(D_\sigma),\pi_S(D_\sigma))$ is
\[\Hom_E\big((\bigwedge\nolimits_E^{\!n-i}\!D_\sigma)/\Fil_i^{\max}D_\sigma, \Fil_i^{\max}D_\sigma\big)\]
(in particular it does not depend on any choice).
\end{prop}
\begin{proof}
We write $\pi_{\alg}$, $\pi_{s_i}$, $\pi_S$ instead of $\pi_{\alg}(D_\sigma)$, $\pi_{s_i}(D_\sigma)$, $\pi_S(D_\sigma)$ and $\Ext^1_{\sigma}$, $\Ext^1_{\sigma,\mathrm{inf}}$ instead of $\Ext^1_{\GL_n(K),\sigma}$, $\Ext^1_{\GL_n(K),\sigma,\mathrm{inf}}$. The fact that the image of $\Ext^1_{\sigma,\mathrm{inf}}(\pi_{\alg},\pi_S)$ under (\ref{eq:comp3}) does not depend on $(\varepsilon_I)_I$ (and thus does not depend on any choice) directly follows from an examination of the proof of Proposition \ref{prop:epsilon}, in particular (\ref{eq:smallI}).\bigskip

\textbf{Step $1$}: We give preliminaries.\\
For $i\in \{1,\dots,n-1\}$ denote by $\pi_{s_i}^{\mathrm{ns}}(D_\sigma)=\pi_{s_i}^{\mathrm{ns}}$ the direct summand on the left hand side of (\ref{mult:decomp}) (``ns'' for ``non-split'') and define similarly to (\ref{eq:amalg4})
\begin{equation}\label{eq:ns}
\pi_S^{\mathrm{ns}}(D_\sigma)=\pi_S^{\mathrm{ns}}:= \bigoplus_{s_i\in S,\pi_{\alg}}\pi_{s_i}^{\mathrm{ns}}
\end{equation}
which is a direct summand of $\pi_S$. We have
\[\Ext^1_{\sigma}(\pi_{\alg},\pi_S) \!= \!\Ext^1_{\sigma}(\pi_{\alg},\pi_S^{\rm ns})\bigoplus \Ext^1_{\sigma}(\pi_{\alg},\pi_S/\pi_S^{\rm ns})\]
and moreover $\Ext^1_{\sigma}(\pi_{\alg},\pi_S/\pi_S^{\rm ns})$ (trivially) lies in $\Ext^1_{\sigma,\mathrm{inf}}(\pi_{\alg},\pi_S)$ (recall any extension of $\pi_{\alg}$ by any $C(I,s_{i,\sigma})$ has an infinitesimal character as the constituents are distinct). By (\ref{mult:fix}) (for $S$ instead of $R$) with (\ref{eq:smoothsplit}) and (\ref{eq:isoindi}) we have an isomorphism
\begin{multline}\label{mult:isoLpi3}
\Hom_{\sm}(T(K),E)\ \bigoplus \ \Hom_{\sigma}(\cO_K^\times,E) \ \bigoplus \ \Big(\!\bigoplus_{\substack{s_{\vert I\vert}\in S\\ I\textrm{ \!non-split}}}\Hom_{\sigma}(\GL_{n-\vert I\vert}(\cO_K),E)\Big)\\
\buildrel\sim\over\longrightarrow \Ext^1_{\sigma}(\pi_{\alg},\pi_S^{\rm ns}).
\end{multline}
(Recall that the restriction of (\ref{mult:isoLpi3}) to the first direct summand depends on the choice of a refinement, see (\ref{eq:restr}), and that its restriction to the second direct summand depends on choices of $(\varepsilon_I)_I$ in (\ref{eq:epsilonI}) via the definition of $\pi_R^{\mathrm{ns}}(D_\sigma)$ and of $\log(p)\in E$. We tacitly make such choices, which won't impact the proof.) Let $\Psi=\psi_{\sm}+ \psi + \sum_{I}\psi_I$ be an element in the left hand side of (\ref{mult:isoLpi3}) (with obvious notation) and $\pi(\Psi)$ a representative of its image in $\Ext^1_{\sigma}(\pi_{\alg},\pi_S^{\rm ns})$ by (\ref{mult:isoLpi3}). Let ${\mathcal Z}_\sigma$ the center of the enveloping algebra $U({\mathfrak g}_\sigma)$ and $\xi:{\mathcal Z}_\sigma\rightarrow E$ the (common) infinitesimal character of $\pi_{\alg}
$ and $\pi_S^{\mathrm{ns}}$. The image of $\Psi$ in $\Ext^1_{\sigma}(\pi_{\alg},\pi_S^{\rm ns})$ lies in $\Ext^1_{\sigma,\mathrm{inf}}(\pi_{\alg},\pi_S^{\rm ns})$ if and only if $z-\xi(z)$ acts by $0$ on $\pi(\Psi)$ for all $z\in {\mathcal Z}_\sigma$.\bigskip

\textbf{Step $2$}: We give necessary and sufficient conditions for an element in $\Ext^1_{\sigma}(\pi_{\alg},\pi_S)$ to lie in $\Ext^1_{\sigma,\mathrm{inf}}(\pi_{\alg},\pi_S)$.\\
By Step $1$ we can replace $\Ext^1_{\sigma}(\pi_{\alg},\pi_S)$ by $\Ext^1_{\sigma}(\pi_{\alg},\pi_S^{\mathrm{ns}})$. Recall that the action of ${\mathcal Z}_\sigma$ commutes with the action of $\GL_n(K)$ (\cite[Prop.~3.7]{ST02}) and that we have an embedding ${\mathcal Z}_\sigma\hookrightarrow U({\mathfrak t}_\sigma)$ (the Harish-Chandra homomorphism). Write $U({\mathfrak t}_\sigma)=E[{\mathfrak t}_{0,\sigma},\dots,{\mathfrak t}_{n-1,\sigma}]\cong E[{\mathfrak t}_{0,\sigma}-\xi({\mathfrak t}_{0,\sigma}),\dots,{\mathfrak t}_{n-1,\sigma}-\xi({\mathfrak t}_{n-1,\sigma})]$ where ${\mathfrak t}_{j,\sigma}\in \textrm{M}_n(K)$ has entries $1$ in coordinate $(j+1,j+1)$ and $0$ elsewhere. Let $z\in {\mathcal Z}_\sigma$ then $z-\xi(z)$ can be written
\begin{equation*}
z-\xi(z)=\sum_{j=0}^{n-1}\lambda_j(z)({\mathfrak t}_{j,\sigma}-\xi({\mathfrak t}_{j,\sigma})) + (\textrm{degree}\geq 2\textrm{ in the }{\mathfrak t}_{j,\sigma}-\xi({\mathfrak t}_{j,\sigma}))
\end{equation*}
for some $\lambda_j(z)\in E$. Write $\psi=\lambda(\psi)\sigma\circ \log$ and $\psi_I=\lambda_I(\psi_I)\sigma\circ \log\circ {\det}$ (see (\ref{eq:log})) with $\lambda(\psi),\lambda_I(\psi_I)\in E$, then it follows from the argument at the end of the proof of \cite[Prop.~3.26]{Di25} that $z-\xi(z)$ acts on $\pi(\Psi)$ by
\[\pi(\Psi)\twoheadrightarrow \pi_{\alg} \buildrel{\delta(z)}\over\longrightarrow \pi_{\alg}\hookrightarrow \pi_S^{\mathrm{ns}}\hookrightarrow \pi(\Psi)\]
(the left surjection and the two right injections come from $\Ext^1_{\sigma}(\pi_{\alg},\pi_S^{\mathrm{ns}})$ and the definition of $\pi_S^{\mathrm{ns}}$) where (we only consider non-split $I$)
\begin{equation}\label{eq:zerocond}
\delta(z):=\sum_{j=0}^{n-1}\bigg(\lambda(\psi) + \sum_{\substack{\vert I\vert\leq j\\s_{\vert I\vert}\in S}}\lambda_I(\psi_I)\bigg)\lambda_j(z).
\end{equation}
Now, there exist elements $z_0,z_1,\dots,z_{n-1}$ in ${\mathcal Z}_\sigma$ such that the matrix $(\lambda_{j}(z_i))_{i,j}\in {\rm M}_n(E)$ lies in $\GL_n(E)$ (this follows from the isomorphism of tangent spaces $X_{\bf h}\buildrel\sim\over\rightarrow X_{\xi}$ in the proof of \cite[Prop.~3.26]{Di25}). Since $z-\xi(z)$ acts by $0$ on $\pi(\Psi)$ if and only if $\delta(z)=0$, we deduce from (\ref{eq:zerocond}) and the previous sentence the following necessary (and clearly sufficient) conditions for the image of $\Psi$ to lie in $\Ext^1_{\sigma,\mathrm{inf}}(\pi_{\alg},\pi_S^{\rm ns})$ (we again only consider non-split $I$ in the sums)
\begin{equation*}
\lambda(\psi) + \sum_{\substack{\vert I\vert\leq j\\s_{\vert I\vert}\in S}}\lambda_I(\psi_I)=0\textrm{\ \ for\ \ }j=0,\dots,n-1.
\end{equation*}
By an obvious induction this is equivalent to
\begin{equation}\label{eq:condinf}
\psi=0\textrm{ \ and }\sum_{\substack{\vert I\vert = j\\I\textrm{ \!non-split}}}\!\!\!\lambda_I(\psi_I)=0\textrm{\ \ for\ }j\textrm{ such that }s_j\in S.
\end{equation}

\textbf{Step $3$}: We prove the statement.\\
Let $\Psi \in \Hom_{\sm}(T(K),E) \bigoplus (\bigoplus_{s_{\vert I\vert}\in S}\Ext^1_{\sigma}(\pi_{\alg},\pi_I/\pi_{\alg}))$ such that its image by (\ref{mult:fix}) (for $S$ instead of $R$) lies in $\Ext^1_{\sigma,\mathrm{inf}}(\pi_{\alg},\pi_S)$, equivalently such that the conditions (\ref{eq:condinf}) are satisfied. Fix $i$ such that $s_i\in S$ and fix a basis $v_i$ of the $1$-dimensional $E$-vector space $\Fil_i^{\max}D_\sigma$ (the choice of which won't matter). Denote by $F_\Psi$ and $F_{\psi_I}$ the image of respectively $\Psi$ and $\psi_I$ in $\Hom_E(\bigwedge\nolimits_E^{\!n-i}\!D_\sigma, \Fil_i^{\max}D_\sigma)$ by (\ref{eq:comp3}) (note that $\psi_{\sm}$ maps to $0$). We obviously have $F_\Psi=\sum_{I}F_{\psi_I}$ and $F_{\psi_I}=0$ if $\vert I\vert \ne i$. By (i) of Remark \ref{rem:forlater} if $\vert I\vert = i$ we have $F_{\psi_I}(v_i)=0$ when $I$ is split and $F_{\psi_I}(v_i)=\lambda_I(\psi_I)v_i$ when $I$ is non-split. It follows that
\[F_\Psi(v_i)=\sum_IF_{\psi_I}(v_i)=\Big(\!\!\sum_{\substack{\vert I\vert = i\\I\textrm{ \!non-split}}}\!\!\!\lambda_I(\psi_I)\Big)v_i\buildrel{(\ref{eq:condinf})}\over =0,\]
i.e.~$F_\Psi\in \Hom_E((\bigwedge\nolimits_E^{\!n-i}\!D_\sigma)/\Fil_i^{\max}D_\sigma, \Fil_i^{\max}D_\sigma)$. The fact that the image of $\Ext^1_{\sigma,\mathrm{inf}}(\pi_{\alg},\pi_S)$ in $\Hom_E((\bigwedge\nolimits_E^{\!n-i}\!D_\sigma), \Fil_i^{\max}D_\sigma)$ is exactly $\Hom_E((\bigwedge\nolimits_E^{\!n-i}\!D_\sigma)/\Fil_i^{\max}D_\sigma, \Fil_i^{\max}D_\sigma)$ follows again easily from (\ref{eq:condinf}) as there are no other conditions on the $\psi_I$.
\end{proof}

\begin{rem}\label{rem:CNSinf}
It follows from the proof of Proposition \ref{prop:inf}, in particular Step $3$, that the image by (\ref{mult:fix}) (for $S$ instead of $R$) of an element
\[\Psi\in \Hom_{\sm}(T(K),E) \bigoplus \Big(\bigoplus_{s_{\vert I\vert}\in S}\Ext^1_{\GL_n(K),\sigma}\big(\pi_{\alg}(D_\sigma),\pi_I(D_\sigma)/\pi_{\alg}(D_\sigma)\big)\Big)\]
in $\Ext^1_{\GL_n(K),\sigma}(\pi_{\alg}(D_\sigma),\pi_S(D_\sigma))$ lies in the subspace $\Ext^1_{\GL_n(K),\sigma,\mathrm{inf}}(\pi_{\alg}(D_\sigma),\pi_S(D_\sigma))$ \emph{if and only if} $F_\Psi\in \Hom_E((\bigwedge\nolimits_E^{\!n-i}\!D_\sigma)/\Fil_i^{\max}D_\sigma, \Fil_i^{\max}D_\sigma)$ for all $i$ such that $s_i\in S$, where $F_\Psi$ is the image of $\Psi$ in $\Hom_E(\bigwedge\nolimits_E^{\!n-i}\!D_\sigma, \Fil_i^{\max}D_\sigma)$ by (\ref{eq:comp3}).
\end{rem}\bigskip

Recall we assumed $S\ne \emptyset$. The partial filtration $(\Fil^{-h_{j,\sigma}}(D_\sigma),s_j\in S)$ on $D_\sigma$ induces a natural decreasing filtration on $\bigwedge\nolimits_E^{\!n-i}\!D_\sigma$. For $i$ such that $s_i\in S$ we denote by
\[\Fil_{S,i}^{2^{\mathrm{nd}}\text{-}\max}D_\sigma\subseteq \bigwedge\nolimits_E^{\!n-i}\!D_\sigma\]
the \emph{one but last step} of this induced filtration. When $S=R$ (in which case $(\Fil^{-h_{j,\sigma}}(D_\sigma),s_j\in R)$ is the full filtration $\Fil^\bullet(D_\sigma)$) we just write $\Fil_{i}^{2^{\mathrm{nd}}\text{-}\max}D_\sigma$. In that case we have 
\begin{eqnarray}\label{eq:2nd}
\nonumber \Fil_{i}^{2^{\mathrm{nd}}\text{-}\max}D_\sigma&=&\Fil^{-h_{n-1,\sigma}}(D_\sigma)\wedge \Fil^{-h_{n-2,\sigma}}(D_\sigma)\wedge\cdots\wedge \Fil^{-h_{i+1,\sigma}}(D_\sigma)\wedge \Fil^{-h_{i-1,\sigma}}(D_\sigma)\\
&=&\big(\bigwedge\nolimits_E^{\!n-i-1}\Fil^{-h_{i+1,\sigma}}(D_\sigma)\big)\wedge \Fil^{-h_{i-1,\sigma}}(D_\sigma)
\end{eqnarray}
(and $\dim_E\Fil_{i}^{2^{\mathrm{nd}}\text{-}\max}D_\sigma=2$). More generally, writing $S=\{s_{i_1},s_{i_2},\dots,s_{i_{\vert S\vert}}\}$ with $i_j<i_{j+1}$ and setting $i_0:=0$, we have
\begin{equation}\label{eq:2ndS}
\Fil_{S,i_{\vert S\vert}}^{2^{\mathrm{nd}}\text{-}\max}D_\sigma=\big(\bigwedge\nolimits_E^{\!n-i_{\vert S\vert}-1}\Fil^{-h_{i_{\vert S\vert},\sigma}}(D_\sigma)\big)\wedge \Fil^{-h_{i_{{\vert S\vert}-1},\sigma}}(D_\sigma)
\end{equation}
(which has dimension $1+(n-i_{\vert S\vert})(i_{\vert S\vert}-i_{\vert S\vert-1})$) and for $j\in \{1,\dots,\vert S\vert-1\}$
\begin{equation}\label{eq:2ndSj}
\Fil_{S,i_j}^{2^{\mathrm{nd}}\text{-}\max}D_\sigma=\big(\bigwedge\nolimits_E^{\!n-i_{j+1}}\Fil^{-h_{i_{j+1},\sigma}}(D_\sigma)\big)\wedge 
\big(\bigwedge\nolimits_E^{\!i_{j+1}-i_j-1}\Fil^{-h_{i_{j},\sigma}}(D_\sigma)\big)\wedge \Fil^{-h_{i_{j-1},\sigma}}(D_\sigma)
\end{equation}
(which has dimension $1+(i_j- i_{j-1})(i_{j+1}-i_j)$).

\begin{rem}\label{rem:maxS}
\hspace{2em}
\begin{enumerate}[label=(\roman*)]
\item
When $\vert S\vert=1$, i.e.~$S=\{i_1\}$, note that we have in particular by (\ref{eq:2ndS})
\[\Fil_{S,i_1}^{2^{\mathrm{nd}}\text{-}\max}D_\sigma=\big(\bigwedge\nolimits_E^{\!n-i_1-1}\Fil^{-h_{i_1,\sigma}}(D_\sigma)\big)\wedge D_\sigma\]
and hence an isomorphism by (\ref{mult:wedge}) and (\ref{mult:wedge2})
\begin{multline*}
\Hom_E\Big((\bigwedge\nolimits_E^{\!n-i_1}\!D_\sigma)/\Fil_{S,i_1}^{2^{\mathrm{nd}}\text{-}\max}D_\sigma, \Fil_{i_1}^{\max}D_\sigma\Big)\buildrel\sim\over\longrightarrow \Ker\bigg(\Hom_E\big(\bigwedge\nolimits_E^{\!n-i_1}\!D_\sigma, \Fil_{i_1}^{\max}D_\sigma\big)\\
\rightarrow \left\{f\in \Hom_E(D_\sigma, \Fil^{-h_{i_1,\sigma}}(D_\sigma)),\ f\vert_{\Fil^{-h_{i_1,\sigma}}(D_\sigma)}\ {\mathrm{scalar}}\right\}\bigg).
\end{multline*}
\item
For $s_i\in S$ we could also define $\Fil_{S,i}^{\max}D_\sigma\subseteq \bigwedge\nolimits_E^{\!n-i}\!D_\sigma$ to be the \emph{last} step of the filtration on $\bigwedge\nolimits_E^{\!n-i}\!D_\sigma$ induced by the partial filtration $(\Fil^{-h_{j,\sigma}}(D_\sigma),s_j\in S)$ on $D_\sigma$. However an exercise analogous to (\ref{eq:filmax}) or (\ref{eq:2ndSj}) shows that we have in fact $\Fil_{S,i}^{\max}D_\sigma=\Fil_{i}^{\max}D_\sigma$.
\end{enumerate}
\end{rem}

\begin{prop}\label{prop:2ndmax}
With \ the \ above \ notation, \ the \ image \ under \ (\ref{eq:comp3}) \ of \ the \ subspace $\Ker(t_{D_\sigma}\vert_{\Ext^1_{\GL_n(K),\sigma}(\pi_{\alg}(D_\sigma),\pi_S(D_\sigma))})$ of $\Ext^1_{\GL_n(K),\sigma}(\pi_{\alg}(D_\sigma),\pi_S(D_\sigma))$ is
\[\Hom_E\Big(\big(\bigwedge\nolimits_E^{\!n-i}\!D_\sigma\big)/\Fil_{S,i}^{2^{\mathrm{nd}}\text{-}\max}D_\sigma, \Fil_i^{\max}D_\sigma\Big)\]
(in particular it does not depend on any choice).
\end{prop}
\begin{proof}
As usual we write $\pi_{\alg}$, $\pi_{s_i}$, $\pi_S$ instead of $\pi_{\alg}(D_\sigma)$, $\pi_{s_i}(D_\sigma)$, $\pi_S(D_\sigma)$, $\Ext^1_{\sigma}$ instead of $\Ext^1_{\GL_n(K),\sigma}$ and $\Fil^{-h_{i_{\vert S\vert},\sigma}}$\!, $\Fil_{i}^{\max}$\!, $\Fil_{S,i}^{2^{\mathrm{nd}}\text{-}\max}$ instead of $\Fil^{-h_{i_{\vert S\vert},\sigma}}(D_\sigma)$, $\Fil_{i}^{\max}D_\sigma$, $\Fil_{S,i}^{2^{\mathrm{nd}}\text{-}\max}D_\sigma$.\bigskip

Note first that the last statement can actually be proved directly: the same proof as in Proposition \ref{prop:log(p)} shows that the subspace $\Ker(t_{D_\sigma}\vert_{\Ext^1_{\sigma}(\pi_{\alg},\pi_S)})$ of $\Ext^1_{\sigma}(\pi_{\alg},\pi_S)$ does not depend on the choice of $\log(p)$, and the image of $\Ker(t_{D_\sigma}\vert_{\Ext^1_{\sigma}(\pi_{\alg},\pi_S)})$ under (\ref{eq:comp3}) does not depend on $(\varepsilon_I)_I$, as follows from the proof of Proposition \ref{prop:epsilon}, in particular (\ref{eq:smallI}).\bigskip

We denote by $\overline t_{D_\sigma,S}$ the composition
\[\Ext^1_{\sigma}(\pi_{\alg},\pi_S/\pi_{\alg})\buildrel{\stackrel{(\ref{eq:amalg4})}{\sim}}\over\longrightarrow \bigoplus_{s_j\in S}\Ext^1_{\sigma}(\pi_{\alg},\pi_{s_j}/\pi_{\alg})\buildrel{(\ref{mult:surj})}\over\longrightarrow \Hom_{\Fil}(D_\sigma,D_\sigma)\]
and recall from the proof of Proposition \ref{prop:map} (in particular Step $2$ of \emph{loc.~cit.}) that we have a canonical isomorphism
\begin{equation}\label{eq:isoker}
\Ker(t_{D_\sigma}\vert_{\Ext^1_{\sigma}(\pi_{\alg},\pi_S)})\buildrel\sim\over\longrightarrow \Ker(\overline t_{D_\sigma,S}).
\end{equation}
Write $S=\{s_{i_1},s_{i_2},\dots,s_{i_{\vert S\vert}}\}$ with $i_j<i_{j+1}$, by (\ref{eq:isoker}) and (\ref{eqn:isoI}) we need to prove that the kernel of the surjection
\begin{equation}\label{eq:surj2}
\bigoplus_{j=1}^{\vert S\vert}\Hom_E\big(\bigwedge\nolimits_E^{\!n-i_j}\!D_\sigma, \Fil_{i_j}^{\max}\big)
\buildrel{(\ref{mult:wedge})+(\ref{mult:wedge2})} \over \twoheadrightarrow \sum_{j=1}^{\vert S\vert}\!\left\{f\in \Hom_E(D_\sigma, \Fil^{-h_{i_j,\sigma}}),\ f\vert_{\Fil^{-h_{i_j,\sigma}}}\ {\mathrm{scalar}}\right\}
\end{equation}
has image $\Hom_E((\bigwedge\nolimits_E^{\!n-i_j}\!D_\sigma)/\Fil_{S,i_j}^{2^{\mathrm{nd}}\text{-}\max}, \Fil_{i_j}^{\max})$ for $j\in \{1,\dots,\vert S\vert\}$ via the projection to the direct summand $\Hom_E((\bigwedge\nolimits_E^{\!n-i_j}\!D_\sigma), \Fil_{i_j}^{\max})$ of the left hand side of (\ref{eq:surj2}). For instance this is clear when $\vert S\vert = 1$ by (i) of Remark \ref{rem:maxS}, hence we can assume $\vert S\vert >1$. Since the kernel of each $\Hom_E(\bigwedge\nolimits_E^{\!n-i_j}\!D_\sigma, \Fil_{i_j}^{\max})\twoheadrightarrow \{f\in \Hom_E(D_\sigma, \Fil^{-h_{i_j,\sigma}}),\ f\vert_{\Fil^{-h_{i_j,\sigma}}}\ {\mathrm{scalar}}\}$ is clearly contained in the kernel of (\ref{eq:surj2}), using (i) of Lemma \ref{lem:surj} together with (\ref{eq:2ndS}) when $j=\vert S\vert$ and (\ref{eq:2ndSj}) when $j<\vert S\vert$, it is equivalent to prove that the image of the kernel of (\ref{eq:surj2}) in $\{f\in \Hom_E(D_\sigma, \Fil^{-h_{i_j,\sigma}}),\ f\vert_{\Fil^{-h_{i_j,\sigma}}}\ {\mathrm{scalar}}\}$~is the subspace
\begin{equation}\label{eq:cond}
\begin{gathered}
\begin{array}{lcl}
\scriptstyle{\left\{f\ \in \ \Hom_E\big(D_\sigma/\Fil^{-h_{i_{{\vert S\vert}},\sigma}},\ \Fil^{-h_{i_{\vert S\vert},\sigma}}\big),\ f\big(\Fil^{-h_{i_{\vert S\vert-1},\sigma}}/\Fil^{-h_{i_{\vert S\vert},\sigma}}\big)\ =\ 0\right\}}&\mathrm{when}&j=\vert S\vert\\
\scriptstyle{\left\{f\ \in \ \Hom_E\big(D_\sigma/\Fil^{-h_{i_j,\sigma}},\ \Fil^{-h_{i_j,\sigma}}\big),\ f\big(\Fil^{-h_{i_{j-1},\sigma}}/\Fil^{-h_{i_j,\sigma}}\big)\ \subseteq \ \Fil^{-h_{i_{j+1},\sigma}}\right\}}&\mathrm{when}&j<\vert S\vert.
\end{array}
\end{gathered}
\end{equation}
Let
\[F_1 + F_2 + \cdots + F_{\vert S\vert}\ \in \ \bigoplus_{j=1}^{\vert S\vert}\Hom_E(\bigwedge\nolimits_E^{\!n-i_j}\!D_\sigma, \Fil_{i_j}^{\max})\]
which maps to $0$ by (\ref{eq:surj2}) and denote by $f_j\in \{f\in \Hom_E(D_\sigma, \Fil^{-h_{i_j,\sigma}}),\ f\vert_{\Fil^{-h_{i_j,\sigma}}}\ {\mathrm{scalar}}\}$ the image of $F_j$. Using that $\Fil^{-h_{i_{\vert S\vert},\sigma}}\subsetneq \Fil^{-h_{i_{\vert S\vert-1},\sigma}}\subsetneq \cdots \subsetneq \Fil^{-h_{i_1,\sigma}}\subsetneq \Fil^{-h_{0,\sigma}}=D_\sigma$, a straightforward induction shows that the equality $f_1+\cdots +f_{\vert S\vert}=0$ in $\Hom_E(D_\sigma,D_\sigma)$ exactly forces the conditions in (\ref{eq:cond}). More precisely the reader may draw the matrix of each $f_j\in \Hom_E(D_\sigma,D_\sigma)$ in an adapted basis of $D_\sigma$ for the above filtration $\Fil^{-h_{i_{\vert S\vert},\sigma}}\subsetneq \cdots \subsetneq \Fil^{-h_{0,\sigma}}=D_\sigma$, then sum up these matrices to get the matrix of $f_1+\cdots +f_{\vert S\vert}$ in this adapted basis, and check that if this last matrix is $0$ then this first implies $f_{\vert S\vert}\vert_{\Fil^{-h_{i_{\vert S\vert},\sigma}}}=0$, $f_{\vert S\vert-1}\vert_{\Fil^{-h_{i_{\vert S\vert-1},\sigma}}}=0,\ \dots,\ f_1\vert_{\Fil^{-h_{i_1,\sigma}}}=0$, and then (\ref{eq:cond}). This proves that the image of the kernel of (\ref{eq:surj2}) in $\{f\in \Hom_E(D_\sigma, \Fil^{-h_{i_j,\sigma}}),\ f\vert_{\Fil^{-h_{i_j,\sigma}}}\ {\mathrm{scalar}}\}$ lands in the subspaces (\ref{eq:cond}). The surjectivity for each $j$ is then an easy exercise (left to the reader) by choosing suitable $f_{j'}$ for $i_{j'}\in S\setminus\{i_j\}$.
\end{proof}

\begin{rem}\label{rem2ndmax}
Let $S'\subseteq S\subseteq R$ and $s_i\in S'$. The natural injection $$\Ext^1_{\GL_n(K), \sigma}(\pi_{\alg}(D_{\sigma}), \pi_{S'}(D_{\sigma})) \hooklongrightarrow \Ext^1_{\GL_n(K),\sigma}(\pi_{\alg}(D_{\sigma}), \pi_{S}(D_{\sigma}))$$ induces an injection
\begin{equation}\label{e:kerSS}
\Ker(t_{D_\sigma}\vert_{\Ext^1_{\GL_n(K),\sigma}(\pi_{\alg}(D_\sigma),\pi_{S'}(D_\sigma))}) \hooklongrightarrow \Ker(t_{D_\sigma}\vert_{\Ext^1_{\GL_n(K),\sigma}(\pi_{\alg}(D_\sigma),\pi_S(D_\sigma))}).
\end{equation}
If $i$ is neither the maximal nor the minimal element in $S$ (identifying $R$ with $\{1,\dots,n-1\}$), let $i_1$, $i_2\in S$ be the two elements which are adjacent to $i$ with $i_1 < i <i_2$. If $i$ is the maximal (resp.~minimal) element in $S$, let $i_3$ be the element in $S$ adjacent to $i$. By Proposition \ref{prop:2ndmax} and the discussion above Remark \ref{rem:maxS}, one can check that the images under (\ref{eq:comp3}) of the two vector spaces in (\ref{e:kerSS}) are equal if and only if $i_1, i_2\in S'$ or $i_3\in S'$ (respectively).
\end{rem}

We denote by
\[\Ext^1_{\GL_n(K),\sigma,\mathrm{inf},Z}(\pi_{\alg}(D_\sigma),\pi_S(D_\sigma))\subset \Ext^1_{\GL_n(K),\sigma,\mathrm{inf}}(\pi_{\alg}(D_\sigma),\pi_S(D_\sigma))\]
the subspaces of $\Ext^1_{\GL_n(K),\sigma}(\pi_{\alg}(D_\sigma),\pi_S(D_\sigma))$ of locally $\sigma$-analytic extensions with an infinitesimal character and a central character (resp.~with an infinitesimal character).

\begin{cor}\label{lem:inf}
For $S\subseteq R$ we have
\[\Ker\big(t_{D_\sigma}\vert_{\Ext^1_{\GL_n(K),\sigma}(\pi_{\alg}(D_\sigma),\pi_S(D_\sigma))}\big)\subset \Ext^1_{\GL_n(K),\sigma,\mathrm{inf},Z}(\pi_{\alg}(D_\sigma),\pi_S(D_\sigma)),\]
in particular the representation $\pi(D_\sigma)(S)$ has an infinitesimal character and a central character.
\end{cor}
\begin{proof}
By Lemma \ref{lem:max} it is enough to prove the statement for $S=R$. Let $\pi_R^{\mathrm{ns}}(D_\sigma)$ as in (\ref{eq:ns}) (for $S=R$), as in Step $1$ of the proof of Proposition \ref{prop:inf} it is enough to prove the statement replacing $\Ext^1_{\GL_n(K),\sigma}(\pi_{\alg}(D_\sigma),\pi_R(D_\sigma))$ by $\Ext^1_{\GL_n(K),\sigma}(\pi_{\alg}(D_\sigma),\pi_R^{\mathrm{ns}}(D_\sigma))$. We have an isomorphism similar to (\ref{mult:isoLpi3})
\begin{multline}\label{mult:isoLpi2}
\Big(\Hom_{\sm}(T(K),E)\!\!\!\!\!\!\bigoplus_{\Hom_{\sm}(K^\times,E)}\!\!\!\!\!\!\Hom_\sigma(K^\times\!\!,E)\Big) \ \bigoplus \ \Big(\!\bigoplus_{I\ \!\textrm{non-split}}\Hom_{\sigma}(\GL_{n-\vert I\vert}(\cO_K),E)\Big)\\
\buildrel\sim\over\longrightarrow \Ext^1_{\GL_n(K),\sigma}(\pi_{\alg}(D_\sigma),\pi_R^{\mathrm{ns}}(D_\sigma)).
\end{multline}
Let $\Psi:=\psi + (\sum_{I\ \!\mathrm{non-split}}\psi_I)$ an element in the left hand side of (\ref{mult:isoLpi2}) (with obvious notation), $c(\Psi)$ its image in $\Ext^1_{\GL_n(K),\sigma}(\pi_{\alg}(D_\sigma),\pi_R^{\mathrm{ns}}(D_\sigma))$ under (\ref{mult:isoLpi2}) and $\pi(\Psi)$ a representative of $c(\Psi)$. We assume $t_{D_\sigma}(c(\Psi))=0$ and we want to prove that $\pi(\Psi)$ has a central character and an infinitesimal character. Note that $t_{D_\sigma}(c(\Psi))=0$ implies $\psi=0$ by Step $2$ in the proof of Proposition \ref{prop:map}, hence we can assume $\Psi=\sum_{I\ \!\mathrm{non-split}}\psi_I$. We write $\psi_I=\lambda_I(\psi_I)\sigma\circ \log\circ {\det}$ where $\lambda_I(\psi_I)\in E$.\bigskip

Note first that $\pi(\Psi)$ has an infinitesimal character by Proposition \ref{prop:2ndmax} and Remark \ref{rem:CNSinf} (both for $S=R$) since $\Fil_{i}^{\max}D_\sigma\subset \Fil_{i}^{2^{\mathrm{nd}}\text{-}\max}D_\sigma$. Let us prove that $\pi(\Psi)$ has a central character. Let $\chi:K^\times\rightarrow E^\times$ be the (common) central character of $\pi_{\alg}(D_\sigma)$ and $\pi_R^{\mathrm{ns}}(D_\sigma)$. Let $\pi(\psi_I)$ a representative of $c(\psi_I)\in \Ext^1_{\GL_n(K),\sigma}(\pi_{\alg}(D_\sigma),\pi_R^{\mathrm{ns}}(D_\sigma))$, then by Step 2 in the proof of Proposition \ref{prop:isoext}, in particular (\ref{eq:comp0}), the action of $\diag(t)-\chi(t)$ on $\pi(\psi_I)$ for $t\in K^\times$ is easily checked to be given by the composition:
\[\pi(\psi_I)\twoheadrightarrow \pi_{\alg}(D_\sigma) \buildrel{\delta_I(t)}\over\longrightarrow \pi_{\alg}(D_\sigma)\hookrightarrow \pi_R^{\mathrm{ns}}(D_\sigma)\hookrightarrow \pi(\psi_I)\]
where $\delta_I(t):=(n-\vert I\vert)\lambda_I(\psi_I)\chi(t)\sigma(\log(t))$ (and where the left surjection and the two right injections come from $\Ext^1_{\GL_n(K),\sigma}(\pi_{\alg}(D_\sigma),\pi_R^{\mathrm{ns}}(D_\sigma))$ and the definition of $\pi_R^{\mathrm{ns}}(D_\sigma)$). It follows that $\diag(t)-\chi(t)$ acts on $\pi(\Psi)$ by
\[\pi(\Psi)\twoheadrightarrow \pi_{\alg}(D_\sigma) \buildrel{\sum_I\delta_I(t)}\over\longrightarrow \pi_{\alg}(D_\sigma)\hookrightarrow \pi_R^{\mathrm{ns}}(D_\sigma)\hookrightarrow \pi(\Psi).\]
But an easy computation yields (remember we only consider non-split $I$):
\[\sum_I\delta_I(t)=\chi(t)\sigma(\log(t))\sum_I(n-\vert I\vert)\lambda_I(\psi_I) = \chi(t)\sigma(\log(t))\sum_{j=0}^{n-1}\big(\sum_{\varphi_j\notin I}\lambda_I(\psi_I)\big)\buildrel {(\ref{eq:crucial})}\over =0,\]
which implies the statement (we can also use (\ref{eq:condinf}) for $S=R$ since we know $\pi(\Psi)$ has an infinitesimal character).
\end{proof}

We can now finally prove one of the most important results of this section.

\begin{thm}\label{thm:fil}
Let $S$ be a subset of the set $R$ of simple reflections of $\GL_n$. The isomorphism class of the locally $\sigma$-analytic representation $\pi(D_\sigma)(S)$ determines and only depends on the Hodge-Tate weights $h_{j,\sigma},j\in \{0,\dots,n-1\}$ and the isomorphism class of the filtered $\varphi^f$-module $D_\sigma$ endowed with the partial filtration $(\Fil^{-h_{i,\sigma}}(D_\sigma),s_i\in S)$.
\end{thm}
\begin{proof}
We \ again \ write \ $\pi_{\alg}$, $\pi_{s_i}$, $\pi_S$, $\widetilde\pi_{S}$ \ for \ $\pi_{\alg}(D_\sigma)$, $\pi_{s_i}(D_\sigma)$, $\pi_S(D_\sigma)$, $\widetilde\pi_{S}(D_\sigma)$, $\Ext^1_{\sigma}$, \ $\Ext^1_{\sigma,\mathrm{inf}}$ \ for \ $\Ext^1_{\GL_n(K),\sigma}$, \ $\Ext^1_{\GL_n(K),\sigma,\mathrm{inf}}$ \ and \ $\Fil_i^{\max}$,\ $\Fil_{S,i}^{2^{\mathrm{nd}}\text{-}\max}$ \ for \ $\Fil_i^{\max}D_\sigma$, \ $\Fil_{S,i}^{2^{\mathrm{nd}}\text{-}\max}D_\sigma$.\bigskip

From its definition the isomorphism class of the representation $\pi(D_\sigma)(S)$ only depends on $\pi_{\alg}$ and $\Ker(t_{D_\sigma}\vert_{\Ext^1_{\sigma}(\pi_{\alg},\pi_S)})$, which by (\ref{eq:algsigma}) and (\ref{mult:surjS}) only depends on the Hodge-Tate weights and the isomorphism class of the filtered $\varphi^f$-module $(D_\sigma,(\Fil^{-h_{i,\sigma}}(D_\sigma),s_i\in S)$ (which means no filtration at all when $S=\emptyset$). We now prove that the latter is \emph{determined} by the isomorphism class of $\pi(D_\sigma)(S)$. Since the Hodge-Tate weights and the eigenvalues of $\varphi^f$ are determined by $\pi_{\alg}(D_\sigma)$ (see (\ref{eq:algsigma})), we can assume $S\ne \emptyset$ and it is enough to prove that one can recover the filtration $(\Fil^{-h_{i,\sigma}}(D_\sigma),s_i\in S)$ from the isomorphism class of $\pi(D_\sigma)(S)$.\bigskip

Denote by $\ol\Ext^1_{\sigma,\mathrm{inf}}$ the image of $\Ext^1_{\sigma,\mathrm{inf}}(\pi_{\alg},\pi_S)$ in $\Ext^1_{\sigma}(\pi_{\alg},\pi_S/\pi_{\alg})$. By Corollary \ref{lem:inf} and using the notation in the proof of Proposition \ref{prop:2ndmax} we have $\Ker(\overline t_{D_\sigma,S})\subset \overline\Ext^1_{\sigma,\mathrm{inf}}$.\bigskip

\textbf{Step $1$}: We prove that the isomorphism class of $\pi(D_\sigma)(S)$ determines the subspaces $\Ker(\overline t_{D_\sigma,S})\subset \overline\Ext^1_{\sigma,\mathrm{inf}}$ of $\Ext^1_{\sigma}(\pi_{\alg},\pi_S/\pi_{\alg})$.\\
First, the isomorphism class of $\pi(D_\sigma)(S)$ determines the isomorphism class of $\pi_S$, which itself (trivially) determines the subspace $\Ext^1_{\sigma,\mathrm{inf}}(\pi_{\alg},\pi_S)$ of $\Ext^1_{\sigma}(\pi_{\alg},\pi_S)$, hence also the subspace $\overline\Ext^1_{\sigma,\mathrm{inf}}$ of $\Ext^1_{\sigma}(\pi_{\alg},\pi_S/\pi_{\alg})$. The isomorphism class of $\pi(D_\sigma)(S)$ also determines the isomorphism class of $\pi(D_\sigma)(S)/\pi_{\alg}$, hence it is enough to prove that the latter determines the subspace $\Ker(\overline t_{D_\sigma,S})$. Recall that by definition of $\pi(D_\sigma)(S)$ we have a commutative diagram (see the comment before Lemma \ref{lem:max})
\begin{equation}\label{eq:pull}
\begin{gathered}
\xymatrix{0\ar[r] & \pi_S/\pi_{\alg}\ar[r]\ar@{=}[d] & \widetilde\pi_{S}/\widetilde\pi_{\emptyset}\ar[r] & \pi_{\alg}\otimes_E\Ext^1_{\sigma}(\pi_{\alg},\pi_S/\pi_{\alg})\ar[r] & 0\\
0\ar[r] & \pi_S/\pi_{\alg}\ar[r]\ar@{=}[u] & \pi(D_\sigma)(S)/\pi_{\alg}\ar[r]\ar@{^{(}->}[u]& \pi_{\alg}\otimes_E\Ker(\overline t_{D_\sigma,S})\ar[r]\ar@{^{(}->}[u] & 0}
\end{gathered}
\end{equation}
where the right square is cartesian. By Lemma \ref{lem:longtime} below applied to the subspace $U=\Ker(\overline t_{D_\sigma,S})$ of $\Ext^1_{\sigma}(\pi_{\alg},\pi_S/\pi_{\alg})$, we have a decomposition of vector spaces $U=U_1\oplus \cdots \oplus U_d$ and an isomorphism
\[\pi(D_\sigma)(S)/\pi_{\alg}\cong \Big(\bigoplus_{I\notin {\mathcal I}_U} C(I, s_{\vert I\vert,\sigma})\otimes_E\Fil_{\vert I\vert}^{\max}\Big) \bigoplus \ (\pi_1\oplus \cdots \oplus \pi_d)\]
where for $j\in \{1,\dots,d\}$ each $\pi_j$ is indecomposable and the image of $\pi_j$ by the composition
\[\pi_j\hookrightarrow \pi(D_\sigma)(S)/\pi_{\alg}\hookrightarrow \widetilde\pi_{S}/\widetilde\pi_{\emptyset}\twoheadrightarrow \pi_{\alg}\otimes_E\Ext^1_{\sigma}(\pi_{\alg},\pi_S/\pi_{\alg})\]
is $\pi_{\alg}\otimes_EU_j$. Since
\begin{equation}\label{eq:End}
\End_{\GL_n(K)}(\pi(D_\sigma)(S)/\pi_{\alg})\cong \big(\bigoplus\End_{\GL_n(K)}(C(I, s_{\vert I\vert,\sigma})\big)\bigoplus \big(\bigoplus \End_{\GL_n(K)}(\pi_j)\big)
\end{equation}
where each $\End(-)$ in (\ref{eq:End}) is $E$, it follows that for \emph{any} injection $\iota: \pi(D_\sigma)(S)/\pi_{\alg}\hookrightarrow \widetilde\pi_{S}/\widetilde\pi_{\emptyset}$ obtained by composing the canonical injection in (\ref{eq:pull}) with an automorphism of $\pi(D_\sigma)(S)/\pi_{\alg}$, the composition with $\widetilde\pi_{S}/\widetilde\pi_{\emptyset}\twoheadrightarrow \pi_{\alg}\otimes_E\Ext^1_{\sigma}(\pi_{\alg},\pi_S/\pi_{\alg})$ still gives the subspace $\pi_{\alg}\otimes_E(U_1\oplus \dots \oplus U_d)=\pi_{\alg}\otimes_E\Ker(\overline t_{D_\sigma,S})$. This proves that the isomorphism class of $\pi(D_\sigma)(S)/\pi_{\alg}$ determines the subspace $\Ker(\overline t_{D_\sigma,S})$.\bigskip

We write $S=\{s_{i_1},s_{i_2},\dots,s_{i_{\vert S\vert}}\}$ with $i_j<i_{j+1}$. For $j\in \{0,\dots,\vert S\vert-1\}$ let $S_j:=\{s_{i_1},\dots,s_{i_j},s_{i_{\vert S\vert}}\}$ (so $S_{\vert S\vert -1}=S$ and $S_0=\{s_{i_{\vert S\vert}}\}$).\bigskip

\textbf{Step $2$}: We prove that the subspaces of $\Ext^1_{\sigma}(\pi_{\alg},\pi_{S}/\pi_{\alg})$
\[\Ker(\overline t_{D_\sigma,S_0})\subset\Ker(\overline t_{D_\sigma,S_1})\subset \cdots\subset \Ker(\overline t_{D_\sigma,S})\subset \overline\Ext^1_{\sigma,\mathrm{inf}}\]
determine the filtration
\[D_\sigma=\Fil^{-h_{0,\sigma}}(D_\sigma)\supset \Fil^{-h_{i_1,\sigma}}(D_\sigma)\supset \cdots \supset \Fil^{-h_{i_{\vert S\vert-1},\sigma}}(D_\sigma)\supset \Fil^{-h_{i_{\vert S\vert},\sigma}}(D_\sigma).\]
Recall first that $\Ker(t_{D_\sigma}\vert_{\Ext^1_{\sigma}(\pi_{\alg},\pi_{S'})})\subset \Ker(t_{D_\sigma}\vert_{\Ext^1_{\sigma}(\pi_{\alg},\pi_{S})})$ for $S'\subset S$ and thus $\Ker(\overline t_{D_\sigma,S'})\subset \Ker(\overline t_{D_\sigma,S})$. \ By \ Proposition \ \ref{prop:inf} \ applied \ with \ $i=i_{\vert S\vert}$ \ the \ image \ of \ $\overline\Ext^1_{\sigma,\mathrm{inf}}$ \ in \ $\Hom_E(\bigwedge\nolimits_E^{\!n-i_{\vert S\vert}}\!D_\sigma, \Fil_{i_{\vert S\vert}}^{\max})$ is the subspace $\Hom_E((\bigwedge\nolimits_E^{\!n-i_{\vert S\vert}}\!D_\sigma)/ \Fil_{i_{\vert S\vert}}^{\max}, \Fil_{i_{\vert S\vert}}^{\max})$. By (\ref{mult:wedge}), (\ref{mult:wedge2}) (with (i) of Lemma \ref{lem:surj}), the image of $\Hom_E((\bigwedge\nolimits_E^{\!n-i_{\vert S\vert}}\!D_\sigma)/ \Fil_{i_{\vert S\vert}}^{\max}, \Fil_{i_{\vert S\vert}}^{\max})$ in $\Hom_{\Fil}(D_\sigma,D_\sigma)$ is the subspace
\[\Hom_E\big(D_\sigma/\Fil^{-h_{i_{\vert S\vert},\sigma}}(D_\sigma),\Fil^{-h_{i_{\vert S\vert},\sigma}}(D_\sigma)\big),\]
which clearly determines $\Fil^{-h_{i_{\vert S\vert},\sigma}}(D_\sigma)$. By Proposition \ref{prop:2ndmax} applied with $i=i_{\vert S\vert}$ and by (\ref{eq:isoker}) the image of $\Ker(\overline t_{D_\sigma,S_j})$ in $\Hom_E(\bigwedge\nolimits_E^{\!n-i_{\vert S\vert}}\!D_\sigma, \Fil_{i_{\vert S\vert}}^{\max})$ for $j\in \{0,\dots,\vert S\vert-1\}$ is the subspace $\Hom_E((\bigwedge\nolimits_E^{\!n-i_{\vert S\vert}}\!D_\sigma)/\Fil_{S_j,i_{\vert S\vert}}^{2^{\mathrm{nd}}\text{-}\max}, \Fil_{i_{\vert S\vert}}^{\max})$. By (\ref{eq:2ndS}) applied with $S_j$ and (\ref{mult:wedge}), (\ref{mult:wedge2}), the image of $\Hom_E((\bigwedge\nolimits_E^{\!n-i_{\vert S\vert}}\!D_\sigma)/\Fil_{S_j,i_j}^{2^{\mathrm{nd}}\text{-}\max}, \Fil_{i_j}^{\max})$ in $\Hom_{\Fil}(D_\sigma,D_\sigma)$ for $j\in \{0,\dots,\vert S\vert-1\}$ is the subspace
\[\Hom_E\big(D_\sigma/\Fil^{-h_{i_j,\sigma}}(D_\sigma),\Fil^{-h_{i_{\vert S\vert},\sigma}}(D_\sigma)\big)\]
(recall $i_0=0$), which again determines $\Fil^{-h_{i_j,\sigma}}(D_\sigma)$. Dualizing, this gives all the steps of the filtration in the statement.\bigskip

\textbf{Step $3$}: We prove the theorem.\\
By Lemma \ref{lem:max} the isomorphism class of $\pi(D_\sigma)(S)$ determines the isomorphism classes of all $\pi(D_\sigma)(S_j)$ for $j\in \{0,\dots,\vert S\vert-1\}$. The statement then follows from Step $1$ and Step~$2$.
\end{proof}

The proof of Theorem \ref{thm:fil} uses the following formal lemma. For $S$ a non-empty subset of $R$ define
\[\mathcal{I}_S:=\left\{I\subseteq \{\varphi_0,\dots,\varphi_{n-1}\},\ \vert I\vert = i\mathrm{\ for\ some\ }i\mathrm{\ such\ that\ }s_i\in S\right\}.\]
For $U$ a vector subspace of $\Ext^1_{\GL_n(K),\sigma}(\pi_{\alg}(D_\sigma),\pi_S(D_\sigma)/\pi_{\alg}(D_\sigma))$ define
${\mathcal I}_U\subseteq {\mathcal I}_S$ as the minimal (for inclusion) subset such that
\[U\subseteq \bigoplus_{I\in \mathcal{I}_U}\Ext^1_{\GL_n(K),\sigma}\big(\pi_{\alg}(D_\sigma),C(I, s_{\vert I\vert,\sigma})\otimes_E\Fil_{\vert I\vert}^{\max}D_\sigma)\big)\]
(recall that $\pi_S(D_\sigma)/\pi_{\alg}=\bigoplus_{I\in \mathcal{I}_S}C(I, s_{\vert I\vert,\sigma})\otimes_E\Fil_{\vert I\vert}^{\max}D_\sigma$ and that each $\Ext^1_{\GL_n(K),\sigma}$ above has dimension $1$ by Lemma \ref{lem:nonsplit}). The following lemma is longer to state than to prove since it is purely formal, we leave its proof to the reader.

\begin{lem}\label{lem:longtime}
Let \ $S$ \ be \ a \ non-empty \ subset \ of \ $R$, \ $U$ \ a \ vector \ subspace \ of \ $\Ext^1_{\GL_n(K),\sigma}(\pi_{\alg}(D_\sigma),\pi_S(D_\sigma)/\pi_{\alg}(D_\sigma))$ \ and \ denote \ by \ $\pi(D_\sigma)(U)$ \ the \ pull-back \ of \ $\widetilde\pi_{S}(D_\sigma)/\widetilde\pi_{\emptyset}(D_\sigma)$ in the top exact sequence of (\ref{eq:pull}) along the canonical injection
\[\pi_{\alg}(D_\sigma)\otimes_EU\hookrightarrow \pi_{\alg}(D_\sigma)\otimes_E\Ext^1_{\GL_n(K),\sigma}(\pi_{\alg}(D_\sigma),\pi_S(D_\sigma)).\]
Write $U=U_1\oplus \cdots \oplus U_d$ where each $U_j$ is non-zero, where ${\mathcal I}_U={\mathcal I}_{U_1} \amalg \cdots \amalg {\mathcal I}_{U_d}$ and where each $U_j$ cannot be decomposed any further (there exist such $d\geq 1$ and $U_i$). Then we have
\[\pi(D_\sigma)(U)\cong \Big(\bigoplus_{I\notin {\mathcal I}_U} C(I, s_{\vert I\vert,\sigma})\otimes_E\Fil_{\vert I\vert}^{\max}D_\sigma)\Big) \bigoplus \ (\pi_1\oplus \cdots \oplus \pi_d)\]
where each $\pi_j$ is indecomposable and the image of $\pi_j$ via the composition
\[\pi_j\hookrightarrow \widetilde\pi_{S}(D_\sigma)/\widetilde\pi_{\emptyset}(D_\sigma)\twoheadrightarrow \pi_{\alg}(D_\sigma)\otimes_E\Ext^1_{\GL_n(K),\sigma}\big(\pi_{\alg}(D_\sigma),\pi_S(D_\sigma)/\pi_{\alg}(D_\sigma)\big)\]
is the subspace $\pi_{\alg}(D_\sigma)\otimes_EU_j$.
\end{lem}

\begin{rem}\label{rem:hodge}\hspace{2em}
\begin{enumerate}[label=(\roman*)]
\item
Let $S=\{s_{i_1},s_{i_2},\dots,s_{i_{\vert S\vert}}\}\ne \emptyset$ with $i_j<i_{j+1}$, it follows from the proof of Theorem \ref{thm:fil} (in particular Step $2$) that the isomorphism class of the representation $\pi(D_\sigma)(S)/\pi_{\alg}$ determines the isomorphism class of the filtered $\varphi^f$-module
\[\Big(D_\sigma, \ D_\sigma=\Fil^{-h_{0,\sigma}}(D_\sigma)\supset\Fil^{-h_{i_1,\sigma}}(D_\sigma)\supset \Fil^{-h_{i_2,\sigma}}(D_\sigma) \supset \cdots \supset \Fil^{-h_{i_{\vert S\vert-1},\sigma}}(D_\sigma)\Big)\]
(that is, the last step $\Fil^{-h_{i_{\vert S\vert},\sigma}}(D_\sigma)$ is missing), but not conversely in general.
\item
We explain more explicitly how to ``visualize'' the Hodge filtration from the proof of Theorem \ref{thm:fil}. We only consider the case $S=R$ (hence $i_{|S|}=n-1$). For $S'\subseteq R$ and $s_{n-1}\in S'$ we denote by $\kappa_{n-1}$ the natural surjection (see for example (\ref{eq:comp3}))
\begin{equation*}
\kappa_{n-1}:\Ext^1_{\GL_n(K),\sigma}\big(\pi_{\alg}(D_{\sigma}), \pi_{S'}(D_{\sigma})\big) \twoheadlongrightarrow \Ext^1_{\GL_n(K),\sigma}\big(\pi_{\alg}(D_{\sigma}), \pi_{s_{n-1}}(D_{\sigma})/\pi_{\alg}(D_{\sigma})\big).
\end{equation*}
In Step 2 of the proof of Theorem \ref{thm:fil}, we consider the subspaces $\kappa_{n-1}(\Ker(\overline t_{D_{\sigma},S_j}))$ for $j\in \{1,\dots, n-2\}$ (with $S_j=\{s_1, \dots, s_j, s_{n-1}\}$). Let $\pi_{s_{n-1}}^{(j)}$ be the tautological extension of $\pi_{\alg}(D_{\sigma}) \otimes_E\kappa_{n-1}(\Ker(\overline t_{D_{\sigma},S_j}))$ by $\pi_{s_{n-1}}(D_{\sigma})/\pi_{\alg}(D_{\sigma})$ similarly as in Definition \ref{def:pi(d)}. By definition, we have 
\begin{equation*}
\pi_{s_{n-1}}^{(j)}(D_{\sigma})\cong \pi(D_{\sigma})(S_j)/\pi(D_{\sigma})(S_j \!\setminus \!\{s_{n-1}\})
\end{equation*}
and by Remark \ref{rem2ndmax} we have in fact $\pi_{s_{n-1}}^{(j)}(D_{\sigma})\cong \pi(D_{\sigma})(\{s_j,s_{n-1}\})/\pi(D_{\sigma})(\{s_j\})$. We let $\pi_{s_{n-1}}^{(0)}(D_\sigma):=\pi_{s_{n-1}}(D_{\sigma})/\pi_{\alg}(D_{\sigma})$ and $\pi_{s_{n-1}}^{(n-1)}(D_\sigma):=\widetilde{\pi}_R(D_{\sigma})_{\inf}/\widetilde{\pi}_{R\setminus \{s_{n-1}\}}(D_{\sigma})_{\inf}$, where for $S'\subseteq R$ we denote by $\widetilde{\pi}_{S'}(D_{\sigma})_{\inf}$ the tautological extension of $\pi_{\alg}(D_{\sigma}) \otimes_E \Ext^1_{\GL_n(K),\sigma,\inf}(\pi_{\alg}(D_{\sigma}), \pi_{S'}(D_{\sigma}))$ by $\pi_{\alg}(D_{\sigma})$. We have an increasing sequence of representations (writing $\pi_{s_{n-1}}^{(j)}$, $\pi_{\alg}$ for $\pi_{s_{n-1}}^{(j)}(D_{\sigma})$, $\pi_{\alg}(D_{\sigma})$)
\begin{equation} \label{e:seqn-1}
\begin{tikzcd}[row sep=2em, column sep=1.2em]
\pi_{s_{n-1}}^{(0)}\arrow[r, hook] \arrow[d, two heads] &\pi_{s_{n-1}}^{(1)} \arrow[r,hook] \arrow[d, two heads] &\cdots \arrow[r,hook] &\pi_{s_{n-1}}^{(j)} \arrow[r,hook] \arrow[d, two heads] &\cdots \arrow[r,hook] &\pi_{s_{n-1}}^{(n-2)} \arrow[r,hook] \arrow[d, two heads] &\pi_{s_{n-1}}^{(n-1)} \arrow[d, two heads] \\
0 \arrow[r, hook] &\pi_{\alg}\arrow[r,hook]&\cdots \arrow[r,hook] &\pi_{\alg}^{\oplus j}\arrow[r,hook] &\cdots \arrow[r,hook] &\pi_{\alg}^{\oplus (n-2)} \arrow[r,hook] &\pi_{\alg}^{\oplus (n-1)}
\end{tikzcd}
\end{equation}
where the kernel of the vertical surjections are all isomorphic to $\pi_{s_{n-1}}^{(0)}(D_{\sigma})$ and where the multiplicities of $\pi_{\alg}(D_{\sigma})$ follow from Proposition \ref{prop:2ndmax} with (\ref{eq:2ndS}) (for $j\neq n-1$), Proposition \ref{prop:inf} (for $j=n-1$). Taking the orthogonal of the image of the subspaces $\kappa_{n-1}(\Ker(\overline t_{D_{\sigma},S_j}))$ of $\Hom_{E}(D_{\sigma}, \Fil^{-h_{n-1,\sigma}}D_{\sigma})$ via (\ref{eqn:isoI}) with respect to the pairing
\begin{equation*}
\Hom_{E}(D_{\sigma}, \Fil^{-h_{n-1,\sigma}}D_{\sigma}) \times D_{\sigma} \longrightarrow \Fil^{-h_{n-1,\sigma}} D_{\sigma} \cong E
\end{equation*}
(which amounts to the argument in Step 2 of the proof of Theorem \ref{thm:fil}), the top sequence in (\ref{e:seqn-1}) corresponds to a sequence of subspaces of $D_{\sigma}$ (with $\pi_{n-1}^{(j)}(D_{\sigma})$ corresponding to $\Fil^{-{h_{j,\sigma}}}D_{\sigma}$)
\begin{multline*}
D_{\sigma} =\Fil^{-h_{0,\sigma}}D_{\sigma} \supsetneq \Fil^{-h_{1,\sigma}} (D_{\sigma})\supsetneq \cdots \supsetneq \Fil^{-h_{n-1-j,\sigma}} (D_{\sigma}) \supsetneq \cdots\supsetneq \Fil^{-h_{n-2,\sigma}} (D_{\sigma})\\
\supsetneq \Fil^{-h_{n-1,\sigma}}(D_{\sigma})
\end{multline*}
which is precisely the Hodge filtration on $D_{\sigma}$.
\end{enumerate}
\end{rem}

We end up this section by a description of the cosocle of $\pi(D_\sigma)$. Recall that the wedge product induces a perfect pairing of finite dimensional vector spaces
\begin{equation}\label{eq:wedge}
\bigwedge\nolimits_E^{\!i}\!D_\sigma\times \bigwedge\nolimits_E^{\!n-i}\!D_\sigma\longrightarrow \bigwedge\nolimits_E^{\!n}\!D_\sigma.
\end{equation}
We say that a subset $I\subset \{\varphi_0,\dots,\varphi_{n-1}\}$ of cardinality $i\in \{1,\dots,n-1\}$ is \emph{split} if the coefficient of $e_{I^c}$ in $\Fil_i^{\max}D_\sigma\subset \bigwedge\nolimits_E^{\!n-i}\!D_\sigma$ is $0$ (equivalently $\Fil_i^{\max}D_\sigma\subset \bigoplus_{\substack{\vert J\vert=n-i\\J\ne I^c}}Ee_J$) and is \emph{cosplit} if the coefficient of $e_{I}$ in any vector of the orthogonal $(\Fil_i^{2^{\mathrm{nd}}\text{-}\max}D_\sigma)^\perp\subset \bigwedge\nolimits_E^{\!i}\!D_\sigma$ of $\Fil_i^{2^{\mathrm{nd}}\text{-}\max}D_\sigma$ under (\ref{eq:wedge}) is $0$ (equivalently $(\Fil_i^{2^{\mathrm nd}-\max}D_\sigma)^\perp\subset \bigoplus_{\substack{\vert J\vert=i \\ J\ne I}}Ee_J$).

\begin{cor}\label{cor:cosoc}
We have
\begin{eqnarray*}
\soc_{\GL_n(K)}\pi(D_\sigma)&\simeq &\pi_{\alg}(D_\sigma)\ \bigoplus \ \big(\bigoplus_{I\ \!{\mathrm{split}}} C(I,s_{i,\sigma})\big)\\
\cosoc_{\GL_n(K)}\pi(D_\sigma)&\simeq &\pi_{\alg}(D_\sigma)^{\oplus 2^n-1-\frac{n(n+1)}{2}}\ \bigoplus \ \big(\bigoplus_{I\ \!{\mathrm{cosplit}}} C(I,s_{i,\sigma})\big).
\end{eqnarray*}
Moreover, when $n\geq 3$, if $C(I,s_{i,\sigma})$ occurs in $\cosoc_{\GL_n(K)}\pi(D_\sigma)$ then $C(I^c,s_{n-i,\sigma})$ occurs in $\soc_{\GL_n(K)}\pi(D_\sigma)$, and this is an equivalence when $n= 3$.
\end{cor}
\begin{proof}
By definition of $\pi(D_\sigma)$ (Definition \ref{def:pi(d)}) we have $\soc_{\GL_n(K)}\pi(D_\sigma)\!\buildrel\sim\over\rightarrow \!\soc_{\GL_n(K)}\pi_R(D_\sigma)$. The first isomorphism then follows from (\ref{mult:decomp}) and (\ref{eq:amalg4}) (for $S=R$). By definition of $\pi(D_\sigma)$ the constituants $C(I, s_{\vert I\vert,\sigma})\cong C(I, s_{\vert I\vert,\sigma})\otimes_E\Fil_{\vert I\vert}^{\max}D_\sigma$ in the cosocle of $\pi(D_\sigma)$ are exactly those $I$ such that the composition
\begin{multline*}
\Ker(t_{D_\sigma})\hookrightarrow \Ext^1_{\GL_n(K),\sigma}(\pi_{\alg}(D_\sigma),\pi_R(D_\sigma))\twoheadrightarrow \Ext^1_{\GL_n(K),\sigma}\big(\pi_{\alg}(D_\sigma),\pi_R(D_\sigma)/\pi_{\alg}(D_\sigma)\big)\\
\twoheadrightarrow \Ext^1_{\GL_n(K),\sigma}\big(\pi_{\alg}(D_\sigma),C(I, s_{\vert I\vert,\sigma})\otimes_E\Fil_{\vert I\vert}^{\max}D_\sigma\big)
\end{multline*}
is $0$, or equivalently such that the image of $\Ker(t_{D_\sigma})$ via (\ref{eq:comp3}) (for $S=R$) in
\[\Hom_E\big(\bigwedge\nolimits_E^{\!n-\vert I\vert}\!D_\sigma, \Fil_{\vert I\vert}^{\max}D_\sigma\big)\buildrel (\ref{eq:wedge}) \over \cong \bigwedge\nolimits_E^{\!\vert I\vert}\!D_\sigma\otimes_E\big(\Fil_{\vert I\vert}^{\max}D_\sigma\otimes_E\bigwedge\nolimits_E^{\!n}\!D_\sigma\big)\]
lands in the subspace $\big(\bigoplus_{\substack{\vert J\vert=\vert I\vert \\ J\ne I}}Ee_J\big)\otimes_E(\Fil_{\vert I\vert}^{\max}D_\sigma\otimes_E\bigwedge\nolimits_E^{\!n}\!D_\sigma)$. By Proposition \ref{prop:2ndmax} (for $S=R$) these are exactly the $C(I, s_{\vert I\vert,\sigma})$ such that $I$ is cosplit. We deduce the second isomorphism. It easily follows from (\ref{eq:filmax}) and (\ref{eq:2nd}) that $\Fil_{n-i}^{\max}D_\sigma\subset (\Fil_i^{2^{\mathrm{nd}}\text{-}\max}D_\sigma)^\perp$ when $n\geq 3$ and $i\in \{1,\dots,n-1\}$, and this is an equality when $n= 3$. In particular if $I$ is cosplit then $I^c$ is split and this is an equivalence when $n=3$. This gives the last statement.
\end{proof}

\begin{rem}
The last statement of Corollary \ref{cor:cosoc} is obviously false when $n=2$ since both $C(\{\varphi_0\},s_{1,\sigma}),C(\{\varphi_1\},s_{1,\sigma})$ always occur in $\cosoc_{\GL_2(K)}\pi(D_\sigma)$ but not necessarily in $\soc_{\GL_2(K)}\pi(D_\sigma)$.
\end{rem}

\subsection{Another definition of \texorpdfstring{$t_{D_\sigma}$}{tDsigma} in terms of \texorpdfstring{$(\varphi,\Gamma)$-modules}{phiGammamodules}}\label{sec:indep}

We give another (equivalent) definition of the map $t_{D_{\sigma}}$ of Proposition \ref{prop:map} in terms of $(\varphi,\Gamma)$-modules over the Robba ring which does not depend on any choice (Theorem \ref{thm:independant}). This alternative definition will be used later.\bigskip

We keep the notation of the previous sections and let $\Gamma:=\Gal(K(\zeta_{p^n},n\geq 1)/K)$.\bigskip

We first need a few reminders on $(\varphi, \Gamma)$-modules over the Robba ring $\cR_{K,E}$. We denote by $\cM(D)$ the $(\varphi, \Gamma)$-module over $\cR_{K,E}$ associated to the filtered $\varphi$-module $(D, \varphi, \Fil^{\bullet}(D_K))$ (cf.~\cite[Thm.~A]{Be081}). For any $(\varphi, \Gamma)$-module $\cM$ over $\cR_{K,E}$, we let $W_{\dR}^+(\cM)$ be the associated $B_{\dR}^+$-representation of $\Gal(\overline K/K)$ (see \cite[Prop.~2.2.6(2)]{Be082}), $W_{\dR}(\cM):=W_{\dR}^+(\cM)[1/t]$ where $t\in B_{\dR}^+$ is Fontaine's ``$2i\pi$'' and $D_{\dR}({\cM}):=W_{\dR}({\cM})^{\Gal(\overline K/K)}$, which is a free $K\otimes_{\Qp}E$-module (see for instance the proof of \cite[Lemma 3.1.4]{BHS19} with \cite[Lemma 3.3.5]{BHS19}). We recall that by definition $\cM$ is de Rham if $D_{\dR}({\cM})$ has rank ${\rm rank}_{\cR_{K,E}}\cM$. For instance $\cM(D)$ is de Rham and $D_{\dR}({\cM(D)})\cong D_K$ (see \cite[Prop.~2.3.4]{Be082}).\bigskip

Following \cite[\S~4.3]{Fo04} we define $B_{\pdR}:=B_{\dR} [\log t]$ and recall that (see \emph{loc.cit.}~for details):
\begin{enumerate}[label=(\roman*)]
\item
the action of $\Gal(\overline K/K)$ on $B_{\dR}$ naturally extends to $B_{\pdR}$ via $g(\log t)=\log t+\log(\varepsilon(g))$;
\item
$B_{\pdR}$ is equipped with a nilpotent $B_{\dR}$-linear operator $\nu_{\pdR}$ such that $\nu_{\pdR}((\log t)^i)=-i (\log t)^{i-1}$ for $i\geq 1$;
\item
the \ filtration \ $\Fil^i(B_{\dR})=t^i B_{\dR}^+$ \ on \ $B_{\dR}$ \ induces \ a \ filtration \ on \ $B_{\pdR}$ \ given \ by $\Fil^i(B_{\pdR}):=t^i B_{\dR}^+[\log t]$ for $i\in \Z$;
\item
$\nu_{\pdR}$ commutes with $\Gal(\overline K/K)$ and both preserve each $\Fil^i(B_{\pdR})$ for $i\in \Z$.
\end{enumerate}
We say that a $(\varphi, \Gamma)$-module $\cM$ over $\cR_{K,E}$ is almost de Rham (cf.~\cite[\S~3.7]{Fo04}) if
\begin{equation}\label{eq:DpdR}
D_{\pdR}({\cM}):=(B_{\pdR} \otimes_{B_{\dR}} W_{\dR}({\cM}))^{\Gal(\overline K/K)}=(B_{\pdR} \otimes_{B_{\dR}^+} W_{\dR}^+({\cM}))^{\Gal(\overline K/K)}
\end{equation}
is free over $K\otimes_{\Qp}E$ of ${\rm rank}_{\cR_{K,E}}\cM$. The nilpotent operator $\nu_{\pdR}$ on $B_{\pdR}$ induces a nilpotent $K \otimes_{\Qp} E$-linear endomorphism $\nu_{{\cM}}$ on $D_{\pdR}({\cM})$ and an almost de Rham $\cM$ is de Rham if and only if $\nu_{{\cM}}=0$. Note also that $W_{\dR}({\cM})$, $D_{\dR}({\cM})$, $D_{\pdR}({\cM})$ in fact only depend on $\cM[1/t]$ and can be defined for any $(\varphi, \Gamma)$-module over $\cR_{K,E}[1/t]$, see for instance the discussion before \cite[Lemma 3.3.5]{BHS19}. In particular we can define in the obvious way de Rham and almost de Rham $(\varphi, \Gamma)$-modules over $\cR_{K,E}[1/t]$.\bigskip

Recall that we can view any extension $\widetilde{\cM}\in \Ext^1_{(\varphi, \Gamma)}(\cM(D), \cM(D))$ (where $\Ext^1_{(\varphi, \Gamma)}$ means extensions as $(\varphi, \Gamma)$-modules over $\cR_{K,E}$) as a deformation of $\cM(D)$ over $\cR_{K,E[\epsilon]/\epsilon^2}$, in particular $\widetilde{\cM}$ is a free $\cR_{K,E[\epsilon]/\epsilon^2}$-module of rank $n$. As $\cM(D)$ is de Rham, we know $\widetilde{\cM}$ is almost de Rham (cf.~\cite[\S~3.7]{Fo04}), equivalently (using the above references in \cite{BHS19}) $D_{\pdR}(\widetilde{\cM})$ is free of rank $n$ over $K \otimes_{\Qp} E[\epsilon]/\epsilon^2$. The filtration $\Fil^\bullet(B_{\pdR})$ induces a filtration on $D_{\pdR}(\widetilde{\cM})$
\[\Fil^i(D_{\pdR}(\widetilde{\cM})):=\big(\Fil^i(B_{\pdR}) \otimes_{B_{\dR}^+} W_{\dR}^+(\widetilde{\cM})\big)^{\Gal(\overline K/K)}\]
by $K\otimes_{\Qp} E[\epsilon]/\epsilon^2$-submodules and there are exact sequences of $K \otimes_{\Qp} E$-modules for $i\in \Z$ (using $D_{\pdR}({\cM(D)})=D_{\dR}({\cM(D)})\cong D_K$)
\begin{multline*}
0\longrightarrow \Fil^i(D_K)\cong \Fil^i(D_{\pdR}({\cM}(D)))\buildrel {.\epsilon}\over\longrightarrow \Fil^i(D_{\pdR}(\widetilde{\cM})) \\
\longrightarrow \Fil^i(D_{\pdR}({\cM}(D)))\cong \Fil^i(D_K) \longrightarrow 0.
\end{multline*}
As above the operator $\nu_{\pdR}$ on $B_{\pdR}$ induces a nilpotent $K \otimes_{\Qp} E[\epsilon]/\epsilon^2$-linear endomorphism
\begin{equation}\label{eq:nu}
\nu_{\widetilde{\cM}}: D_{\pdR}(\widetilde{\cM}) \longrightarrow D_{\pdR}(\widetilde{\cM})
\end{equation}
which preserves each $\Fil^i(D_{\pdR}(\widetilde{\cM}))$. Since $\nu_{\widetilde{\cM}}\vert_{D_{\pdR}({\cM(D)})}=0$ (as ${\cM(D)}$ is de Rham) we deduce that $\nu_{\widetilde{\cM}}$ factors as follows
\[\nu_{\widetilde{\cM}} : D_{\pdR}(\widetilde{\cM})\twoheadrightarrow D_{\pdR}({\cM}(D))\cong D_K\longrightarrow D_K\cong D_{\pdR}({\cM}(D))\buildrel {.\epsilon}\over \hookrightarrow D_{\pdR}(\widetilde{\cM})\]
where the endomorphism $D_K\rightarrow D_K$ respects $\Fil^\bullet(D_K)$.\bigskip

We now use the canonical isomorphism 
\begin{equation}\label{eq:dec}
K \otimes_{\Qp} E[\epsilon]/\epsilon^2\longrightarrow \bigoplus_{\sigma\in \Sigma} E[\epsilon]/\epsilon^2,\ \lambda \otimes x \mapsto (\sigma(\lambda)x)_{\sigma\in \Sigma}
\end{equation}
which induces canonical decompositions $\Fil^i(D_{\pdR}(\widetilde{\cM}))\cong \bigoplus_{\sigma\in \Sigma} \Fil^i(D_{\pdR}(\widetilde{\cM})_{\sigma})$ for $i\in \Z$ where each $\Fil^i(D_{\pdR}(\widetilde{\cM})_{\sigma})$ is a free $E[\epsilon]/\epsilon^2$-module, with $D_{\pdR}(\widetilde{\cM})_{\sigma}$ of rank $n$. Likewise $\nu_{\widetilde{\cM}}$ induces a nilpotent $E[\epsilon]/\epsilon^2$-linear endomorphism $\nu_{\widetilde{\cM},\sigma}$ on $D_{\pdR}(\widetilde{\cM})_{\sigma}$ which factors as:
\[\nu_{\widetilde{\cM}, \sigma} : D_{\pdR}(\widetilde{\cM})_{\sigma}\twoheadrightarrow D_{\pdR}({\cM}(D))_\sigma\cong D_\sigma\longrightarrow D_\sigma\cong D_{\pdR}({\cM}(D))_\sigma\buildrel {.\epsilon}\over \hookrightarrow D_{\pdR}(\widetilde{\cM})_{\sigma}.\]
We still denote by $\nu_{\widetilde{\cM}, \sigma}$ the induced nilpotent endomorphism in $\Hom_{\Fil}(D_{\sigma}, D_{\sigma})$. We have thus obtained a canonical $E$-linear morphism 
\begin{equation}\label{eq:Enil}
\Ext^1_{(\varphi, \Gamma)}(\cM(D), \cM(D)) \longrightarrow \bigoplus_{\sigma\in \Sigma} \Hom_{\Fil}(D_{\sigma}, D_{\sigma}), \ \ \widetilde{\cM} \longmapsto (\nu_{\widetilde{\cM}, \sigma})_{\sigma\in \Sigma}.
\end{equation}

\begin{lem}\label{lem:LpdR}
The \ map \ (\ref{eq:Enil}) \ is \ surjective \ and \ its \ kernel \ is \ the \ subspace \ $\Ext^1_g(\cM(D),\cM(D))$ of de Rham extensions.
\end{lem}
\begin{proof}
By definition, $\widetilde{\cM}$ is de Rham if and only if $\nu_{\widetilde{\cM}}=0$ if and only if $\nu_{\widetilde{\cM}, \sigma}=0$ for all $\sigma\in \Sigma$. The second part of the statement follows. By \cite[Thm.~0.2(a)]{Li07} and using the fact that $(D, \varphi, \Fil^\bullet(D_K))$ is regular and satisfies (\ref{eq:phi}) we have
\[\dim_E\Ext^1_{(\varphi, \Gamma)}(\cM(D), \cM(D))=\dim_E \Hom_{(\varphi, \Gamma)}(\cM(D),\cM(D))+n^2 [K:\Qp].\]
By \cite[Cor.~A.4]{Di192} applied to $J=\Sigma$ and $W$ the $B$-pair associated to the $(\varphi,\Gamma)$-module $\cM(D) \otimes_{\cR_{K,E}} \cM(D)^{\vee}$ where $\cM(D)^{\vee}$ is the dual of $\cM(D)$ we have
\[\dim_E \Ext^1_g(\cM(D),\cM(D))=\dim_E \Hom_{(\varphi, \Gamma)}(\cM(D),\cM(D))+\frac{n(n-1)}{2} [K:\Qp].\]
Since $\dim_E(\bigoplus_{\sigma\in \Sigma} \Hom_{\Fil}(D_{\sigma}, D_{\sigma}))=\frac{n(n+1)}{2} [K:\Qp]$ (which is obvious as $\Fil^\bullet(D_\sigma)$ is a full flag on $D_\sigma$), the first part of the statement follows by comparing dimensions.
\end{proof}

Recall \ that \ a \ $(\varphi,\Gamma)$-module \ $\cM$ \ over \ $\cR_{K,E}$ \ is \ crystalline \ if \ the \ $K_0\otimes_{\Qp}E$-module $D_{\cris}({\cM}):=(\cM[1/t])^{\Gamma}$ has dimension $[K_0\!:\!\Qp]{\rm rank}_{\cR_{K,E}}\cM$ over $E$, see \cite[\S~1.2.3]{Be11}. Equivalently, using $D_{\dR}({\cM})=D_{\cris}({\cM})\otimes_{K_0}K$ (see \emph{loc.~cit.}) and the above freeness of $D_{\dR}({\cM})$ over $K\otimes_{\Qp}E$, $\cM$ is crystalline if $D_{\cris}({\cM})$ is free over $K_0\otimes_{\Qp}E$ of rank ${\rm rank}_{\cR_{K,E}}\cM$.

\begin{lem}\label{lem:cris}
Any extension in $\Ext^1_g(\cM(D), \cM(D))$ is automatically crystalline.
\end{lem}
\begin{proof}
The statement is well known but we include a proof for the reader's convenience and to introduce several maps that will be used in the sequel. By inverting $t$, we have a natural morphism 
\begin{equation}\label{eq:Einvt}
\Ext^1_{(\varphi, \Gamma)}(\cM(D), \cM(D)) \longrightarrow \Ext^1_{(\varphi, \Gamma)}(\cM(D)[1/t], \cM(D)[1/t])
\end{equation}
where the second $\Ext^1_{(\varphi, \Gamma)}$ means extensions of $(\varphi, \Gamma)$-modules over $\cR_{K,E}[1/t]$ as defined at the beginning of \cite[\S~3.3]{BHS19} and where the map in (\ref{eq:Einvt}) comes by functoriality from the inclusion $\cM(D)\hookrightarrow \cM(D)[1/t]$. As $\cM(D)$ is crystalline and $D$ satisfies (\ref{eq:phi}), by \cite[Lemma 3.4.7]{BHS19} there is an isomorphism of $(\varphi, \Gamma)$-modules over $\cR_{K,E}[1/t]$
\begin{equation*}
\cM(D)[1/t]\cong \bigoplus_{i=0}^{n-1} \cR_{K,E}(\unr(\varphi_i))[1/t]
\end{equation*}
where $\cR_{K,E}(\unr(\varphi_i))$ is the rank one $(\varphi, \Gamma)$-module associated to the character $\unr(\varphi_i):K^\times\rightarrow E^\times$ (\cite[Cons.~6.2.4]{KPX14}). Hence we have
\begin{multline}\label{mult:Etri}
\Ext^1_{(\varphi, \Gamma)}(\cM(D)[1/t], \cM(D)[1/t]) \\
\cong \bigoplus_{i,j}\Ext^1_{(\varphi, \Gamma)}\big(\cR_{K,E}(\unr(\varphi_i))[1/t], \cR_{K,E}(\unr(\varphi_j))[1/t]\big).
\end{multline}
Using \cite[Cor.~1.4.6]{Be11} together with (\ref{eq:phi}) (which implies that all $H^0$ in \emph{loc.~cit.}~are $0$), by d\'evissage we see that the image of $\Ext^1_g(\cM(D), \cM(D))$ in
\[\Ext^1_{(\varphi, \Gamma)}\big(\cR_{K,E}(\unr(\varphi_i))[1/t], \cR_{K,E}(\unr(\varphi_j))[1/t]\big)\]
via (\ref{eq:Einvt}) and (\ref{mult:Etri}) coincides with the image of $\Ext^1_{g}(\cR_{K,E}(\unr(\varphi_i)), \cR_{K,E}(\unr(\varphi_j)))$. Using $\Ext^1_{g}(\cR_{K,E}(\unr(\varphi_i)), \cR_{K,E}(\unr(\varphi_j)))=0$ when $i\ne j$ (which uses (\ref{eq:phi})), we obtain that the image of $\Ext^1_g(\cM(D), \cM(D))$ in $\Ext^1_{(\varphi, \Gamma)}(\cM(D)[1/t], \cM(D)[1/t])$ via (\ref{eq:Einvt}) lands into the direct summand $\bigoplus_{i=0}^{n-1} \Ext^1_{(\varphi, \Gamma)}(\cR_{K,E}(\unr(\varphi_i))[1/t], \cR_{K,E}(\unr(\varphi_i))[1/t])$. Using the last isomorphism in \cite[(3.12)]{BHS19}, there are canonical isomorphisms
\begin{multline}\label{eq:isorobbai}
\Hom(K^{\times},E) \buildrel\sim\over\longrightarrow \Ext^1_{(\varphi, \Gamma)}\big(\cR_{K,E}(\unr(\varphi_i), \cR_{K,E}(\unr(\varphi_i))\big) \\
\buildrel\sim\over\longrightarrow \Ext^1_{(\varphi, \Gamma)}\big(\cR_{K,E}(\unr(\varphi_i))[1/t], \cR_{K,E}(\unr(\varphi_i))[1/t]\big)
\end{multline}
given by sending $\psi\in \Hom(K^\times,E)$ to $\cR_{K,E[\epsilon]/\epsilon^2}(\unr(\varphi_i)(1+\psi \epsilon))[1/t]$. And one easily checks (for example see \cite[\S~1.3.1]{Di172}) that (\ref{eq:isorobbai}) induces isomorphisms
\begin{multline*}
\Hom_{\sm}(K^{\times},E)\buildrel\sim\over\longrightarrow \Ext^1_g\big(\cR_{K,E}(\unr(\varphi_i)), \cR_{K,E}(\unr(\varphi_i))\big)\\
\buildrel\sim\over\longrightarrow \Ext^1_{g}\big(\cR_{K,E}(\unr(\varphi_i))[1/t], \cR_{K,E}(\unr(\varphi_i))[1/t]\big)
\end{multline*}
(the \ latter \ being \ the \ subspace \ of \ de \ Rham \ extensions). \ It follows \ that, \ for \ $\widetilde{\cM}$ \ in $\Ext^1_g(\cM(D), \cM(D))$, we can write
\begin{equation}\label{edecom}
\widetilde{\cM}[1/t]\ \cong \ \bigoplus_{i=0}^{n-1} \widetilde{\cM}_i[1/t] \ \cong \ \bigoplus_{i=0}^{n-1} \cR_{K,E[\epsilon]/\epsilon^2}(\unr(\varphi_i)(1+\psi_i \epsilon))[1/t]
\end{equation}
for some $\psi_i \in \Hom_{\sm}(K^{\times},E)$. However, as $\psi_i$ is trivial on $\cO_K^{\times}$, one directly computes
\[\dim_E \big(\cR_{K,E[\epsilon]/\epsilon^2}(\unr(\varphi_i)(1+\psi_i \epsilon))[1/t]\big)^{\Gamma}=2[K_0:\Q_p],\]
which implies that $\cR_{K,E[\epsilon]/\epsilon^2}(\unr(\varphi_i)(1+\psi_i \epsilon))$ is crystalline (see \cite[\S~1.2.3]{Be11}). Using (\ref{edecom}), we finally obtain that $\widetilde{\cM}$ is also crystalline.
\end{proof} 

Let $\widetilde{\cM}\in \Ext^1_g(\cM(D), \cM(D))$, by Lemma \ref{lem:cris} and the freeness of $D_{\pdR}(\widetilde{\cM})$ over $K \otimes_{\Qp} E[\epsilon]/\epsilon^2$, we have that $D_{\cris}(\widetilde{\cM})$ is a free $K_0 \otimes_{\Qp} E[\epsilon]/\epsilon^2$-module of rank $n$. It is also endowed with the Frobenius $\varphi$ coming from the one on $\widetilde{\cM}$. Using the isomorphism of $\varphi$-module $D_{\cris}({\cM}(D))\cong D$ (which follows for instance from \cite[Thm.~A]{Be081}) with (\ref{eq:dec}) and (\ref{eq:K0}), we obtain again a canonical decomposition $D_{\cris}(\widetilde{\cM})\otimes_{K_0}K\cong \bigoplus_{\sigma\in \Sigma} D_{\cris}(\widetilde{\cM})_{\sigma}$ of $\varphi^f$-modules over $E[\epsilon]/\epsilon^2$ where the $\varphi^f$-module $D_{\cris}(\widetilde{\cM})_{\sigma}\cong D_{\cris}(\widetilde{\cM}) \otimes_{K_0\otimes E,\sigma\vert_{K_0}\otimes \id}E$ for $\sigma\in \Sigma$ is a deformation over $E[\epsilon]/\epsilon^2$ of the $\varphi^f$-module $D_\sigma$. For $\sigma\in \Sigma$, we have thus obtained a canonical $E$-linear morphism
\begin{equation}\label{eq:Ephi}
\Ext^1_g(\cM(D), \cM(D)) \longrightarrow \Ext^1_{\varphi^f}(D_{\sigma}, D_{\sigma}), \ \ \widetilde{\cM} \longmapsto D_{\cris}(\widetilde{\cM})_{\sigma}.
\end{equation}
Note that the $\varphi^f$-module $D_{\cris}(\widetilde{\cM})_{\sigma}$ only depends on the restriction $\sigma\vert_{K_0}$ of $\sigma:K\hookrightarrow E$, and using the isomorphism of $\varphi^f$-modules
\[\varphi\otimes \id:D_{\cris}(\widetilde{\cM}) \otimes_{K_0\otimes E,\sigma\vert_{K_0}\otimes \id}E\buildrel\sim\over\longrightarrow D_{\cris}(\widetilde{\cM}) \otimes_{K_0\otimes E,\sigma\circ\varphi^{-1}\vert_{K_0}\otimes \id}E\]
in fact does not depend on $\sigma$ at all (like the $\varphi^f$-module $D_\sigma$). 

\begin{lem}\label{lem:surj2}
The map (\ref{eq:Ephi}) is surjective.
\end{lem}
\begin{proof}
By (the $E[\epsilon]/\epsilon^2$-version of) \cite[Thm.~A]{Be081}, it suffices to show that for any deformation $\widetilde{D}_{\sigma} \in \Ext^1_{\varphi^f}(D_{\sigma}, D_{\sigma})$ of $\varphi^f$-modules, there exists a filtered $\varphi$-module $(\widetilde{D}, \varphi, \Fil^{\bullet}(\widetilde{D}_K))$ (with $\widetilde{D}_K := K \otimes_{K_0} \widetilde{D}$) over $E[\epsilon]/\epsilon^2$ deforming $(D, \varphi, \Fil^{\bullet}(D_K))$. For $i = 0, \ldots, n-1$, let $\widetilde{e}_{i,\sigma} \in \widetilde{D}_{\sigma}$ be a lift of $e_{i,\sigma}$ that is a generalized $\varphi_i$-eigenvector, then $\widetilde{D}_{\sigma} \cong \bigoplus_{i=0}^{n-1} E[\epsilon]/\epsilon^2 \,\widetilde{e}_{i,\sigma}$. For $\tau = \sigma \circ \varphi^{-j} : K_0 \hookrightarrow E$, define $\widetilde{D}_{\tau}$ to be the same $\varphi^f$-module as $\widetilde{D}_{\sigma}$ with basis labelled by $\widetilde{e}_{i,\tau}$. Define the $K_0$-semi-linear and $E$-linear endomorphism
\[\varphi : \widetilde{D} := \prod_{\tau : K_0 \hookrightarrow E} \widetilde{D}_\tau \xlongrightarrow{\sim} \prod_{\tau : K_0 \hookrightarrow E} \widetilde{D}_\tau,\]
by sending $\widetilde{e}_{i,\tau}$ to $\widetilde{e}_{i,\tau\circ \varphi^{-1}}$. Then $\widetilde{D}$ is a deformation of the $\varphi$-module $D$ over $E[\epsilon]/\epsilon^2$. In particular, there is a $\varphi$-equivariant surjection $\widetilde{D} \twoheadrightarrow D$ sending $\widetilde{e}_{i,\sigma\circ \varphi^{-j}}$ to $\varphi^{j}(e_{i,\sigma})$ for $j = 0, \ldots, f-1$. Choose a filtration $\Fil^{\bullet}(\widetilde{D}_K)$ of $\widetilde{D}_K = \widetilde{D} \otimes_{K_0} K$ by $K \otimes_{\Q_p} E[\epsilon]/\epsilon^2$-submodules which agrees with $\Fil^{\bullet}(D_K)$ modulo $\epsilon$ (via $\widetilde{D}_K \twoheadrightarrow D_K$). Then $(\widetilde{D}, \varphi, \Fil^{\bullet}(\widetilde{D}_K))$ is a deformation of $(D, \varphi, \Fil^{\bullet}(D_K))$ over $E[\epsilon]/\epsilon^2$ and the lemma follows.
\end{proof}

We denote by $\Ext^1_0(\cM(D), \cM(D))$ the kernel of (\ref{eq:Ephi}), which does not depend on $\sigma\in \Sigma$, and for any subspace $\Ext^1_*(\cM(D), \cM(D))$ of $\Ext^1_{(\varphi, \Gamma)}(\cM(D), \cM(D))$ which contains $\Ext^1_0(\cM(D), \cM(D))$ we define
\begin{equation}\label{eq:surligne}
\ol{\Ext}^1_*(\cM(D), \cM(D)):=\Ext^1_*(\cM(D), \cM(D))\ /\ \Ext^1_0(\cM(D), \cM(D)).
\end{equation}
Hence, by Lemma \ref{lem:surj2}, (\ref{eq:Ephi}) induces an isomorphism for any $\tau\in \Sigma$)
\begin{equation}\label{eq:isog}
\ol{\Ext}^1_g(\cM(D), \cM(D))\buildrel\sim\over\longrightarrow \Ext^1_{\varphi^f}(D_{\tau},D_{\tau})
\end{equation}
and together with Lemma \ref{lem:LpdR} we obtain an exact sequence of $E$-vector spaces
\begin{equation}\label{eq:EbarpdR}
0 \longrightarrow \Ext^1_{\varphi^f}(D_{\tau},D_{\tau}) \longrightarrow \ol{\Ext}^1_{(\varphi, \Gamma)}(\cM(D), \cM(D)) \longrightarrow \bigoplus_{\sigma\in \Sigma} \Hom_{\Fil}(D_{\sigma}, D_{\sigma}) \longrightarrow 0.
\end{equation}
The exact sequence (\ref{eq:EbarpdR}) is part of a commutative diagram that we explain now. Using that $D_{\pdR}(\widetilde{\cM})$ and $\nu_{\widetilde{\cM}}$ in (\ref{eq:nu}) in fact only depend on $\widetilde\cM[1/t]$ (see the discussion below (\ref{eq:DpdR})), the map (\ref{eq:Enil}) factors through a map
{\small
\begin{equation}\label{eq:Enilt}
\Ext^1_{(\varphi, \Gamma)}(\cM(D), \cM(D)) \buildrel{(\ref{eq:Einvt})}\over \rightarrow \Ext^1_{(\varphi, \Gamma)}(\cM(D)[1/t], \cM(D)[1/t]) \rightarrow \bigoplus_{\sigma\in \Sigma} \Hom_{E}(D_{\sigma}, D_{\sigma}).
\end{equation}}
\!\!\!As in \S~\ref{sec:def}, for $\sigma\in \Sigma$ we fix a basis $e_{0,\sigma}, \dots, e_{n-1,\sigma}$ of $\varphi^f$-eigenvectors of $D_{\sigma}$ such that $\varphi^f(e_{i,\sigma})=\varphi_i e_{i,\sigma}$. If $i\ne j$, it easily follows from the last isomorphism in \cite[(3.12)]{BHS19} that the second map in (\ref{eq:Enilt}) with (\ref{mult:Etri}) induce an isomorphism
\[\Ext^1_{(\varphi, \Gamma)}\big(\cR_{K,E}(\unr(\varphi_i))[1/t], \cR_{K,E}(\unr(\varphi_j))[1/t]\big) \buildrel \sim\over\longrightarrow \bigoplus_\sigma \Hom_E(Ee_{i,\sigma}, Ee_{j,\sigma}).\]
If $i=j$, it follows from the proof of Lemma \ref{lem:cris} that the inclusion $\Hom_{\sm}(K^\times,E)\hookrightarrow \Hom(K^\times,E)$ induces a short exact sequence
\begin{multline}\label{mult:i=j}
0\longrightarrow \Hom_{\sm}(K^\times,E)\longrightarrow \Ext^1_{(\varphi, \Gamma)}\big(\cR_{K,E}(\unr(\varphi_i))[1/t], \cR_{K,E}(\unr(\varphi_i))[1/t]\big)\\
\longrightarrow \bigoplus_\sigma \Hom_E(Ee_{i,\sigma}, Ee_{i,\sigma})\longrightarrow 0.
\end{multline}
Using
\[\Ext^1_{\varphi^f}(D_\tau,D_\tau)\cong \bigoplus_{i=0}^{n-1}\Ext^1_{\varphi^f}(Ee_{i,\tau}, Ee_{i,\tau})\cong \bigoplus_{i=0}^{n-1}\Hom_{\sm}(K^\times,E)\]
where the last isomorphism is (\ref{eq:phiepsilon}) (for the refinement $(\varphi_{j_1},\dots,\varphi_{j_n})=(\varphi_{0},\dots,\varphi_{n-1})$) we obtain a short exact sequence
{\small
\begin{equation}\label{eq:EbarpdRt}
0 \longrightarrow \Ext^1_{\varphi^f}(D_{\tau},D_{\tau}) \longrightarrow {\Ext}^1_{(\varphi, \Gamma)}(\cM(D)[1/t], \cM(D)[1/t]) \longrightarrow \bigoplus_{\sigma\in \Sigma} \Hom_E(D_{\sigma}, D_{\sigma}) \longrightarrow 0.
\end{equation}}
\!\!\!We \ deduce \ from \ (\ref{eq:isog}), \ (\ref{eq:EbarpdR} \ and \ (\ref{eq:EbarpdRt}) \ that \ (\ref{eq:Einvt}) \ factors \ through \ the \ quotient $\ol{\Ext}^1_{(\varphi, \Gamma)}(\cM(D), \cM(D))$ and that we have a canonical commutative diagram of short exact sequences (for any $\tau\in \Sigma$)
\[\xymatrix{
0\ar[r] & \Ext^1_{\varphi^f}(\!D_{\tau},\!D_{\tau}\!)\ar[r]\ar@{=}[d] & \ol{\Ext}^1_{(\varphi, \Gamma)}(\cM(D), \!\cM(D))\ar@{^{(}->}^{(\ref{eq:Einvt})}[d]\ar[r] & \displaystyle{\bigoplus_{\sigma\in \Sigma}} \Hom_{\Fil}(D_{\sigma}, \!D_{\sigma}) \ar@{^{(}->}[d] \ar[r] & 0 \\
0\ar[r] & \Ext^1_{\varphi^f}(\!D_{\tau},\!D_{\tau}\!)\ar[r] & \Ext^1_{(\varphi, \Gamma)}(\cM(D)[1/t], \!\cM(D)[1/t]) \ar[r] & \displaystyle{\bigoplus_{\sigma\in \Sigma}} \Hom_{E}(D_{\sigma}, \!D_{\sigma}) \ar[r] & 0.}\]

\begin{prop}\label{prop:Psplit}
There is a splitting of the exact sequence (\ref{eq:EbarpdR}) which only depends on a choice of $\log(p)\in E$. 
\end{prop}
\begin{proof}
By the above commutative diagram it is enough to construct a splitting of the second exact sequence. From (\ref{mult:Etri}) we have a surjection 
\begin{multline}\label{eq:Etri}
\Ext^1_{(\varphi, \Gamma)}(\cM(D)[1/t], \cM(D)[1/t]) \twoheadrightarrow \bigoplus_{i=0}^{n-1} \Ext^1_{(\varphi, \Gamma)}\big(\cR_{K,E}(\unr(\varphi_i))[1/t], \cR_{K,E}(\unr(\varphi_i))[1/t]\big) \\ \buildrel\sim\over\longrightarrow \bigoplus_{i=0}^{n-1}\Hom(K^\times,E)
\end{multline}
where the last isomorphism sends $(\cR_{K,E[\epsilon]/\epsilon^2}(\unr(\varphi_i)(1+\psi_i \epsilon)))_{i\in \{0,\dots,n-1\}}$ to $(\psi_i)_{i\in \{0,\dots,n-1\}}$. The choice of $\log(p)$ gives a projection $\Hom(K^\times,E) \cong E\val \bigoplus (\bigoplus_{\tau\in \Sigma}\tau\circ \log)\twoheadrightarrow E\val\cong \Hom_{\sm}(K^{\times},E)$ sending all $\tau\circ \log$ to $0$ where $\tau\circ \log$ is the branch of the logarithm in $\Hom_\tau(K^\times,E)$ associated to $\log(p)$. This projection induces a surjection
\[\Ext^1_{(\varphi, \Gamma)}(\cM(D)[1/t], \cM(D)[1/t]) \twoheadlongrightarrow \bigoplus_{i=0}^{n-1}\Hom_{\sm}(K^\times,E) \buildrel (\ref{eq:phiepsilon}) \over \cong \Ext^1_{\varphi^f}(D_{\sigma},D_{\sigma})\]
which gives the sought after splitting.
\end{proof}

For $\sigma\in \Sigma$ denote by $\Ext^1_{\sigma}(\cM(D),\cM(D))$ the kernel of the composition 
\[\Ext^1_{(\varphi, \Gamma)}(\cM(D), \cM(D)) \buildrel(\ref{eq:Enil})\over \longrightarrow \bigoplus_{\tau\in \Sigma} \Hom_{\Fil}(D_{\tau}, D_{\tau})\twoheadrightarrow \bigoplus_{\tau\neq \sigma} \Hom_{\Fil}(D_{\tau}, D_{\tau}),\]
or equivalently the preimage of $\Hom_{\Fil}(D_{\sigma}, D_{\sigma})$ via (\ref{eq:Enil}). The subspace $\Ext^1_{\sigma}(\cM(D),\cM(D))$ consists of those $\widetilde{\cM}\in \Ext^1_{(\varphi, \Gamma)}(\cM(D),\cM(D))$ such that $D_{\dR}(\widetilde{\cM})_{\tau}$ is free of rank $n$ over $E[\epsilon]/\epsilon^2$ (hence equal to $D_{\pdR}(\widetilde{\cM})_{\tau}$) for all $\tau \in \Sigma \setminus \{\sigma\}$ (such extensions are called $\Sigma_L \setminus \{\sigma\}$-de Rham). It obviously contains $\Ext^1_g(\cM(D), \cM(D))$ hence $\Ext^1_0(\cM(D), \cM(D))$, and by (\ref{eq:EbarpdR}) with Proposition \ref{prop:Psplit} we deduce

\begin{cor}\label{cor:splitsigma}
Fix $\sigma\in\Sigma$. There is an isomorphism which only depends on a choice of $\log(p)\in E$
\begin{equation}\label{eq:Esplit}
\ol{\Ext}^1_{\sigma}(\cM(D),\cM(D)) \buildrel\sim\over\longrightarrow \Ext^1_{\varphi^f}(D_{\sigma},D_{\sigma}) \bigoplus \Hom_{\Fil}(D_{\sigma}, D_{\sigma})
\end{equation}
and which is (\ref{eq:isog}) in restriction to $\ol{\Ext}^1_{g}(\cM(D),\cM(D))$.
\end{cor}

We can now state the main result of this section.

\begin{thm}\label{thm:independant}
Fix $\sigma\in \Sigma$. The composition
\begin{equation}\label{eq: tDsgima'}(\ref{eq:Esplit})^{-1} \circ \ t_{D_\sigma}:\Ext^1_{\GL_n(K),\sigma}(\pi_{\alg}(D_\sigma),\pi_R(D_\sigma))\twoheadlongrightarrow \ol{\Ext}^1_{\sigma}(\cM(D),\cM(D))\end{equation}
(for the same choice of $\log(p)$ in (\ref{eq:Esplit}) and $t_{D_\sigma}$) does not depend on any choice up to isomorphism.
\end{thm}
\begin{proof}
First, the ``up to isomorphism'' in the statement means that there is a commutative diagram as (\ref{eq:small}) with $\ol{\Ext}^1_{\sigma}(\cM(D),\cM(D))$ instead of $\Ext^1_{\varphi^f}(D_\sigma,D_\sigma) \bigoplus\Hom_{\Fil}(D_\sigma,D_\sigma)$.\bigskip

Since $\Ext^1_{\GL_n(K),\sigma}(\pi_{\alg}(D_{\sigma}), \pi_R(D_{\sigma}))$ is spanned by the $\Ext^1_{\GL_n(K),\sigma}(\pi_{\alg}(D_{\sigma}), \pi_I(D_{\sigma}))$ for $I\subseteq \{\varphi_0,\dots,\varphi_{n-1}\}$ of cardinality in $\{1,\dots,n-1\}$ (see (\ref{eq:amalgi}) and (\ref{eq:amalg4})), it is enough to prove Theorem \ref{thm:independant} with $\pi_I(D_\sigma)$ instead of $\pi_R(D_\sigma)$.\bigskip

We fix $I$ of cardinality $i\in \{1,\dots,n-1\}$. Fix a choice of $\log(p)\in E$, from the definition of (\ref{eq:Esplit}), it is enough to prove that the composition
\begin{multline}\label{mult:compo}
\Ext^1_{\GL_n(K),\sigma}(\pi_{\alg}(D_\sigma),\pi_I(D_\sigma))\buildrel t_{D_\sigma}\over \longrightarrow \Ext^1_{\varphi^f}(D_{\tau},D_{\tau})\bigoplus \Hom_{\Fil}(D_{\sigma}, D_{\sigma})\\
\ \ \ \ \ \ \ \ \ \ \ \ \ \ \ \ \ \ \ \ \ \ \ \ \ \ \ \ \ \ \ \hooklongrightarrow \Ext^1_{\varphi^f}(D_{\tau},D_{\tau})\bigoplus \big(\bigoplus_{\sigma\in \Sigma}\Hom_{E}(D_{\sigma}, D_{\sigma})\big)\\
\buildrel \sim \over \longrightarrow \Ext^1_{(\varphi, \Gamma)}(\cM(D)[1/t], \cM(D)[1/t]) 
\end{multline}
does not depend on $\log(p)$, where $t_{D_\sigma}$ is defined using $\log(p)$ and the last isomorphism is the splitting of the bottom exact sequence of the diagram above Proposition \ref{prop:Psplit} associated to $\log(p)$ (see the proof of \emph{loc.~cit.}). Indeed, using (\ref{eq:smallI}) we see that (\ref{mult:compo}) does not depend on the choice of an isomorphism $\varepsilon_I$ as in (\ref{eq:epsilonI}).\bigskip

Let $c\in \Ext^1_{\GL_n(K),\sigma}(\pi_{\alg}(D_\sigma),\pi_I(D_\sigma))$ and write $t_{D_\sigma}(c)=e(c)_{\log(p)} + f(c)$ where $e(c)_{\log(p)}\in \Ext^1_{\varphi^f}(D_{\tau},D_{\tau})$ and $f(c)\in \Hom_{\Fil}(D_{\sigma}, D_{\sigma})$ (as the notation suggests $e(c)_{\log(p)}$ may depend on $\log(p)$ while $f(c)$ does not). By Lemma \ref{lem:nuI} there is $\lambda(c)\in E$ such that
\begin{equation}\label{eq:reviemlog}
\left\{\begin{array}{ccll}
f(c)(e_{j,\sigma})& \in &\lambda(c) e_{j,\sigma} + \bigoplus_{j'\ne j}Ee_{j',\sigma}& \text{ if } \pi_I(D_{\sigma})\text{ is non-split and }\varphi_j\notin I\\
f(c)(e_{j,\sigma})&\in & \bigoplus_{j'\ne j}Ee_{j',\sigma}& \text{ otherwise}.
\end{array}\right.
\end{equation}
Assume first that $\pi_I(D_{\sigma})$ is split. Then using (\ref{eq:splitI}) we see from the proof of Proposition \ref{prop:map} that $t_{D_\sigma}(c)$ does not depend on $\log(p)$. Let $\Hom_{E}(D_{\sigma}, D_{\sigma})_0\subset \Hom_{E}(D_{\sigma}, D_{\sigma})$ be the (canonical) subspace of endomorphisms $f$ such that $f(e_{j,\sigma}) \in \bigoplus_{j'\ne j}Ee_{j',\sigma}$ for $j\in \{0,\dots,n-1\}$, from (\ref{eq:reviemlog}) we have $f(c)\in \Hom_{E}(D_{\sigma}, D_{\sigma})_0$. Let $\Ext^1_{(\varphi, \Gamma)}(\cM(D)[1/t], \cM(D)[1/t])_0$ be the inverse image of $\bigoplus_{\sigma\in \Sigma}\Hom_{E}(D_{\sigma}, D_{\sigma})_0$ via the bottom exact sequence of the diagram above Proposition \ref{prop:Psplit}. It readily follows from the proof of \emph{loc.~cit.}~that there is a \emph{canonical} splitting
\[\Ext^1_{\varphi^f}(D_{\tau},D_{\tau})\bigoplus \big(\bigoplus_{\sigma\in \Sigma}\Hom_{E}(D_{\sigma}, D_{\sigma})_0\big) \buildrel\sim\over\longrightarrow \Ext^1_{(\varphi, \Gamma)}(\cM(D)[1/t], \cM(D)[1/t])_0.\]
Since $t_{D_\sigma}(c)\in \Ext^1_{\varphi^f}(D_{\tau},D_{\tau})\bigoplus \Hom_{E}(D_{\sigma}, D_{\sigma})_0$, we see that the image of $t_{D_\sigma}(c)$ by the composition \ref{mult:compo} does not depend on $\log(p)$. Assume now that $\pi_I(D_{\sigma})$ is non-split. It follows from the proof of Proposition \ref{prop:Psplit} that the splitting of $\Ext^1_{(\varphi, \Gamma)}(\cM(D)[1/t], \cM(D)[1/t])$ associated to $\log(p)$ constructed there induces a splitting
\begin{multline}\label{mult:parti}
\Ext^1_{\varphi^f}(D_{\tau},D_{\tau})\bigoplus \Big(\bigoplus_{j=0}^{n-1}\bigoplus_{\sigma\in \Sigma} \Hom_E(Ee_{j,\sigma},Ee_{j,\sigma})\Big)\\
\buildrel \sim\over \longrightarrow \bigoplus_{j=0}^{n-1}\Ext^1_{(\varphi, \Gamma)}\big(\cR_{K,E}(\unr(\varphi_j))[1/t], \cR_{K,E}(\unr(\varphi_j))[1/t]\big)
\end{multline}
and that it is enough to check that the projection of the image of $t_{D_\sigma}(c)$ under (\ref{mult:parti}) does not depend on $\log(p)$. By (\ref{eq:subtil}) with (\ref{eq:extphi}), (\ref{eq:phiepsilon}), we have
\begin{equation}\label{eq:e(c)}
e(c)_{\log(p)'}=e(c)_{\log(p)}+\lambda(c)(\log(p)-\log(p)')E_I
\end{equation}
where $E_I\in \Ext^1_{\varphi^f}(D_{\tau},D_{\tau})\cong \bigoplus_{j=0}^{n-1}\Hom_{\sm}(K^\times,E)$ has entry $\val$ for $j$ such that $\varphi_j\notin I$ and $0$ elsewhere. For $j\in \{0,\dots,n-1\}$ denote by $e(c)_{\log(p),j}$ (resp.~$e(c)_{\log(p)',j}$) the $j$-th entry of $e(c)_{\log(p)}$ (resp.~$e(c)_{\log(p)'}$) in $\Hom_{\sm}(K^\times,E)$ by the above isomorphism. Let $\sigma \circ \log$ (resp.~$\sigma\circ \log'$) be the branch of the logarithm in $\Hom_{\sigma}(K^\times,E)$ associated to $\log(p)$ (resp.~$\log(p)'$). By (\ref{eq:reviemlog}) with (\ref{mult:i=j}) and the discussion that follows (\ref{eq:Etri}), the image of $t_{D_\sigma}(c)$ by (\ref{mult:parti}) for the choice of $\log(p)$ is
\begin{equation}\label{eq:e(c)j}
\sum_{\varphi_j\in I}e(c)_{\log(p),j} + \sum_{\varphi_j\notin I}\big(e(c)_{\log(p),j} + \lambda(c) \sigma\circ \log \big)\in \bigoplus_{j=0}^{n-1}\Hom_{\sigma}(K^\times, E)\subset \bigoplus_{j=0}^{n-1}\Hom(K^\times, E)
\end{equation}
and similarly with $e(c)_{\log(p)',j}$ for the choice of $\log(p)'$. Using (\ref{eq:e(c)}), we can rewrite (\ref{eq:e(c)j}) as 
\begin{multline*}
\sum_{\varphi_j\in I}e(c)_{\log(p)',j} + \sum_{\varphi_j\notin I}\big(e(c)_{\log(p)',j} + \lambda(c)(\log(p)'-\log(p))\val + \lambda(c) \sigma\circ \log \big)\\
=\sum_{\varphi_j\in I}e(c)_{\log(p)',j} + \sum_{\varphi_j\notin I}\big(e(c)_{\log(p)',j} + \lambda(c) \sigma\circ \log' \big)
\end{multline*}
which shows that the image of $t_{D_\sigma}(c)$ under (\ref{mult:parti}) does not depend on $\log(p)$. This finishes the proof.
\end{proof}

\subsection{Trianguline deformations and comparison with \texorpdfstring{\cite{Di25}}{Di25}}\label{sec:tri}

We prove that the map in Theorem \ref{thm:independant} gives back the map $t_{\cM(D),\sigma}$ of \cite[(3.39)]{Di25} and \cite[Cor.~3.29(1)]{Di25} when all the refinements on $D_\sigma$ for all $\sigma\in \Sigma$ are non-critical (Corollary \ref{cor:Di25}). We also prove several results (not necessarily in the non-critical case) which will be used later.\bigskip

We keep the notation of \S~\ref{sec:indep} and we fix $I\subseteq \{\varphi_0,\dots,\varphi_{n-1}\}$ of cardinality in $\{1,\dots,n-1\}$. As they are heavily used in \cite{Di25}, we need $\fR$-trianguline deformations for $\fR$ a refinement, and we also fix a refinement $\mathfrak{R}$ compatible with the fixed subset $I$. In order to simplify notation, up to renumbering the $\varphi_i$ we can and do assume $\mathfrak{R}=(\varphi_0, \dots, \varphi_{n-1})$ (and $I=\{\varphi_0,\dots,\varphi_{i-1}\}$). Correspondingly, we have a filtration of $\cM(D)[1/t]$ by free $\cR_{K,E}[1/t]$-submodules: 
\begin{multline*}
\cR_{K,E}(\unr(\varphi_0))[1/t]\subset \cR_{K,E}(\unr(\varphi_0))[1/t]\bigoplus \cR_{K,E}(\unr(\varphi_1))[1/t]\\
\subset \cdots \subset \bigoplus_{i=0}^{n-1} \cR_{K,E}(\unr(\varphi_i))[1/t]\cong \cM(D)[1/t].
\end{multline*}
An extension $\widetilde{\cN}\in \Ext^1_{(\varphi, \Gamma)}(\cM(D)[1/t], \cM(D)[1/t])$ is called an $\cR$-trianguline deformation of $\cM(D)[1/t]$ over $\cR_{K,E[\epsilon]/\epsilon^2}[1/t]$ if $\widetilde{\cN}$ admits an increasing filtration $0=\Fil_{-1}\subset \Fil_0\subset \Fil_1\subset \cdots \subset \Fil_{n-1}=\widetilde{\cN}$ by $(\varphi, \Gamma)$-submodules over $\cR_{K,E[\epsilon]/\epsilon^2}[1/t]$ which are direct summands as $\cR_{K,E[\epsilon]/\epsilon^2}[1/t]$-modules and such that $\Fil_i/\Fil_{i-1}$, $i\in \{0,\dots,n-1\}$ is isomorphic to $\cR_{K,E[\epsilon]/\epsilon^2}(\unr(\varphi_i)(1+\psi_i\epsilon))[1/t]$ for some $\psi_i\in \Hom(K^{\times},E)$. We call $(\unr(\varphi_0)(1+\psi_0\epsilon),\dots,\unr(\varphi_{n-1})(1+\psi_{n-1}\epsilon))$ a trianguline parameter of $\widetilde{\cN}$. We define
\[\Ext^1_{\mathfrak{R}}(\cM(D)[1/t],\cM(D)[1/t])\subseteq \Ext^1_{(\varphi, \Gamma)}(\cM(D)[1/t], \cM(D)[1/t])\]
the subspace of $\cR$-trianguline deformations of $\cM(D)[1/t]$ over $\cR_{K,E[\epsilon]/\epsilon^2}[1/t]$. We denote by $\Ext^1_{\fR}(\cM(D), \cM(D))$ the preimage of $\Ext^1_{\fR}(\cM(D)[1/t], \cM(D)[1/t])$ via (\ref{eq:Einvt}). So $\widetilde{\cM}$ lies in $\Ext^1_{\fR}(\cM(D), \cM(D))$ if and only if $\widetilde{\cM}[1/t]$ is $\fR$-trianguline. It follows from the discussion below (\ref{mult:i=j}) that we have $\Ext^1_{g}(\cM(D), \cM(D))\subseteq \Ext^1_{\fR}(\cM(D), \cM(D))$. For $\sigma\in \Sigma$ we define the subspaces
\begin{equation*}
\Hom_{\fR}(D_{\sigma}, D_{\sigma})=\left\{f\in \Hom_{E}(D_{\sigma}, D_{\sigma}),\ f(e_{j,\sigma})\in \bigoplus_{k=0}^{j}Ee_{k,\sigma}, \forall \ 0\leq j\leq n-1\right\}
\end{equation*}
and $\Hom_{\Fil,\fR}(D_{\sigma}, D_{\sigma}):=\Hom_{\Fil}(D_{\sigma}, D_{\sigma})\cap \Hom_{\fR}(D_{\sigma}, D_{\sigma})$. Then the second map in (\ref{eq:Enilt}) induces a map
\begin{equation}\label{eq:Rt}
\Ext^1_{\fR}(\cM(D)[1/t], \cM(D)[1/t])\longrightarrow \bigoplus_{\sigma\in \Sigma}\Hom_{\fR}(D_{\sigma}, D_{\sigma})
\end{equation}
and one easily checks using (\ref{mult:Etri}) that the commutative diagram above Proposition \ref{prop:Psplit} induces another commutative diagram 
\[\xymatrix{
0\ar[r] & \Ext^1_{\varphi^f}(\!D_{\tau},\!D_{\tau}\!)\ar[r]\ar@{=}[d] & \ol{\Ext}^1_{\fR}(\cM(D), \!\cM(D))\ar@{^{(}->}^{(\ref{eq:Einvt})}[d]\ar[r] & \displaystyle{\bigoplus_{\sigma\in \Sigma}} \Hom_{\Fil,\fR}(D_{\sigma}, \!D_{\sigma}) \ar@{^{(}->}[d] \ar[r] & 0 \\
0\ar[r] & \Ext^1_{\varphi^f}(\!D_{\tau},\!D_{\tau}\!)\ar[r] & \Ext^1_{\fR}(\cM(D)[1/t], \!\cM(D)[1/t]) \ar[r] & \displaystyle{\bigoplus_{\sigma\in \Sigma}} \Hom_{\fR}(D_{\sigma}, \!D_{\sigma}) \ar[r] & 0.}\]

There is also a canonical map
\begin{equation}\label{eq:Etri2}
\Ext^1_{\fR}(\cM(D)[1/t], \cM(D)[1/t]) \longrightarrow \Hom(T(K),E)
\end{equation}
sending $\widetilde{\cN}$ of trianguline parameter $(\unr(\varphi_i)(1+\psi_i\epsilon))_{i\in \{0,\dots,n-1\}}$ to $(\psi_i)_{i\in \{0,\dots,n-1\}}$. It is easy to check that (\ref{eq:Etri2}) coincides with the restriction of (\ref{eq:Etri}) and it follows from (\ref{mult:Etri}) that the map (\ref{eq:Etri2}) is (still) surjective. For $\sigma\in \Sigma$ and $f\in \Hom_{\fR}(D_{\sigma}, D_{\sigma})$ let $(a_j)_{j\in \{0,\dots,n-1\}}\!\in E^{\oplus n}$ such that $f(e_{j,\sigma})-a_je_{j,\sigma}\in \bigoplus_{k=0}^{j-1}Ee_{k,\sigma}$, $0\leq j\leq n-1$. Sending $f$ to $(a_j \sigma \circ \log)_{j\in \{0,\dots,n-1\}}$ defines a canonical surjection
\begin{equation}\label{eq:EhomfilR}
\Hom_{\fR}(D_{\sigma}, D_{\sigma}) \twoheadlongrightarrow \Hom_{\sigma}(T(\cO_K), E)
\end{equation}
and one readily checks that there is a commutative diagram of surjective maps
\begin{equation}\label{eq:diagR}
\begin{gathered}
\xymatrix{
\Ext^1_{\fR}(\cM(D)[1/t], \cM(D)[1/t])\ar@{->>}[r]^{\ \ \ \ \ \ \ \ \ (\ref{eq:Etri2})}\ar@{->>}[d]^{(\ref{eq:Rt})} & \Hom(T(K),E) \ar@{->>}[d]^{\text{res}} \\
\displaystyle{\bigoplus_{\sigma\in \Sigma}}\Hom_{\fR}(D_{\sigma}, D_{\sigma})\ar@{->>}[r]^{\!\!\!\!\!\!\!\!\!\!\!\!\!\!\!\!\!\!\!\!\!\!\!\!\!\!\!\!\!\!\!\!\!\!\!\!\!\!\!\!\!\!(\ref{eq:EhomfilR})} & \Hom(T(\cO_K),E)\cong \displaystyle{\bigoplus_{\sigma\in \Sigma}}\Hom_{\sigma}(T(\cO_K), E).}
\end{gathered}
\end{equation}
For $i\in \{1,\dots, n-1\}$ we also define $\Hom_{\Fil,\fR}^{i}(D_{\sigma}, D_{\sigma})\subseteq \Hom_{\Fil,\fR}(D_{\sigma}, D_{\sigma})$ by 
\begin{multline}\label{mult:explicit}
\!\!\!\!\Hom_{\Fil,\fR}^{i}(D_{\sigma}, D_{\sigma}):=\{f\in \Hom_{\Fil}(D_{\sigma}, D_{\sigma})\text{ such that }
\exists \ a, b \in E\text{ with }\\
f(e_{j,\sigma})=a e_{j,\sigma} \ \forall \ 0\leq j \leq i-1,\ f(e_{j,\sigma})-b e_{j,\sigma}\in \bigoplus_{k=0}^{i-1}Ee_{k,\sigma}\ \forall \ i\leq j\leq n-1\},
\end{multline}
and we note that (\ref{eq:EhomfilR}) restricts to a canonical map
\begin{equation}\label{eq:fisigma}
f_{i,\sigma}: \Hom_{\Fil,\fR}^{i}(D_{\sigma}, D_{\sigma})\longrightarrow \Hom_{\sigma}(L_{P_i}(\cO_K), E)\ (\buildrel\res\over \hookrightarrow \Hom_{\sigma}(T(\cO_K), E)).
\end{equation}
Using the basis $(e_{i,\sigma})_i$ we identify $\Hom_E(D_{\sigma}, D_{\sigma})$ with $\fg_{\sigma}$, hence $\Hom_{\fR}(D_{\sigma}, D_{\sigma})$ is identified with $\fb_\sigma$. For $\sigma\in \Sigma$ we choose $g_\sigma\in G(E)$ such that $g_{\sigma} B_{\sigma} \in G_{\sigma}/B_{\sigma}$ gives the ``coordinate" of the Hodge flag (\ref{eq:fil}) in the basis $(e_{i,\sigma})_i$, where the flag $E e_{0,\sigma} \subset E e_{0,\sigma} \bigoplus E e_{1,\sigma} \subset \cdots \subset \bigoplus_{i=0}^{n-1} E e_{i,\sigma}=D_{\sigma}$ has coordinate $1 B_{\sigma}\in G_{\sigma}/B_{\sigma}$. The following descriptions of $\Hom_{\Fil,\fR}^{i}(D_{\sigma}, D_{\sigma})\subset \Hom_{\Fil,\fR}(D_{\sigma}, D_{\sigma})\subset \Hom_{\Fil}(D_{\sigma}, D_{\sigma})$ will be convenient:
\begin{eqnarray}
\Ad_{g_{\sigma}} (\fb_{\sigma}) &\xlongrightarrow{\sim}& \Hom_{\Fil}(D_{\sigma}, D_{\sigma})\nonumber\\
\fb_{\sigma} \cap \Ad_{g_{\sigma}} (\fb_{\sigma}) &\xlongrightarrow{\sim}& \Hom_{\Fil,\fR}(D_{\sigma}, D_{\sigma})\label{eq:middle}\\
\mathfrak{r}_{P_i,\sigma}\cap \Ad_{g_{\sigma}} (\fb_{\sigma}) &\xlongrightarrow{\sim}& \Hom_{\Fil,\fR}^{i}(D_{\sigma}, D_{\sigma}).\label{mult:missing}
\end{eqnarray}
For instance on (\ref{mult:missing}) the map $f_{i,\sigma}$ in (\ref{eq:fisigma}) is immediately checked to be surjective.\bigskip

We let $w_{\fR}:=(w_{\fR,\sigma})_{\sigma}\in S_n^{\Sigma}$ such that $g_\sigma B_\sigma\subset B_\sigma w_{\fR,\sigma}B_\sigma$ for $\sigma\in \Sigma$. More intrinsically the permutation $w_{\fR,\sigma}$ measures the relative position of the Hodge flag on $D_\sigma$ with respect to the flag determined by the refinement $\fR$. We write $w_{0,\sigma}$ the longest element of the Weyl group of ${\GL_n}\times_{K,\sigma}E$.

\begin{prop}\label{prop:Psmooth}
Let $\sigma\in \Sigma$, the map $f_{i,\sigma}$ in (\ref{eq:fisigma}) is an isomorphism if and only if the simple reflection $s_{i,\sigma}$ does not appear in some (equivalently any) reduced expression of $w_{\fR,\sigma}w_{0,\sigma}$.
\end{prop}
\begin{proof}
Let $b_{\sigma}\in B(E)$ such that $g_{\sigma}B=b_{\sigma} w_{\fR, \sigma}B$, we have
\begin{equation*}
\fb_{\sigma} \cap \Ad_{g_{\sigma}}( \fb_{\sigma})=\Ad_{b_{\sigma}}\big(\fb_{\sigma} \cap \Ad_{w_{\fR, \sigma}}(\fb_{\sigma})\big)
=\Ad_{b_{\sigma}}\Big(\ft_{\sigma} \bigoplus (\fn_{\sigma} \cap \Ad_{w_{\fR,\sigma}} (\fn_{\sigma}))\Big)
\end{equation*}
and hence
\begin{equation*}
\mathfrak{r}_{P_i,\sigma}\cap \Ad_{g_{\sigma}} (\fb_{\sigma})=\Ad_{b_{\sigma}}\Big(\fz_{P_i,\sigma} \bigoplus (\fn_{P_i, \sigma} \cap \Ad_{w_{\fR,\sigma}} (\fn_{\sigma}))\Big).
\end{equation*}
We easily check that
\begin{multline}\label{mult:nPiw}
\dim_E (\fn_{P_i, \sigma} \cap \Ad_{w_{\fR,\sigma}} (\fn_{\sigma}))\\
=\vert\{(j,k)\!\in\!\{0,\dots, i-1\}\times \{i, \dots, n-1\}, w_{\fR,\sigma}w_{0,\sigma}(k)<w_{\fR,\sigma}w_{0,\sigma}(j)\}\vert.
\end{multline}
Assume that $s_{i,\sigma}$ does not appear in some (equivalently any) reduced expression of $w_{\fR,\sigma} w_{0,\sigma}$. From (\ref{mult:nPiw}) we deduce $\fn_{P_i, \sigma} \cap \Ad_{w_{\fR,\sigma}} (\fn_{\sigma})=0$. Since $\dim_E \fz_{P_i, \sigma} = 2$, using (\ref{mult:missing}) and the surjectivity of $f_{i,\sigma}$, comparing dimensions we deduce that $f_{i,\sigma}$ is an isomorphism. Assume that $s_{i,\sigma}$ appears in some (equivalently any) reduced expression of $w_{\fR,\sigma} w_{0,\sigma}$. Then from (\ref{mult:nPiw}) we get $\dim_E (\fn_{P_i, \sigma} \cap \Ad_{w_{\fR,\sigma}} (\fn_{\sigma}))\geq 1$, hence $\dim_E\Hom_{\Fil,\fR}^{i}(D_{\sigma}, D_{\sigma})\geq \dim_E \fz_{P_i, \sigma}+1 = 3$, which implies $\dim_E\Ker (f_{i,\sigma})\geq 1$. In particular $f_{i,\sigma}$ is not an isomorphism.
\end{proof}

\begin{rem}\label{rem:Rsmooth}
\hspace{2em}
\begin{enumerate}[label=(\roman*)]
\item
If $s_{i,\sigma}$ does not appear in $w_{\fR, \sigma} w_{0,\sigma}$, it follows Proposition \ref{prop:Psmooth} and (\ref{mult:explicit}) that there is a unique element $h_{\log, i}\in \Hom_{\Fil,\fR}^{i}(D_{\sigma}, D_{\sigma})$ such that $h_{\log, i}(e_{j,\sigma})=0$ for $0\leq j \leq i-1$ and $h_{\log, i}(e_{j,\sigma})-e_{j,\sigma}\in \bigoplus_{k=0}^{i-1} E e_{k,\sigma}$ for $i\leq j \leq n-1$. 
\item
If $s_{i,\sigma}$ appears with multiplicity $1$ in some reduced expression of $w_{\fR, \sigma} w_{0,\sigma}$, using (\ref{mult:nPiw}) we have $\dim_E (\fn_{P_i, \sigma} \cap \Ad_{w_{\fR,\sigma}} (\fn_{\sigma}))=1$, and it follows that there is, up to multiplication by an element of $E^\times$, a unique non-zero element $h_i \in \Hom_{\Fil,\fR}^{i}(D_{\sigma}, D_{\sigma})$ such that $h_{i}(e_{j,\sigma})=0$ for all $0\leq j \leq i-1$ and $h_{i}(e_{j,\sigma})\in \bigoplus_{k=0}^{i-1} E e_{k,\sigma}$ for all $i\leq j \leq n-1$ (such an element generates $\Ker(f_{i,\sigma})$).
\end{enumerate}
\end{rem}

Till the end of this section we fix $\sigma\in \Sigma$ and write $e_j$ for $e_{j,\sigma}$ (as in \S~\ref{sec:def}). The following lemma will be useful.

\begin{lem}\label{lem:Lfilmax}
Let $i\in \{1,\dots,n-1\}$, $I\subset \{\varphi_j,\ 0\leq j \leq n-1\}$ of cardinality $i$ and $\fR'$ a refinement compatible with $I$ (Definition \ref{def:compatible}). The following statements are equivalent:
\begin{enumerate}[label=(\roman*)]
\item
$s_{i,\sigma}$ does not appear in some (equivalently any) reduced expression of $w_{\fR',\sigma}w_{0,\sigma}$;
\item
$\Fil^{-h_{i,\sigma}}(D_{\sigma}) \cap (\bigoplus_{\varphi_j\in I}E e_{j})=0$;
\item
the coefficient of $e_{I^c}$ in $\Fil_i^{\max} D_{\sigma}=\bigwedge\nolimits_E^{\!n-i}\Fil^{-h_{i,\sigma}}(D_{\sigma})$ is non-zero (see (\ref{eq:eI}) for $e_{I^c}$).
\end{enumerate}
\end{lem}
\begin{proof}
The last two statements are equivalent by Lemma \ref{lem:triveq}, hence it is enough to prove that (i) is equivalent to (ii). Since $\fR'$ is compatible with $I$, renumbering the $\varphi_j$ and the $e_j$ we can assume $I=\{\varphi_0,\dots,\varphi_{i-1}\}$ and $\fR'=(\varphi_0, \dots, \varphi_{n-1})$. Multiplying $g_{\sigma}$ by an element of $B(E)$ on the right, we can assume $g_{\sigma}=b_{\sigma} w_{\fR',\sigma}$ for some $b_\sigma\in B(E)$. Define a new basis of $D_\sigma$ by $(e_0', \dots, e_{n-1}'):=(e_0, \dots, e_{n-1}) b_{\sigma}$, hence $e_j'-a_je_j\in \bigoplus_{j'< j} E e_{j'}$ for $j\in \{0,\dots,n-1\}$ and some $a_j\in E^{\times}$. By definition of $g_\sigma$
\[(f_{n-1}, \dots, f_0):=(e_0,\dots, e_{n-1}) g_{\sigma}=(e_{0}', \dots, e_{n-1}')w_{\fR', \sigma}\]
is such that $\Fil^{-h_{j,\sigma}}(D_{\sigma})=Ef_j\oplus Ef_{j+1}\oplus \cdots \oplus Ef_{n-1}$ for $j\in \{0,\dots,n-1\}$. Equivalently
\begin{equation}\label{eq:fj}
(f_0, \dots, f_{n-1})=(e_{0}', \dots, e_{n-1}')w_{\fR', \sigma}w_{0,\sigma}.
\end{equation}
Assume that $s_{i,\sigma}$ does not appear in $w_{\fR', \sigma}w_{0,\sigma} $, then from (\ref{eq:fj}) one has $\Fil^{-h_i,\sigma}(D_{\sigma})=Ee'_i \oplus Ee'_{i+1}\oplus \cdots \oplus Ee'_{n-1}$. By the form of $e_j'$ above, it is straightforward that (ii) holds. Assume that $s_{i,\sigma}$ appears in $w_{\fR', \sigma}w_{0,\sigma}$. By (\ref{eq:fj}) there exist $j\geq i$ and $k \leq i-1$ such that $f_j=e'_k \in \bigoplus_{j'\leq i-1} E e_{j'}$, which gives a non-zero element in $\Fil^{-h_{i,\sigma}}(D_{\sigma}) \cap (E e_{0} \oplus \cdots \oplus E e_{i-1})$.
\end{proof}

We now assume that $s_{i,\sigma}$ does not appear in some (equivalently any) reduced expression of $w_{\fR,\sigma}w_{0,\sigma}$. By Proposition \ref{prop:Psmooth}, $f_{i, \sigma}$ is an isomorphism, hence we can consider the isomorphism 
\begin{equation}\label{mult:EtD'}
\Hom_{\sm}(T(K), E) \bigoplus \Hom_{\sigma}(L_{P_i}(\cO_K),E) \!\!\!\! \buildrel \substack{(\ref{eq:extphi}) \bigoplus f_{i, \sigma}^{-1} \\ \sim}\over\longrightarrow \!\!\!\! \Ext^1_{\varphi^f}(D_{\sigma},D_{\sigma}) \bigoplus \Hom_{\Fil,\fR}^{i}(D_{\sigma}, D_{\sigma}).
\end{equation}
Recall that the length $2$ locally $\sigma$-analytic representation $\pi_I(D_\sigma)$ of $\GL_n(K)$ over $E$ defined above (\ref{eq:amalgi}) is non-split by (iii) of Lemma \ref{lem:Lfilmax} (see (\ref{eq:VipiI}) and the definition of $V_I$ below (\ref{mult:formal})). As for (\ref{eq:smoothsplit}), the choice of $\log(p)$ gives an isomorphism
\begin{multline}\label{mult:Pisplit}
\Hom_{\sm}(T(K), E) \bigoplus \Hom_{\sigma}(L_{P_i}(\cO_K),E) \\
\xlongrightarrow{\sim} \Hom_{\sm}(T(K),E) \!\!\!\!\!\!\bigoplus_{\Hom_{\sm}(L_{P_i}(K),E)} \!\!\!\!\!\!\Hom_{\sigma}(L_{P_i}(K),E).
\end{multline}
Hence by (\ref{eq:amalg2}) (applied with the fixed refinement $\fR$) we deduce an isomorphism which depends on a choice of $\log(p)$
\begin{equation}\label{eq:depends}
\Hom_{\sm}(T(K), E) \bigoplus \Hom_{\sigma}(L_{P_i}(\cO_K),E) \xlongrightarrow{\sim} \Ext^1_{\GL_n(K),\sigma}(\pi_{\alg}(D_{\sigma}), \pi_I(D_{\sigma})).
\end{equation}

\begin{prop}\label{prop:Pcompa}
Assume that $s_{i,\sigma}$ does not appear in some (equivalently any) reduced expression of $w_{\fR,\sigma}w_{0,\sigma}$. Fixing the same choice of $\log(p)$ in (\ref{eq:depends}) and in the definition of the map $t_{D_{\sigma}}$ of Proposition \ref{prop:map}, the isomorphism (\ref{mult:EtD'}) coincides with the composition (via $\Hom_{\Fil,\fR}^{i}(D_{\sigma}, D_{\sigma})\subset \Hom_{\Fil}(D_{\sigma}, D_{\sigma})$ and for arbitrary isomorphisms $(\varepsilon_J)_J$ as in (\ref{eq:epsilonI}))
\begin{multline}\label{mult:compare}
\Hom_{\sm}(T(K), E) \bigoplus \Hom_{\sigma}(L_{P_i}(\cO_K),E) \buildrel \substack{(\ref{eq:depends})\\ \sim} \over \longrightarrow \Ext^1_{\GL_n(K),\sigma}(\pi_{\alg}(D_{\sigma}), \pi_I(D_{\sigma})) \\
\xlongrightarrow{t_{D_{\sigma}}} \Ext^1_{\varphi^f}(D_{\sigma},D_{\sigma}) \bigoplus \Hom_{\Fil}(D_{\sigma}, D_{\sigma}).
\end{multline}
(In particular the composition (\ref{mult:compare}) does not depend on the choice of $\log(p)$.)
\end{prop}
\begin{proof}
We first check that both compositions coincide when restricted to the subspace $\Hom_{\sm}(T(K),E)\bigoplus \Hom_{\sigma}(\GL_n(\cO_K),E)$. \ By \ its \ very \ definition, \ the \ composition \ (\ref{mult:EtD'}) \ sends \ $\Hom_{\sm}(T(K),E)$ \ (resp.~$\Hom_{\sigma}(L_{P_i}(\cO_K),E)$) \ to \ the \ subspace \ $\Ext^1_{\varphi^f}(D_{\sigma}, D_{\sigma})$ (resp.~to $\Hom_{\Fil, \fR}^i(D_{\sigma}, D_{\sigma})$). The analogous statement holds for the composition (\ref{mult:compare}) by (\ref{eq:extphi}) and (\ref{eq:scalar}). By Step $2$ in the proof of Proposition \ref{prop:map} the restriction to $\Hom_{\sm}(T(K),E)$ of (\ref{mult:EtD'}) and (\ref{mult:compare}) coincide. An examination of the map $f_{i,\sigma}$ and of (\ref{eq:EhomfilR}) show that $f_{i,\sigma}(\id)=\sigma\circ \log\circ \det\in \Hom_{\sigma}(\GL_n(\cO_K),E)$, in particular the restriction of (\ref{mult:EtD'}) to $\Hom_{\sigma}(\GL_n(\cO_K),E)$ sends $\sigma\circ \log \circ \det$ to $\id \in \Hom_{\Fil}(D_{\sigma},D_{\sigma})$, which coincides with that of (\ref{mult:compare}) by (\ref{eq:scalar}).\bigskip

It remains to show that the images of
\[\log\in \Hom_{\sigma}(\GL_{n-i}(\cO_K),E)\hookrightarrow \Hom_{\sigma}(L_{P_i}(\cO_K),E)\]
in $\Hom_{\Fil,\fR}^{i}(D_{\sigma},D_{\sigma})$ are the same under the two compositions (where we use the notation $\log$ for $\sigma\circ \log \circ \det$ as below (\ref{eq:log})). By (i) of Remark \ref{rem:Rsmooth} and (\ref{eq:EhomfilR}), (\ref{mult:EtD'}) sends $\log\in \Hom_{\sigma}(\GL_{n-i}(\cO_K),E)$ to the unique element $h_{\log, i}\in \Hom_{\Fil}(D_{\sigma}, D_{\sigma})$ such that $h_{\log, i}(e_{j})=0$ for all $0\leq j \leq i-1$ and $h_{\log, i}(e_{j})-e_{j}\in \bigoplus_{k=0}^{i-1} Ee_{k}$ for $i\leq j \leq n-1$ (using that, by (\ref{mult:explicit}), such an element is automatically in $\Hom_{\Fil,\fR}^{i}(D_{\sigma}, D_{\sigma})$). It suffices to show that $t_{D_{\sigma}}(\log)$ satisfies the same properties, where here we also use $\log$ to denote the image of $\log\in \Hom_{\sigma}(\GL_{n-i}(\cO_K),E)$ in $\Ext^1_{\GL_n(K),\sigma}(\pi_{\alg}(D_{\sigma}), \pi_I(D_{\sigma})/\pi_{\alg}(D_{\sigma}))$ by (\ref{eq:isoindi}). But this follows from (iii) of Lemma \ref{lem:nuI}.
\end{proof}

Until the rest of this section we assume that \emph{all refinements on $D_\sigma$ are non-critical for all $\sigma\in \Sigma$}, which is the running assumption of \cite{Di25}. In \emph{loc.~cit.}~the locally $\Qp$-algebraic representation $\pi_{\alg}(D)$ in (\ref{eq:alg}) is denoted $\pi_{\alg}(\phi ,\textbf{h})\otimes_E\varepsilon^{1-n}$. The locally $\Qp$-analytic representation $\bigoplus_{\sigma, \ \!\pi_{\alg}(D)} (\pi_R(D_{\sigma}) \otimes_E \otimes_{\tau\neq \sigma} L(\lambda_{\tau}))$ is isomorphic to the representation denoted $\pi_1(\phi, \textbf{h})\otimes_E\varepsilon^{1-n}$ in \cite[\S~3.1.2]{Di25}. We fix an isomorphism $\bigoplus_{\sigma, \ \!\pi_{\alg}(D)} (\pi_R(D_{\sigma}) \otimes_E \otimes_{\tau\neq \sigma} L(\lambda_{\tau}))\buildrel\sim\over\rightarrow \pi_1(\phi, \textbf{h}))\otimes_E\varepsilon^{1-n}$, which is defined up to multiplication by a scalar in $E^\times$ since, by the non-criticality hypothesis, these representations are easily checked to have scalar endomorphisms. We then deduce an isomorphism (see \cite[\S~3.1.4]{Di25} for the $\Ext^1_{\sigma}$ on the right hand side)
\begin{equation}\label{eq:Esigma}
\Ext^1_{\GL_n(K),\sigma}(\pi_{\alg}(D_{\sigma}), \pi_R(D_{\sigma}))\xlongrightarrow{\sim} \Ext^1_{\sigma}(\pi_{\alg}(\phi, \textbf{h}), \pi_1(\phi,\textbf{h})). 
\end{equation}
Recall also from \cite[Cor.~3.29(1)]{Di25} and \cite[(3.39)]{Di25} that there is a canonical $E$-linear surjection $t_{\cM(D),\sigma}:\Ext^1_{\sigma}(\pi_{\alg}(\phi, \textbf{h}), \pi_1(\phi,\textbf{h})) \twoheadlongrightarrow \ol{\Ext}^1_{\sigma}(\cM(D), \cM(D))$ (see Corollary \ref{cor:splitsigma} for the right hand side).

\begin{prop}\label{prop:noncritical}
Assume that all refinements on $D_\sigma$ are non-critical for all $\sigma\in \Sigma$ and let $\sigma\in \Sigma$. Then for any choice of isomorphism $\varepsilon_I$ in (\ref{eq:epsilonI}) the map $t_{D_{\sigma}}\vert_{\Ext^1_{\GL_n(K),\sigma}(\pi_{\alg}(D_{\sigma}), \pi_I(D_{\sigma}))}$ coincides, up to multiplication by a scalar in $E^\times$, with the composition
\begin{multline*}
\Ext^1_{\GL_n(K),\sigma}(\pi_{\alg}(D_{\sigma}), \pi_I(D_{\sigma}))\buildrel (\ref{eq:Esigma})\over \hooklongrightarrow \Ext^1_{\sigma}(\pi_{\alg}(\phi, \textbf{h}), \pi_1(\phi,\textbf{h}))\\ 
\buildrel {t_{\cM(D),\sigma}} \over \twoheadlongrightarrow \ol{\Ext}^1_{\sigma}(\cM(D), \cM(D))\buildrel \substack{(\ref{eq:Esplit})\\\sim} \over\longrightarrow \Ext^1_{\varphi^f}(D_{\sigma},D_{\sigma}) \bigoplus \Hom_{\Fil}(D_{\sigma}, D_{\sigma})
\end{multline*}
for the same choice of $\log(p)$ in (\ref{eq:Esplit}) and in $t_{D_{\sigma}}$.
\end{prop}
\begin{proof}
Since $D_\sigma$ is non-critical we have $w_{\fR,\sigma}w_{0,\sigma}=1$. In particular, using the notation in the proof of Proposition \ref{prop:Psmooth}, we have $\fb_{\sigma}\cap \Ad_{g_{\sigma}}(\fb_{\sigma})=\Ad_{b_{\sigma}}(\ft_{\sigma})$. We then easily deduce from (\ref{eq:middle}) that the map (\ref{eq:EhomfilR}) induces an isomorphism $\Hom_{\Fil,\fR}(D_{\sigma}, D_{\sigma})\buildrel \sim \over \rightarrow \Hom_{\sigma}(T(\cO_K), E)$. Using the commutative diagram (\ref{eq:diagR}) we deduce that the composition
\begin{multline}\label{eq:Etri5}
\Ext^1_{\fR}(\cM(D), \cM(D)) \cap \Ext^1_{\sigma}(\cM(D), \cM(D))\buildrel (\ref{eq:Enil}) \over \longrightarrow \Hom_{\Fil,\fR}(D_{\sigma}, D_{\sigma})\\
\buildrel \substack{(\ref{eq:EhomfilR})\\\sim} \over \longrightarrow \Hom_{\sigma}(T(\cO_K), E)
\end{multline}
coincides with the composition
\begin{multline}\label{eq:Etri3}
\Ext^1_{\fR}(\cM(D), \cM(D)) \cap \Ext^1_{\sigma}(\cM(D), \cM(D))\buildrel (\ref{eq:Einvt})\over \longrightarrow \Ext^1_{\fR}(\cM(D)[1/t], \cM(D)[1/t]) \\
\ \ \ \ \ \ \ \ \ \ \ \buildrel (\ref{eq:Etri2})\over \longrightarrow \Hom(T(K),E) \buildrel\res\over \twoheadlongrightarrow \Hom(T(\cO_K), E) \cong \bigoplus_{\tau \in \Sigma} \Hom_{\tau}(T(\cO_K),E)\\
\twoheadrightarrow \Hom_{\sigma}(T(\cO_K),E).
\end{multline} 
It follows that the kernel of (\ref{eq:Etri3}) is $\Ext^1_g(\cM(D), \cM(D))$ which is the kernel of (\ref{eq:Enil}) (Lemma \ref{lem:LpdR}). Moreover the first map in (\ref{eq:Etri5}) is surjective using the surjectivity in the first exact sequence of the commutative diagram below (\ref{eq:Rt}). Also (\ref{eq:Etri2}) induces an isomorphism\ $\ol{\Ext}^1_g(\cM(D), \cM(D)) \xrightarrow{\sim} \Hom_{\sm}(T(K),E)$ (see (\ref{eq:isog}) with the discussion after (\ref{mult:i=j})). We deduce from all this that (\ref{eq:Etri2}) induces a canonical isomorphism
\begin{equation}\label{eq:isoextra}
\ol{\Ext}^1_{\fR}(\cM(D), \cM(D))\cap \ol\Ext^1_{\sigma}(\cM(D), \cM(D))\buildrel\sim\over\longrightarrow \Hom_\sigma(T(K),E)
\end{equation}
(the intersection being inside $\ol{\Ext}^1_{\varphi,\Gamma}(\cM(D), \cM(D))$) which fits into a commutative diagram (writing $\ol{\Ext}^1_{\fR}\cap \ol\Ext^1_{\sigma}$ for $\ol{\Ext}^1_{\fR}(\cM(D), \cM(D))\cap \ol\Ext^1_{\sigma}(\cM(D), \cM(D))$)
\[\xymatrix{
0 \ar[r] & \Ext^1_{\varphi^f}(D_{\tau},D_{\tau})\ar[r]\ar^{\wr}[d] & \ol{\Ext}^1_{\fR}\cap \ol\Ext^1_{\sigma} \ar^{\wr\ (\ref{eq:isoextra})}[d]\ar[r] & \Hom_{\Fil,\fR}(D_{\sigma}, D_{\sigma}) \ar^{\wr\ (\ref{eq:EhomfilR})}[d] \ar[r] & 0 \\
0 \ar[r] & \Hom_{\sm}(T(K),E)\ar[r] & \Hom_\sigma(T(K),E) \ar[r] & \Hom_{\sigma}(T(\cO_K),E)\ar[r] & 0.}\]
Moreover the splitting of the top exact sequence associated to $\log(p)$ induced by Corollary \ref{cor:splitsigma} corresponds to the splitting $\Hom_\sigma(T(K),E)\cong \Hom_{\sm}(T(K),E)\bigoplus \Hom_{\sigma}(T(\cO_K),E)$ (associated to $\log(p)$).\bigskip

Using the notation of \cite[\S~2.3.1]{Di25}, let $\Ext^1_w(\cM(D), \cM(D))$ be the extension group of (genuine) trianguline deformations of $\cM(D)$ with respect to the refinement $\fR$ (in \emph{loc.~cit.}~$w$ is a \ permutation \ related \ to \ $\fR$). \ As \ $\Ext^1_w(\cM(D), \cM(D))$ \ is \ obviously \ sent \ to \ $\Ext^1_{\fR}(\cM(D)[1/t], \cM(D)[1/t])$ via (\ref{eq:Einvt}), we have the inclusion
\[\Ext^1_w(\cM(D),\cM(D)) \!\subseteq \Ext^1_{\fR}(\cM(D),\cM(D)).\]
Moreover the map $\Ext^1_w(\cM(D),\cM(D)) \rightarrow \Hom (T(K),E)$ defined in \cite[(2.12)]{Di25} is surjective by \cite[Prop.~2.10(2)]{Di25} and coincides with the composition
\begin{multline*}
\Ext^1_w(\cM(D),\cM(D)) \hookrightarrow \Ext^1_{\fR}(\cM(D),\cM(D))\rightarrow \Ext^1_{\fR}(\cM(D)[1/t], \cM(D)[1/t])\\\buildrel(\ref{eq:Etri2})\over\longrightarrow \Hom(T(K),E).
\end{multline*}
A proof analogous to the proof of (\ref{eq:isoextra}) shows that the kernel of the composition
\[\Ext^1_{\fR}(\cM(D),\cM(D))\rightarrow \Ext^1_{\fR}(\cM(D)[1/t], \cM(D)[1/t]) \buildrel(\ref{eq:Etri2})\over\longrightarrow \Hom(T(K),E)\]
is \ $\Ext^1_0(\cM(D), \cM(D))\subset \Ext^1_g(\cM(D), \cM(D))$. \ Since \ we \ have \ $\Ext^1_g(\cM(D), \cM(D))\subset \Ext^1_w(\cM(D), \cM(D))$ \ by \ \cite[Prop.~2.10(3)]{Di25}, \ we \ deduce \ $\Ext^1_w(\cM(D), \cM(D)) = \Ext^1_{\fR}(\cM(D),\cM(D))$.\bigskip
	 	
Set (for all $i\in \{1,\dots,n-1\}$)
\begin{equation}\label{eqtisigma}
\Hom_{\sigma,i}(T(K),E):=\Hom_{\sm}(T(K),E) \!\!\bigoplus_{\Hom_{\sm}(L_{P_i}(K),E)} \!\!\Hom_{\sigma}(L_{P_i}(K),E)
\end{equation}
and consider now the composition
\begin{multline}\label{mult:Ecomp3}
\Ext^1_{\GL_n(K),\sigma}(\pi_{\alg}(D_{\sigma}), \pi_I(D_{\sigma})) \buildrel \substack{(\ref{eq:amalg2})^{-1}\\\sim}\over\longrightarrow \Hom_{\sigma,i}(T(K),E)\hooklongrightarrow \Hom_\sigma(T(K),E) \\
\buildrel \substack{(\ref{eq:isoextra})^{-1}\\\sim} \over\longrightarrow \ol{\Ext}^1_{\fR}(\cM(D), \cM(D))\cap \ol\Ext^1_{\sigma}(\cM(D),\cM(D))\\
\hookrightarrow \Ext^1_{\varphi^f}(D_{\sigma}, D_{\sigma})\bigoplus \Hom_{\Fil}(D_{\sigma}, D_{\sigma})
\end{multline}
where (\ref{eq:amalg2}) is applied with the refinement $\fR$ and the last injection is induced by Corollary \ref{cor:splitsigma} (and depends on a choice of $\log(p)$). By the discussion around (\ref{eq:isoextra}), the composition $\Hom_{\sigma,i}(T(K),E)\rightarrow \Ext^1_{\varphi^f}(D_{\sigma}, D_{\sigma})\bigoplus \Hom_{\Fil}(D_{\sigma}, D_{\sigma})$ in (\ref{mult:Ecomp3}) precomposed with (\ref{mult:Pisplit}) coincides \ with \ the \ composition \ (\ref{mult:compare}). \ Hence \ by \ Proposition \ \ref{prop:Pcompa} \ the \ map $t_{D_{\sigma}}\vert_{\Ext^1_{\GL_n(K),\sigma}(\pi_{\alg}(D_{\sigma}), \pi_I(D_{\sigma}))}$ coincides with (\ref{mult:Ecomp3}) (for the same choice of $\log(p)$). But it follows from \cite[Cor.~3.29(1)]{Di25} (and its proof) that the map $t_{\cM(D),\sigma}\vert_{\Ext^1_{\GL_n(K),\sigma}(\pi_{\alg}(D_{\sigma}), \pi_I(D_{\sigma}))}$ of \emph{loc.~cit.}~lands in
\[\ol\Ext^1_w(\cM(D), \cM(D))\cap \ol\Ext^1_{\sigma}(\cM(D),\cM(D))\buildrel\sim\over\rightarrow \ol\Ext^1_{\fR}(\cM(D), \cM(D))\cap \ol\Ext^1_{\sigma}(\cM(D),\cM(D))\]
and coincides with the composition $\Ext^1_{\GL_n(K),\sigma}(\pi_{\alg}(D_{\sigma}), \pi_I(D_{\sigma}))\rightarrow \ol\Ext^1_{\fR}(\cM(D), \cM(D))\cap \ol\Ext^1_{\sigma}(\cM(D),\cM(D))$ in (\ref{mult:Ecomp3}) up to multiplication by a scalar in $E^\times$. In particular its composition with the map (\ref{eq:Esplit}) coincides with (\ref{mult:Ecomp3}) (up to a scalar).
\end{proof}

We finally obtain the main result of that section.

\begin{cor}\label{cor:Di25}
Assume that all refinements on $D_\sigma$ are non-critical for all $\sigma\in \Sigma$ and let $\sigma\in \Sigma$. The canonical composition of Theorem \ref{thm:independant} coincides with the composition
\begin{equation*}
\Ext^1_{\GL_n(K),\sigma}(\pi_{\alg}(D_{\sigma}), \pi_R(D_{\sigma}))\buildrel \substack{(\ref{eq:Esigma})\\\sim}\over \longrightarrow \Ext^1_{\sigma}(\pi_{\alg}(\phi, \textbf{h}), \pi_1(\phi,\textbf{h}))
\buildrel{t_{\cM(D),\sigma}} \over \twoheadlongrightarrow \ol{\Ext}^1_{\sigma}(\cM(D), \cM(D))
\end{equation*}
up to multiplication by a scalar in $E^\times$.
\end{cor}
\begin{proof}
As the $E$-vector space $\Ext^1_{\GL_n(K),\sigma}(\pi_{\alg}(D_{\sigma}), \pi_R(D_{\sigma}))$ is spanned by the subspaces $\Ext^1_{\GL_n(K),\sigma}(\pi_{\alg}(D_{\sigma}), \pi_I(D_{\sigma}))$, the result follows from Proposition \ref{prop:noncritical}, noting that, since $\Ext^1_{\GL_n(K),\sigma}(\pi_{\alg}(D_{\sigma}), \pi_{\alg}(D_{\sigma}))$ canonically embeds into $\Ext^1_{\GL_n(K),\sigma}(\pi_{\alg}(D_{\sigma}), \pi_I(D_{\sigma}))$ for all $I$, the scalar in Proposition \ref{prop:noncritical} won't depend on $I$.
\end{proof}

\subsection{The direct summands \texorpdfstring{$\pi(D_\sigma)^\flat$}{piDsigmabemol} and \texorpdfstring{$\pi(D)^\flat$}{piDbemol}}

We define a canonical direct summand $\pi(D_\sigma)^\flat$ of $\pi(D_\sigma)$ (as a representation of $\GL_n(K)$) which still determines the isomorphism class of the filtered $\varphi^f$-module $D_\sigma$ and which coincides with $\pi(D_\sigma)$ when $D_\sigma$ is not too critical. We use it to define a direct summand $\pi(D)^\flat$ of~$\pi(D)$.\bigskip

We keep the notation of the previous sections and fix an embedding $\sigma\in \Sigma$. Recall that, in Step $3$ of the proof of Proposition \ref{prop:map}, for a subset $I\subseteq \{\varphi_0,\dots,\varphi_{n-1}\}$ of cardinality in $\{1,\dots,n-1\}$ we defined a canonical map (still denoted)
\[t_{D_\sigma}: \Ext^1_{\GL_n(K),\sigma}(\pi_{\alg}(D_\sigma),\pi_I(D_\sigma)/\pi_{\alg}(D_\sigma))\longrightarrow \Hom_{\Fil}(D_\sigma,D_\sigma).\]
Recall also that if $\fR$ is any refinement we defined the permutation $w_{\fR,\sigma}\in S_n$ just above Proposition \ref{prop:Psmooth}.

\begin{prop}\label{prop:mul2}
Let $I\subseteq \{\varphi_0,\dots,\varphi_{n-1}\}$ of cardinality $i\in \{1,\dots,n-1\}$ and $\fR$ a refinement compatible with $I$. The simple reflection $s_{i,\sigma}$ appears with multiplicity $\geq 2$ in all reduced expressions of $w_{\fR,\sigma}w_{0,\sigma}$ if and only if we have
\begin{equation}\label{eq:tD0}
t_{D_\sigma}\big(\Ext^1_{\GL_n(K),\sigma}(\pi_{\alg}(D_{\sigma}), \pi_I(D_{\sigma})/\pi_{\alg}(D_\sigma))\big)=0.
\end{equation}
\end{prop}
\begin{proof}
Note that we require $s_{i,\sigma}$ to appear with multiplicity $\geq 2$ in \emph{any} reduced expression of $w_{\fR,\sigma}w_{0,\sigma}$, which is stronger than to appear with multiplicity $\geq 2$ in \emph{some} reduced expression (think about $s_{i,\sigma}s_{i+1,\sigma}s_{i,\sigma}=s_{i+1,\sigma}s_{i,\sigma}s_{i+1,\sigma}$). To simplify notation we write $w:=w_{\fR,\sigma}w_{0,\sigma}$ in this proof. Recall from the proof of Proposition \ref{prop:Psmooth} and from (ii) of Remark \ref{rem:Rsmooth}) that $s_{i,\sigma}$ appears with multiplicity $\leq 1$ in \emph{some} reduced expression of $w$ if and only if $\vert\{(j,k)\!\in\!\{0,\dots, i-1\}\times \{i, \dots, n-1\}, w(k)<w(j)\}\vert\leq 1$. Hence $s_{i,\sigma}$ appears with multiplicity $\geq 2$ in \emph{any} reduced expression of $w$ if and only if $\vert\{(j,k)\!\in\!\{0,\dots, i-1\}\times \{i, \dots, n-1\}, w(k)<w(j)\}\vert\geq 2$, or equivalently
\begin{equation}\label{eq:mult2}
\vert w(\{i,\dots,n-1\})\cap \{i,\dots,n-1\}\vert\leq n-i-2.
\end{equation}
Hence we need to prove (\ref{eq:mult2})$\Longleftrightarrow$(\ref{eq:tD0}).\bigskip

By (ii) of Remark \ref{rem:forlater} we need to prove that (\ref{eq:mult2}) is equivalent to the following fact: for any $\varphi_j\notin I$ the coefficient of $e_{I^c \setminus \{\varphi_j\}}$ is $0$ in any vector of $\wedge_E^{\!n-i-1}\Fil^{-h_{i,\sigma}}(D_\sigma)$, where we fix a basis $e_{0}, \dots, e_{n-1}$ of $\varphi^f$-eigenvectors of $D_{\sigma}$ such that $\varphi^f(e_{j})=\varphi_j e_{j}$ as in \S~\ref{sec:def}. Changing this numbering if necessary, we assume $\fR=(\varphi_0,\dots,\varphi_{n-1})$ and $I=\{\varphi_0,\dots,\varphi_{i-1}\}$. Thus we need to prove
\begin{multline}\label{mult:equivbis}
(\ref{eq:mult2})\Longleftrightarrow\forall\ j\in \{i,\dots,n-1\}\text{ the coefficient of }e_{I^c \setminus \{\varphi_j\}}\text{ is }0\\
\text{ in any vector of }\wedge_E^{\!n-i-1}\Fil^{-h_{i,\sigma}}\!(D_\sigma).
\end{multline}

It follows from (\ref{eq:fj}) that there exists a basis $f_0,\dots,f_{n-1}$ of $D_\sigma$ such that, for $j\in \{0,\dots,n-1\}$, $f_{j}\in \Fil^{-h_{j,\sigma}}(D_{\sigma})$ and
\begin{equation}\label{eq:renum}
f_{w^{-1}(j)} - e_j \in \bigoplus_{j'<j}Ee_{j'}.
\end{equation}
We can simplify (\ref{eq:renum}) when $\fR$ is critical (i.e.~$w_{\fR,\sigma}\ne w_{0,\sigma}$). For $j=0$ \emph{loc.~cit.}~gives $e_0\in \Fil^{-h_{w^{-1}(0),\sigma}}(D_{\sigma})$. If $w^{-1}(0)>w^{-1}(1)$, then \emph{a fortiori} $e_0\in \Fil^{-h_{w^{-1}(1),\sigma}}(D_{\sigma})$ and we can forget $e_0$ on the right hand side of (\ref{eq:renum}) for $f_{w^{-1}(1)}$. Let us look at (\ref{eq:renum}) for $j=2$. If $w^{-1}(1)>w^{-1}(2)$, using (\ref{eq:renum}) for $j=1$ we see that we can forget $e_1$ on the right hand side of (\ref{eq:renum}) for $f_{w^{-1}(2)}$ (possibly modifying the coefficient of $e_0$). If $w^{-1}(0)>w^{-1}(2)$ we can forget $e_0$ (as previously in $f_{w^{-1}(1)}$), and so on. Hence we see that on the right hand side of (\ref{eq:renum}) we can furthermore assume $w^{-1}(j')<w^{-1}(j)$, or equivalently for $j\in \{0,\dots,n-1\}$
\begin{equation}\label{eq:fjbis}
f_j - e_{w(j)} \in \bigoplus_{\substack{j'<j\\w(j')<w(j)}}Ee_{w(j')}.
\end{equation}

The $E$-vector space $\wedge_E^{\!n-i-1}\Fil^{-h_{i,\sigma}}(D_\sigma)$ is generated by the following $n-i$ vectors
\begin{equation}\label{eq:wedgen-i-1}
\begin{gathered}
\left\{\begin{array}{l}
f_{i}\wedge f_{i+1}\wedge \cdots \wedge f_{n-3} \wedge f_{n-2},\\
f_{i}\wedge \cdots \wedge f_{k-1} \wedge f_{k+1} \wedge \cdots \wedge f_{n-1},\ k\in \{i+1,\dots,n-2\}\\
f_{i+1}\wedge f_{i+2}\wedge \cdots \wedge f_{n-2} \wedge f_{n-1}.
\end{array}\right.
\end{gathered}
\end{equation}
Assume $\vert w(\{i,\dots,n-1\})\cap \{i,\dots,n-1\}\vert\geq n-i-1$ and let $j_1,\dots,j_{n-i-1}\in \{i,\dots,n-1\}$ such that $w(j_k)\in \{i,\dots,n-1\}$ for all $k$. Then $\{w(j_1),\dots,w(j_{n-1-i})\}=\{i,\dots,n-1\}\setminus\{j\}$ for some $j$, and by (\ref{eq:fjbis}) $e_{I^c \setminus \{\varphi_j\}}$ has a non-zero coefficient in the vector $\wedge_{k=1}^{n-i-1}f_{j_k}$ (which is in the list (\ref{eq:wedgen-i-1})). Assume $\vert w(\{i,\dots,n-1\})\cap \{i,\dots,n-1\}\vert\leq n-i-2$ and let $j_1,j_2\in \{i,\dots,n-1\}$ such that $j_1\ne j_2$, $w(j_1)<i$ and $w(j_2)<i$. Then by (\ref{eq:fjbis}) the vector $f_{j_1}$ only ``contains'' vectors $e_{j'}$ with $j'\leq w(j_1)<i$ (hence $\varphi_{j'}\in I$), and similarly with the vector $f_{j_2}$. If follows that if a vector $\wedge_kf_{k}$ in (\ref{eq:wedgen-i-1}) is such that $f_k=f_{j_1}$ or $f_k=f_{j_2}$ for some $k$, then all $e_{I^c \setminus \{\varphi_j\}}$ for $j\in \{i,\dots,n-1\}$ have coefficient $0$ in $\wedge_kf_{j_k}$. But clearly any vector in (\ref{eq:wedgen-i-1}) is like this. This proves (\ref{mult:equivbis}).
\end{proof}

\begin{definit}\label{def:critical}
Let $I\subset \{\varphi_j,\ 0\leq j \leq n-1\}$ of cardinality $i\in \{1,\dots,n-1\}$. We say that $I$ is \emph{very critical} for $\sigma$ if, for one (equivalently any by Proposition \ref{prop:mul2}) refinement $\fR$ compatible with $I$, $s_{i,\sigma}$ appears with multiplicity $\geq 2$ in \emph{all} reduced expressions of $w_{\fR,\sigma}w_{0,\sigma}$.
\end{definit}

Recall that when $s_{i,\sigma}$ appears with multiplicity $\geq 1$ in some (equivalently any) reduced expressions of $w_{\fR,\sigma}w_{0,\sigma}$ we say that $I$ is critical for $\sigma$ (this does not depend on the refinement compatible with $I$ for $\sigma$). When $\sigma$ is fixed (as in this section) we just say that $I$ is very critical, resp.~$I$ is critical. We define
\[\pi(D_\sigma)^\flat\subseteq \pi(D_\sigma)\]
as the maximal subrepresentation of $\pi(D_\sigma)$ which does not contain any $C(I,s_{i,\sigma})$ with $I$ very critical in its Jordan-H\"older constituents.

\begin{prop}\label{prop:flat}
\hspace{2em}
\begin{enumerate}[label=(\roman*)]
\item
We have an isomorphism (with non-split extensions on the right)
\begin{multline*}
\pi(D_\sigma)\cong \pi(D_\sigma)^\flat \ \ \bigoplus \ \ \bigoplus_{I\text{ very critical}}\Big(\big(C(I, s_{\vert I\vert,\sigma})\otimes_E\Fil_{\vert I\vert}^{\max}D_\sigma\big)\!\begin{xy} (30,0)*+{}="a"; (40,0)*+{}="b"; {\ar@{-}"a";"b"}\end{xy}\!\pi_{\alg}(D_\sigma)\Big).
\end{multline*}
In particular $\pi(D_\sigma)^\flat$ is a direct summand of $\pi(D_\sigma)$.
\item
The isomorphism class of the locally $\sigma$-analytic representation $\pi(D_\sigma)^\flat$ determines the one of $\pi(D_\sigma)$. In particular the isomorphism class of $\pi(D_\sigma)^\flat$ determines and only depends on the isomorphism class of the filtered $\varphi^f$-module $D_\sigma$.
\end{enumerate}
\end{prop}
\begin{proof}
Similarly to (\ref{eq:amalgi}) or (\ref{eq:amalg4}) we define
\begin{equation}\label{eq:amalgcrit}
\pi_{\flat}(D_\sigma) \ := \!\!\!\!\bigoplus_{I\text{ not v.~c., } \pi_{\alg}(D_\sigma)}\!\!\!\!\!\!\!\pi_I(D_\sigma)
\end{equation}
where v.~c.~means very critical and $\pi_{\alg}(D_\sigma)$ embeds into $\pi_I(D_\sigma)$ via $\iota_I$ (see above (\ref{eq:amalgi}) for $\iota_I$). Similarly to (\ref{mult:decomp}) we have a canonical isomorphism
\begin{multline*}
\bigg(\pi_{\alg}(D_\sigma)\begin{xy} (0,0)*+{}="a"; (12,0)*+{}="b"; {\ar@{-}"a";"b"}\end{xy}\!\Big(\!\!\bigoplus_{I\text{ non-split}}\!\!\!\!\!C(I, s_{{\vert I\vert},\sigma})\otimes_E\Fil_{\vert I\vert}^{\max}D_\sigma\Big)\bigg) \ \bigoplus \ \Big( \!\!\!\!\!\bigoplus_{\substack{I\text{ split}\\ \text{and not v.~c.}}} \!\!\!\!C(I, s_{{\vert I\vert},\sigma})\otimes_E\Fil_{\vert I\vert}^{\max}D_\sigma\Big)\\
\buildrel\sim\over\longrightarrow \pi_\flat(D_\sigma).
\end{multline*}
Similarly \ to \ (\ref{eq:allinj}) \ the \ canonical \ injection \ $\pi_{\flat}(D_\sigma)\hookrightarrow \pi_R(D_\sigma)$ \ induces \ an \ injection $\Ext^1_{\GL_n(K),\sigma}(\pi_{\alg}(D_\sigma),\pi_\flat(D_\sigma))\hookrightarrow \Ext^1_{\GL_n(K),\sigma}(\pi_{\alg}(D_\sigma),\pi_R(D_\sigma))$ and similarly to Lemma \ref{lem:max} the representation $\pi(D_\sigma)^\flat$ is isomorphic to the representation of $\GL_n(K)$ over $E$ associated to the image in
\[\Ext^1_{\GL_n(K),\sigma}\Big(\pi_{\alg}(D_\sigma)\otimes_E\Ker\big(t_{D_\sigma}\vert_{\Ext^1_{\GL_n(K),\sigma}(\pi_{\alg}(D_\sigma),\pi_\flat(D_\sigma))}\big), \pi_\flat(D_\sigma)\Big)\]
of the canonical vector of $\Ext^1_{\GL_n(K),\sigma}(\pi_{\alg}(D_\sigma),\pi_\flat(D_\sigma))\otimes_E\Ext^1_{\GL_n(K),\sigma}(\pi_{\alg}(D_\sigma),\pi_\flat(D_\sigma))^\vee$. Using the equivalence (i)$\Leftrightarrow$(iii) in Lemma \ref{lem:Lfilmax} we have moreover
\begin{equation}\label{eq:decomppiR}
\pi_R(D_\sigma)\cong \pi_\flat(D_\sigma) \bigoplus \bigoplus_{I\text{ v.~c.}}\big(C(I, s_{{\vert I\vert},\sigma})\otimes_E\Fil_{\vert I\vert}^{\max}D_\sigma\big)
\end{equation}
and from Proposition \ref{prop:mul2} we deduce
\begin{multline}\label{mult:decomker}
\Ker(t_{D_\sigma}) = \Ker(t_{D_\sigma}\vert_{\Ext^1_{\GL_n(K),\sigma}(\pi_{\alg}(D_\sigma),\pi_\flat(D_\sigma))}) \\\bigoplus \bigoplus_{I\text{ v.~c.}}\Ext^1_{\GL_n(K),\sigma}\big(\pi_{\alg}(D_\sigma),C(I, s_{{\vert I\vert},\sigma})\otimes_E\Fil_{\vert I\vert}^{\max}D_\sigma\big).
\end{multline}
Using (\ref{eq:decomppiR}), (\ref{mult:decomker}) with Lemma \ref{lem:nonsplit} and Definition \ref{def:pi(d)}, it is formal to check that the image of the canonical vector of $\Ext^1_{\GL_n(K),\sigma}(\pi_{\alg}(D_\sigma),\pi_R(D_\sigma))\otimes_E\Ext^1_{\GL_n(K),\sigma}(\pi_{\alg}(D_\sigma),\pi_R(D_\sigma))^\vee$ in $\Ext^1_{\GL_n(K),\sigma}(\pi_{\alg}(D_\sigma)\otimes_E\Ker(t_{D_\sigma}), \pi_R(D_\sigma))$ has a representative given as in (i) (use that, if $V,W$ are finite dimensional $E$-vector spaces, then the canonical vector in $(V\oplus W)\otimes_E(V\oplus W)^\vee$ lies in $(V\otimes_EV^\vee)\oplus (W\otimes_EW^\vee)$ and is the sum of the two respective canonical vectors). Finally (ii) follows readily from (i) by Lemma \ref{lem:nonsplit} and Theorem \ref{thm:fil}.
\end{proof}

\begin{rem}\label{rem:forlater3}
Though the set of irreducible constituents of $\pi(D_\sigma)$ (with the multiplicity of $\pi_{\alg}(D_\sigma)$) only depends on the Frobenius eigenvalues and the Hodge-Tate weights, this is not the case of $\pi(D_\sigma)^\flat$. Let $D'_\sigma$ be a $\varphi^f$ filtered module as in \S\ref{sec:prel} with same Frobenius eigenvalues and same Hodge-Tate weights as $D_\sigma$ but distinct from $D_\sigma$. If we cannot have $\pi(D'_\sigma)^\flat\cong \pi(D_\sigma)^\flat$ by (ii) of Proposition \ref{prop:flat}, we could still possibly have a proper $\GL_n(K)$-equivariant injection $\pi(D'_\sigma)^\flat\hookrightarrow \pi(D_\sigma)$, or even $\pi(D'_\sigma)^\flat\hookrightarrow \pi(D_\sigma)^\flat$. We do not expect the latter (at least) to occur, but this would require a closer examination of the map $t_{D_\sigma}$ than what is done in Proposition \ref{prop:mul2}.
\end{rem}

\begin{cor}\label{cor:isoflat}
\hspace{2em}
\begin{enumerate}[label=(\roman*)]
\item
The map $\Ext^1_{\GL_n(K),\sigma}(\pi_{\alg}(D_\sigma),\pi_R(D_\sigma))\longrightarrow \Ext^1_{\GL_n(K),\sigma}(\pi_{\alg}(D_\sigma),\pi(D_\sigma))$ induced by the injection $\pi_R(D_\sigma)\hookrightarrow \pi(D_\sigma)$ factors as an isomorphism
\[\Ext^1_{\varphi^f}(D_\sigma,D_\sigma) \bigoplus \Hom_{\Fil}(D_\sigma,D_\sigma)\buildrel\sim\over\longrightarrow \Ext^1_{\GL_n(K),\sigma}\big(\pi_{\alg}(D_\sigma),\pi(D_\sigma)\big).\]
\item
The injection $\pi(D_\sigma)^\flat\hookrightarrow \pi(D_\sigma)$ induces an isomorphism
\[\Ext^1_{\GL_n(K),\sigma}\big(\pi_{\alg}(D_\sigma),\pi(D_\sigma)^\flat\big)\buildrel\sim\over\longrightarrow \Ext^1_{\GL_n(K),\sigma}\big(\pi_{\alg}(D_\sigma),\pi(D_\sigma)\big).\]
\end{enumerate}
\end{cor}
\begin{proof}
Let us first prove that, for any $I\subset \{\varphi_j,\ 0\leq j \leq n-1\}$ of cardinality $\in \{1,\dots,n-1\}$, we have:
\begin{equation}\label{eq:=0}
\Ext^1_{\GL_n(K),\sigma}\big(\pi_{\alg}(D_\sigma),\begin{xy} (0,0)*+{\big(C(I, s_{{\vert I\vert},\sigma})\otimes_E\Fil_{\vert I\vert}^{\max}D_\sigma)\big)}="a"; (39,0)*+{\pi_{\alg}(D_\sigma)}="b"; {\ar@{-}"a";"b"}\end{xy}\!\big)=0
\end{equation}
where the representation on the right hand side is the unique non-split extension (Lemma \ref{lem:nonsplit}). If (\ref{eq:=0}) is wrong, this means there exists an indecomposable locally $\sigma$-analytic representation of the form
\[\begin{xy} (0,0)*+{\big(C(I, s_{{\vert I\vert},\sigma})\otimes_E\Fil_{\vert I\vert}^{\max}D_\sigma)\big)}="a"; (39,0)*+{\pi_{\alg}(D_\sigma)}="b"; (62,0)*+{\pi_{\alg}(D_\sigma)}="c"; {\ar@{-}"a";"b"}; {\ar@{-}"b";"c"}\end{xy}\!.\]
But if such a representation exists, this implies
\[\dim_E\Ext^1_{\GL_n(K),\sigma}\Big(\!\begin{xy}(0,0)*+{\pi_{\alg}(D_\sigma)}="b"; (23,0)*+{\pi_{\alg}(D_\sigma)}="c"; {\ar@{-}"b";"c"}\end{xy}\!,\big(C(I, s_{{\vert I\vert},\sigma})\otimes_E\Fil_{\vert I\vert}^{\max}D_\sigma)\big)\Big)\geq 2\]
and thus \emph{a fortiori} $\dim_E\Ext^1_{\GL_n(K)}(\text{same}) \geq 2$ which contradicts \cite[Lemma 3.5(2)]{Di25} with Lemma \ref{lem:nonsplit} (note that \cite[Lemma 3.5(2)]{Di25} is applied with $\pi_{\alg}(D)$ in (\ref{eq:alg}) instead of $\pi_{\alg}(D_\sigma)$ and $C(I, s_{{\vert I\vert},\sigma})\otimes_E (\otimes_{\tau\ne \sigma}L(\lambda_\tau))$ instead of $C(I, s_{{\vert I\vert},\sigma})$, but one can always take $L(\lambda_\tau)=1$ for $\tau\in \Sigma\setminus\{\sigma\}$ to apply \emph{loc.~cit.}).\bigskip

We now prove the corollary. (ii) immediately follows from (\ref{eq:=0}) and (i) of Proposition \ref{prop:flat}. We prove (i). It formally follows from Definition \ref{def:pi(d)} that the kernel of the map $\Ext^1_{\GL_n(K),\sigma}(\pi_{\alg}(D_\sigma),\pi_R(D_\sigma))\longrightarrow \Ext^1_{\GL_n(K),\sigma}(\pi_{\alg}(D_\sigma),\pi(D_\sigma))$ is isomorphic to $\ker(t_{D_\sigma})$ (we leave this to the reader). Hence it is enough to prove that this map is surjective. We have a commutative diagram of exact sequences (writing $\Ext^1_{\sigma}$ for $\Ext^1_{\GL_n(K),\sigma}$ and $\pi_{\alg},\pi_R, \pi$ for $\pi_{\alg}(D_\sigma),\pi_R(D_\sigma), \pi(D_\sigma)$)
\begin{equation}
\begin{gathered}\label{eq:diag}
\xymatrix{
\Ext^1_{\sigma}(\pi_{\alg},\pi_R)\ar[d]\ar[r] & \Ext^1_{\sigma}(\pi_{\alg},\pi) \ar[d]\ar[r] & \Ext^1_{\sigma}\big(\pi_{\alg},\pi/\pi_R)\ar@{=}[d] \\
\Ext^1_{\sigma}(\pi_{\alg},\pi_R/\pi_{\alg})\ar[r] & \Ext^1_{\sigma}(\pi_{\alg},\pi/\pi_{\alg}) \ar[r] & \Ext^1_{\sigma}(\pi_{\alg},\pi/\pi_R).}
\end{gathered}
\end{equation}
Define the locally $\sigma$-analytic representation (using Lemma \ref{lem:nonsplit})
\[\widetilde\pi :=\bigoplus_I \Big(\!\begin{xy} (0,0)*+{\big(C(I, s_{{\vert I\vert},\sigma})\otimes_E\Fil_{\vert I\vert}^{\max}D_\sigma)\big)}="a"; (39,0)*+{\pi_{\alg}(D_\sigma)}="b"; {\ar@{-}"a";"b"}\end{xy}\!\Big).\]
We have an injection $\pi(D_\sigma)/\pi_{\alg}(D_\sigma)\hooklongrightarrow \widetilde\pi$ which induces another commutative diagram of exact sequences
\begin{equation}
\begin{gathered}\label{eq:diag2}
\xymatrix{
\Ext^1_{\sigma}(\pi_{\alg},\pi_R/\pi_{\alg})\ar[r]\ar@{=}[d] & \Ext^1_{\sigma}(\pi_{\alg},\pi/\pi_{\alg})\ar[r]\ar[d] & \Ext^1_{\sigma}(\pi_{\alg},\pi/\pi_R)\ar@{^{(}->}[d]\\
\Ext^1_{\sigma}(\pi_{\alg},\pi_R/\pi_{\alg})\ar[r]& \Ext^1_{\sigma}(\pi_{\alg},\widetilde\pi)\ar[r] & \Ext^1_{\sigma}(\pi_{\alg},\widetilde\pi/(\pi_R/\pi_{\alg}))}
\end{gathered}
\end{equation}
where the right vertical map is injective as $\pi(D_\sigma)/\pi_R(D_\sigma)\cong \pi_{\alg}(D_\sigma)^{\oplus 2^n-1-\frac{n(n+1)}{2}}$ is a direct summand of $\widetilde\pi/(\pi_R(D_\sigma)/\pi_{\alg}(D_\sigma))\cong \pi_{\alg}(D_\sigma)^{\oplus 2^n-2}$. It then follows from (\ref{eq:=0}) and an obvious diagram chase that the map $\Ext^1_{\sigma}(\pi_{\alg},\pi_R/\pi_{\alg})\rightarrow \Ext^1_{\sigma}(\pi_{\alg},\pi/\pi_{\alg})$ in (\ref{eq:diag2}) is surjective. Hence so is the map $\Ext^1_{\sigma}(\pi_{\alg},\pi_R)\rightarrow \Ext^1_{\sigma}(\pi_{\alg},\pi)$ in (\ref{eq:diag}) by another obvious diagram chase.
\end{proof}

For any subset $S$ of the set $R$ of simple reflections of $\GL_n$ we also define
\begin{equation}\label{eq:flatS}
\pi(D_\sigma)(S)^\flat := \pi(D_\sigma)^\flat \cap \pi(D_\sigma)(S)
\end{equation}
where $\pi(D_\sigma)(S)$ is defined in \S~\ref{sec:property} and the intersection is in $\pi(D_\sigma)$. By Lemma \ref{lem:max} $\pi(D_\sigma)(S)^\flat$ is the maximal subrepresentation of $\pi(D_\sigma)$ which does not contain any $C(I,s_{\vert I\vert,\sigma})$ which is either very critical or such that $s_{\vert I\vert}\notin S$ in its Jordan-H\"older constituents. We have a decomposition for $\pi(D_\sigma)(S)$ analogous to (i) of Proposition \ref{prop:flat} (adding the condition $s_{\vert I\vert}\notin S$ on the right hand side), hence $\pi(D_\sigma)(S)^\flat$ is a direct summand of $\pi(D_\sigma)(S)$. As in (i) of Proposition \ref{prop:flat}, the isomorphism class of $\pi(D_\sigma)(S)^\flat$ determines the one of $\pi(D_\sigma)(S)$ and thus determines (and only depends on) the Hodge-Tate weights $h_{j,\sigma},j\in \{0,\dots,n-1\}$ and the isomorphism class of the filtered $\varphi^f$-module $D_\sigma$ endowed with the partial filtration $(\Fil^{-h_{i,\sigma}}(D_\sigma),s_i\in S)$ by Theorem \ref{thm:fil}. And as in Corollary \ref{cor:isoflat} the injection $\pi(D_\sigma)(S)^\flat\hookrightarrow \pi(D_\sigma)(S)$ induces an isomorphism $\Ext^1_{\GL_n(K),\sigma}(\pi_{\alg}(D_\sigma),\pi(D_\sigma)(S)^\flat)\buildrel\sim\over\rightarrow \Ext^1_{\GL_n(K),\sigma}(\pi_{\alg}(D_\sigma),\pi(D_\sigma)(S))$.\bigskip

Similarly to (\ref{eq:pi(D)}) we define the locally $\Qp$-analytic representation of $\GL_n(K)$ over $E$:
\begin{equation}\label{eq:pi(D)flat}
\pi(D)^\flat:= \bigoplus_{\sigma, \ \!\pi_{\alg}(D)}\big(\pi(D_\sigma)^\flat\otimes_E (\otimes_{\tau\ne \sigma}L(\lambda_\tau))\big)
\end{equation}
where the amalgamated sum is over $\sigma\in \Sigma$ and where $\pi_{\alg}(D)$ embeds into $\pi(D_\sigma)^\flat\otimes_E (\otimes_{\tau\ne \sigma}L(\lambda_\tau))$ via the composition $\pi_{\alg}(D_\sigma)\hookrightarrow \pi_\flat(D_\sigma)\hookrightarrow\pi(D_\sigma)^\flat$ tensored by $\otimes_{\tau\ne \sigma}L(\lambda_\tau)$ (see (\ref{eq:amalgcrit}) for $\pi_\flat(D_\sigma)$). It is obviously a direct summand of $\pi(D)$ and its isomorphism class still determines the isomorphism classes of all of the filtered $\varphi^f$-module $D_\sigma$ for $\sigma\in \Sigma$ by (ii) of Proposition \ref{prop:flat}. For later use we also define
\begin{equation}\label{eq:pi(D)_flat}
\pi_\flat(D):= \bigoplus_{\sigma, \ \!\pi_{\alg}(D)}\big(\pi_\flat(D_\sigma)\otimes_E (\otimes_{\tau\ne \sigma}L(\lambda_\tau))\big) \hooklongrightarrow \pi(D)^\flat.
\end{equation}

We end up this section with an application of Proposition \ref{prop:mul2} which will be used later.\bigskip

We fix a refinement $\mathfrak{R}$, and renumbering the Frobenius eigenvalues if necessary we assume $\mathfrak{R}=(\varphi_0, \dots, \varphi_{n-1})$. Recall from \S~\ref{sec:tri} that for any $i\in \{1,\dots,n-1\}$ we have a surjection
\begin{multline}\label{mult:Ecomp}
\Ext^1_{\varphi^f}(D_{\sigma},D_{\sigma}) \bigoplus \Hom_{\Fil,\fR}^{i}(D_{\sigma}, D_{\sigma})\\
\buildrel (\ref{eq:extphi})^{-1} \bigoplus f_{i,\sigma}\over \twoheadlongrightarrow\Hom_{\sm}(T(K), E) \bigoplus \Hom_{\sigma}(L_{P_i}(\cO_K),E)
\end{multline}
where $\Hom_{\Fil,\fR}^{i}(D_{\sigma}, D_{\sigma})$ is defined in (\ref{mult:explicit}) and $f_{i,\sigma}$ is defined in (\ref{eq:fisigma}). We let $I:=\{\varphi_0,\dots,\varphi_{i-1}\}$ and note that by (\ref{mult:explicit}) and Lemma \ref{lem:nuI} we have
\begin{equation}\label{eq:subsetRi}
t_{D_{\sigma}}\big(\Ext^1_{\GL_n(K),\sigma}(\pi_{\alg}(D_{\sigma}), \pi_I(D_{\sigma})/\pi_{\alg}(D_{\sigma})\big)\subseteq \Hom_{\Fil,\fR}^{i}(D_{\sigma}, D_{\sigma})\subset \Hom_{\Fil}(D_{\sigma}, D_{\sigma}).
\end{equation}

\begin{prop}\label{prop:crit}
Assume that $s_{i,\sigma}$ appears with multiplicity $1$ in some reduced expression of $w_{\fR,\sigma}w_{0,\sigma}$. Then the kernel of (\ref{mult:Ecomp}) is equal to the $1$-dimensional $E$-vector space $t_{D_{\sigma}}(\Ext^1_{\GL_n(K),\sigma}(\pi_{\alg}(D_{\sigma}), \pi_I(D_{\sigma})/\pi_{\alg}(D_{\sigma}))$. 
\end{prop}
\begin{proof}
Let $0 \neq c_I\in \Ext^1_{\GL_n(K),\sigma}(\pi_{\alg}(D_{\sigma}), \pi_I(D_{\sigma})/\pi_{\alg}(D_{\sigma}))$ (recall the latter has dimension $1$ by Lemma \ref{lem:nonsplit}). As $s_{i,\sigma}$ appears in $w_{\fR,\sigma}w_{0,\sigma}$ with multiplicity $1$, $t_{D_{\sigma}}(c_I)$ is non-zero by Proposition \ref{prop:mul2}. By (\ref{eq:subsetRi}) and the uniqueness in (ii) of Remark \ref{rem:Rsmooth}, it suffices to show that $t_{D_{\sigma}}(c_I)$ also satisfies $t_{D_{\sigma}}(c_I)(e_j)=0$ for $0\leq j\leq i-1$ and $t_{D_{\sigma}}(c_I)(e_j)\in \bigoplus_{k=0}^{i-1} E e_k$ for $i\leq j \leq n-1$. The first property is satisfied by (i) of Lemma \ref{lem:nuI} and the second by (ii) of Lemma \ref{lem:nuI} (which can be applied by the equivalence (i)$\Leftrightarrow$(iii) in Lemma \ref{lem:Lfilmax}).
\end{proof}

\section{Local-global compatibility}\label{sec:global0}

For a filtered $\varphi$-module $D$ as in \S~\ref{sec:local} coming from an automorphic Galois representation for a compact unitary group, we prove that the representation $\pi(D)^\flat$ in (\ref{eq:pi(D)flat}) occurs in the associated Hecke-eigenspace of the completed $H^0$. Using Appendix \ref{Wu} of Z.~Wu we give evidence that the larger $\pi(D)$ in (\ref{eq:pi(D)}) should be there too.

\subsection{The global setting}\label{sec:global}

We introduce the global setting, which is (almost) the same as in \cite[\S~5.1]{BHS19}, \cite[\S~5]{HHS25}, and many other references.\bigskip

We let $F^+$ be a totally real number fields, $F/F^+$ a CM extension and $G/F^+$ a unitary group attached to the quadratic extension $F/F^+$ such that $G\times_{F^+} F\cong \GL_n$ ($n\geq 2$) and $G(F^+\otimes_{\mathbb{Q}} \mathbb{R})$ is compact. For a finite place $v$ of $F^+$ which is split in $F$, we have natural isomorphisms $\iota_{\widetilde{v}}: G(F^+_v)\xrightarrow{\sim} G(F_{\widetilde{v}})\xrightarrow{\sim} \GL_n(F_{\widetilde{v}})$ where $\widetilde{v}$ is a place of $F$ above $v$. We denote by $S_p$ the set of places of $F^+$ dividing $p$ and we assume that each place in $S_p$ is split in $F$.\bigskip

We let $U^p=\prod_{v\nmid p} U_v$ be a sufficiently small (cf.~\cite[\S~3.3]{CHT08}) compact open subgroup of $G(\mathbb{A}_{F^+}^{\infty, p})$ where $\mathbb{A}_{F^+}^{\infty, p}$ means the finite ad\`eles of $F^+$ outside $p$. We let $S^p$ be the set of places such that $U_v$ is not hyperspecial and $S:=S_p\cup S^p$. For each $v\in S$, we fix a place $\widetilde{v}$ of $F$ above $v$. For $*\in \{\cO_E, E\}$, we define
\[\widehat{S}(U^p,*):=\{f: G(F^+)\backslash G(\mathbb{A}_{F^+}^{\infty})/U^p \longrightarrow *, \ f \text{ is continuous}\},\]
which is a Banach space over $E$ equipped with a continuous left action of $G(F_p^+):=G(F^+ \otimes_{\mathbb{Q}} \Qp)$ by right translation on functions. We let $\mathbb{T}(U^S)$ be the polynomial $\cO_E$-algebra generated by the Hecke operators $T_{\widetilde{v}}^{(j)}=\Big[U_v \iota_{\widetilde{v}}^{-1}\smat{1_{n-j} & 0 \\ 0 &\varpi_{\widetilde{v}}1_j} U_v\Big]$ for $v\notin S$ which splits as $\widetilde{v}\widetilde{v}^c$ in $F$ and $j=1, \dots, n$, where $\varpi_{\widetilde{v}}$ is a uniformizer of $F_{\widetilde{v}}$. Then $\widehat{S}(U^p,*)$ is equipped with an action of $\mathbb{T}(U^S)$ (by double coset operators) with commutes with $G(F^+_p)$. Recall that for any finite extension $E$ of $\Qp$ in $\overline\Qp$ we have a $G(F^+_p)$-equivariant isomorphism (e.g.~see \cite[Prop.~5.1]{Br15})
\begin{equation}\label{Elalg0}
\widehat{S}(U^p,E)^{\Qp\text{-}\alg} \otimes_E \overline\Qp\cong \bigoplus_{\pi} \big((\pi^{\infty,p})^{U^p} \otimes_{\overline\Q} (\otimes_{v\in S_p}(\pi_v \otimes_{\overline\Q} W_v))\big)^{\oplus m(\pi)}
\end{equation}	
where ``$\Qp$-$\alg$'' means locally $\Qp$-algebraic vectors, $\pi=\pi_{\infty} \otimes_{\mathbb{C}} (\mathbb{C}\otimes_{\overline\Q}\pi^{\infty})=\pi_{\infty} \otimes_{\overline\Q} \pi^{\infty, p} \otimes_{\overline\Q}( \otimes_{v\in S_p}\pi_v)$ runs through automorphic representations of $G(\mathbb{A}_{F^+})$ and where $\otimes_{v\in S_p} W_v$ is the algebraic representation of $G(F_p^+)$ (seen over $\overline{\Qp}$) ``associated'' to $\pi_{\infty}$ (with respect to a fixed isomorphism $\mathbb{C}\xrightarrow{\sim} \overline{\Qp}$, see the discussion in \cite[\S~5]{Br15}).\bigskip

We fix a place $\wp$ of $F^+$ lying above $p$. For each $v\in S_p$, $v\neq \wp$, we fix a dominant weight $\xi_v$ of $\Res_{F^+_v/\Qp} \GL_n$ with respect to the upper Borel $\Res_{F^+_v/\Qp} B$ and an inertial type $\tau_v: I_{F_v^+} \rightarrow \GL_n(E)$ where $I_{F_v^+}$ is the inertia subgroup of $\Gal(\overline{F_v^+}/F_v^+)$. Recall that to $\tau_v$ one can associate a smooth irreducible representation $\sigma(\tau_v)$ of $\GL_n(\cO_{F_v^+})$ over $E$ as in \cite[Thm.~3.7]{CEGGPS16} where $E$ is a sufficiently large finite extension of $\Qp$ in $\overline\Qp$. We let $L(\xi_v)$ be the algebraic representation of $\Res_{F^+_v/\Qp} \GL_n$ over $E$ of highest weight $\xi_v\in (\Z^n)^{\{F_v^+\hookrightarrow E\}}$, and $W_{\xi_v,\tau_v}$ a $\GL_n(\cO_{F_v^+})$-invariant $\cO_E$-lattice of the finite dimensional locally algebraic representation $\sigma(\tau_v)^{\vee} \otimes_E L(\xi_v)^{\vee}$. We let $U_p^{\wp}:=\prod_{v\in S_p\setminus \{\wp\}} \GL_n(\cO_{F_v^+})$, $U^{\wp}:=U^pU_p^{\wp}$ and $W_{\xi,\tau}:=\otimes_{v\in S_p\setminus \{\wp\},\cO_E}W_{\xi_\nu,\tau_\nu}$ (we use the notation of \cite[\S~2.3]{CEGGPS16}). For $*\in \{\cO_E, E\}$, we define 
\begin{equation}\label{Sxitau}
\widehat{S}_{\xi,\tau}(U^{\wp}, *):=(\widehat{S}(U^p, *) \otimes_{\cO_E} W_{\xi,\tau})^{U_p^{\wp}}
\end{equation}
which is a representation of $G(F^+_{\wp})\buildrel \stackrel{\iota_{\widetilde{\wp}}}{\sim}\over \rightarrow \GL_n(F_{\widetilde{\wp}})$ equipped with an action of $\mathbb{T}(U^S)$ commuting with $\GL_n(F_{\widetilde{\wp}})$.\bigskip

We let $\pi$ be an automorphic representation of $G(\mathbb{A}_{F^+})$ satisfying the conditions
\begin{enumerate}[label=(\roman*)]
\item
$(\pi^{\infty,p})^{U^p}\neq 0\text{ \ and \ }\big(\otimes_{v\in S_p\setminus \{\wp\}} (\pi_v \otimes_E \sigma(\tau_v))\big)^{U_p^{\wp}}\neq 0$;
\item
the representation $\otimes_{v\in S_p} W_v$ of $G(F_p^+)$ in (\ref{Elalg0}) satisfies $W_v\cong L(\xi_v)\otimes_E \overline\Qp$ for $v\neq \wp$.
\end{enumerate}
Let $\fm_{\pi}$ be the maximal ideal of $\mathbb{T}(U^S)[1/p]$ such that the $\mathbb{T}(U^S)$-action on $(\pi^{\infty,p})^{U^p}$ via double coset operators coincides with $\omega_{\pi}: \mathbb{T}(U^S)\rightarrow \mathbb{T}(U^S)/\fm_{\pi}\cong \overline\Qp$ (using that $\pi_v^{U_v}$ is $1$-dimensional for $v\notin S$ totally split in $F$). By the work of many people (see for instance \cite[Thm.~7.2.1]{EGH13}), at least for $F^+\ne \Q$\footnote{As is well-known this assumption $F^+\ne \Q$ comes from \cite[Cor.~3.11]{La17}. Maybe it is useless by now but it is not clear to the authors which reference(s) to quote.} one can associate to $\pi$ a continuous semi-simple representation $\rho_{\pi}: \Gal(\overline F/F)\rightarrow \GL_n(E)$ (enlarging $E$ if necessary), whose isomorphism class is uniquely determined by the conditions
\begin{enumerate}[label=(\roman*)]
\item
$\rho_{\pi}^c\cong \rho^{\vee} \otimes_E \varepsilon^{1-n}$ where $\rho^c(g):=\rho(cgc)$ for $g\in \Gal(F^S/F)$ with $c$ being the complex conjugation;
\item
for $v\notin S$ and $v=\widetilde{v} \widetilde{v}^c$, $\rho_{\pi,\widetilde{v}}:=\rho_\pi\vert_{\Gal(\overline{F_{\widetilde{v}}}/F_{\widetilde{v}})}$ is unramified and the characteristic polynomial of $\rho_{\pi}(\Frob_{\widetilde{v}})$ for a geometric Frobenius $\Frob_{\widetilde{v}}$ at $\widetilde{v}$ is $X^n+\omega_{\pi}(T_{\widetilde{v}}^{(1)}) X^{n-1}+\cdots+\omega_{\pi}(T_{\widetilde{v}}^{(n-1)}) X+\omega_{\pi}(T_{\widetilde{v}}^{(n)})$;
\item 
for $v\in S_p\setminus \{\wp\}$, $\rho_{\pi,\widetilde{v}}$ is potentially crystalline of inertial type $\tau_v$ and of Hodge-Tate weights $\xi_v-(0,\dots, n-1)^{\{F_{\widetilde{v}}\hookrightarrow E\}}$ (with obvious notation).
\end{enumerate}
Note that (ii) implies $\omega_\pi$ is $E$-valued.\bigskip

We now assume $\rho_{\pi}$ absolutely irreducible and $\rho_{\pi,\widetilde{\wp}}$ crystalline such that the filtered $\varphi$-module $D:=D_{\cris}(\rho_{\pi,\widetilde{\wp}})$ is regular and satisfies (\ref{eq:phi}). We use the notation of the previous sections to this specific $D$: $K:=F^+_\wp=F_{\widetilde{\wp}}$ with $f:=[K_0:\Qp]$, $\{\varphi_j,0\leq j\leq n-1\}$ is the set of $\varphi^f$-eigenvalues, $\{h_{j,\sigma},j\in \{0,\dots,n-1\}\}$ is the set of Hodge-Tate weights of $D_\sigma$ for $\sigma\in \Sigma$, $\lambda_{\sigma}=(\lambda_{j,\sigma})=(h_{j,\sigma}-(n-1-j))$, $\pi_p,\pi_{\alg}(D_\sigma),\pi_{\alg}(D)$ are the $\GL_n(K)$-representations in (\ref{eq:pip}), (\ref{eq:algsigma}), (\ref{eq:alg}) etc. It easily follows from (\ref{Elalg0}), (\ref{Sxitau}) and the local-global compatibility in the classical local Langlands correspondence (see again \cite[Thm.~7.2.1]{EGH13}) that there exists $m\geq 1$ such that
\begin{equation}\label{Elalg}
(\pi_{\alg}(D) \otimes_E \varepsilon^{n-1})^{\oplus m} = \big(\pi_p \otimes_E (\otimes_{\sigma}L(\lambda_\sigma))\otimes_E \varepsilon^{n-1}\big)^{\oplus m} \xlongrightarrow{\sim} \widehat{S}_{\xi,\tau}(U^{\wp},E)[\fm_{\pi}]^{\Qp\text{-}\alg}.
\end{equation}

Let $\widehat{S}_{\xi,\tau}(U^{\wp},E)[\fm_{\pi}]^{\Qp\text{-}\an}$ be the locally $\Qp$-analytic vectors of the continuous representation $\widehat{S}_{\xi,\tau}(U^{\wp},E)[\fm_{\pi}]$. Here is our main conjecture:

\begin{conj}\label{conj:main}
With the above notation, the isomorphism (\ref{Elalg}) extends to an injection of locally $\Qp$-analytic representations of $\GL_n(K)$: 
\[(\pi(D) \otimes_E \varepsilon^{n-1})^{\oplus m}\hooklongrightarrow \widehat{S}_{\xi,\tau}(U^{\wp},E)[\fm_{\pi}]^{\Qp\text{-}\an}.\]
Moreover, for any rank $n$ regular filtered $\varphi$-module $D'$ satisfying (\ref{eq:phi}) as in \S~\ref{sec:prel}, we have a $\GL_n(K)$-equivariant injection
\[\pi(D') \otimes_E \varepsilon^{n-1}\hooklongrightarrow \widehat{S}_{\xi,\tau}(U^{\wp},E)[\fm_{\pi}]^{\Qp\text{-}\an}\]
if and only if $D'_{\sigma}\cong D_{\sigma}$ for all $\sigma\in \Sigma$.
\end{conj}

Conjecture \ \ref{conj:main} \ in \ particular \ implies \ that \ the \ locally \ $\Qp$-analytic \ representation $\widehat{S}_{\xi,\tau}(U^{\wp},E)[\fm_{\pi}]^{\Qp\text{-}\an}$ determines the collection of $\varphi^f$-filtered modules $\{D_{\sigma}\}_{\sigma\in \Sigma}$. (It is of course expected that $\widehat{S}_{\xi,\tau}(U^{\wp},E)[\fm_{\pi}]^{\Qp\text{-}\an}$ actually determines the full $\varphi$-filtered module $D$, which is much stronger statement when $K=F^+_\wp\ne \Qp$.)\bigskip

We will prove non-trivial results towards Conjecture \ref{conj:main} under the so-called Taylor-Wiles assumptions (as in \cite{BHS19} or \cite{HHS25}), which we recall now. We denote by $\overline{\rho}$ the mod $p$ semi-simplification of $\rho_{\pi}$ (we should write $\overline{\rho}_\pi$, but $\overline{\rho}$ will lighten notation). We let $\fm_{\overline{\rho}}$ be the maximal ideal of $\mathbb{T}(U^S)$ such that the characteristic polynomial of $\overline{\rho}(\Frob_{\widetilde{v}})$ for $v\notin S$ and $v=\widetilde{v} \widetilde{v}^c$ is $X^n+\omega_{\overline{\rho}}(T_{\widetilde{v}}^{(1)}) X^{n-1}+\cdots+\omega_{\overline{\rho}}(T_{\widetilde{v}}^{(n-1)}) X+\omega_{\overline{\rho}}(T_{\widetilde{v}}^{(n)})$, where $\omega_{\overline{\rho}}$ denotes the natural map $\mathbb{T}(U^S) \twoheadrightarrow \mathbb{T}(U^S)/\fm_{\overline{\rho}}\cong k_E$. For a $\mathbb{T}(U^S)$-module $M$, we denote by $M_{\overline{\rho}}$ its localization at $\fm_{\overline{\rho}}$. By the proof of \cite[Lemma 6.5]{BD20} (and the discussion that follows \emph{loc.~cit.}), $\widehat{S}_{\xi,\tau}(U^{\wp},E)_{\overline{\rho}}$ (resp.~$\widehat{S}(U^p,E)_{\overline{\rho}}$) is a $\GL_n(K) \times \mathbb{T}(U^S)$-equivariant direct summand of $\widehat{S}_{\xi,\tau}(U^{\wp},E)$ (resp.~of $\widehat{S}(U^p,E)$). We assume the following (Taylor-Wiles) assumptions: 

\begin{hyp}\label{TayWil0}\ 
\begin{enumerate}[label=(\roman*)]
\item
$p>2$;
\item
the field $F$ is unramified over $F^+$ and $F$ does not contain a non trivial root $\sqrt[p]{1}$ of $1$;
\item
$G$ is quasi-split at all finite places of $F^+$;
\item
$U_{v}$ is hyperspecial when the finite place $v$ of $F^+$ is inert in $F$;
\item
$\overline{\rho}$ is absolutely irreducible and $\overline{\rho}(\Gal_{F(\sqrt[p]{1})})$ is adequate (\cite[Def.~2.20]{Th17}).
\end{enumerate}
\end{hyp}

Under Hypothesis \ref{TayWil0}, the action of $\mathbb{T}(U^S)_{\overline{\rho}}$ on $\widehat{S}(U^p,E)_{\overline{\rho}}$ factors through a faithful action of a certain noetherian local complete $\cO_E$-algebra $\widetilde{\mathbb{T}}(U^S)_{\overline{\rho}}$, and there is a natural surjection $R_{\overline{\rho},\cS}\twoheadrightarrow \widetilde{\mathbb{T}}(U^S)_{\overline{\rho}}$, where $R_{\overline{\rho}, \cS}$ denotes the universal Galois deformation ring of deformations associated to the deformation problem (cf.~\cite[\S~2.3]{CHT08})
\begin{equation*}
\cS=\big(F/F^+, S, \widetilde{S}, \cO_E, \overline{\rho}, \varepsilon^{1-n} \delta_{F/F^+}^n, \{R_{\overline{\rho}_{\widetilde{v}}}\}_{v\in S}\big)
\end{equation*}
where $\widetilde{S}:=\{\widetilde{v}\ |\ v\in S\}$, $\delta_{F/F^+}$ is the quadratic character of $\Gal(\overline{F}/F^+)$ associated to $F/F^+$, and $R_{\overline{\rho}_{\widetilde{v}}}$ denotes the maximal $p$-torsion free reduced quotient of the framed deformation ring of $\overline{\rho}_{\widetilde{v}}:=\rhobar\vert_{\Gal(\overline{F_{\widetilde{v}}}/F_{\widetilde{v}})}$ over $\cO_E$. We let $R_{\overline{\rho},\cS}(\xi,\tau)$ be the universal Galois deformation ring of deformations associated to the deformation problem 
\begin{equation*}
\big(F/F^+, S, \widetilde{S}, \cO_E, \overline{\rho}, \varepsilon^{1-n} \delta_{F/F^+}^n, \{R_{\overline{\rho}_{\widetilde{v}}}\}_{v\in S^p\cup \{\wp\}}, \{R_{\overline{\rho}_{\widetilde{v}}}(\tau_v, \xi_v)\}_{v\in S_p \setminus \{\wp\}}\big),
\end{equation*}
where $R_{\overline{\rho}_{\widetilde{v}}}(\tau_v, \xi_v)$ denotes the universal potentially crystalline framed deformation ring of $\overline{\rho}_{\widetilde{v}}$ of inertial type $\tau_v$ and of Hodge-Tate weights $\xi_v-(0,\dots, n-1)^{\{F_{\widetilde{v}}\hookrightarrow E\}}$ (\cite{Ki08}). Then $R_{\overline{\rho},\cS}(\xi,\tau)$ is a quotient of $R_{\overline{\rho},\cS}$ and the action of $R_{\overline{\rho},\cS}$ on $\widehat{S}_{\xi,\tau}(U^{\wp},E)_{\overline{\rho}}$ factors through $R_{\overline{\rho},\cS}(\xi,\tau)$.\bigskip

Let $g\in \Z_{\geq 1}$, $R^{\loc}:=\widehat{\otimes}_{v\in S} R_{\overline{\rho}_{\widetilde{v}}}$ and $R_{\infty}:=R^{\loc}[[x_1, \dots, x_g]]$. Let $q:=g+[F^+:\mathbb{Q}]\frac{n(n-1)}{2}+|S|n^2$ and $S_{\infty}:=\cO_E[[y_1, \dots, y_q]]$. By \cite[Thm.~3.5]{BHS171} (which is a slight generalization of \cite[\S~2]{CEGGPS16}), there exist $g\geq 1$, a unitary $R_{\infty}$-admissible representation $\Pi_\infty$ of $G(F^+_p)$ over $E$ (\cite[D\'ef.~3.1]{BHS171}) and morphisms of $\cO_E$-algebras $S_{\infty} \rightarrow R_{\infty}$ and $R_{\infty} \rightarrow R_{\overline{\rho},\cS}$ such that:
\begin{enumerate}[label=(\roman*)]
\item
There exists an $\cO_E$-lattice $\Pi_{\infty}^0$ of $\Pi_{\infty}$ stable by $G(F_p^+)$ and $R_{\infty}$ such that $M_{\infty}:=\Hom_{\cO_E}(\Pi_{\infty}^0,\cO_E)$ is a finite type projective $S_{\infty}[[K_p]]$-module (via $S_{\infty} \rightarrow R_{\infty}$) where $K_p$ is a compact open subgroup of $G(F_p^+)$.
\item
There exist an ideal $\fa$ of $R_{\infty}$ together with a surjection $R_{\infty}/\fa R_{\infty} \twoheadrightarrow R_{\overline{\rho}, \cS}$, and an $R_{\infty}/\fa$-equivariant isomorphism of $G(F_p^+)$-representations $\Pi_{\infty}[\fa] \cong \widehat{S}(U^p, E)_{\overline{\rho}}$.
\end{enumerate}
We define
\[R^{\loc}_{\xi,\tau}:=\big(\widehat{\otimes}_{v\in S^p\cup \{\wp\}} R_{\overline{\rho}_{\widetilde{v}}}\big) \widehat{\otimes}_{\cO_E} \big(\widehat{\otimes}_{v\in S_p\setminus \{\wp\}} R_{\overline{\rho}_{\widetilde{v}}}(\xi_v,\tau_v)\big)\text{\ and\ }R_{\infty}(\xi,\tau):=R^{\loc}_{\xi,\tau}[[x_1, \dots, x_g]]\]
(note that $R^{\loc}_{\xi,\tau}$ is a quotient of $R^{\loc}$ and the surjection $R_\infty \twoheadrightarrow R_{\overline{\rho}, \cS}\twoheadrightarrow R_{\overline{\rho}, \cS}(\xi,\tau)$ factors through $R_\infty \twoheadrightarrow R_{\infty}(\xi,\tau) \twoheadrightarrow R_{\overline{\rho}, \cS}(\xi,\tau)$). We also define
\begin{equation}\label{eq:piinfini}
\Pi_{\infty}(\xi,\tau):=(\Pi_{\infty} \otimes_E W_{\xi,\tau})^{U_p^{\wp}}\text{\ and\ }\Pi_{\infty}^0(\xi,\tau):=(\Pi_{\infty}^0 \otimes_{\cO_E} W_{\xi,\tau})^{U_p^{\wp}}.
\end{equation}
Then $\Pi_{\infty}(\xi,\tau)$ is an $R_{\infty}$-admissible continuous unitary representation of $\GL_n(K)$ over $E$ and by an argument similar to the one in the beginning of the proof of \cite[Lemma 4.18.1]{CEGGPS16}, it is not difficult to check that
\begin{equation}\label{eq:patched}
M_{\infty}(\xi,\tau):=\Hom_{\cO_E}(\Pi_{\infty}^0(\xi,\tau), \cO_E)
\end{equation}
is a finite type projective $S_{\infty}[[\GL_n(\cO_K)]]$-module (see also the proof of \cite[Lemma 6.1]{BD20}). Moreover, we have by (\ref{Sxitau})
\begin{equation}\label{Epatchedauto}
\Pi_{\infty}(\xi,\tau)[\fa]\cong \widehat{S}_{\xi,\tau}(U^{\wp},E)_{\overline{\rho}}.
\end{equation}
Finally, we note that the action of $R_{\infty}$ on $\Pi_{\infty}(\xi,\tau)$ factors through its quotient $R_{\infty}(\xi,\tau)$ since its action on the dense subspace $\Pi_{\infty}(\xi,\tau)^{\Qp\text{-}\alg}$ of locally algebraic vectors for $G(F_p^+)$ factors through $R_{\infty}(\xi,\tau)$ by (the proof of) \cite[Lemma 4.17.1]{CEGGPS16}.

\begin{rem}\label{rem:setting}
The present global setting slightly varies from the global setting of \cite[\S~5]{BHS19}. In \emph{loc.~cit.}~one treats all $p$-adic places together (there is no fixed place $\wp$) and consequently $\rho_{\pi}$ is assumed crystalline regular satisfying (\ref{eq:phi}) at all $p$-adic places. But the proofs of \cite[\S~5]{BHS19} essentially remain unchanged (and are even simpler) in the present setting replacing $\Pi_\infty$ of \cite[\S~5.1]{BHS19} by $\Pi_{\infty}(\xi,\tau)$, $R_\infty$ of \emph{loc.~cit.}~by $R_{\infty}(\xi,\tau)$ and recalling that the rigid analytic variety associated to $\widehat{\otimes}_{v\in S_p\setminus \{\wp\}} R_{\overline{\rho}_{\widetilde{v}}}(\xi_v,\tau_v)$ is smooth by \cite[Thm.~3.3.8]{Ki08}.
\end{rem}

\subsection{Patched eigenvarieties}\label{sec:patched}

We briefly recall the construction of the patched eigenvariety from \cite{BHS171}, along with some of its (partially classical) closed subspaces as described in \cite{Wu24}. A minor difference with \emph{loc.~cit.}~is that we fix a locally algebraic type at the $p$-adic places other than $\wp$ (see Remark \ref{rem:setting}). But all the results of \cite{BHS171} and \cite{Wu24} carry over to our case with only minor adjustments.\bigskip

We keep all previous notation. For a local complete noetherian $\cO_E$-algebra or $E$-algebra $R$ we denote by $\Spf R$ the associated formal scheme over $\cO_E$ or $E$ respectively. For a local complete noetherian $\cO_E$-algebra $R$ we denote by $(\Spf R)^{\rig}$ Raynaud's associated rigid analytic space over $E$.\bigskip

We consider the $T(K)$-representation $J_B\big(\Pi_{\infty}(\xi,\tau)^{R_{\infty}(\xi,\tau)\text{-}\an}\big)$ where $\Pi_{\infty}(\xi,\tau)$ is as in (\ref{eq:piinfini}), ``$R_{\infty}(\xi,\tau)\text{-}\an$" denotes the locally $R_{\infty}(\xi,\tau)$-analytic vectors for $\GL_n(K)$ in the sense of \cite[D\'ef.~3.2]{BHS171} and $J_B$ is Emerton's locally $\Qp$-analytic Jacquet functor with respect to $B(K)$ (\cite{Em06}). We let $\widehat{T}$ be the rigid analytic space (taken over $E$) parametrizing locally $\Qp$-analytic characters of $T(K)$. There exists a coherent sheaf $\cM_{\infty}(\xi,\tau)$ over the quasi-Stein rigid analytic space $(\Spf R_{\infty}(\xi,\tau))^{\rig} \times \widehat{T}$ uniquely determined by
\[\Gamma\big((\Spf R_{\infty}(\xi,\tau))^{\rig} \times \widehat{T}, \cM_{\infty}(\xi,\tau)\big) = J_B\big(\Pi_{\infty}(\xi,\tau)^{R_{\infty}(\xi,\tau)\text{-}\an}\big)^{\vee}\]
(where $J_B(-)^\vee$ is the continuous dual of $J_B(-)$), see \cite[Prop.~3.4]{BHS171}. We let $\cE_{\infty}(\xi,\tau)$ be the scheme theoretic support of $\cM_{\infty}(\xi,\tau)$ (see the discussion above \cite[D\'ef.~3.6]{BHS171}), which we call the patched eigenvariety for $\Pi_{\infty}(\xi,\tau)$. By similar arguments as in \cite[Cor.~3.12]{BHS171}, \cite[Cor.~3.20]{BHS171} and \cite[Lemma 3.8]{BHS172}, we have 

\begin{prop}\label{P: Esigma}\ 
\begin{enumerate}[label=(\roman*)]
\item
The rigid analytic space $\cE_{\infty}(\xi,\tau)$ is reduced and equidimensional of dimension $g+|S|n^2+[F^+:\mathbb{Q}]\frac{n(n-1)}{2}+[K:\Qp]n$.
\item
The coherent sheaf $\cM_{\infty}(\xi,\tau)$ is Cohen-Macaulay over $\cE_{\infty}(\xi,\tau)$.
\end{enumerate}
\end{prop}

We let $\overline{r}:=\overline{\rho}_{\widetilde{\wp}}$ and $X_{\tri}(\overline{r})\hookrightarrow (\Spf R_{\overline{r}})^{\rig} \times \widehat{T}$ be the (framed) trianguline variety of \cite[\S~2.2]{BHS171} (see \S~\ref{sec:global} for $R_{\overline{r}}$ and recall that $X_{\tri}(\overline{r})$ is by definition reduced, see \cite[D\'ef.~2.4]{BHS171}). We let $\iota_p$ be the following automorphism where $\delta_B=\boxtimes_{i=0}^{n-1} \vert \cdot \vert_K^{n-1-2i}$ is the modulus character of $B(K)$:
\begin{equation}\label{eq:iotap}
\iota_p: (\Spf R_{\overline{r}})^{\rig} \times \widehat{T} \xlongrightarrow{\sim} (\Spf R_{\overline{r}})^{\rig} \times \widehat{T}, \ \ (r, \delta) \mapsto (r, \delta \delta_B^{-1}(\boxtimes_{i=0}^{n-1} \varepsilon^{i})^{-1}).
\end{equation}
Let $R^{\loc,\wp}_{\xi,\tau}:=\big(\widehat{\otimes}_{v\in S^p} R_{\overline{\rho}_{\widetilde{v}}}\big) \widehat{\otimes}_{\cO_E} \big(\widehat{\otimes}_{v\in S_p\setminus \{\wp\}} R_{\overline{\rho}_{\widetilde{v}}}(\xi_v,\tau_v)\big)\text{\ and\ }R_{\infty}^{\wp}(\xi,\tau):=R^{\loc,\wp}_{\xi,\tau}[[x_1, \dots, x_g]]$. Si\-milarly as in \cite[Thm.~3.21]{BHS171}, we then have

\begin{prop}\label{prop:pembd}
The natural closed embedding
\[\cE_{\infty}(\xi,\tau)\hooklongrightarrow(\Spf R_{\infty}(\xi,\tau))^{\rig} \times \widehat{T}\cong (\Spf R_{\infty}^{\wp}(\xi,\tau))^{\rig} \times (\Spf R_{\overline{r}})^{\rig} \times \widehat{T}\]
factors through a closed embedding
\begin{equation}\label{Einjtri}
\cE_{\infty}(\xi,\tau)\hooklongrightarrow (\Spf R_{\infty}^{\wp}(\xi,\tau))^{\rig} \times \iota^{-1}_p(X_{\tri}(\overline{r}))
\end{equation}
which \ identifies \ $\cE_{\infty}(\xi,\tau)$ \ with \ a \ union \ of \ irreducible \ components \ of \ $(\Spf R_{\infty}^{\wp}(\xi,\tau))^{\rig}\times \iota^{-1}_p(X_{\tri}(\overline{r}))$.
\end{prop}

We now introduce some closed subspaces $\cE_{\infty}(\xi,\tau)_{\sigma,i}$ of $\cE_{\infty}(\xi,\tau)$.\bigskip

We fix $\sigma\in \Sigma$ and $i\in \{1,\dots, n-1\}$. We consider the following locally $\Qp$-analytic representation of $\GL_n(K)$
\begin{equation}\label{Elamsigalg}
\Pi_{\infty}(\xi,\tau)^{\lambda^{\sigma}\text{-}\alg}:=\big(\Pi_{\infty}(\xi,\tau)^{R_{\infty}(\xi,\tau)\text{-}\an} \otimes_E (\otimes_{\tau\neq \sigma} L(\lambda_{\tau})^{\vee})\big)^{\sigma\text{-}\an}\otimes_E (\otimes_{\tau\neq \sigma} L(\lambda_{\tau}))
\end{equation}
(recall $\lambda_\tau$ is as above (\ref{eq:t}) and $L(\lambda_\tau)$ as below (\ref{eq:t})) where ``$\sigma\text{-}\an$'' means the locally $\sigma$-analytic vectors. It follows from \cite[Prop.~6.1.3]{Di171} that $\Pi_{\infty}(\xi,\tau)^{\lambda^{\sigma}\text{-}\alg}$ is a closed subrepresentation of the locally $\Qp$-analytic vectors $\Pi_{\infty}(\xi,\tau)^{\Qp\text{-}\an}$ and by similar arguments as in \cite[Lemma 7.2.12]{Di171}, we have
{\small
\begin{equation}\label{eq:Ding}
J_{P_i}\big(\Pi_{\infty}(\xi,\tau)^{\lambda^{\sigma}\text{-}\alg}\big)\cong J_{P_i}\Big(\big(\Pi_{\infty}(\xi,\tau)^{R_{\infty}(\xi,\tau)\text{-}\an} \otimes_E (\otimes_{\tau\neq \sigma} L(\lambda_{\tau})^{\vee})\big)^{\sigma\text{-}\an}\Big) \otimes_E (\otimes_{\tau\neq \sigma} L_{i}(\lambda_{\tau}))
\end{equation}}
\!\!where $L_{i}(\lambda_{\tau})$ is the algebraic representation of $(L_{P_i})_\tau=L_{P_i} \times_{\Spec K, \tau} \Spec E$ over $E$ of highest weight $\lambda_\tau$ with respect to the upper Borel. Let $L_{P_i}'\cong \SL_{i}\times \SL_{n-i}$ be the derived subgroup of $L_{P_i}$ (seen over $K$) and $\fl_{P_i}'$ be its Lie algebra over $K$. We have an injection of locally $\sigma$-analytic representations of $L_{P_i}(K)$ over $E$ (see \cite[Prop.~4.2.10]{Em17} for the injectivity) 
\begin{multline}\label{EPiclass}
\Big(J_{P_i}\Big(\big(\Pi_{\infty}(\xi,\tau)^{R_{\infty}(\xi,\tau)\text{-}\an} \otimes_E (\otimes_{\tau\neq \sigma} L(\lambda_{\tau})^{\vee})\big)^{\sigma\text{-}\an}\Big) \otimes_E L_i(\lambda_{\sigma})^{\vee}\Big)^{\fl_{P_i,\sigma}'}\!\! \otimes_E L_i(\lambda_{\sigma}) \\
\hooklongrightarrow J_{P_i}\Big(\big(\Pi_{\infty}(\xi,\tau)^{R_{\infty}(\xi,\tau)\text{-}\an} \otimes_E (\otimes_{\tau\neq \sigma} L(\lambda_{\tau})^{\vee})\big)^{\sigma\text{-}\an}\Big).
\end{multline}
Applying $J_{B\cap L_{P_i}}(-)$ we finally obtain the following injections of $T(K)$-representations
{\scriptsize
\begin{eqnarray}\label{Eparabolic}
\nonumber V_{\sigma,i}\!\!&:=&\!\!J_{B\cap L_{P_i}}\bigg(\Big(J_{P_i}\Big(\big(\Pi_{\infty}(\xi,\tau)^{R_{\infty}(\xi,\tau)\text{-}\an} \otimes_E (\otimes_{\tau\neq \sigma} L(\lambda_{\tau})^{\vee})\big)^{\sigma\text{-}\an}\Big) \otimes_E L_i(\lambda_{\sigma})^{\vee}\Big)^{\fl_{P_i}'} \bigg) \otimes_E \prod_{\tau\in \Sigma}L(\lambda_\tau)^{N(K)} \\
\nonumber \!\!&\cong &\!\!J_{B\cap L_{P_i}}\bigg(\Big(J_{P_i}\Big(\big(\Pi_{\infty}(\xi,\tau)^{R_{\infty}(\xi,\tau)\text{-}\an} \otimes_E (\otimes_{\tau\neq \sigma} L(\lambda_{\tau})^{\vee})\big)^{\sigma\text{-}\an}\Big) \otimes_E L_i(\lambda_{\sigma})^{\vee}\Big)^{\fl_{P_i}'}\otimes_E L_{i}(\lambda_{\sigma}) \otimes_E (\otimes_{\tau\neq \sigma} L_{i}(\lambda_{\tau}))\bigg) \\
\nonumber \!\!&\hookrightarrow &\!\!J_{B\cap L_{P_i}}\Big(J_{P_i}\Big(\big(\Pi_{\infty}(\xi,\tau)^{R_{\infty}(\xi,\tau)\text{-}\an} \otimes_E (\otimes_{\tau\neq \sigma} L(\lambda_{\tau})^{\vee})\big)^{\sigma\text{-}\an}\Big) \otimes_E (\otimes_{\tau\neq \sigma} L_{i}(\lambda_{\tau}))\Big)\\
\nonumber \!\!&\cong &\!\!J_{B \cap L_{P_i}}\big(J_{P_i}(\Pi_{\infty}(\xi,\tau)^{\lambda^{\sigma}\text{-}\alg})\big) \\
\nonumber \!\!&\cong &\!\!J_B\big(\Pi_{\infty}(\xi,\tau)^{\lambda^{\sigma}\text{-}\alg}\big)\\
\!\!&\hookrightarrow &\!\!J_B\big(\Pi_{\infty}(\xi,\tau)^{R_{\infty}(\xi,\tau)\text{-}\an}\big)
\end{eqnarray}}
\!\!where the injections follow from \cite[Lemme 3.4.7(ii)]{Em06} and (\ref{EPiclass}), the first isomorphism follows from \cite[Prop.~4.3.6]{Em06}, the second from (\ref{eq:Ding}) and the third from \cite[Thm.~5.3.2]{HL10} (note that the $T(K)$-representation $L(\lambda_\tau)^{N(K)}$ is the highest weight of $L(\lambda_\tau)$). All the above representations inherit from $\Pi_{\infty}(\xi,\tau)$ a left action of $R_{\infty}(\xi,\tau)$, and all the above morphisms are (clearly) $R_{\infty}(\xi,\tau)$-equivariant.\bigskip

We then denote by $\cM_{\infty}(\xi,\tau)_{\sigma,i}$ the unique coherent sheaf on the quasi-Stein rigid analytic space $(\Spf R_{\infty}(\xi,\tau))^{\rig} \times \widehat{T}$ such that
\begin{equation}\label{eq:sections}
\Gamma\big((\Spf R_{\infty}(\xi,\tau))^{\rig} \times \widehat{T}, \cM_{\infty}(\xi,\tau)_{\sigma,i}\big)= (V_{\sigma,i}\otimes_E\varepsilon^n)^{\vee},
\end{equation}
and we let $\cE_{\infty}(\xi,\tau)_{\sigma,i}$ be the scheme theoretic support of $\cM_{\infty}(\xi,\tau)_{\sigma,i}$ (the twist by the character $\varepsilon^n\circ{\det}$ of $T(K)$ comes from the same twist in (\ref{Elalg})). Then $\cM_{\infty}(\xi,\tau)_{\sigma,i}$ is a quotient of $\cM_{\infty}(\xi,\tau)$ and $\cE_{\infty}(\xi,\tau)_{\sigma,i}$ is a closed rigid analytic subspace of $\cE_{\infty}(\xi,\tau)$. (Note that both $\cM_{\infty}(\xi,\tau)_{\sigma,i}$ and $\cE_{\infty}(\xi,\tau)_{\sigma,i}$ also depend on $(\lambda_\tau)_\tau$, but this weight will be fixed later and we drop it from the notation.)

\begin{prop}\label{P: Esigmai}\ 
\begin{enumerate}[label=(\roman*)]
\item
The rigid space $\cE_{\infty}(\xi,\tau)_{\sigma,i}$ is reduced and equidimensional of dimension $g+|S|n^2+[F^+\!:\mathbb{Q}]\frac{n(n-1)}{2}+2$.
\item
The coherent sheaf $\cM_{\infty}(\xi,\tau)_{\sigma,i}$ is Cohen-Macaulay over $\cE_{\infty}(\xi,\tau)_{\sigma,i}$.
\end{enumerate}
\end{prop}
\begin{proof}
The equidimensional part in (i) follows from \cite[Prop.~5.9]{Wu24}. The reducedness in (i) follows by the same argument as in the proof of \cite[Cor.~3.20]{BHS171}, with Theorem 3.19 of \emph{loc.~cit.}~replaced by (an easy variation of) \cite[Prop.~5.11]{Wu24}. Part (ii) follows by the same argument as in \cite[Lemma 3.8]{BHS172} with \cite[Prop.~3.11]{BHS171} replaced by \cite[Prop.~5.9]{Wu24}. Finally, the dimension of $\cE_{\infty}(\xi,\tau)_{\sigma,i}$ is the same as the dimension of ${\mathcal W}_{\lambda_J'}$ above \cite[Prop.~5.9]{Wu24}, which can be checked in our case to be $\dim\cE_{\infty}(\xi,\tau)-[K:\Qp]n+2$, whence the formula in (i) by (i) of Proposition \ref{P: Esigma}.
\end{proof}

\subsection{Local model of trianguline varieties}\label{sec:model}

We apply the local model theory of trianguline varieties developed in \cite{BHS19} (see also \cite{Wu24}) to establish a smoothness result for $\cE_{\infty}(\xi,\tau)_{\sigma,i}$ (Corollary \ref{C: smooth}) which will play a role in the proof of our main local-global compatibility result.\bigskip

We fix a crystalline $E$-valued point $r$ of $(\Spf R_{\overline{r}})^{\rig}$ such that the filtered $\varphi$-module $D:=D_{\cris}(r)$ is regular and satisfies (\ref{eq:phi}) (we use the notation of \S~\ref{sec:prel} for $D$). We also fix a refinement $\fR$ of $D$. By reordering the eigenvalues of $\varphi^f$, we assume $\fR=(\varphi_0, \dots, \varphi_{n-1})$. With notation as in (\ref{eq:pip}), (\ref{eq:t}) we define the following $E$-valued characters of $T(K)$:
\begin{eqnarray}
\nonumber\unr(\underline{\varphi})&:=&\unr(\varphi_{0})\boxtimes \unr(\varphi_{1})\boxtimes \cdots \boxtimes \unr(\varphi_{n-1})\\
\label{eq:dominant}t^{\bf h}&:=&\prod_{i=0}^{n-1}\Big(\prod_{\sigma\in \Sigma}\sigma(t_i)^{h_{i,\sigma}}\Big).
\end{eqnarray}
By \cite[Thm.~4.2.3]{BHS19}, the ``dominant'' point
\begin{equation}\label{EptyR}
y_{\fR}:=(r, \unr(\underline{\varphi})t^{\textbf{h}})\in (\Spf R_{\overline{r}})^{\rig} \times \widehat{T}
\end{equation}
lies in the closed trianguline subspace $X_{\tri}(\overline{r})$.\bigskip

We briefly recall the local model of \cite{BHS19} and refer the reader to \cite[\S~3]{BHS19} for more details and references. We let $G:=\GL_n$ (seen over $K$) and define the algebraic variety $\widetilde{\fg}_{\Sigma}:=G_{\Sigma} \times^{B_{\Sigma}} \mathfrak{b}_{\Sigma}$ over $E$, where $B_{\Sigma}$ acts on the left on $\mathfrak{b}_{\Sigma}$ via the adjoint action. We have $\widetilde{\fg}_{\Sigma}\cong \prod_{\sigma\in \Sigma} G\times^B \mathfrak{b}_{\sigma}$ as $E$-schemes and we set $\widetilde{\fg}_{\sigma}:=G\times^B \mathfrak{b}_{\sigma}$. Recall that $\widetilde{\fg}_{\Sigma}$ (resp.~$\widetilde{\fg}_{\sigma}$) is isomorphic to the closed reduced subscheme of $G_{\Sigma}/B_{\Sigma} \times \fg_{\Sigma}$ (resp.~$G_{\sigma}/B_{\sigma} \times \fg_{\sigma}$) consisting of those $(gB_{\Sigma}, \psi)$ with $\Ad_{g^{-1}}(\psi) \in \fb_{\Sigma}$ (resp.~those $(gB_{\sigma}, \psi)$ with $\Ad_{g^{-1}}(\psi) \in \fb_{\sigma}$). We let $X_{\Sigma}:=\widetilde{\fg}_{\Sigma} \times_{\fg_{\Sigma}} \widetilde{\fg}_{\Sigma}$ and $X_{\sigma}:=\widetilde{\fg}_{\sigma} \times_{\fg_{\sigma}} \widetilde{\fg}_{\sigma}$ for $\sigma\in \Sigma$, then $X_\Sigma\cong \prod_{\sigma\in \Sigma}X_\sigma$ and there are natural closed embeddings
\[X_{\Sigma} \hookrightarrow G_{\Sigma}/B_{\Sigma} \times G_{\Sigma}/B_{\Sigma} \times \fg_{\Sigma},\ \ X_{\sigma} \hookrightarrow G_{\sigma}/B_{\sigma} \times G_{\sigma}/B_{\sigma} \times \fg_{\sigma}.\]
We denote by $\kappa: X_{\Sigma} \rightarrow G_{\Sigma}/B_{\Sigma} \times G_{\Sigma}/B_{\Sigma}$ the induced morphism. For $w=(w_{\sigma})\in S_n^{\Sigma}=$ the Weyl group of $G_\Sigma$, we let $U_w=\prod_{\sigma} U_{w_{\sigma}}\subset G_{\Sigma}/B_{\Sigma} \times G_{\Sigma}/B_{\Sigma}$ be the $G_{\Sigma}$-orbit of $(1 B_{\Sigma}, w B_{\Sigma})$ for the diagonal action, where we also denote by $w\in G_\Sigma$ the permutation matrix associated to $w$. We let $V_w :=\kappa^{-1} (U_w)\subset X_\Sigma$ and denote by $X_w$ the Zariski-closure of $V_w$ in $X_{\Sigma}$. Then $\{X_w\}_{w\in S_n^{\Sigma}}$ is exactly the set of irreducible components of $X_{\Sigma}$ (cf.~\cite[Prop.~2.2.5]{BHS19}). For $\sigma\in \Sigma$ we define in a similar way $X_{w_\sigma}\subset X_\sigma$ for $w_\sigma\in S_n$, and likewise $\{X_{w_\sigma}\}_{w_\sigma\in S_n}$ is the set of irreducible components of $X_{\sigma}$. For $w\in S_n^{\Sigma}$ we have $X_w\cong \prod_{\sigma\in \Sigma}X_{w_\sigma}$.\bigskip

We let $\cC_E$ be the category of local artinian $E$-algebras. We recall that $K\otimes_{K_0}D\cong D_{\dR}(r)\buildrel\sim\over\rightarrow D_{\pdR}(r):=(B_{\pdR}\otimes_{\Qp}r)^{\Gal(\overline K/K)}$. As below (\ref{mult:Etri}) for $\sigma\in \Sigma$ we fix a basis $e_{0,\sigma}, \dots, e_{n-1,\sigma}$ of $\varphi^f$-eigenvectors of $D_{\sigma}$ such that $\varphi^f(e_{i,\sigma})=\varphi_i e_{i,\sigma}$. With respect to this basis, we have a bijection
\[\alpha: \bigoplus_{\sigma\in \Sigma} E^{\oplus n}\buildrel\sim\over \longrightarrow D_{\dR}(r).\]
If $A\in \cC_E$ with maximal ideal ${\fm_A}$, recall that $\cR_{K,A}$ is the Robba ring over $K$ with $A$-coefficients (see \cite[Def.~6.2.1]{KPX14}). We let $X_{r,\fR}^{\square}$ be the groupoid over $\cC_E$ (denoted $X_{r,\cM_{\bullet}}^{\square}$ in \cite[\S~3.6]{BHS19}) of deformations $(r_A, \cF^{\bullet}_A, \alpha_A)$ such that
\begin{enumerate}[label=(\roman*)]
\item
$r_A$ is a framed deformation of $r$ over $A\in \cC_E$;
\item
$\cF^{\bullet}_A$ is an increasing filtration by projective $(\varphi, \Gamma)$-submodules of $D_{\rig}(r_A)[1/t]$ over $\cR_{K,A}[1/t]$ such that $\cF^i_A/\cF^{i-1}_A \cong \cR_{K,A}(\delta_{i,A})[1/t]$ with $\cF^0_A=0$ and $\delta_{i,A}\equiv \unr(\varphi_i) \pmod{\fm_A}$ for $i=1, \dots, n$;
\item
$\alpha_A$ is an $A$-linear isomorphism $\bigoplus_{\sigma \in \Sigma} A^{\oplus n} \xrightarrow {\sim}D_{\pdR}(r_A)$ such that $\alpha_A \equiv \alpha \pmod{\fm_A}$.
\end{enumerate}
Similarly as below (\ref{mult:explicit}) we let $g_{\sigma}\in G(E)$ such that $g_{\sigma} B_{\sigma}\in G_{\sigma}/B_{\sigma}$ gives the coordinates of the Hodge flag (\ref{eq:fil}) in the basis $(e_{i,\sigma})_i$ and $g_{\Sigma}:=(g_{\sigma})\in G_{\Sigma}$. We then define the point
\[z_{\fR}:=(1 B_{\Sigma}, g_\Sigma B_{\Sigma}, 0)\in X_{\Sigma}(E)\]
and let $\widehat{X}_{\Sigma, z_{\fR}}$ be the groupoid over $\cC_E$ pro-represented by the noetherian local complete $E$-algebra given by the completion of the $E$-scheme $X_\Sigma$ at $z_{\fR}$.\bigskip

By \cite[Cor.~3.5.8(ii), (3.28)]{BHS19} with \cite[(3.28)]{BHS19} we have a natural formally smooth morphism of groupoids over $\cC_E$:
\begin{equation}\label{E: localmodel}
X_{r, \fR}^{\square} \longrightarrow \widehat{X}_{\Sigma, z_{\fR}}, \ (r_A, \cF_A^{\bullet}, \alpha_A) \longmapsto (g_{1,A} B_{\Sigma}(A), g_{2,A} B_{\Sigma}(A), \nu_A)
\end{equation}
where, under the isomorphism $\alpha_A: \oplus_{\sigma\in \Sigma} A^{\oplus n} \xrightarrow{\sim} D_{\pdR}(r_A)$:
\begin{enumerate}[label=(\roman*)]
\item
$g_{1,A}B_{\Sigma}(A)$ gives the coordinates of the flag $D_{\pdR}(\cF^{\bullet}_A)$;
\item
$g_{2,A} B_{\Sigma}(A)$) gives the coordinates of the Hodge flag
\[\Fil^{-h_j, \sigma} D_{\pdR}(r_A)_{\sigma}:=\big((t^{-h_{j,\sigma}}B_{\dR}^+[\log t]\otimes_{\Qp}r_A)^{\Gal(\overline K/K)}\big)_\sigma\]
where for a $K\otimes_{\Qp}A$-module $\mathcal D$ we define $\mathcal D_\sigma$ for $\sigma\in \Sigma$ as in (\ref{eq:dec}) replacing $E[\epsilon]/\epsilon^2$~by~$A$;
\item
$\nu_A\in \fg_{\Sigma}(A)$ is the matrix of Fontaine's nilpotent operator on $D_{\pdR}(A)$ induced by the nilpotent operator $\nu_{\pdR}$ on $B_{\pdR}$ (see the beginning of \S~\ref{sec:indep}). 
\end{enumerate}
For $w\in S_n^{\Sigma}$ we define the groupoid $X_{r, \fR}^{\square,w}:=X_{r, \fR}^{\square} \times_{\widehat{X}_{\Sigma, z_{\fR}}} \widehat{X}_{w, z_{\fR}}$. Note that we have a formally smooth morphism $X_{r, \fR}^{\square,w}\longrightarrow \widehat{X}_{w, z_{\fR}}$ and that $X_{r, \fR}^{\square,w}$ is empty if $z_{\fR}$ does not lie in $X_w$. We define the groupoid $X_{r,\fR}$ as we defined $X_{r, \fR}^{\square}$ but forgetting the framing $\alpha$. We have a forgetful morphism of groupoids $X_{r, \fR}^{\square}\rightarrow X_{r, \fR}$ and we define the groupoid $X_{r, \fR}^{w}\subset X_{r,\fR}$ as the image of $X_{r, \fR}^{\square,w}$ in $X_{r,\fR}$. Then as in \cite[(3.26)]{BHS19} we have an equivalence of groupoids over $\cC_E$
\begin{equation}\label{eq:group}
X_{r, \fR}^{\square,w}\buildrel\sim\over\longrightarrow X_{r, \fR}^{w}\times_{X_{r, \fR}}X_{r, \fR}^{\square}.
\end{equation}

We refer the reader to \cite[Def.~(A.5.1)]{Ki09} and \cite[Def.~(A.7.1)(1)]{Ki09} for the definition of pro-representable groupoids over $\cC_E$ and recall that, if $\cG$ is a pro-representable groupoid over $\cC_E$, then the natural morphism of groupoids $\cG\rightarrow \vert\cG\vert$ is an equivalence, where $\vert \cG\vert$ denotes the associated functor of isomorphism classes. In that case we won't distinguish $\cG$ and $\vert \cG\vert$. By \cite[Thm.~3.6.2(i)]{BHS19} $X_{r,\fR}$, $X_{r, \fR}^{w}$ are pro-representable, hence are equivalent to their associated deformation functors $|X_{r,\fR}|$, $|X_{r,\fR}^w|$. Let $R_{r, \fR}$, $R_{r, \fR}^{w}$ be the noetherian local complete $E$-algebras pro-representing the functors $|X_{r,\fR}|$, $|X_{r,\fR}^w|$, which are quotients of $R_r:=$ the noetherian local complete $E$-algebra pro-representing framed deformations of $r$ over artinian $E$-algebras (see for instance \cite[(3.33)]{BHS19}). Denote by $\widehat{X}_{\tri}(\overline{r})_{y_{\fR}}$ the completion of ${X}_{\tri}(\overline{r})$ at the point $y_{\fR}$, we have a morphism of affine formal $E$-schemes $\widehat{X}_{\tri}(\overline{r})_{y_{\fR}} \rightarrow \Spf R_r$ (see above \cite[Prop.~3.7.2]{BHS19}). By \cite[Cor.~3.7.8]{BHS19} together with (\ref{eq:dominant}) and the definition of the permutation $w$ above \cite[Lemma 3.7.4]{BHS19}, the morphism $\widehat{X}_{\tri}(\overline{r})_{y_{\fR}} \rightarrow \Spf R_r$ factors through an isomorphism of affine formal $E$-schemes
\begin{equation}\label{modeltri}
\widehat{X}_{\tri}(\overline{r})_{y_{\fR}}\xlongrightarrow{\sim} \Spf R_{r, \fR}^{w_0}. 
\end{equation}

The study of the tangent spaces of the previous groupoids over $\cC_E$, that is, of their values at $A=E[\epsilon]/\epsilon^2$, is very important for our arguments. As in \S~\ref{sec:indep} denote by $\cM(D)$ the $(\varphi, \Gamma)$-module over $\cR_{K,E}$ associated to the filtered $\varphi$-module $D$. We have a commutative diagram of finite dimensional $E$-vector spaces
\begin{equation}\label{diag: R}
\adjustbox{scale=0.86}{
\begin{tikzcd}
&X_{r,\fR}(E[\epsilon]/\epsilon^2) \arrow[r, two heads, "f_{\fR}"] &\Ext^1_{\fR}(\cM(D), \cM(D)) \arrow[r, "\ref{eq:Einvt}+\ref{eq:Etri2}"] \arrow[d, two heads, "\ref{eq:Enil}"] &\Hom(T(K),E)\arrow[d, two heads, "\text{res}"] \\
X_{r, \fR}^{\square}(E[\epsilon]/\epsilon^2) \arrow[r, two heads, "\ref{E: localmodel}"] \arrow[ru, two heads]	& T_{z_{\fR}} X_{\Sigma} \arrow[r, two heads] & \oplus_{\sigma\in \Sigma} \Hom_{\Fil,\fR}(D_{\sigma}, D_{\sigma}) \arrow[r, "\ref{eq:EhomfilR}"] & \oplus_{\sigma\in \Sigma} \Hom_{\sigma}(T(\cO_K),E)
\end{tikzcd}}
\end{equation}
where the first bottom horizontal map is surjective by formal smoothness of (\ref{E: localmodel}), the second bottom horizontal map is the composition
\begin{equation}\label{Etangentmodel}
T_{z_{\fR}} X_{\Sigma} \twoheadlongrightarrow \oplus_{\sigma} \big( \fb_{\sigma} \cap \Ad_{g_{\sigma}}(\fb_{\sigma}) \big) \buildrel {\stackrel{\ref{eq:middle}}{\sim}}\over \longrightarrow \oplus_{\sigma\in \Sigma}\Hom_{\Fil,\fR}(D_{\sigma},D_{\sigma})
\end{equation}
(noting that the map of tangent spaces $T_{z_{\fR}} X_{\Sigma} \rightarrow \fg_{\Sigma}$ induced by $X_{\Sigma} \hookrightarrow G_{\Sigma}/B_{\Sigma} \times G_{\Sigma}/B_{\Sigma} \times \fg_{\Sigma}\twoheadrightarrow \fg_{\Sigma}$ has image equal to $\fb_{\Sigma} \cap \Ad_g(\fb_{\Sigma})=\bigoplus_{\sigma\in \Sigma}\fb_{\sigma} \cap \Ad_{g_{\sigma}} (\fb_{\sigma})$), and where the map $f_{\fR}$ is induced by the composition
\[X_{r,\fR}(E[\epsilon]/\epsilon^2) \hooklongrightarrow X_r(E[\epsilon]/\epsilon^2)\twoheadlongrightarrow \Ext^1_{\Gal(\overline K/K)}(r,r)\cong \Ext^1_{(\varphi, \Gamma)}(\cM(D), \cM(D))\]
where $X_r$ is the groupoid over $\cC_E$ of framed deformations of $r$ (pro-represented by $R_r)$. Moreover it is not difficult to check that $X_{r,\fR}(E[\epsilon]/\epsilon^2)\hookrightarrow X_r(E[\epsilon]/\epsilon^2)$ \ is \ the \ preimage \ of $\Ext^1_{\fR}(\cM(D), \cM(D))\subset \Ext^1_{(\varphi, \Gamma)}(\cM(D), \cM(D))$, from which one deduces $\dim_E \Ker f_{\fR}=n^2-\dim_E \Hom_{(\varphi, \Gamma)}(\cM(D), \cM(D))$. The left part of (\ref{diag: R}) is easily checked to commute and the right part commutes by (\ref{eq:diagR}). Note that the surjectivity of the middle vertical map in (\ref{diag: R}) follows by an obvious diagram chase. We deduce by Lemma \ref{lem:LpdR} and its proof
\begin{equation}\label{dimker}
\begin{array}{rcl}
\dim_E \Ker ((\ref{eq:Enil}) \circ f_{\fR})&=&\dim_E \Ext^1_g(\cM(D), \cM(D))+\dim_E \Ker f_{\fR}\\
&=&\big(\dim_E \Hom_{(\varphi, \Gamma)}(\cM(D), \cM(D))+\frac{n(n-1)}{2} [K:\Qp]\big)\\
&&\ \ \ \ \ \ \ \ \ \ \ \ \ \ \ \ \ \ \ \ +\big(n^2-\dim_E \Hom_{(\varphi, \Gamma)}(\cM(D), \cM(D))\big)\\
&=&n^2+\frac{n(n-1)}{2} [K:\Qp].
\end{array}
\end{equation}
The image of $X_{r, \fR}^{\square,w_0}(E[\epsilon]/\epsilon^2)$ in $\oplus_{\sigma\in \Sigma}\Hom_{\Fil,\fR}(D_{\sigma}, D_{\sigma})$ via (\ref{diag: R}) coincides with the image of $T_{z_{\fR}} X_{w_0}=\oplus_{\sigma\in \Sigma}T_{(1 B_{\sigma}, g_\sigma B_{\sigma}, 0)}X_{w_{0,\sigma}}$ via (\ref{Etangentmodel}), hence has the form $\oplus_{\sigma\in \Sigma}\Hom_{\Fil,\fR,w_0}(D_{\sigma}, D_{\sigma})$ for some subspaces $\Hom_{\Fil,\fR,w_0}(D_{\sigma}, D_{\sigma})\subseteq \Hom_{\Fil,\fR}(D_{\sigma}, D_{\sigma})$. By (\ref{diag: R}) this is also the image of $\Ext^1_{\fR,w_0}(\cM(D), \cM(D)):=f_{\fR}(X_{r,\fR}^{w_0}(E[\epsilon]/\epsilon^2))$ by the map \ref{eq:Enil}. Note that one can again check that $X_{r,\fR}^{w_0}(E[\epsilon]/\epsilon^2)\subset X_{r,\fR}(E[\epsilon]/\epsilon^2)$ is the preimage of $\Ext^1_{\fR,w_0}(\cM(D), \cM(D))$ via $f_{\fR}$.

\begin{lem}\label{lem:Lzw0}
We have $\Ext^1_g(\cM(D), \cM(D)) \subset \Ext^1_{\fR,w_0}(\cM(D), \cM(D))$.
\end{lem}
\begin{proof}
As $G_{\Sigma}(1 B_{\Sigma}, w_0 B_{\Sigma})\times \{0\}\subseteq V_{w_0}=\kappa^{-1}(U_{w_0})$ and $G_{\Sigma}(1 B_{\Sigma}, w_0 B_{\Sigma})$ is Zariski-dense in $G_{\Sigma}/B_{\Sigma} \times G_{\Sigma}/B_{\Sigma}$, we have $Z_{w_0}:=G_{\Sigma}/B_{\Sigma} \times G_{\Sigma}/B_{\Sigma} \times \{0\}\subset X_{w_0}$. In particular, $z_{\fR}\in Z_{w_0}\subset X_{w_0}$. Using (\ref{diag: R}) (and Lemma \ref{lem:LpdR}), we see that the preimage of $\Ext^1_g(\cM(D), \cM(D))$ in $X_{r, \fR}^{\square}(E[\epsilon]/\epsilon^2)$ by $X_{r, \fR}^{\square}(E[\epsilon]/\epsilon^2) \rightarrow X_{r, \fR}(E[\epsilon]/\epsilon^2) \twoheadrightarrow \Ext^1_{\fR}(\cM(D), \cM(D))$ is exactly the preimage of $T_{z_{\fR}} Z_{w_0}\subset T_{z_{\fR}} X_{\Sigma}$ by $	X_{r, \fR}^{\square}(E[\epsilon]/\epsilon^2) \twoheadrightarrow T_{z_{\fR}} X_{\Sigma}$. The lemma then follows from $Z_{w_0} \subset X_{w_0}$ and the definition of $X_{r, \fR}^{\square, w_0}$. 
\end{proof}

By Lemma \ref{lem:Lzw0} and (\ref{diag: R}) (with Lemma \ref{lem:LpdR}), $\Ext^1_{\fR,w_0}(\cM(D), \cM(D))$ is also the preimage of $\oplus_{\sigma\in \Sigma} \Hom_{\Fil,\fR,w_0}(D_{\sigma}, D_{\sigma})$ via (\ref{eq:Enil}). Then (\ref{diag: R}) induces a commutative ``subdiagram''
\begin{equation} \label{diag: Rw0}
\adjustbox{scale=0.73}{
\begin{tikzcd}
&X_{r,\fR}^{w_0}(E[\epsilon]/\epsilon^2)\cong T_{y_{\fR}} X_{\tri}(\overline{r}) \arrow[r, two heads] &\Ext^1_{\fR,w_0}(\cM(D), \cM(D)) \arrow[r] \arrow[d, two heads] &\Hom(T(K),E)\arrow[d, two heads] \\
X_{r, \fR}^{\square,w_0}(E[\epsilon]/\epsilon^2) \arrow[r, two heads] \arrow[ru, two heads]	& T_{z_{\fR}} X_{w_0} \arrow[r, two heads] & \oplus_{\sigma\in \Sigma} \Hom_{\Fil,\fR,w_0}(D_{\sigma}, D_{\sigma}) \arrow[r] & \oplus_{\sigma\in \Sigma} \Hom_{\sigma}(T(\cO_K),E)
\end{tikzcd}}
\end{equation}
where the isomorphism $X_{r,\fR}^{w_0}(E[\epsilon]/\epsilon^2)\cong T_{y_{\fR}} X_{\tri}(\overline{r})$ follows from (\ref{modeltri}) and the discussion before it. Note that the composition of the top horizontal maps coincides with the tangent map of the composition $X_{\tri}(\overline{r}) \hookrightarrow (\Spf R_{\overline{r}})^{\rig} \times \widehat{T}\twoheadrightarrow \widehat{T}$ at the points $y_{\fR}\mapsto \unr(\underline{\varphi})t^{\textbf{h}}$.\bigskip

Now we fix $\sigma\in \Sigma$ and $i\in \{1,\dots, n-1\}$. We consider the subgroupoid $X_{r,\fR,\sigma,i}^{\square,w_0}\subset X_{r,\fR}^{\square,w_0}$ of $(r_A, \cF^{\bullet}_A, \alpha_A)$ such that (see below (\ref{eq:DpdR}) for a de Rham $(\varphi,\Gamma)$-module over $\cR_{K,E}[1/t]$):
\begin{enumerate}[label=(\roman*)]
\item
$\cF^i_A$ and $D_{\rig}(r_A)[1/t]/\cF^i_A $ are de Rham up to twist by a character;
\item
$D_{\rig}(r_A)[1/t]$ is $\tau$-de Rham (i.e.~$\dim_ED_{\dR}\big(D_{\rig}(r_A)[1/t]\big)_{\tau}=n\dim_EA$) for $\tau \neq \sigma$.
\end{enumerate}
(``Up to twist by a character'' in (i) means that a twist by a rank one $(\varphi,\Gamma)$-over $\cR_{K,A}[1/t]$ is de Rham.) We define the groupoid $X_{r,\fR,\sigma,i}^{w_0}$ as the image of $X_{r,\fR,\sigma,i}^{\square,w_0}$ in $X_{r,\fR}$. We have $X_{r,\fR,\sigma,i}^{w_0}\subset X_{r,\fR}^{w_0}$, and by (\ref{eq:group}) and since conditions (i) and (ii) above do not concern the framing $\alpha_A$ we again have an equivalence of groupoids over $\cC_E$
\begin{equation}\label{eq:groupi}
X_{r, \fR,\sigma,i}^{\square,w_0}\buildrel\sim\over\longrightarrow X_{r, \fR,\sigma,i}^{w_0}\times_{X_{r, \fR}}X_{r, \fR}^{\square}\ (\cong X_{r, \fR,\sigma,i}^{w_0}\times_{X_{r, \fR}^{w_0}}X_{r, \fR}^{\square,w_0}).
\end{equation}
By \cite[Lemma 3.11]{Wu24} and the discussion after \emph{loc.~cit.}~both $X_{r,\fR,\sigma,i}^{\square,w_0}$ and $X_{r,\fR,\sigma,i}^{w_0}$ are pro-representable. Consider the closed $E$-subscheme of $X_\Sigma$
\[X_{\Sigma}^{\sigma}:=X_{\Sigma} \times \prod_{\tau\neq \sigma} (G_{\tau}/B_{\tau} \times G_{\tau}/B_{\tau} \times \{0\})\subset \prod_{\tau\in \Sigma}X_\tau\cong X_{\Sigma}.\]
Let $\widetilde{\fg}_{P_i,\sigma}:=G_{\sigma} \times^{B_{\sigma}} \mathfrak{r}_{P_i,\sigma}\subset \widetilde{\fg}_{\sigma}$, $X_{\sigma,i}:=\widetilde{\fg}_{P_i,\sigma} \times_{\fg_{\sigma}} \widetilde{\fg}_{\sigma}\subset X_{\sigma}$ (a closed $E$-subscheme of $X_\sigma$), and define the $E$-schemes
\begin{equation}\label{eq:tangent}
\begin{gathered}
X_{\Sigma,i}^{\sigma}:=X_{i,\sigma}\times \prod_{\tau\neq \sigma} (G_{\tau}/B_{\tau} \times G_{\tau}/B_{\tau} \times \{0\})\\
X_{w_0,i}^{\sigma}:=(X_{i,\sigma}\times_{X_\sigma}X_{w_{0,\sigma}}) \times \prod_{\tau\neq \sigma} (G_{\tau}/B_{\tau} \times G_{\tau}/B_{\tau} \times \{0\}) \ \cong \ X_{\Sigma, i}^{\sigma}\times_{X_\Sigma}X_{w_0}.
\end{gathered}
\end{equation}
We have closed immersions $Z_{w_0}\subset X_{w_0,i}^{\sigma}\subset X_{\Sigma,i}^{\sigma}\subset X_{\Sigma}^{\sigma}\subset X_{\Sigma}$, and in particular $z_{\fR}\in X_{w_0,i}^{\sigma}$. The following lemma easily follows from the above definitions and from (\ref{E: localmodel}):

\begin{lem} \label{Lemma: modelsigmai}
We have 
\begin{equation*}
X_{r, \fR, \sigma,i}^{\square, w_0} \cong X_{r,\fR}^{\square,w_0} \times_{\widehat{X}_{w_0, z_{\fR}}} (\widehat{X}_{w_0,i}^{\sigma})_{z_{\fR}} \cong X_{r,\fR}^{\square} \times_{\widehat{X}_{\Sigma, z_{\fR}}} (\widehat{X}_{w_0,i}^{\sigma})_{z_{\fR}}.		\end{equation*}
In particular, $X_{r, \fR, \sigma, i}^{\square, w_0}$ is formally smooth over $(\widehat{X}_{w_0, i}^{\sigma})_{z_{\fR}}$. 
\end{lem}

We will need the following formula:

\begin{lem}\label{L: dimsgimai}
We have $\dim X_{w_0,i}^{\sigma}= n(n-1)[K:\Qp]+2$.
\end{lem} 
\begin{proof}
By (\ref{eq:tangent}) it suffices to show $\dim (X_{i,\sigma} \times_{X_{\sigma}} X_{w_{0,\sigma}})=2+n(n-1)$. One easily checks that the following diagram is Cartesian
\begin{equation*}
\begin{tikzcd}
\widetilde{\fg}_{P_i,\sigma}=G_{\sigma} \times^{B_{\sigma}} \mathfrak{r}_{P_i,\sigma} \arrow[r] \arrow[d] & 	G_{\sigma}/{B_{\sigma}} \times \fg_{\sigma} \arrow[r] \arrow[d] & G_{\sigma}/B_{\sigma} \arrow[d]\\
\widetilde{\fg}_{P_i,\sigma}':=	G_{\sigma} \times^{P_{i,\sigma}} \mathfrak{r}_{P_i,\sigma} \arrow[r] & 	G_{\sigma}/{P_{i,\sigma}} \times \fg_{\sigma} \arrow[r] & G_{\sigma}/P_{i,\sigma}
\end{tikzcd}
\end{equation*}
where the first top (resp.~bottom) horizontal map sends $(g, \psi)$ with $\Ad_{g^{-1}}(\psi)\in \mathfrak{r}_{P_i,\sigma}$ to $(g B_{\sigma}, \psi)$ (resp.~to $(g P_{i,\sigma}, \psi)$). Hence we have
\begin{equation}\label{eq:xisigma}
X_{i,\sigma}\cong (G_{\sigma}/B_{\sigma}) \times_{G_{\sigma}/P_{i,\sigma}} \big(\widetilde{\fg}'_{P_i,\sigma} \times_{\fg_{\sigma}} \widetilde{\fg}_{\sigma}\big).
\end{equation}
It follows from \cite[Cor.~5.2.2]{BD24} that $X_{i,\sigma}$ is equidimensional of dimension $\frac{n(n-1)}{2}+\dim \mathfrak{r}_{P_i,\sigma} + \dim (L_{P_i,\sigma}\cap N_\sigma)= 2+n(n-1)$ and that its irreducible components are indexed by the longest elements in $W_{L_{P_{i,\sigma}}}\!\backslash S_n$. If the (reduced) irreducible component associated to some $w_\sigma\ne w_{0,\sigma}$ lies in $X_{w_{0,\sigma}}$, then by \cite[Lemma 5.2.6]{BD24} its image in $\ft_{\sigma}\times \ft_{\sigma}$ by the map $X_\sigma\rightarrow \ft_{\sigma}\times \ft_{\sigma}$ (see the beginning of \cite[\S~2.5]{BHS19}) contains a point of the form $(\Ad(w_\sigma^{-1})t,t)$ which is distinct from $(\Ad(w_{0,\sigma}^{-1})t,t)$. But this contradicts the last equality in \cite[Lemma 2.5.1]{BHS19}. We deduce that $(X_{w_0,i}^{\sigma})_{\mathrm{red}}$ is an irreducible component of $X_{i,\sigma}$ and the lemma follows.
\end{proof}

The image of $X_{r, \fR,\sigma,i}^{\square,w_0}(E[\epsilon]/\epsilon^2)$ in $\oplus_{\sigma\in \Sigma}\Hom_{\Fil,\fR,w_0}(D_{\sigma}, D_{\sigma})$ via (\ref{diag: Rw0}) coincides with the image of $T_{z_{\fR}} X_{w_0,i}^{\sigma}$ via (\ref{Etangentmodel}) by Lemma \ref{Lemma: modelsigmai}, hence has the form $\Hom_{\Fil,\fR,w_0}^i(D_{\sigma}, D_{\sigma})$ for some subspace $\Hom_{\Fil,\fR,w_0}^i(D_{\sigma}, D_{\sigma})\subseteq \Hom_{\Fil,\fR,w_0}(D_{\sigma}, D_{\sigma})$. By (\ref{diag: Rw0}) this is also the image of $\Ext^1_{\fR,w_0,\sigma,i}(\cM(D), \cM(D)):=f_{\fR}(X_{r,\fR,\sigma,i}^{w_0}(E[\epsilon]/\epsilon^2))$ by the map \ref{eq:Enil} (see (\ref{diag: R}) for $f_{\fR}$). In fact it follows from (\ref{eq:tangent}), (\ref{eq:xisigma}) and (\ref{mult:missing}) that we have inside $\Hom_{\Fil,\fR}(D_{\sigma}, D_{\sigma})$:
\begin{equation*}
\Hom_{\Fil,\fR,w_0}^{i}(D_{\sigma}, D_{\sigma})=\Hom_{\Fil,\fR,w_0}(D_{\sigma}, D_{\sigma})\cap \Hom_{\Fil,\fR}^i(D_{\sigma}, D_{\sigma}).
\end{equation*}
Moreover one can again check that $X_{r, \fR,\sigma,i}^{w_0}(E[\epsilon]/\epsilon^2)\subset X_{r,\fR}^{w_0}(E[\epsilon]/\epsilon^2)\subset X_{r,\fR}(E[\epsilon]/\epsilon^2)$ is the preimage of $\Ext^1_{\fR,w_0,\sigma,i}(\cM(D), \cM(D))$ via $f_{\fR}$. By the proof of Lemma \ref{lem:Lzw0} with the inclusion $Z_{w_0}\subset X_{w_0,i}^{\sigma}$ and Lemma \ref{Lemma: modelsigmai}, we have
\begin{equation*}
\Ext^1_g(\cM(D), \cM(D))\subset \Ext^1_{\fR,w_0,\sigma,i}(\cM(D), \cM(D)).
\end{equation*}
Then (\ref{diag: Rw0}) induces another commutative ``subdiagram'' (see (\ref{eqtisigma}) for $\Hom_{\sigma,i}(T(K),E)$): 
\begin{equation}\label{diag: model2}
\adjustbox{scale=0.82}{
\begin{tikzcd}
&X_{r,\fR,\sigma,i}^{w_0}(E[\epsilon]/\epsilon^2) \arrow[r, two heads] &\Ext^1_{\fR,w_0,\sigma,i}(\cM(D), \cM(D)) \arrow[r, "\ref{eq:Einvt}+\ref{eq:Etri2}"] \arrow[d, two heads, "(\ref{eq:Enil})"] &\Hom_{\sigma,i}(T(K),E)\arrow[d, two heads] \\
X_{r, \fR,\sigma,i}^{\square,w_0}(E[\epsilon]/\epsilon^2) \arrow[r, two heads] \arrow[ru, two heads] & T_{z_{\fR}} X_{w_0,i}^{\sigma} \arrow[r, two heads] & \Hom_{\Fil,\fR,w_0}^i(D_{\sigma}, D_{\sigma}) \arrow[r, "\ref{eq:fisigma}"] & \Hom_{\sigma}(L_{P_i}(\cO_K),E)
\end{tikzcd}}
\end{equation}
where $\Ext^1_{\fR,w_0,\sigma,i}(\cM(D), \cM(D))$ is also the preimage of $\Hom_{\Fil,\fR,w_0}^i(D_{\sigma}, D_{\sigma})$ via (\ref{eq:Enil}). In particular $\Ext^1_{\fR,w_0,\sigma,i}(\cM(D), \cM(D))\subset \Ext^1_{\sigma}(\cM(D), \cM(D))$ and Corollary \ref{cor:splitsigma} induces a splitting (depending on a choice of $\log(p)$ and see (\ref{eq:surligne}) for $\ol{\Ext}^1_{\fR,w_0,\sigma,i}(\cM(D), \cM(D))$)
\begin{equation}\label{eq:missingiso}
\ol{\Ext}^1_{\fR,w_0,\sigma,i}(\cM(D),\cM(D)) \buildrel\sim\over\longrightarrow \Ext^1_{\varphi^f}(D_{\sigma},D_{\sigma}) \bigoplus \Hom_{\Fil,\fR,w_0}^i(D_{\sigma}, D_{\sigma}).
\end{equation}

Recall we defined $w_{\fR,\sigma}\in S_n$ just above Proposition \ref{prop:Psmooth}.

\begin{prop}\label{prop:Psmoothbis}
Assume that the multiplicity of $s_{i,\sigma}$ in some reduced expression of $w_{\fR,\sigma}w_{0,\sigma}$ is at most one.\begin{enumerate}[label=(\roman*)]
\item
We have $\dim_E \Hom_{\Fil,\fR,w_0}^i(D_{\sigma}, D_{\sigma}) =2$, $X_{w_0,i}^{\sigma}$ is smooth at the point $z_{\fR}$ and the local complete noetherian $E$-algebra pro-representing $X_{r, \fR, \sigma, i}^{w_0}$ is formally smooth.
\item
If $s_{i,\sigma}$ does not appear in some (equivalently any) reduced expression of $w_{\fR,\sigma} w_{0,\sigma}$, then the natural inclusion $\Hom_{\Fil,\fR,w_0}^i(D_{\sigma}, D_{\sigma}) \hookrightarrow \Hom_{\Fil,\fR}^i(D_{\sigma},D_{\sigma})$ is bijective. 
\item
If $s_{i,\sigma}$ has multiplicity one in some reduced expression of $w_{\fR,\sigma} w_{0,\sigma}$, then the composition $\Hom_{\Fil,\fR,w_0}^i(D_{\sigma}, D_{\sigma}) \hookrightarrow \Hom_{\Fil,\fR}^i(D_{\sigma}, D_{\sigma}) \buildrel (\ref{eq:fisigma}) \over\rightarrow \Hom_{\sigma}(L_{P_i}(\cO_K),E)$ has image equal to $\Hom_{\sigma}(\GL_n(\cO_K),E)$ and kernel equal to $\Ker (\ref{eq:fisigma})$.
\end{enumerate}
\end{prop}
\begin{proof}
By the same argument as in the proof of \cite[Prop.~2.5.3]{BHS19} using \cite[Lemma 2.3.4]{BHS19}, the tangent space $T_{z_{\fR}} X_{w_0,i}^{\sigma}$ of $X_{w_0,i}^{\sigma}$ at $z_{\fR}$, as a subspace of the tangent space
\[T_{z_{\fR}} (G_{\Sigma}/B_{\Sigma} \times G_{\Sigma}/B_{\Sigma} \times \fg_{\sigma})\cong T_{(1B_{\Sigma}, g B_{\Sigma})} (G_{\Sigma}/B_{\Sigma} \times G_{\Sigma}/B_{\Sigma}) \bigoplus \fg_{\sigma},\]
is contained in
\[V:=T_{(1B_{\Sigma}, g B_{\Sigma})} (G_{\Sigma}/B_{\Sigma} \times G_{\Sigma}/B_{\Sigma})\bigoplus \{\psi\in \mathfrak{r}_{P_i,\sigma} \cap \Ad_{g_{\sigma}}(\fb_{\sigma}),\ \overline\psi \in \ft_{\sigma}^{w_{\fR,\sigma}w_{0,\sigma}}\}\]
where $\overline\psi$ is the image of $\psi$ via the composition $ \mathfrak{r}_{P_i,\sigma} \cap \Ad_{g_{\sigma}}(\fb_{\sigma})\hookrightarrow \fb_{\sigma} \cap \Ad_{g_{\sigma}} (\fb_{\sigma}) \hookrightarrow \fb_{\sigma} \twoheadrightarrow \ft_{\sigma}$. Let $b_{\sigma}\in B(E)$ as in the proof of Proposition \ref{prop:Psmooth}, we have
\[V_1:=\{\psi\in \mathfrak{r}_{P_i,\sigma} \cap \Ad_{g_{\sigma}}(\fb_{\sigma}),\ \overline\psi \in \ft_{\sigma}^{w_{\fR,\sigma}w_{0,\sigma}}\} = \Ad_{b_{\sigma}}\Big((\ft_{\sigma}^{w_{\fR,\sigma}w_{0,\sigma}} \cap \fz_{P_i, \sigma}) \bigoplus (\fn_{P_i,\sigma} \cap \Ad_{w_{\fR,\sigma}}(\fn_{\sigma}))\Big).\]

Assume $s_{i,\sigma}$ does not appear in $w_{\fR,\sigma} w_{0,\sigma}$. We then check that
\[\ft_{\sigma}^{w_{\fR,\sigma}w_{0,\sigma}} \cap \fz_{P_i,\sigma}=\fz_{P_i, \sigma}\text{ \ and \ }\fn_{P_i, \sigma} \cap \Ad_{w_{\fR,\sigma}} (\fn_{\sigma})=0\]
(see the proof of Proposition \ref{prop:Psmooth} for the second equality). Hence $\dim_EV_1=2$ and $\dim_E V=n(n-1)[K:\Qp]+2=\dim X_{w_0,i}^{\sigma}$ by Lemma \ref{L: dimsgimai}. We deduce that $X_{w_0,i}^{\sigma}$ is smooth at the point $z_{\fR}$ and $T_{z_{\fR}} X_{w_0,i}^{\sigma} \xrightarrow{\sim} V$. Let $R_{r,\fR,\sigma,i}^{\square,w_0}$ (resp.~$R_{r,\fR,\sigma,i}^{w_0}$) be the local complete noetherian $E$-algebra pro-representing $X_{r,\fR,\sigma,i}^{\square,w_0}$ (resp.~$X_{r,\fR,\sigma,i}^{w_0}$). By (\ref{eq:groupi}) $X_{r,\fR,\sigma,i}^{\square,w_0}$ is obtained from $X_{r,\fR,\sigma,i}^{w_0}$ by adding a framing $\alpha_A$, hence the ring $R_{r,\fR,\sigma,i}^{\square,w_0}$ is a formal power series ring over $R_{r,\fR,\sigma,i}^{w_0}$. Since $R_{r,\fR,\sigma,i}^{\square,w_0}$ is formally smooth by Lemma \ref{Lemma: modelsigmai} and the previous result, it follows that $R_{r,\fR,\sigma,i}^{w_0}$ is also formally smooth. This proves (i) in this case. By definition $\Hom_{\Fil,\fR,w_0}^i(D_{\sigma}, D_{\sigma})$ is the image of $T_{z_{\fR}} X_{w_0,i}^{\sigma}$ via (\ref{Etangentmodel}), hence coincides with the image of $V_1$ via (\ref{eq:middle}), hence has dimension $2$ by the above results. By Proposition \ref{prop:Psmooth} with $\dim_E\Hom_{\sigma}(L_{P_i}(\cO_K), E)=2$, (ii) follows.\bigskip
	
Assume now $s_{i,\sigma}$ appears in $w_{\fR,\sigma} w_{0,\sigma}$ with multiplicity $1$. Then we easily check that $\ft_{\sigma}^{w_{\fR,\sigma}w_{0,\sigma}} \cap \fz_{P_i,\sigma}=\fz_{\sigma}$ (of dimension $1$). Together with $\dim_E (\fn_{P_i, \sigma} \cap \Ad_{w_{\fR,\sigma}} (\fn_{\sigma}))=1$, we still have $\dim_E V_1=2$ and by the same arguments as above, (i) follows. Moreover, as the composition of (\ref{mult:missing}) and (\ref{eq:fisigma}) coincides with the natural isomorphism $\mathfrak{r}_{P_i,\sigma}\xrightarrow{\sim} \Hom_{\sigma}(L_{P_i}(\cO_K),E)$, we see that the image of the composition in (iii), which is the image of $V_1$ by (\ref{mult:missing}), is the $1$-dimensional subspace $\Hom_{\sigma}(\GL_n(\cO_K),E)$. As the kernel of (\ref{eq:fisigma}) in this case is $1$-dimensional by (ii) of Remark \ref{rem:Rsmooth}, comparing dimensions we must have $\Ker (\ref{eq:fisigma}) \subset \Hom_{\Fil,\fR,w_0}^i(D_{\sigma}, D_{\sigma})$ and $\Hom_{\Fil,\fR,w_0}^i(D_{\sigma}, D_{\sigma}) \twoheadrightarrow \Hom_{\sigma}(\GL_n(\cO_K),E)$. This finishes~the~proof. 
\end{proof}

We denote by $\Ext^1_{\sigma,0}(\cM(D), \cM(D))$ the subspace of $\Ext^1_{\sigma}(\cM(D), \cM(D))$ (see above Corollary \ref{cor:splitsigma}) of $\widetilde{\cM}$ which are $\sigma$-de Rham up to twist by (the $(\varphi,\Gamma)$-module over $\cR_{K,E[\epsilon]/\epsilon^2}$ associated to) a locally $\sigma$-analytic character of $K^\times$ over $E[\epsilon]/\epsilon^2$ (equivalently $\widetilde{\cM}$ is de Rham up to twist by such a locally $\sigma$-analytic character). We denote by $\Hom^0_{\Fil}(D_{\sigma}, D_{\sigma})$ the $1$-dimensional subspace of $\Hom_{\Fil}(D_{\sigma}, D_{\sigma})$ spanned by the identity map. Obviously $\Hom^0_{\Fil}(D_{\sigma}, D_{\sigma})\subset \Hom_{\Fil,\fR}^i(D_{\sigma}, D_{\sigma})$ (for all $i\in \{1,\dots, n-1\}$) and $\Ext^1_{\sigma,0}(\cM(D), \cM(D))$ is the preimage of $\Hom^0_{\Fil}(D_{\sigma}, D_{\sigma})$ via (\ref{eq:Enil}). By similar argument as in the proof of Lemma \ref{lem:Lzw0} with $Z_{w_0}$, $X_{w_0}$, (\ref{diag: R}) replaced by $(G_{\sigma}/B_{\sigma} \times G_{\sigma}/B_{\sigma} \times \fz_{\sigma}) \times \prod_{\tau \neq \sigma}(G_{\tau}/B_{\tau} \times G_{\tau}/B_{\tau} \times \{0\})$, $X_{w_0,i,}^{\sigma}$, (\ref{diag: model2}), we have $\Hom^0_{\Fil}(D_{\sigma}, D_{\sigma})\subset \Hom^i_{\Fil,\fR,w_0}(D_{\sigma}, D_{\sigma})$ and
\begin{equation*}
\Ext^1_{\sigma,0}(\cM(D), \cM(D)) \subset \Ext^1_{\fR, w_0, \sigma, i}(\cM(D), \cM(D)).
\end{equation*}
We define
\begin{equation}\label{eqt0sigma}
\Hom_{\sigma,0}(T(K),E):=\Hom_{\sm}(T(K),E) \oplus_{\Hom_{\sm}(\GL_n(K),E)} \Hom_{\sigma}(\GL_n(K),E)
\end{equation}
and note that $\Hom_{\sigma,0}(T(K),E)\subset \Hom_{\sigma,i}(T(K),E)$ for $i\in \{1,\dots,n-1\}$ (see (\ref{eqtisigma})).
Moreover, as in (\ref{eq:missingiso}), we have a splitting (depending on a choice of $\log(p)$)
\begin{equation*}
\ol{\Ext}^1_{\sigma,0}(\cM(D),\cM(D)) \buildrel\sim\over\longrightarrow \Ext^1_{\varphi^f}(D_{\sigma},D_{\sigma}) \bigoplus \Hom^0_{\Fil}(D_{\sigma}, D_{\sigma})
\end{equation*}
from which it follows (with (\ref{eq:extphi}) and $\Hom^0_{\Fil}(D_{\sigma}, D_{\sigma})\buildrel\sim\over\rightarrow \Hom_{\sigma}(\GL_n(\cO_K),E)$) that the map (\ref{eq:Einvt}) composed with (\ref{eq:Etri2}) induces an isomorphism
\begin{equation}\label{Eisom13}
\ol{\Ext}^1_{\sigma,0}(\cM(D), \cM(D)) \xlongrightarrow{\sim} \Hom_{\sigma,0}(T(K),E).
\end{equation}

\begin{cor}
Assume that the multiplicity of $s_{i,\sigma}$ in some reduced expression of $w_{\fR,\sigma}w_{0,\sigma}$ is at most one.
\begin{enumerate}[label=(\roman*)]
\item
If $s_{i,\sigma}$ does not appear in some (equivalently any) reduced expression of $w_{\fR,\sigma} w_{0,\sigma}$, the map (\ref{eq:Einvt}) composed with (\ref{eq:Etri2}) induces an isomorphism
\begin{equation}\label{Eisom11}
\ol{\Ext}^1_{\fR,w_0,\sigma,i}(\cM(D), \cM(D)) \xlongrightarrow{\sim} \Hom_{\sigma,i}(T(K),E).
\end{equation}
\item
If $s_{i,\sigma}$ appears with multiplicity one in some reduced expression of $w_{\fR,\sigma} w_{0,\sigma}$, the map (\ref{eq:Einvt}) composed with (\ref{eq:Etri2}) induces a surjection with kernel isomorphic to $\Ker(\ref{mult:Ecomp})$: 
\begin{equation}\label{Esurj12}
\ol{\Ext}^1_{\fR,w_0,\sigma,i}(\cM(D), \cM(D)) \twoheadlongrightarrow \Hom_{\sigma,0}(T(K),E).
\end{equation}
Moreover, we have a splitting (only depending on the refinement $\fR$)
\begin{equation}\label{Eisom12} 
\ol{\Ext}^1_{\sigma,0}(\cM(D), \cM(D)) \oplus \Ker(\ref{Esurj12})\buildrel\sim\over\longrightarrow \ol{\Ext}^1_{\fR, w_0, \sigma, i}(\cM(D), \cM(D)).
\end{equation}
\end{enumerate}
\end{cor}
\begin{proof}
Part (i) follows from (\ref{eq:missingiso}) with (\ref{eq:extphi}), (ii) of Proposition \ref{prop:Psmoothbis} and Proposition \ref{prop:Psmooth} (using (\ref{eqtisigma})). The first part of (ii) follows from (\ref{eq:missingiso}) with (\ref{eq:extphi}), (iii) of Proposition \ref{prop:Psmoothbis} and noting that, by the second half of (\ref{diag: model2}), one can identify $\Ker(\ref{Esurj12})$ with $\Ker(\ref{eq:fisigma})$, which is also $\Ker(\ref{mult:Ecomp})$. Finally (\ref{Eisom13}) gives the splitting (\ref{Eisom12}).
\end{proof}

Recall that $\pi$ is an automorphic representation as in Conjecture \ref{conj:main}. From now on we set $r:=\rho_{\pi,\widetilde \wp}$. From the definitions in \S~\ref{sec:global} the action of $R_{\overline{\rho}, \cS}(\xi,\tau)[1/p]$ on $\widehat{S}_{\xi,\tau}(U^{\wp},E)[\fm_{\pi}]$ factors through a quotient map that we still denote $\omega_\pi:R_{\overline{\rho}, \cS}(\xi,\tau)[1/p]\twoheadlongrightarrow E$. We let $\fm_{\pi}^{\wp}\subset R_{\infty}^{\wp}(\xi,\tau)[1/p]$ (resp.~$\fm_{\pi,\wp}\subset R_{\overline{r}}[1/p]$) be the maximal ideal defined as the kernel of the composition
\[R_{\infty}^{\wp}(\xi,\tau)[1/p]\hookrightarrow R_{\infty}(\xi,\tau)[1/p] \twoheadlongrightarrow R_{\overline{\rho}, \cS}(\xi,\tau)[1/p]\buildrel \omega_\pi\over \twoheadlongrightarrow E\] 
(resp.~$R_{\overline{r}}[1/p]\hookrightarrow R_{\infty}(\xi,\tau)[1/p]\twoheadlongrightarrow \cdots \buildrel \omega_\pi\over\twoheadlongrightarrow E$). Applying Emerton's Jacquet functor $J_B$ to (\ref{Elalg}), via (\ref{Epatchedauto}) and (\ref{Einjtri}) we obtain a point for all refinements $\fR$ of $D=D_{\cris}(r)$, $\sigma\in \Sigma$ and $i\in \{1, \dots, n-1\}$:
\begin{equation}\label{EptxR}
x_{\fR}:=\big(\fm_{\pi}^{\wp}, \fm_{\pi,\wp}, \delta_{\fR}\big)\in \cE_{\infty}(\xi,\tau)_{\sigma,i}
\end{equation}
where $\delta_{\fR}:= (\unr(\underline{\varphi})t^{\textbf{h}} \delta_B(\boxtimes_{j=0}^{n-1} \varepsilon^{j}))\in \widehat{T}$ (see (\ref{eq:dominant})). Indeed, either by definition or using \cite[Prop.~5.5]{Wu24}, it is easy to see that $J_B$ applied to the left hand side of (\ref{Elalg}) is contained in the space $V_{\sigma,i}$ of (\ref{Eparabolic}). Note that the image of $x_{\fR}$ in $X_{\tri}(\overline{r})$ via (\ref{Einjtri}) (and $\iota_p$) is the point $y_{\fR}$ of (\ref{EptyR}). The following corollary is the main result of that section.

\begin{cor}\label{C: smooth}
Assume that the multiplicity of $s_{i,\sigma}$ in some reduced expression of $w_{\fR,\sigma}w_{0,\sigma}$ is at most one. The rigid analytic space $\cE_{\infty}(\xi,\tau)_{\sigma,i}$ is smooth at the point $x_{\fR}$ and we have
\begin{equation}\label{diag: Esigmai}
\begin{tikzcd}
T_{x_{\fR}} \cE_{\infty}(\xi,\tau)_{\sigma,i} \arrow[r, two heads] &\Ext^1_{\fR,w_0,\sigma,i}(\cM(D), \cM(D)) \arrow[r, "\ref{eq:Einvt}+\ref{eq:Etri2}"] &\Hom_{\sigma,i}(T(K),E)
\end{tikzcd}
\end{equation}
where the first top horizontal map is surjective and is induced by the tangent space at $x_{\fR}$ of the composition 
\[\cE_{\infty}(\xi,\tau)\buildrel (\ref{Einjtri}) \over \hooklongrightarrow (\Spf R_{\infty}^{\wp}(\xi,\tau))^{\rig} \times \iota^{-1}_p(X_{\tri}(\overline{r}))\twoheadlongrightarrow (\Spf R_{\overline{r}})^{\rig},\]
and where the composition in (\ref{diag: Esigmai}) is induced by the tangent map at $x_{\fR}$ of the composition
\[\cE_{\infty}(\xi,\tau)\buildrel (\ref{Einjtri}) \over \hooklongrightarrow (\Spf R_{\infty}^{\wp}(\xi,\tau))^{\rig} \times \iota^{-1}_p(X_{\tri}(\overline{r}))\twoheadlongrightarrow \widehat{T}.\]
\end{cor}
\begin{proof}
By (\ref{Einjtri}) and (\ref{modeltri}), we have natural injections of finite dimensional $E$-vector spaces (identifying the maximal ideal $\fm_{\pi}^{\wp}$ with the corresponding point on $(\Spf R_{\infty}^{\wp}(\xi,\tau))^{\rig}$)
\begin{equation}\label{eq: tan1}
T_{x_{\fR}}\cE_{\infty}(\xi,\tau)_{\sigma,i} \hooklongrightarrow T_{x_{\fR}}\cE_{\infty}(\xi,\tau) \hooklongrightarrow T_{\fm_{\pi}^{\wp}}(\Spf R_{\infty}^{\wp}(\xi,\tau))^{\rig} \bigoplus X_{r,\fR}^{w_0}(E[\epsilon]/(\epsilon^2).
\end{equation}
By the same argument as in the proof of \cite[Prop.~5.13]{Wu24} but applying \cite[Prop.~A.10]{Wu24} to the tangent space of $\cE_{\infty}(\xi,\tau)_{\sigma,i}$ at $x_{\fR}$ instead of just to points, the composition (\ref{eq: tan1}) has image in 
\begin{equation*}
T_{\fm_{\pi}^{\wp}}(\Spf R_{\infty}^{\wp}(\xi,\tau))^{\rig} \bigoplus X_{r,\fR,\sigma,i}^{w_0}(E[\epsilon]/(\epsilon^2).
\end{equation*}
The second part of the statement, except the surjectivity, then follows from (\ref{diag: Rw0}) and the discussion below it. It follows from \cite[Lemma 1.3.2(1)]{BLGHT11} (see for instance the argument in \cite[p.~1633]{BHS171}) and from \cite[Thm.~3.3.8]{Ki08} that the maximal ideal $\fm_{\pi}^{\wp}$ defines a smooth point on $(\Spf R_{\infty}^{\wp}(\xi,\tau))^{\rig}$. Hence the tangent space of $(\Spf R_{\infty}^{\wp}(\xi,\tau))^{\rig}$ at this point has dimension $\dim (\Spf R_{\infty}^{\wp}(\xi,\tau))^{\rig}$. From (\ref{eq: tan1}), (\ref{dimker}) and $\dim_E \Hom_{\Fil,\fR,w_0}^i(D_{\sigma}, D_{\sigma}) =2$ (see the first statement in (i) of Proposition \ref{prop:Psmoothbis}) we deduce
{\scriptsize
\begin{eqnarray*}
\dim_E T_{x_{\fR}}\cE_{\infty}(\xi,\tau)_{\sigma,i} &\leq &\dim (\Spf R_{\infty}^{\wp}(\xi,\tau))^{\rig}+\Big(n^2+\frac{n(n-1)}{2}[K:\Qp]\Big)+2\\
&=&\big(g+(|S|-1) n^2+\Big([F^+:\mathbb{Q}]-[K:\Qp])\frac{n(n-1)}{2}\Big)+\Big(n^2+\frac{n(n-1)}{2}[K:\Qp]\Big)+2\\
&=&\dim \cE_{\infty}(\xi,\tau)_{\sigma,i}
\end{eqnarray*}}
\!\!where the first equality is a standard formula for $\dim (\Spf R^{\loc}_{\xi,\tau})^{\rig}$ and the second equality follows from (i) of Proposition \ref{P: Esigmai}. The first part of the statement follows. Finally the smoothness and the above calculation imply that (\ref{eq: tan1}) induces an isomorphism
\begin{equation}\label{eq:inclWu}
T_{x_{\fR}}\cE_{\infty}(\xi,\tau)_{\sigma,i} \xlongrightarrow{\sim}	T_{\fm_{\pi}^{\wp}}(\Spf R_{\infty}^{\wp}(\xi,\tau))^{\rig} \bigoplus X_{r,\fR,\sigma,i}^{w_0}(E[\epsilon]/(\epsilon^2).
\end{equation}
Together with (\ref{diag: model2}) this implies the surjectivity of the first map in (\ref{diag: Esigmai}).
\end{proof}

\begin{rem}
The isomorphism (\ref{eq:inclWu}) can be upgraded to an isomorphism of groupoids over $\cC_E$ between the (groupoids pro-represented by the) completed local ring of $\cE_{\infty}(\xi,\tau)_{\sigma,i}$ at $x_{\fR}$ and the completed local ring of $(\Spf R_{\infty}^{\wp}(\xi,\tau))^{\rig}$ at $\fm_{\pi}^{\wp}$ times $\Spf R_{r,\fR,\sigma,i}^{w_0}$. It is likely that this isomorphism (but not the smoothness at $x_{\fR}$) holds without assumptions on $w_{\fR,\sigma}w_{0,\sigma}$.
\end{rem}

\subsection{Universal extensions} 

Using the maps $(t_{D_{\sigma}})_{\sigma\in \Sigma}$ of Proposition \ref{prop:map} (or of Theorem \ref{thm:independant}), we equip the universal extension of $\pi_{\alg}(D)$ by $\pi_{\flat}(D)$ (see (\ref{eq:alg}) and (\ref{eq:pi(D)_flat})) with an action of a variant of the local Galois deformation ring. We then study this action in detail. In the next section, as a crucial step in our proof of local-global compatibility, we will show that this universal extension embeds into the representation $\Pi_{\infty}(\xi,\tau)$ of (\ref{eq:piinfini}).\bigskip

We keep the notation of the previous sections. Similarly to (\ref{eq:pi(D)_flat}) we define
\begin{equation}\label{eq:pi(D)_R}
\pi_R(D):=\bigoplus_{\sigma\in \Sigma, \ \!\pi_{\alg}(D)} \big(\pi_{R}(D_{\sigma}) \otimes_E (\otimes_{\tau\neq \sigma} L(\lambda_{\tau}))\big)
\end{equation}
(which contains $\pi_\flat(D)$) and similarly to (\ref{eq:amalg4}) (for $S=R$) with (\ref{eq:amalgi}) we have 
\[\pi_R(D)\cong \bigoplus_{\sigma\in \Sigma, I\subset\{\varphi_0, \dots, \varphi_{n-1}\}, \pi_{\alg}(D)} \!\!\!\pi_{\sigma,I}(D)\]
where $\pi_{\sigma,I}(D):=\pi_I(D_{\sigma}) \otimes_E (\otimes_{\tau\neq \sigma} L(\lambda_{\tau}))$ (strictly speaking, to define $\pi_I(D_{\sigma})$ we tacitly choose isomorphisms as in (\ref{eq:epsilonI}) for each $\sigma,I$). Recall that $\Ext^1_{\alg}(\pi_{\alg}(D),\pi_{\alg}(D))\subset \Ext^1_{\GL_n(K)}(\pi_{\alg}(D),\pi_{\alg}(D))$ \ is \ the \ subspace \ of \ locally \ $\Qp$-algebraic \ extensions. We \ denote by $\Ext^1_0(\pi_{\alg}(D), \pi_*(D))$ for $*\in \{\flat, R\}$ its image via the natural injective push-forward map
\[\Ext^1_{\GL_n(K)}(\pi_{\alg}(D), \pi_{\alg}(D)) \hooklongrightarrow \Ext^1_{\GL_n(K)}(\pi_{\alg}(D), \pi_*(D))\]
(recall injectivity comes from $\Hom_{\GL_n(K)}(\pi_{\alg}(D),\pi_*(D)/\pi_{\alg}(D))=0$), so that we have
\begin{equation}\label{isotrivial}
\Ext^1_{\alg}(\pi_{\alg}(D),\pi_{\alg}(D))\buildrel\sim\over\longrightarrow \Ext^1_0(\pi_{\alg}(D), \pi_*(D)).
\end{equation}
For $\sigma \in \Sigma$ and $*\in \{\flat, R\}$ we denote by $\Ext^1_{\sigma}(\pi_{\alg}(D), \pi_*(D))$ the image of the composition
\begin{multline}\label{eq: twtau}
\Ext^1_{\GL_n(K), \sigma}\big(\pi_{\alg}(D_{\sigma}), \pi_*(D_{\sigma})\big) \longrightarrow \Ext^1_{\GL_n(K)}\big(\pi_{\alg}(D), \pi_*(D_{\sigma}) \otimes_E (\otimes_{\tau\neq \sigma} L(\lambda_{\tau}))\big)\\
\hooklongrightarrow \Ext^1_{\GL_n(K)}\big(\pi_{\alg}(D), \pi_*(D)\big)
\end{multline}
where the first map sends $V$ to $V \otimes_E (\otimes_{\tau \neq \sigma} L(\lambda_{\tau}))$ and the second is the natural (injective) \ push-forward \ map. \ We \ denote \ by \ $\Ext^1_{\sigma,0}(\pi_{\alg}(D), \pi_*(D))$ \ the \ image \ of $\Ext^1_{\GL_n(K), \sigma}(\pi_{\alg}(D_{\sigma}), \pi_{\alg}(D_{\sigma}))$ via the composition (\ref{eq: twtau}). Tensoring with $\otimes_{\tau\neq \sigma} L(\lambda_\tau)$ also induces an isomorphism (using Lemma \ref{lem:isoextalg} with \cite[Prop.~3.3(1)]{Di25})
\begin{equation}\label{eq: twiso}
\Ext^1_{*}(\pi_{\alg}(D_{\sigma}), \pi_{\alg}(D_{\sigma})) \xlongrightarrow{\sim} \Ext^1_{*}(\pi_{\alg}(D), \pi_{\alg}(D)), \ \ *\in \{\alg, \GL_n(K)\}
\end{equation}
and we denote by $\Ext^1_{\sigma}(\pi_{\alg}(D), \pi_{\alg}(D))$ the image of $\Ext^1_{\GL_n(K),\sigma}(\pi_{\alg}(D_{\sigma}), \pi_{\alg}(D_{\sigma}))$ when $*=\GL_n(K)$. In particular we have inclusions for $\sigma\in \Sigma$ and $*\in \{\flat, R\}$
\begin{equation}\label{eq:incl0}
\Ext^1_0(\pi_{\alg}(D), \pi_*(D))\subset \Ext^1_{\sigma,0}(\pi_{\alg}(D), \pi_*(D)) \subset \Ext^1_{\sigma}(\pi_{\alg}(D), \pi_*(D)).
\end{equation}

\begin{lem}\label{lem: tD}
Let $*\in \{\flat, R\}$.
\begin{enumerate}[label=(\roman*)]
\item
We have a canonical isomorphism
\begin{equation*}
\bigoplus_{\sigma\in \Sigma, \ \!\Ext^1_{0}(\pi_{\alg}(D), \pi_*(D))} \Ext^1_{\sigma}(\pi_{\alg}(D), \pi_*(D))\xlongrightarrow{\sim} \Ext^1_{\GL_n(K)}(\pi_{\alg}(D), \pi_*(D))
\end{equation*}
where the amalgamated sum is via (\ref{eq:incl0}).
\item
The maps (\ref{eq: tDsgima'}) for all $\sigma \in \Sigma$ induce a surjection via (i)
\begin{equation*}
t_{D,*}: \Ext^1_{\GL_n(K)}(\pi_{\alg}(D), \pi_*(D)) \twoheadlongrightarrow \ol{\Ext}^1_{(\varphi, \Gamma)}(\cM(D), \cM(D)).
\end{equation*}
\end{enumerate}
\end{lem}
\begin{proof}
We only prove the lemma for $\pi_{\flat}(D)$ (the case $\pi_R(D)$ being similar).\bigskip

(i) Let $\sigma\in \Sigma$, by d\'evissage the composition (\ref{eq: twtau}) sits in a commutative diagram of exact sequences (writing $\Ext^1_{\sigma}$, $\Ext^1$ for $\Ext^1_{\GL_n(K),\sigma}$, $\Ext^1_{\GL_n(K)}$)
\begin{equation}\label{eq: amalsigma}
\adjustbox{scale=0.84}{
\begin{tikzcd}
\Ext^1_{\sigma}(\pi_{\alg}(D_{\sigma}), \pi_{\alg}(D_{\sigma})) \arrow[r, hook] \arrow[d] & \Ext^1_{\sigma}(\pi_{\alg}(D_{\sigma}), \pi_{\flat}(D_{\sigma})) \arrow[r, two heads] \arrow[d] & \bigoplus_{I}\Ext^1_{\sigma}(\pi_{\alg}(D_{\sigma}), C(I,s_{i,\sigma})) \arrow[d]\\
\Ext^1(\pi_{\alg}(D), \pi_{\alg}(D)) \arrow[r, hook] & \Ext^1(\pi_{\alg}(D), \pi_{\flat}(D))\arrow[r]& \bigoplus_{\tau\in \Sigma} \bigoplus_J \Ext^1(\pi_{\alg}(D), \widetilde{C}(J,s_{j, \tau}))
\end{tikzcd}}
\end{equation}
where the vertical maps are all induced by tensoring with $\otimes_{\tau\neq \sigma} L(\lambda_{\tau})$, where $\widetilde{C}(J,s_{j,\tau}):=C(J,s_{j,\tau})\otimes_E (\otimes_{\tau'\neq \tau} L(\lambda_{\tau'}))$ and where $I$ (resp.~$J$) runs through the subsets which are \emph{not} very critical for $\sigma$ (resp.~for $\tau$), see Definition \ref{def:critical}. Note that the surjectivity of the top second horizontal map follows for instance from (\ref{mult:fix}) (with the second isomorphism in (\ref{eq:amalg})). The left vertical map is injective by (\ref{eq:amalg}) and (\ref{eq: twiso}). It follows from Lemma \ref{lem:nonsplit} and \cite[Lemma 3.5(1)]{Di25} that we have an isomorphism (for any $I$)
\begin{equation}\label{eq: twi3}
\Ext^1_{\GL_n(K),\sigma}(\pi_{\alg}(D_{\sigma}), C(I,s_{i,\sigma})) \xlongrightarrow{\sim} \Ext^1_{\GL_n(K)}\big(\pi_{\alg}(D), \widetilde{C}(I,s_{i,\sigma})\big),
\end{equation}
and thus the right vertical map is also injective. By a trivial diagram chase the middle vertical map in (\ref{eq: amalsigma}) is again injective, hence (\ref{eq: twtau}) induces an isomorphism
\begin{equation}\label{eq:isoflat}
\Ext^1_{\GL_n(K),\sigma}(\pi_{\alg}(D_{\sigma}), \pi_{\flat}(D_{\sigma})) \buildrel\sim\over\longrightarrow \Ext^1_{\sigma}(\pi_{\alg}(D), \pi_{\flat}(D)).
\end{equation}
Using \cite[Prop.~3.3(1)]{Di25} with (\ref{eq:amalg}), (\ref{eq: twiso}) (and the discussion below it) we have an isomorphism
\[\Ext^1_{\GL_n(K)}(\pi_{\alg}(D), \pi_{\alg}(D))\buildrel\sim\over\longleftarrow \!\!\!\!\bigoplus_{\sigma\in \Sigma, \ \!\Ext^1_{\alg}(\pi_{\alg}(D), \pi_{\alg}(D)) } \!\!\!\Ext^1_{\sigma}(\pi_{\alg}(D), \pi_{\alg}(D)).\]
Taking the direct sum over $\sigma\in \Sigma$ of the top exact sequence in (\ref{eq: amalsigma}), it then follows from the above isomorphism with (\ref{isotrivial}) and (\ref{eq: twi3}) (and an obvious diagram chase) that the canonical map in (i) is an isomorphism (and the bottom second map in (\ref{eq: amalsigma}) is then also surjective).\bigskip

(ii) By Proposition \ref{prop:mul2}, the surjection (\ref{eq: tDsgima'}) remains surjective when $\pi_R(D_{\sigma})$ is replaced by $\pi_{\flat}(D_{\sigma})$. Moreover by (\ref{eq:isog}) and an examination of Step 2 in the proof of Proposition \ref{prop:map} the following composition does not depend on $\sigma\in \Sigma$:
\begin{multline*}
\Ext^1_0(\pi_{\alg}(D), \pi_{\flat}(D)) \cong \Ext^1_{\alg}(\pi_{\alg}(D), \pi_{\alg}(D)) \xlongrightarrow{\sim} \Ext^1_{\alg}(\pi_{\alg}(D_{\sigma}), \pi_{\alg}(D_{\sigma}))\\
\hooklongrightarrow \Ext^1_{\GL_n(K),\sigma}(\pi_{\alg}(D_{\sigma}), \pi_{\flat}(D_{\sigma})) \xlongrightarrow{(\ref{eq: tDsgima'})} \ol{\Ext}^1_{(\varphi, \Gamma)}(\cM(D), \cM(D)).
\end{multline*}
The statement then follows from (i) with Proposition \ref{prop:Psplit} and Corollary \ref{cor:splitsigma}.
\end{proof}

By Theorem \ref{thm:independant}, the map (\ref{eq: tDsgima'}) is unique but only up to isomorphism. In the sequel, we \textit{fix} a choice of (\ref{eq: tDsgima'}) for each $\sigma\in \Sigma$, which determines maps $t_{D,R}$ and $t_{D,\flat}$ by (ii) of Lemma \ref{lem: tD}. Note that the representation $\pi(D)^{\flat}$ in (\ref{eq:pi(D)flat}) (resp.~$\pi(D)$ in (\ref{eq:pi(D)})) is no other than the tautological extension of $\pi_{\alg}(D)\otimes_E \Ker (t_{D,\flat})$ (resp.~of $\pi_{\alg}(D) \otimes_E \Ker(t_{D,R})$) by $\pi_{\flat}(D)$ in (\ref{eq:pi(D)_flat}) (resp.~by $\pi_R(D)$ in (\ref{eq:pi(D)_R})) defined similarly as in Definition \ref{def:pi(d)}.\bigskip

Recall from \S~\ref{sec:model} that $R_r$ is the noetherian local complete $E$-algebra pro-representing framed deformations of $r$ over artinian $E$-algebras and let $\fm_{R_r}$ be its maximal ideal. Consider the natural composition
\begin{equation*}
(\fm_{R_r}/\fm_{R_r}^2)^{\vee} \twoheadlongrightarrow \Ext^1_{\Gal(\overline K/K)}(r,r)\buildrel\sim\over\longrightarrow \Ext^1_{(\varphi, \Gamma)}(\cM(D), \cM(D)) \twoheadlongrightarrow \ol{\Ext}^1_{(\varphi, \Gamma)}(\cM(D), \cM(D)).
\end{equation*}
There exists a (unique) local Artinian $E$-subalgebra $A_D$ of $R_r/\fm_{R_r}^2$ of maximal ideal $\fm_{A_D}$ such that 
\begin{equation}\label{Esurj1}
(\fm_{R_r}/\fm_{R_r}^2)^{\vee}\twoheadlongrightarrow (\fm_{A_D})^{\vee}\cong \ol{\Ext}^1_{(\varphi, \Gamma)}(\cM(D), \cM(D)).
\end{equation}
We \ \ denote \ \ by \ \ $\widetilde{\pi}_\flat(D)$ \ \ (resp.~\ $\widetilde{\pi}_R(D)$) \ \ the \ \ tautological \ \ extension \ \ of \ $\pi_{\alg}(D) \otimes_E \Ext^1_{\GL_n(K)}(\pi_{\alg}(D), \pi_{\flat}(D))$ (resp.~of $\pi_{\alg}(D) \otimes_E \Ext^1_{\GL_n(K)}(\pi_{\alg}(D), \pi_R(D))$) by $\pi_{\flat}(D)$ (resp.~by $\pi_R(D)$) defined as in (the first map of) Definition \ref{def:pi(d)} or as
above Lemma \ref{lem:max} (it is also sometimes called the universal extension). As in the proof of \emph{loc.~cit.}~we have an injection $\widetilde{\pi}_\flat(D)\hookrightarrow \widetilde{\pi}_R(D)$. We define a $\GL_n(K)$-equivariant left action of $A_D$ on $\widetilde{\pi}_R(D)$ as follows: $x\in \fm_{A_D}\cong \ol{\Ext}^1_{(\varphi, \Gamma)}(\cM(D), \cM(D))^{\vee}$ acts on $\widetilde{\pi}_R(D)$ via
\begin{multline}\label{E: action}
\widetilde{\pi}_R(D) \twoheadlongrightarrow \pi_{\alg}(D) \otimes_E \Ext^1_{\GL_n(K)}(\pi_{\alg}(D), \pi_{R}(D))\\
 \xlongrightarrow{\id \otimes t_{D,R}} \pi_{\alg}(D) \otimes_E \ol{\Ext}^1_{(\varphi, \Gamma)}(\cM(D), \cM(D))\\ 
\xlongrightarrow{\id \otimes x} \pi_{\alg}(D) \hooklongrightarrow \pi_R(D) \hooklongrightarrow\widetilde{\pi}_R(D).
\end{multline}
The subrepresentation $\widetilde{\pi}_\flat(D)$ is preserved by $A_D$ since the $A_D$-action on it can be described as in (\ref{E: action}) with $t_{D,R}$ replaced by $t_{D,\flat}$. It is then formal to check that the subrepresentation $\widetilde{\pi}_R(D)[\fm_{A_D}]$ of elements cancelled by $\fm_{A_D}$ is exactly the subrepresentation $\pi(D)$. Likewise we have $\widetilde{\pi}_\flat(D)[\fm_{A_D}]\cong \pi(D)^{\flat}$.\bigskip

For $\sigma\in \Sigma$ and $I\subset \{\varphi_0, \dots, \varphi_{n-1}\}$ of cardinality $i\in \{1,\dots,n-1\}$, we denote by $W_{\sigma,I}$ the unique (up to isomorphism) non-split extension of $\pi_{\alg}(D)$ by $\widetilde{C}(I,s_{i,\sigma})$ (\cite[Lemma 3.5(1)]{Di25}). We let $A_D$ act on $W_{\sigma,I}$ via $A_D\twoheadrightarrow A_D/\fm_{A_D}\cong E$ (in particular $\fm_{A_D}$ cancels $W_{\sigma,I}$). Using Proposition \ref{prop:mul2}, we easily check that we have a $\GL_n(K) \times A_D$-equivariant isomorphism (see Definition \ref{def:critical} for $I$ very critical)
\begin{equation*}
\widetilde{\pi}_\flat(D) \ \ \bigoplus \ \ \bigg(\bigoplus_{\sigma \in \Sigma, \ \!\text{$I$ \!very \!critical \!for \!$\sigma$}}W_{\sigma,I}\bigg) \xlongrightarrow{\sim} \widetilde{\pi}_R(D)
\end{equation*}
which induces a $\GL_n(K)$-equivariant isomorphism (compare with (i) of Proposition \ref{prop:flat})
\begin{equation}\label{eq:pi(D)W}
{\pi}(D)^\flat \ \ \bigoplus \ \ \bigg(\bigoplus_{\sigma \in \Sigma, \ \!\text{$I$ \!very \!critical \!for \!$\sigma$}}W_{\sigma,I}\bigg) \xlongrightarrow{\sim} {\pi}(D).
\end{equation}

We now decompose $\widetilde{\pi}_\flat(D)$ into $A_D$-equivariant subrepresentations (see (\ref{eq:decopflat}) below). Similarly as in the discussion before Lemma \ref{lem: tD}, for $\sigma\in \Sigma$ and $I\subset \{\varphi_0, \dots, \varphi_{n-1}\}$ (of cardinality $\in \{1,\dots,n-1\}$) we denote by $\Ext^1_{\sigma}(\pi_{\alg}(D), \pi_{\sigma,I}(D))$ the image of
\[\Ext^1_{\GL_n(K),\sigma}(\pi_{\alg}(D_{\sigma}), \pi_I(D_{\sigma}))\longrightarrow \Ext^1_{\GL_n(K)}(\pi_{\alg}(D), \pi_{\sigma,I}(D)), \ \ V \longmapsto V\otimes_E (\otimes_{\tau\neq \sigma} L(\lambda_{\tau})).\]
For $\sigma\in \Sigma$ we let $\widetilde{\pi}_{\alg,\sigma}(D)$ be the tautological extension of $\pi_{\alg}(D) \otimes_E \Ext^1_{\sigma}(\pi_{\alg}(D), \pi_{\alg}(D))$ by $\pi_{\alg}(D)$. Following the notation above Lemma \ref{lem:max}, we could also write $\widetilde{\pi}_{\sigma,\emptyset}(D)$ but the former notation is better in the present context. Likewise we let $\widetilde{\pi}_{\alg}(D)$ be the tautological extension of $\pi_{\alg}(D) \otimes_E \Ext^1_{\alg}(\pi_{\alg}(D), \pi_{\alg}(D))$ by $\pi_{\alg}(D)$. Fix $I\subset \{\varphi_0, \dots, \varphi_{n-1}\}$ such that $I$ is \emph{not} very critical for $\sigma$ and denote by $\widetilde{\pi}_{\sigma,I}(D)$ the tautological extension of $\pi_{\alg}(D)\otimes_E \Ext^1_{\sigma}(\pi_{\alg}(D), \pi_{\sigma,I}(D))$ by $\pi_{\sigma,I}(D)$. Using the injections
\begin{multline*}
\Ext^1_{\alg}(\pi_{\alg}(D), \pi_{\alg}(D)) \hooklongrightarrow \Ext^1_{\sigma}(\pi_{\alg}(D), \pi_{\alg}(D))\hooklongrightarrow \Ext^1_{\sigma}(\pi_{\alg}(D), \pi_{\sigma,I}(D)) \\
\hooklongrightarrow \Ext^1_{\sigma}(\pi_{\alg}(D), \pi_{\flat}(D))
\end{multline*}
and arguing as in the proof of Lemma \ref{lem:max} we have natural $\GL_n(K)$-equivariant injections
\begin{equation}\label{Einjuniv}
\widetilde{\pi}_{\alg}(D) \hooklongrightarrow \widetilde{\pi}_{\alg,\sigma}(D) \hooklongrightarrow \widetilde{\pi}_{\sigma,I}(D) \hooklongrightarrow \widetilde{\pi}_\flat(D).
\end{equation}
Note also that (\ref{eq:isoflat}) induces by restriction an isomorphism
\begin{equation}\label{eq:isoI}
\Ext^1_{\GL_n(K),\sigma}(\pi_{\alg}(D_{\sigma}), \pi_I(D_{\sigma}))\buildrel\sim\over\longrightarrow \Ext^1_{\sigma}(\pi_{\alg}(D), \pi_{\sigma,I}(D)).
\end{equation}
Let $i:=\vert I\vert$ and $\fR$ a refinement compatible with $I$ (Definition \ref{def:compatible}). Recall that, since $I$ is not very critical for $\sigma$, the multiplicity of $s_{i,\sigma}$ in some reduced expression of $w_{\fR,\sigma} w_{0,\sigma}$ is at most one. By the definition of $t_{D,R}$ in (ii) of Lemma \ref{lem: tD} and unravelling all the definitions, the following corollary is a consequence of Proposition \ref{prop:Pcompa} and (\ref{Eisom11}) (with (\ref{eqtisigma})) when $I$ is non-critical for $\sigma$, of Proposition \ref{prop:crit} and (the proof of) (\ref{Eisom12}) when $I$ is critical for $\sigma$ (the last statement being a consequence of (\ref{Eisom13}) with (\ref{eq:amalg})):

\begin{cor}\label{cor:tDsigmaI}
With the above notation and assumptions the map $t_{D,R}$ induces by restriction an isomorphism (see below (\ref{diag: model2}) for ${\Ext}^1_{\fR,w_0, \sigma, i}(\cM(D), \cM(D))$):
\begin{equation}\label{tDsigmaI}
t_{D,\sigma,I}: \Ext^1_{\sigma}(\pi_{\alg}(D), \pi_{\sigma,I}(D)) \xlongrightarrow{\sim} \ol{\Ext}^1_{\fR,w_0, \sigma, i}(\cM(D),\cM(D))
\end{equation}
which itself induces by restriction an isomorphism
\begin{equation*}
t_{D,\sigma,0}: \Ext^1_{\sigma}(\pi_{\alg}(D), \pi_{\alg}(D)) \xlongrightarrow{\sim} \ol{\Ext}^1_{\sigma,0}(\cM(D),\cM(D)).
\end{equation*}
\end{cor}

Let $A_{D,\sigma,\fR,i}$ be the (local) Artinian $E$-algebra defined as the unique quotient of $A_D$ such that
\begin{equation}\label{eq:ADSRi}
\fm_{A_{D,\sigma, \fR,i}}\cong \ol{\Ext}^1_{\fR,w_0, \sigma, i}(\cM(D), \cM(D))^{\vee}
\end{equation}
and $A_{D,\sigma,0}$ the (local) Artinian $E$-algebra defined as the unique quotient of $A_{D,\sigma,\fR,i}$ such that $\fm_{A_{D,\sigma,0}}\cong \ol{\Ext}^1_{\sigma,0}(\cM(D), \cM(D))^{\vee}$. Similarly as in (\ref{E: action}) with $t_{D,R}$ replaced by the map $t_{D,\sigma,I}$ or by the map $t_{D,\sigma,0}$ of Corollary \ref{cor:tDsigmaI}, we see that the action of $A_D$ on $\widetilde{\pi}_\flat(D)$ factors as an action of its quotient $A_{D, \sigma, \fR,i}$ on the subrepresentation $\widetilde{\pi}_{\sigma,I}(D)$ which itself factors as an action of its quotient $A_{D,\sigma,0}$ on the subrepresentation $\widetilde{\pi}_{\alg,\sigma}(D)$. Finally, it is easy to see that we have a $\GL_n(K) \times A_D$-equivariant isomorphism
\begin{equation}\label{eq:decopflat}
\bigoplus_{\sigma\in \Sigma, \ \!\widetilde{\pi}_{\alg}(D)} \bigg(\bigoplus_{\text{$I$ \!not \!v.~\!c.~\!for \!$\sigma$}, \ \!\widetilde{\pi}_{\alg,\sigma}(D)} \!\!\!\!\widetilde{\pi}_{\sigma,I}(D)\bigg) \xlongrightarrow{\sim} \widetilde{\pi}_\flat(D)
\end{equation}
which induces a $\GL_n(K)$-equivariant isomorphism (compare with (\ref{eq:amalgcrit}))
\begin{equation*}
\bigoplus_{\sigma\in \Sigma,\ \!\text{$I$ \!not \!v.~\!c.~\!for \!$\sigma$}, \ \!{\pi}_{\alg}(D)} \!\!\!\!{\pi}_{\sigma,I}(D)\xlongrightarrow{\sim} {\pi}_\flat(D).
\end{equation*}

We now prove two lemmas on the structure of $\widetilde{\pi}_{\sigma,I}(D)$ which will be used in the local-global compatibility of \S\ref{sec:loc-glob}. We keep the previous notation, in particular $\sigma\in \Sigma$, $I$ is not very critical for $\sigma$, $i=\vert I\vert$ and $\fR$ is a refinement compatible with $I$. Renumbering the Frobenius eigenvalues if necessary we assume $\mathfrak{R}=(\varphi_0, \dots, \varphi_{n-1})$ and we let $\delta_{\fR}\in \widehat T$ as below (\ref{EptxR}).\bigskip

We first assume that $s_{i,\sigma}$ does not appear in some (equivalently any) reduced expression of $w_{\fR,\sigma} w_{0,\sigma}$. As for $\widetilde{\pi}_\flat(D)$, $\widetilde{\pi}_R(D)$ above, we let $\widetilde{\delta}_{\fR,\sigma,i}$ be the tautological extension of $\delta_{\fR} \otimes_E \Hom_{\sigma,i}(T(K),E)$ by $\delta_{\fR}$ where we identify $\Hom_{\sigma,i}(T(K),E)$ (see (\ref{eqtisigma})) with a subspace of $\Ext^1_{T(K)}(\delta_{\fR}, \delta_{\fR})$ using the isomorphism $\Hom(T(K),E)\cong \Ext^1_{T(K)}(\delta_{\fR}, \delta_{\fR})$ (similarly as in the proof of Lemma \ref{lem:isoextalg}). As in (\ref{E: action}) but replacing $t_{D,R}$ by the inverse of the isomorphism (\ref{Eisom11}) (which uses the assumption on $s_{i,\sigma}$), we equip $\widetilde{\delta}_{\fR,\sigma,i}$ with an action of $A_{D,\sigma, \fR,i}$. For an admissible locally $\Qp$-analytic representation $V$ of $L_P(K)$ over $E$ where $P\subset \GL_n$ is a standard parabolic subgroup, we let $I_{P^-}^{\GL_n}(V)$ be the closed $\GL_n(K)$-subrepresentation of the locally $\Qp$-analytic parabolic induction $(\Ind_{P^-(K)}^{\GL_n(K)} V)^{\Qp\text{-}\an}$ generated by the image of the natural embedding (cf.~\cite[Lemma 0.3]{Em07}, \cite[\S~2.8]{Em07} and recall $\delta_P$ is the modulus character of $P(K)$ seen as a (smooth) character of $L_P(K)$)
\begin{equation}\label{eq:jpcan}
V\otimes_E\delta_P \hooklongrightarrow J_P\big((\Ind_{P^-(K)}^{\GL_n(K)} V)^{\Qp\text{-}\an}\big).
\end{equation}
We consider $I_{B^-}^{\GL_n} (\widetilde{\delta}_{\fR,\sigma,i}\delta_B^{-1})\subset (\Ind_{B^-(K)}^{\GL_n(K)}\widetilde{\delta}_{\fR, \sigma,i} \delta_B^{-1})^{\Qp\text{-}\an}$ which are both equipped with a left action of $A_{D,\sigma, \fR,i}$ induced by the action on (the underlying $E$-vector space of) $\widetilde{\delta}_{\fR,\sigma,i}\delta_B^{-1}$. We can check from (\ref{eq:t}), (\ref{eq:alg}) and the definition of $\delta_{\fR}$ that we have a $\GL_n(K)$-equivariant injection
\begin{equation}\label{eq: iota1}
\iota: \pi_{\alg}(D) \otimes_E\varepsilon^{n-1} \hooklongrightarrow \big(\Ind_{B^-(K)}^{\GL_n(K)}\delta_{\fR} \delta_B^{-1}\big)^{\Qp\text{-}\an}.
\end{equation}
Using \cite[Rk.~5.1.8]{Em07} we check that the image of (\ref{eq: iota1}) contains the image of $\delta_{\fR}\hookrightarrow J_B((\Ind_{B^-(K)}^{\GL_n(K)}\delta_{\fR} \delta_B^{-1})^{\Qp\text{-}\an})$, and since $\pi_{\alg}(D)$ is irreducible we deduce that (\ref{eq: iota1}) induces an isomorphism $\pi_{\alg}(D) \otimes_E\varepsilon^{n-1} \buildrel\sim\over\longrightarrow I_{B^-}^{\GL_n}(\delta_{\fR}\delta_B^{-1})$. We fix the injection (\ref{eq: iota1}) in the sequel.

\begin{lem}\label{Luniv1}
Assume that $s_{i,\sigma}$ does not appear in some (equivalently any) reduced expression of $w_{\fR,\sigma} w_{0,\sigma}$.There is a $\GL_n(K) \times A_{D,\sigma, \fR, i}$-equivariant isomorphism
\begin{equation}\label{isoIPG1}
I_{B^-}^{\GL_n}(\widetilde{\delta}_{\fR,\sigma,i} \delta_B^{-1}) \xlongrightarrow{\sim} \widetilde{\pi}_{\sigma,I}(D) \otimes_E \varepsilon^{n-1}
\end{equation}
which restricts to the identity map on $\pi_{\alg}(D) \otimes_E \varepsilon^{n-1}$.
\end{lem}
\begin{proof}
We first get rid of the factor $\otimes_{\tau\neq \sigma} L(\lambda_{\tau})$ on each side of (\ref{isoIPG1}). Let $\widetilde{\pi}_{I}(D_\sigma)$ be the tautological extension of $\pi_{\alg}(D_{\sigma}) \otimes_E \Ext^1_{\GL_n(K),\sigma}(\pi_{\alg}(D_{\sigma}), \pi_I(D_{\sigma}))$ by $\pi_I(D_{\sigma})$ defined as $\widetilde{\pi}_{S}(D_\sigma)$ before Lemma \ref{lem:max}. It is a locally $\sigma$-analytic representation and by (\ref{eq:isoI}) we have a natural $\GL_n(K)\times A_{D,\sigma,\fR,i}$-equivariant isomorphism
\begin{equation*}
\widetilde{\pi}_{I}(D_\sigma)\otimes_E (\otimes_{\tau\neq \sigma} L(\lambda_{\tau}))\cong \widetilde{\pi}_{\sigma,I}(D)
\end{equation*}	
where the $A_{D,\sigma, \fR,i}$-action on $\widetilde{\pi}_{I}(D_\sigma)$ is induced by (\ref{eq:isoI}) composed with (\ref{tDsigmaI}). Set
\[\delta_{\fR,\sigma}':=\delta_{\fR} \prod_{j=0}^{n-1}\Big(\prod_{\tau\ne \sigma}\tau(t_j)^{-h_{j,\tau}-j}\Big)\]
which is a locally $\sigma$-analytic character of $T(K)$. One easily checks that the locally $\sigma$-analytic character
\[\widetilde{\delta}_{\fR,\sigma,i}':=\widetilde{\delta}_{\fR,\sigma,i} \prod_{j=0}^{n-1}\Big(\prod_{\tau\ne \sigma}\tau(t_j)^{-h_{j,\tau}-j}\Big)\]
is the tautological extension of $\delta_{\fR,\sigma}' \otimes_E \Hom_{\sigma,i}(T(K),E)$ by $\delta_{\fR,\sigma}'$ (using $\Hom_\sigma(T(K),E)\cong \Ext^1_{T(K),\sigma}(\delta_{\fR,\sigma}', \delta_{\fR,\sigma}')$) and that there is a natural $\GL_n(K) \times A_{D,\sigma, \fR,i}$-equivariant injection 
\begin{equation}\label{eq: twalg2}
\big(\Ind_{B^-(K)}^{\GL_n(K)} \widetilde{\delta}_{\fR,\sigma,i}' \delta_{B}^{-1}\big)^{\sigma\text{-}\an} \otimes_E (\otimes_{\tau\neq \sigma} L(\lambda_{\tau})) \hooklongrightarrow \big(\Ind_{B^-(K)}^{\GL_n(K)} \widetilde{\delta}_{\fR,\sigma,i}\delta_{B}^{-1}\big)^{\Q_p\text{-}\an}.
\end{equation}
As $\widetilde{\delta}_{\fR,\sigma,i}'$ is locally $\sigma$-analytic, the injection $\widetilde{\delta}_{\fR,\sigma,i}'\hookrightarrow J_B((\Ind_{B^-(K)}^{\GL_n(K)} \widetilde{\delta}_{\fR,\sigma,i}'\delta_B^{-1})^{\Q_p\text{-}\an})$ has image in the subspace $J_B((\Ind_{B^-(K)}^{\GL_n(K)} \widetilde{\delta}_{\fR,\sigma,i}'\delta_B^{-1})^{\sigma\text{-}\an})$ (recall $J_B$ is left exact). By definition of $I_{B^-}^{\GL_n}(-)$, we deduce that $I_{B^-}^{\GL_n} ( \widetilde{\delta}_{\fR,\sigma,i}' \delta_{B}^{-1})$ is contained in $(\Ind_{B^-(K)}^{\GL_n(K)} \widetilde{\delta}_{\fR,\sigma,i}' \delta_{B}^{-1})^{\sigma\text{-}\an}$. Moreover, by definition of $I_{B^-}^{\GL_n}(-)$ again, one easily sees that (\ref{eq: twalg2}) induces a $\GL_n(K) \times A_{D,\sigma, \fR,i}$-equivariant isomorphism
\begin{equation*}
I_{B^-}^{\GL_n}(\widetilde{\delta}_{\fR,\sigma,i}' \delta_{B}^{-1}) \otimes_E (\otimes_{\tau\neq \sigma} L(\lambda_{\tau})) \xlongrightarrow{\sim} I_{B^-}^{\GL_n} (\widetilde{\delta}_{\fR,\sigma,i}\delta_{B}^{-1}).
\end{equation*}
Moreover, as for (\ref{eq: iota1}), one has an injection $\pi_{\alg}(D_\sigma) \otimes_E\varepsilon^{n-1} \hooklongrightarrow \big(\Ind_{B^-(K)}^{\GL_n(K)}\delta_{\fR,\sigma}'\delta_B^{-1}\big)^{\sigma\text{-}\an}$ which likewise induces an isomorphism
\begin{equation}\label{eq:algI}
\pi_{\alg}(D_\sigma) \otimes_E\varepsilon^{n-1} \buildrel\sim\over\longrightarrow I_{B^-}^{\GL_n}(\delta_{\fR,\sigma}'\delta_B^{-1}).
\end{equation}
Therefore, tensoring everything by $\otimes_{\tau\neq \sigma} L(\lambda_{\tau})$, we are reduced to prove a statement as in the lemma but replacing $\widetilde{\delta}_{\fR,\sigma,i}$, $\widetilde{\pi}_{\sigma,I}(D)$, $\pi_{\alg}(D)$ by $\widetilde{\delta}_{\fR,\sigma,i}'$, $\widetilde{\pi}_{I}(D_\sigma)$, $\pi_{\alg}(D_\sigma)$.\bigskip

We now claim that we have a commutative diagram (writing $\GL_n$, $B^-$ for $\GL_n(K)$, $B^-(K)$ in the inductions and $\Ext^1_{\sigma}$ for $\Ext^1_{\GL_n(K),\sigma}$)
\begin{equation}\label{diagtildepiI}
\adjustbox{scale=0.81}{
\begin{tikzcd}
\pi_I(D_{\sigma}) \otimes_E \varepsilon^{n-1}\arrow[r, hook] \arrow[d, hook, "\iota"] & \widetilde{\pi}_{I}(D_\sigma)\otimes_E \varepsilon^{n-1} \arrow[r, two heads] \arrow[d, hook] & (\pi_{\alg}(D_{\sigma})\otimes_E \varepsilon^{n-1}) \otimes_E \Ext^1_{\sigma}(\pi_{\alg}(D_{\sigma}), \pi_I(D_{\sigma})) \arrow[d, hook, "\iota \otimes (\ref{eq:amalg2})^{-1}"] \\
(\Ind_{B^-}^{\GL_n}\delta_{\fR,\sigma}' \delta_B^{-1})^{\sigma\text{-}\an} \arrow[r, hook] & (\Ind_{B^-}^{\GL_n}\widetilde{\delta}_{\fR, \sigma,i}' \delta_B^{-1})^{\sigma\text{-}\an} \arrow[r, two heads] &(\Ind_{B^-}^{\GL_n}\delta_{\fR,\sigma}'\delta_B^{-1} )^{\sigma\text{-}\an}\otimes_E \Hom_{\sigma,i}(T(K),E) \\
I_{B^-}^{\GL_n}(\delta_{\fR,\sigma}'\delta_B^{-1}) \arrow[u, hook] \arrow[r, hook] &I_{B^-}^{\GL_n}(\widetilde{\delta}_{\fR,\sigma,i}' \delta_B^{-1}) \arrow[u, hook] \arrow[r, two heads] &I_{B^-}^{\GL_n}(\delta_{\fR,\sigma}'\delta_B^{-1})\otimes_E \Hom_{\sigma,i}(T(K),E) \arrow[u, hook]
\end{tikzcd}}
\end{equation}
where the top two sequences are exact, while the bottom sequence is exact on the left and right but not necessarily in the middle. The exactness statements are clear, as is the commutativity of the bottom two squares. The existence and commutativity of the top two squares follow from (\ref{eqtisigma}) and a close examination of the proof of Proposition \ref{prop:isoext}, which shows that the isomorphism (\ref{eq:amalg2}) in \emph{loc.~cit.}~is obtained by identifying each extension of $\pi_{\alg}(D_{\sigma})$ by $\pi_I(D_{\sigma})$ as a subrepresentation of $(\Ind_{B^-(K)}^{\GL_n(K)} \delta_{\fR,\sigma}' \delta_B^{-1}(1+\psi \epsilon)\varepsilon^{1-n})^{\sigma\text{-}\an}$ for the associated $\psi\in \Hom_{\sigma,i}(T(K),E)$. Moreover this last observation implies that, similarly as below (\ref{eq: iota1}) or as for (\ref{eq:algI}), the image of $\widetilde{\pi}_{I}(D_\sigma)\otimes_E \varepsilon^{n-1}$ in $ (\Ind_{B^-}^{\GL_n}\widetilde{\delta}_{\fR, \sigma,i}' \delta_B^{-1})^{\sigma\text{-}\an}$ contains the image of $\widetilde{\delta}_{\fR, \sigma,i}'$, and hence that we have an inclusion $I_{B^-}^{\GL_n}(\widetilde{\delta}_{\fR,\sigma,i}' \delta_B^{-1})\subseteq \widetilde{\pi}_{I}(D_\sigma)\otimes_E \varepsilon^{n-1}$ by definition of $I_{B^-}^{\GL_n}(-)$. Since both representations surject onto $(\pi_{\alg}(D_{\sigma})\otimes_E \varepsilon^{n-1}) \otimes_E \Ext^1_{\GL_n(K),\sigma}(\pi_{\alg}(D_{\sigma}), \pi_I(D_{\sigma}))$ using (\ref{eq:algI}), it formally follows from the definition of $\widetilde{\pi}_{I}(D_\sigma)$ that this inclusion is an equality (note the similarity here with \cite[Lemma 3.35]{Di25}). Finally, using Proposition \ref{prop:Pcompa} and (\ref{Eisom11}) (and unravelling the various definitions), it is easy to check that (\ref{diagtildepiI}) is moreover $A_{D,\fR,\sigma, i}$-equivariant. 
\end{proof}

\begin{rem}\label{Rnuni}
An isomorphism as in Lemma \ref{Luniv1} is not unique. Indeed, composing (\ref{isoIPG1}) with the action of any element $a\in A_{D,\sigma, \fR,i}$ such that $a\equiv 1 \!\!\mod{\fm_{A_{D,\sigma, \fR,i}}}$ is still a $\GL_n(K) \times A_{D,\sigma, \fR,i}$-equivariant isomorphism which restricts to the identity on $\pi_{\alg}(D) \otimes_E \varepsilon^{n-1}$. In fact the action of $A_{D,\sigma, \fR,i}$ on $\widetilde{\pi}_{\sigma,I}(D)$ induces an \emph{isomorphism} 
\begin{equation}\label{eq:isoA}
A_{D,\sigma, \fR,i} \buildrel\sim\over \longrightarrow \End_{\GL_n(K)}(\widetilde{\pi}_{\sigma,I}(D)),
\end{equation}
in particular $\End_{\GL_n(K)}(\widetilde{\pi}_{\sigma,I}(D))$ is commutative. One can argue as follows. For any $f\in \End_{\GL_n(K)}(\widetilde{\pi}_{\sigma,I}(D))$, let $\lambda\in E$ such that $f|_{\pi_{\alg}(D)}=\lambda (\id)$. As we have
\[\Hom_{\GL_n(K)}\big(\pi_{\sigma,I}(D)/\pi_{\alg}(D), \widetilde{\pi}_{\sigma,I}(D)\big)=0\]
(since $\soc_{\GL_n(K)}\widetilde{\pi}_{\sigma,I}(D)\cong \pi_{\alg}(D)$), the morphism $f - \lambda(\id)$ necessarily factors through $\cosoc_{\GL_n(K)}\widetilde{\pi}_{\sigma,I}(D)$ which by definition of $\widetilde{\pi}_{\sigma,I}(D)$ is $\pi_{\alg}(D)$-isotypic. Hence $f - \lambda(\id)$ must have image in $\pi_{\alg}(D)\cong \soc_{\GL_n(K)}\widetilde{\pi}_{\sigma,I}(D)$, equivalently $f-\lambda(\id)\in \End_{\GL_n(K)}(\widetilde{\pi}_{\sigma,I}(D))$ lies in the subspace $\Hom_{\GL_n(K)}(\widetilde{\pi}_{\sigma,I}(D), \pi_{\alg}(D))$. But since 
\begin{equation*}
\Hom_{\GL_n(K)}(\widetilde{\pi}_{\sigma,I}(D), \pi_{\alg}(D)) \cong \Ext^1_{\sigma}(\pi_{\alg}(D), \pi_{\sigma,I}(D))^\vee \cong \ol{\Ext}^1_{\fR,w_0, \sigma, i}(\cM(D),\cM(D))^\vee
\end{equation*}
where the first isomorphism follows from the definition of $\widetilde{\pi}_{\sigma,I}(D)$ and the second from (\ref{tDsigmaI}), we deduce (\ref{eq:isoA}) using $\ol{\Ext}^1_{\fR,w_0, \sigma, i}(\cM(D),\cM(D))^\vee\cong \fm_{A_{D,\sigma, \fR,i}}$.
\end{rem}

As for $\widetilde{\delta}_{\fR,\sigma,i}$, we let $\widetilde{\delta}_{\fR, \sigma,0}$ be the tautological extension of $\delta_{\fR} \otimes_E \Hom_{\sigma,0}(T(K),E)$ by $\delta_{\fR}$ (see (\ref{eqt0sigma})). As for $\widetilde{\delta}_{\fR,\sigma,i}$ but replacing the inverse of (\ref{Eisom11}) by the inverse of (\ref{Eisom13}), we equip $\widetilde{\delta}_{\fR, \sigma,0}$ with an action of $A_{D,\sigma,0}$. As for (\ref{isoIPG1}) we can prove a $\GL_n(K) \times A_{D,\sigma,0}$-equivariant isomorphism
\begin{equation}\label{isoIPG2}
I_{B^-}^{\GL_n}(\widetilde{\delta}_{\fR,\sigma,0} \delta_B^{-1}) \xlongrightarrow{\sim} \widetilde{\pi}_{\alg,\sigma}(D) \otimes_E \varepsilon^{n-1}
\end{equation}
(note that we do not need any assumption on $s_{i,\sigma}$ here). We have $\widetilde{\delta}_{\fR,\sigma,0}\hookrightarrow \widetilde{\delta}_{\fR,\sigma,i}$ and, when $s_{i,\sigma}$ does not appear in $w_{\fR,\sigma}w_{0,\sigma}$, the action of $A_{D,\sigma, \fR,i}$ on $\widetilde{\delta}_{\fR,\sigma,i}$ factors through its quotient $A_{D,\sigma,0}$ on $\widetilde{\delta}_{\fR,\sigma,0}$ (and any isomorphism as in (\ref{isoIPG1}) induces an isomorphism as in (\ref{isoIPG2})).\bigskip

We now assume that $s_{i,\sigma}$ appears with multiplicity one in some reduced expression of $w_{\fR,\sigma} w_{0,\sigma}$. We denote by $E_{\fR,\sigma,i}\subset \ol{\Ext}^1_{\fR,w_0,\sigma,i}(\cM(D), \cM(D))$ the $1$-dimensional kernel of the map (\ref{Esurj12}) (or equivalently of the map (\ref{mult:Ecomp}) and by $B_{D,\fR, \sigma,i}$ the (local) Artinian $E$-algebra defined as the unique quotient of $A_{D,\fR,\sigma,i}$ such that
\begin{equation}\label{eq:BDSRi}
\fm_{B_{D,\fR,\sigma,i}}\cong (E_{\fR,\sigma,i})^{\vee}.
\end{equation}
It follows from (\ref{Eisom12}) and the definition of $A_{D,\sigma,0}$ that we have $\fm_{A_{D,\fR,\sigma,i}}\buildrel\sim\over\longrightarrow \fm_{A_{D,\sigma,0}}\bigoplus \fm_{B_{D,\fR,\sigma,i}}$ from which we deduce an isomorphism of local Artinian $E$-algebras:
\begin{equation}\label{Eartdec}
A_{D,\fR,\sigma,i} \buildrel\sim\over\longrightarrow B_{D,\fR,\sigma,i} \times_{E} A_{D,\sigma,0}
\end{equation}
where the fiber product on the right is for the two natural maps onto the residue field $E$. We let $V_{\fR,\sigma,i}$ be the preimage of $E_{\fR, \sigma,i}$ in $\Ext^1_{\sigma}(\pi_{\alg}(D), \pi_{\sigma,I}(D))$ via (\ref{tDsigmaI}). Recall from (\ref{eq:splitI}) that we have $\pi_{\sigma,I}(D)\cong \pi_{\alg}(D) \oplus \widetilde{C}(I,s_{i,\sigma})$ where $\widetilde{C}(I,s_{i,\sigma})=C(I,s_{i,\sigma})\otimes_E (\otimes_{\tau\neq \sigma} L(\lambda_{\tau}))$. Then by the analogue of Proposition \ref{prop:crit} with the map $t_{D,\sigma,I}$ in (\ref{tDsigmaI}) we see that $V_{\fR,\sigma,i}$ is the image of $\Ext^1_{\GL_n(K),\sigma}(\pi_{\alg}(D_{\sigma}), {C}(I,s_{i,\sigma}))$ via the isomorphism (\ref{eq:isoI}), or equivalently by \cite[Lemma 3.5(1)]{Di25} is the $1$-dimensional $E$-vector space $\Ext^1_{\GL_n(K)}(\pi_{\alg}(D), \widetilde{C}(I,s_{i,\sigma}))$. We let $\widetilde{\pi}_{\sigma,I,1}(D)$ be the tautological extension of $\pi_{\alg}(D) \otimes_E V_{\fR,\sigma,i}$ by $\pi_{\sigma,I}(D)$, that is, we have
\begin{equation}\label{eq:piI1}
\widetilde\pi_{\sigma,I,1}(D)\cong \pi_{\alg}(D)\ \bigoplus \ \big(\widetilde{C}(I,s_{i,\sigma})\!\begin{xy} (30,0)*+{}="a"; (40,0)*+{}="b"; {\ar@{-}"a";"b"}\end{xy}\!(\pi_{\alg}(D)\otimes_E V_{\fR,\sigma,i})\big)
\end{equation}
where the direct summand on the right is the unique non-split extension (i.e.~isomorphic to $W_{\sigma,I}$, see above (\ref{eq:pi(D)W})). We equip $\widetilde{\pi}_{\sigma,I,1}(D)$ with an action of $B_{D,\fR,\sigma,i}$ as in (\ref{E: action}). As in (\ref{Einjuniv}) we have a natural $\GL_n(K)\times A_{D,\fR,\sigma,i}$-equivariant injection $\widetilde{\pi}_{\sigma,I,1}(D)\hookrightarrow \widetilde{\pi}_{\sigma,I}(D)$ such that the action of $A_{D,\fR,\sigma,i}$ on $\widetilde{\pi}_{\sigma,I,1}(D)$ factors through $B_{D,\fR,\sigma,i}$. For later use, we recall that $x \in \fm_{B_{D,\fR,\sigma,i}}\buildrel\sim\over\rightarrow V_{\fR,\sigma,i}^{\vee}$ acts on $\widetilde{\pi}_{\sigma,I,1}(D)$ by 
\begin{equation}\label{eq:kappax1}
\widetilde{\pi}_{\sigma,I,1}(D) \buildrel \kappa_x \over \longrightarrow \pi_{\alg}(D) \hooklongrightarrow \widetilde{\pi}_{\sigma,I,1}(D)
\end{equation}
where $\kappa_x$ is the composition
\[\widetilde\pi_{\sigma,I,1}(D)\twoheadlongrightarrow \big(\widetilde{C}(I,s_{i,\sigma})\!\begin{xy} (30,0)*+{}="a"; (38,0)*+{}="b"; {\ar@{-}"a";"b"}\end{xy}\!(\pi_{\alg}(D)\otimes_E V_{\fR,\sigma,i})\big)\twoheadlongrightarrow\pi_{\alg}(D)\otimes_E V_{\fR,\sigma,i} \buildrel \id \otimes x \over \longrightarrow \pi_{\alg}(D).\]

Recall $\delta_{\fR}= \unr(\underline{\varphi})t^{\textbf{h}} \delta_B(\boxtimes_{j=0}^{n-1} \varepsilon^{j})\in \widehat{T}$ and denote by $\wt(\delta_{\fR})=(\wt(\delta_{\fR})_\tau)_{\tau\in \Sigma}$ the $1$-dimensional $\text{U}(\ft_\Sigma)$-module over $E$ which is the derivative of $\delta_{\fR}$, or equivalently of $t^{\textbf{h}}(\boxtimes_{j=0}^{n-1}\varepsilon^{j})$. We \ also \ write \ $\delta_{\fR}= \unr(\underline{\varphi})t^{\wt(\delta_{\fR})} \delta_B(\boxtimes_{j=0}^{n-1} \vert \cdot\vert_K^{j})$ \ using \ the notation \ of \ (\ref{eq:dominant}). We let $X_{\sigma,i}^-(-\wt(\delta_{\fR}))$ be the unique quotient of $\text{U}(\fg_{\Sigma})\otimes_{\text{U}(\fb_\Sigma^-)}(-\wt(\delta_{\fR}))$ which is a non-split extension \ of \ $L(\wt(\delta_{\fR}))^{\vee}\cong L^-(-\wt(\delta_{\fR}))$ \ by \ $L^-(- s_{i,\sigma}\!\cdot\!\wt(\delta_{\fR}))$. \ Here \ $L(\wt(\delta_{\fR}))$ \ (resp.~$L^-(-\wt(\delta_{\fR}))$) is the (finite dimensional) simple $\text{U}(\fg_{\Sigma})$-module over $E$ of highest weight $\wt(\delta_{\fR})$ (resp.~$-\wt(\delta_{\fR})$) with respect to the upper Borel $\fb_\Sigma$ (resp.~the lower Borel $\fb_\Sigma^-$), and $L^-(- s_{i,\sigma}\!\cdot\!\wt(\delta_{\fR}))$ is the unique simple $\text{U}(\fg_{\Sigma})$-module in the BGG category $\cO$ for $\fg_\Sigma$ with respect to $\fb_\Sigma^-$ of highest weight $- s_{i,\sigma}\!\cdot\!\wt(\delta_{\fR}):=(- s_{i,\sigma}\!\cdot\!\wt(\delta_{\fR})_\sigma,(-\wt(\delta_{\fR})_\tau)_{\tau\ne \sigma}))$ where $ s_{i,\sigma}\!\cdot\!\wt(\delta_{\fR})_\sigma$ is defined as in (\ref{eq:dotaction}). We define
\begin{equation}\label{eq:W}
W'_{\sigma,I}:=\pi_{\alg}(D)\otimes_E \varepsilon^{n-1} \ \bigoplus \ \cF_{B^-}^{\GL_n}\big(X_{\sigma,i}^-(-\wt(\delta_{\fR}))^{\vee}, \delta_{\fR}^{\sm} \delta_B^{-1}\big)
\end{equation}
where $X_{\sigma,i}^-(-\wt(\delta_{\fR}))^{\vee}$ is here the dual of $X_{\sigma,i}^-(-\wt(\delta_{\fR}))$ in the sense of \cite[\S~3.2]{Hu08} and $\delta_{\fR}^{\sm}:=\unr(\underline{\varphi})\delta_B(\boxtimes_{j=0}^{n-1} \vert \cdot\vert_K^{j})$ (a smooth character of $T(K)$). It follows from \cite[Prop.~4.1.2]{Or20} with \cite[Lemma 3.5(1)]{Di25} \ that \ $\cF_{B^-}^{\GL_n}\big(X_{\sigma,i}^-(-\wt(\delta_{\fR}))^{\vee}, \delta_{\fR}^{\sm} \delta_B^{-1}\big)\cong W_{\sigma,I}\otimes_E\varepsilon^{n-1}$.\bigskip

To any surjection (unique up to scalar) $\kappa: W_{\sigma,I}\otimes_E\varepsilon^{n-1}\twoheadlongrightarrow \pi_{\alg}(D) \otimes_E \varepsilon^{n-1}$ and to any non-zero element $x\in \fm_{B_{D,\fR,\sigma,i}}$ we associate an $E$-linear action of $B_{D,\fR,\sigma,i}$ on $W_{\sigma,I}'$ such that $W'_{\sigma,I}[\fm_{B_{D,\fR,\sigma,i}}]\cong \big(\pi_{\alg}(D) \oplus \widetilde{C}(I,s_{i,\sigma})\big)\otimes_E \varepsilon^{n-1}$ by making $x$ act by (recall $\dim_E\fm_{B_{D,\fR,\sigma,i}}=1$):
\begin{equation}\label{eq:kappax2}
W'_{\sigma,I} \twoheadlongrightarrow W_{\sigma,I}\otimes_E\varepsilon^{n-1} \xlongrightarrow{\kappa} \pi_{\alg}(D) \otimes_E \varepsilon^{n-1}\hooklongrightarrow W'_{\sigma,I}.
\end{equation}
This action of $B_{D,\fR,\sigma,i}$ depends on the choices of $\kappa$ and $x$. However, we have:

\begin{lem}\label{lem: OStildepi}
Assume that $s_{i,\sigma}$ appears with multiplicity one in some reduced expression of $w_{\fR,\sigma} w_{0,\sigma}$. There is a $\GL_n(K) \times B_{D,\fR,\sigma,i}$-equivariant isomorphism 
\begin{equation}\label{EOStildepi}
W'_{\sigma,I} \xlongrightarrow{\sim} \widetilde{\pi}_{\sigma,I,1}(D) \otimes_E \varepsilon^{n-1}
\end{equation}
which restricts to the identity map on $\pi_{\alg}(D)\otimes_E \varepsilon^{n-1}$.
\end{lem}
\begin{proof}
By (\ref{eq:piI1}) and (\ref{eq:W}) the two $\GL_n(K)$-representations in (\ref{EOStildepi}) are isomorphic. Fix a $\GL_n(K)$-equivariant isomorphism $f:W'_{\sigma,I} \buildrel\sim\over \longrightarrow \widetilde{\pi}_{\sigma,I,1}(D) \otimes_E \varepsilon^{n-1}$ such that $f|_{\pi_{\alg}(D)\otimes_E \varepsilon^{n-1}}=\id$. We need to compare the $B_{D,\fR,\sigma,i}$-action on each side of $f$. Let $0 \neq x \in \fm_{B_{D,\fR,\sigma,i}}$, comparing (\ref{eq:kappax1}) and (\ref{eq:kappax2}), we see there exists $\lambda\in E^{\times}$ such that the following diagram commutes
\begin{equation}\label{diag:actiontw}
\begin{tikzcd}
W'_{\sigma,I} \arrow[r, "f", "\sim" swap] \arrow[d, "x"]& \widetilde{\pi}_{\sigma,I,1}(D) \otimes_E \varepsilon^{n-1} \arrow[d, "\lambda x"]\\
W'_{\sigma,I} \arrow[r, "f", "\sim"'] & \widetilde{\pi}_{\sigma,I,1}(D) \otimes_E \varepsilon^{n-1}.
\end{tikzcd}
\end{equation}
It is easy to see from (\ref{eq:piI1}) that $\End_{\GL_n(K)}(\widetilde{\pi}_{\sigma,I,1}(D) \otimes_E \varepsilon^{n-1})\cong\smat{E & E \\ 0 & E}$ with $x\in \smat{0 & E \\ 0 & 0}$ and $\smat{E & 0 \\ 0 & 0}$ acting on the direct summand $\pi_{\alg}(D)\otimes_E \varepsilon^{n-1}$. Composing $f$ with the automorphism $\smat{1 & 0 \\ 0 & \lambda}$ of $\widetilde{\pi}_{\sigma,I,1}(D) \otimes_E \varepsilon^{n-1}$ (which is the identity on $\pi_{\alg}(D) \otimes_E \varepsilon^{n-1}$), we get an isomorphism $f'$ such that (\ref{diag:actiontw}) holds with $f$ replaced by $f'$ and $\lambda x$ replaced by $x$. As $\fm_{B_{D,\fR,\sigma,i}}$ is spanned by $x$, $f'$ satisfies the property in the lemma.
\end{proof}

From (\ref{Eisom12}) and (\ref{Eartdec}) with the definitions of $\widetilde{\pi}_{\sigma,I}(D)$ and $\widetilde{\pi}_{\alg, \sigma}(D)$, it is formal to deduce a $\GL_n(K) \times A_{D,\fR,\sigma,i}$-equivariant isomorphism
\begin{equation}\label{eq:tildepiI}
\widetilde{\pi}_{\sigma,I}(D) \cong \widetilde{\pi}_{\sigma,I,1}(D) \bigoplus_{\pi_{\alg}(D)} \widetilde{\pi}_{\alg, \sigma}(D).
\end{equation}
With Lemma \ref{lem: OStildepi}, (\ref{eq:W}) and (\ref{isoIPG2}), we then deduce a $\GL_n(K) \times A_{D,\fR,\sigma,i}$-equivariant isomorphism (which is the analogue of (\ref{isoIPG1}) when $s_{i,\sigma}$ appears with multiplicity one in $w_{\fR,\sigma} w_{0,\sigma}$)
\begin{equation}\label{Eisomuniv2}
I_{B^-}^{\GL_n}(\widetilde{\delta}_{\fR,\sigma,0} \delta_B^{-1}) \ \bigoplus \ W_{\sigma,I}\otimes_E\varepsilon^{n-1} \xlongrightarrow{\sim} \widetilde{\pi}_{\sigma,I}(D) \otimes_E \varepsilon^{n-1}.
\end{equation}

\subsection{Local-global compatibility for \texorpdfstring{$\pi(D)^{\flat}$}{piDflat} and main results}\label{sec:loc-glob}

We state and prove our main results (Theorem \ref{T: lg} and Corollary \ref{cor:main}), which give a weak form of Conjecture \ref{conj:main} under the Taylor-Wiles assumptions.\bigskip

We keep all previous notation, in particular $D=D_{\cris}(r)=D_{\cris}(\rho_{\pi,\widetilde \wp})$.

\begin{thm}\label{T: lg}
Assume the Taylor-Wiles assumptions Hypothesis \ref{TayWil0}. The isomorphism (\ref{Elalg}) extends to an injection of locally $\Qp$-analytic representations of $\GL_n(K)$ over $E$
\begin{equation}\label{Elg}
(\pi(D)^{\flat} \otimes_E \varepsilon^{n-1})^{\oplus m} \hooklongrightarrow \widehat{S}_{\xi,\tau}(U^{\wp},E)[\fm_{\pi}]^{\Qp\text{-}\an}
\end{equation}
such that $\Hom_{\GL_n(K)}\big(\pi_{\alg}(D)\otimes_E \varepsilon^{n-1}, \widehat{S}_{\xi,\tau}(U^{\wp},E)[\fm_{\pi}]^{\Qp\text{-}\an}/(\pi(D)^{\flat} \otimes_E \varepsilon^{n-1})^{\oplus m}\big)=0$.
\end{thm}

For a regular filtered $\varphi$-module $D'$ satisfying (\ref{eq:phi}) as in \S~\ref{sec:prel} we define the finite sets
\begin{eqnarray*}
S^{\mathrm{nc}}(D')&:=&\{(\sigma,I)\ |\ \text{$I$ is not critical for $\sigma$}\}\\
S^{\flat}(D')&:=&\{(\sigma,I) \ |\ \text{$I$ is not very critical for $\sigma$}\}.
\end{eqnarray*}
Here $\sigma\in \Sigma$, $I$ is a subset of the set of Frobenius eigenvalues of $D'$ of cardinality $\in \{1,\dots,n-1\}$ and criticality is (of course) with respect to $D'$ (see Definition \ref{def:critical}).

\begin{prop}\label{prop:nasty}
Keep the setting of Theorem \ref{T: lg}. Let $D'$ be a regular filtered $\varphi$-module satisfying (\ref{eq:phi}) as in \S~\ref{sec:prel} and assume $S^{\flat}(D')=S^{\flat}(D)$, $S^{\mathrm{nc}}(D')=S^{\mathrm{nc}}(D)$. If there is a $\GL_n(K)$-equivariant injection $\pi(D')^{\flat} \otimes_E \varepsilon^{n-1}\hookrightarrow \widehat{S}_{\xi,\tau}(U^{\wp},E)[\fm_{\pi}]^{\Qp\text{-}\an}$ then for any $\sigma\in \Sigma$ we have isomorphisms of filtered $\varphi^f$-modules $D'_{\sigma}\cong D_{\sigma}$.
\end{prop}

We do not know if the statement of Proposition \ref{prop:nasty} still holds without the assumptions $S^{\flat}(D')=S^{\flat}(D)$ or $S^{\mathrm{nc}}(D')=S^{\mathrm{nc}}(D)$: there is the issue mentioned in Remark \ref{rem:forlater3}, but we also do not know how to rule out $(\pi_{\alg}(D)\!\begin{xy} (30,0)*+{}="a"; (38,0)*+{}="b"; {\ar@{-}"a";"b"}\end{xy}\!\widetilde{C}(I,s_{i,\sigma})) \otimes_E \varepsilon^{n-1}\hookrightarrow \widehat{S}_{\xi,\tau}(U^{\wp},E)[\fm_{\pi}]^{\Qp\text{-}\an}$ when $I$ is critical for $\sigma$ (with a non-split extension on the left hand side). Fortunately, if we consider the socle of $\widehat{S}_{\xi,\tau}(U^{\wp},E)[\fm_{\pi}]^{\Qp\text{-}\an}$, we can still deduce:

\begin{cor}\label{cor:main}
Assume the Taylor-Wiles assumptions Hypothesis \ref{TayWil0}. The isomorphism class of the $\GL_n(K)$-representation $\widehat{S}_{\xi,\tau}(U^{\wp},E)[\fm_{\pi}]^{\Qp\text{-}\an}$ determines the isomorphism classes of all the filtered $\varphi^f$-modules $D_{\sigma}$ for $\sigma\in \Sigma$. In particular if $K=\Qp$ the $\GL_n(\Qp)$-representation $\widehat{S}_{\xi,\tau}(U^{\wp},E)[\fm_{\pi}]^{\Qp\text{-}\an}$ determines the $\Gal(\overline\Qp/\Qp)$-representation $r=\rho_{\pi,\widetilde{\wp}}$.
\end{cor}
\begin{proof}
It follows from \cite[Thm.~1.4]{BHS19} (with Remark \ref{rem:setting}) that the ``finite slope'' socle of $\widehat{S}_{\xi,\tau}(U^{\wp},E)[\fm_{\pi}]^{\Qp\text{-}\an}$ determines the permutations $w_{\fR,\sigma}$ for all refinements $\fR$ and all $\sigma\in \Sigma$. In particular 
the $\GL_n(K)$-representation $\widehat{S}_{\xi,\tau}(U^{\wp},E)[\fm_{\pi}]^{\Qp\text{-}\an}$ determines the sets $S^{\flat}(D)$ and $S^{\mathrm{nc}}(D)$. By Theorem \ref{T: lg} and Proposition \ref{prop:nasty}, it then determines the isomorphism classes of all the $D_{\sigma}$. The last assertion follows from Lemma \ref{lem:sad}.
\end{proof}

The rest of the section is devoted to the proofs of Theorem \ref{T: lg} and Proposition \ref{prop:nasty}.\bigskip

We first prove Theorem \ref{T: lg}. We use the notation of the previous sections. Recall $\fm_{\pi,\wp}\subset R_{\overline{r}}[1/p]\hookrightarrow R_{\infty}(\xi,\tau)[1/p]$ and $\fm^{\wp}_{\pi}\subset R_{\infty}^{\wp}(\xi,\tau)[1/p]\hookrightarrow R_{\infty}(\xi,\tau)[1/p]$ are defined above (\ref{EptxR}). From (\ref{Epatchedauto}) we deduce an isomorphism
\[\widehat{S}_{\xi,\tau}(U^{\wp},E)[\fm_{\pi}]^{\Qp\text{-}\an}=\widehat{S}_{\xi,\tau}(U^{\wp},E)^{\Qp\text{-}\an}[\fm_{\pi}]\cong \Pi_{\infty}(\xi,\tau)^{R_{\infty}(\xi,\tau)\text{-}\an}[\fm_{\pi}^{\wp}+\fm_{\pi,\wp}].\]
Using (\ref{Elalg}) we fix an injection
\begin{multline}\label{Einjlalg}
(\pi_{\alg}(D) \otimes_E \varepsilon^{n-1})^{\oplus m} \cong \Pi_{\infty}(\xi,\tau)^{R_{\infty}(\xi,\tau)\text{-}\an}[\fm_{\pi}^{\wp}+\fm_{\pi,\wp}]^{\Qp\text{-}\alg} \\
\hooklongrightarrow\Pi_{\infty}(\xi,\tau)^{R_{\infty}(\xi,\tau)\text{-}\an}[\fm_{\pi}^{\wp}+\fm_{\pi,\wp}].
\end{multline}
Recall there is a surjection
\[(\fm_{R_r}/\fm_{R_r}^2)^{\vee} \twoheadlongrightarrow \ol{\Ext}^1_{(\varphi, \Gamma)}(\cM(D), \cM(D))\cong (\fm_{A_D})^{\vee}.\]
We let $\fa_D$ be an ideal of $R_{r}$ containing $\fm_{R_r}^2$ and such that one has an isomorphism of finite dimensional $E$-vector spaces
\begin{equation}\label{eq:aDsplit}
\fa_D/\fm_{R_r}^2\ \bigoplus \ \fm_{A_D} \buildrel\sim\over\longrightarrow \fm_{R_r}/\fm_{R_r}^2.
\end{equation}
Note that a choice of $\fa_D$ is equivalent to a choice of splitting of the surjection of $E$-vector spaces (\ref{Esurj1}). It follows from (\ref{eq:aDsplit}) that the composition 
\begin{equation}\label{ADRr}
A_D \hooklongrightarrow R_r/\fm_{R_r}^2 \twoheadlongrightarrow R_r/\fa_D
\end{equation}
is an isomorphism of local Artinian $E$-algebras. We also denote $\fa_D$ the associated ideal of $R_{\overline{r}}[1/p]$ with $\fm_{\pi, {\wp}}^2\subset \fa_D \subset \fm_{\pi,{\wp}}$ and we define the ideal
\begin{equation}\label{eq:api}
\fa_{\pi}:=(\fm^{\wp}_{\pi}, \fa_D)\subset (\fm_{\pi}^{\wp},\fm_{\pi,\wp})\subset R_{\infty}(\xi,\tau)[1/p].
\end{equation}
The composition (\ref{ADRr}) induces $A_D[1/p] \xrightarrow{\sim} R_{\infty}(\xi,\tau)[{1}/{p}]/\fa_{\pi}$, \ hence \ the $\GL_n(K)$-representa\-tion $\Pi_{\infty}(\xi,\tau)[\fa_{\pi}]$ is equipped with an equivariant action of $A_D$ induced from the action of $R_{\infty}(\xi,\tau)$.\bigskip

The following proposition is crucial for Theorem \ref{T: lg}.

\begin{prop}\label{prop:T: lg}
Keep the setting of Theorem \ref{T: lg}. The injection (\ref{Einjlalg}) extends to a $\GL_n(K) \times A_D$-equivariant injection
\begin{equation}\label{injuniv}
(\widetilde{\pi}_\flat(D) \otimes_E \varepsilon^{n-1})^{\oplus m} \hooklongrightarrow \Pi_{\infty}(\xi,\tau)^{R_{\infty}(\xi,\tau)\text{-}\an}[\fa_{\pi}]
\end{equation}
where $\widetilde{\pi}_\flat(D)$ is defined below (\ref{Esurj1}).
\end{prop}

The existence of an injection (\ref{Elg}) as in Theorem \ref{T: lg} then immediately follows from Proposition \ref{prop:T: lg} by taking the subspaces annihilated by $\fm_{A_D}$ on both sides of (\ref{injuniv}) since $\Pi_{\infty}(\xi,\tau)^{R_{\infty}(\xi,\tau)\text{-}\an}[\fa_{\pi}][\fm_{A_D}]\cong \Pi_{\infty}(\xi,\tau)^{R_{\infty}(\xi,\tau)\text{-}\an}[\fm_{\pi}^{\wp}+\fm_{\pi,\wp}]\cong \widehat{S}_{\xi,\tau}(U^{\wp},E)^{\Qp\text{-}\an}[\fm_{\pi}]$.\bigskip

We now start the proof of Proposition \ref{prop:T: lg}. Since it is quite long we divide it into steps. But the strategy is similar to the proof of many results in this paper (for instance Corollary \ref{cor:Di25}): for each $(\sigma,I)\in S^{\flat}(D)$ we will show that (\ref{Einjlalg}) extends to a $\GL_n(K) \times A_D$-equivariant injection
\begin{equation*}
(\widetilde{\pi}_{\sigma,I}(D)\otimes_E \varepsilon^{n-1})^{\oplus m} \hooklongrightarrow \Pi_{\infty}(\xi,\tau)^{R_{\infty}(\xi,\tau)\text{-}\an}[\fa_{\pi}]
\end{equation*}
such that its restriction to $(\widetilde{\pi}_{\alg,\sigma}(D) \otimes_E \varepsilon^{n-1})^{\oplus m}$ does not depend on the choice of $I$ (both $\widetilde{\pi}_{\alg,\sigma}(D)\subset \widetilde{\pi}_{\sigma,I}(D)$ are defined above (\ref{Einjuniv})). The existence of (\ref{injuniv}) will then follow by amalgamating all these injections for $(\sigma,I)\in S^{\flat}(D)$ using (\ref{eq:decopflat}).\bigskip

\noindent \textbf{Step 0: Preliminaries.}\\
Let $(\sigma,I)\in S^{\flat}(D)$, $i:=|I|$ and $\fR$ be a refinement compatible with $I$. As before, renumbering the Frobenius eigenvalues if necessary we assume $\fR=(\varphi_0, \dots, \varphi_{n-1})$ and we let $x_{\fR}$ be the point of $\cE_{\infty}(\xi,\tau)_{\sigma,i}$ associated to $\fR$ in (\ref{EptxR}). Applying the functor $J_B(-)$ to the injection (\ref{Einjlalg}) we deduce an injection (see the comment below (\ref{eq: iota1}))
\begin{equation}\label{ExRfibre1}
\delta_{\fR}^{\oplus m} \hooklongrightarrow J_B\big(\Pi_{\infty}(\xi,\tau)^{R_{\infty}(\xi,\tau)\text{-}\an}[ \fm_{\pi}^{\wp}+\fm_{\pi,\wp}]\big).
\end{equation}

\begin{lem}\label{lem:free}
The coherent sheaf $\cM_{\infty}(\xi,\tau)_{\sigma,i}$ is locally free of rank $m$ at $x_{\fR}$, and the map (\ref{ExRfibre1}) factors as
\begin{equation}\label{ExRfibre}
\delta_{\fR}^{\oplus m} \cong (x_{\fR}^* \cM_{\infty}(\xi,\tau)_{\sigma,i})^{\vee} \hooklongrightarrow J_B\big(\Pi_{\infty}(\xi,\tau)^{R_{\infty}(\xi,\tau)\text{-}\an}[ \fm_{\pi}^{\wp}+\fm_{\pi,\wp}]\big)
\end{equation}
where $x_{\fR}^* \cM_{\infty}(\xi,\tau)_{\sigma,i}$ is the fiber $\cM_{\infty}(\xi,\tau)_{\sigma,i}\otimes_{\cO_{\cE_{\infty}(\xi,\tau)_{\sigma,i}}}\!E$ at $x_{\fR}$.
\end{lem}
\begin{proof}
By (the first statement of) Corollary \ref{C: smooth} the rigid variety $\cE_{\infty}(\xi,\tau)_{\sigma,i}$ is smooth at $x_{\fR}$, hence by (ii) of Proposition \ref{P: Esigmai} the coherent sheaf $\cM_{\infty}(\xi,\tau)_{\sigma,i}$ on $\cE_{\infty}(\xi,\tau)_{\sigma,i}$ is locally free at $x_{\fR}$. Moreover the map (\ref{ExRfibre1}) factors as (see (\ref{Eparabolic}) for $V_{\sigma,i}$)
\begin{multline}\label{mult:later}
\delta_{\fR}^{\oplus m} \cong \delta_{\fR} \otimes_E \Hom_{\GL_n(K)}\big(\pi_{\alg}(D) \otimes_E \varepsilon^{n-1}, \Pi_{\infty}(\xi,\tau)^{R_{\infty}(\xi,\tau)-\an}[ \fm_{\pi}^{\wp}+\fm_{\pi,\wp}]\big)\\
\begin{array}{rl}
\hooklongrightarrow &\delta_{\fR} \otimes_E \Hom_{T(K)}(\delta_{\fR}, (V_{\sigma,i}\otimes_E\varepsilon^n)[ \fm_{\pi}^{\wp}+\fm_{\pi,\wp}]) \cong (x_{\fR}^* \cM_{\infty}(\xi,\tau)_{\sigma,i})^{\vee} \\
\hooklongrightarrow &\Gamma(\cE_{\infty}(\xi,\tau)_{\sigma,i}, \cM_{\infty}(\xi,\tau)_{\sigma,i})^{\vee}[ \fm_{\pi}^{\wp}+\fm_{\pi,\wp}]\cong (V_{\sigma,i}\otimes_E\varepsilon^n)[ \fm_{\pi}^{\wp}+\fm_{\pi,\wp}] \\
\hooklongrightarrow & J_B\big(\Pi_{\infty}(\xi,\tau)^{R_{\infty}(\xi,\tau)\text{-}\an}[ \fm_{\pi}^{\wp}+\fm_{\pi,\wp}]\big),
\end{array}
\end{multline}
where the first injection is induced by taking $J_B(-)$ (see the comment below (\ref{EptxR})), and the two isomorphisms and injections that follow are by definition and from (\ref{Eparabolic}) (note that \emph{loc.~cit.}~implies $V_{\sigma,i}\otimes_E\varepsilon^n\hookrightarrow J_B(\Pi_{\infty}(\xi,\tau)^{\wt(\delta_{\fR})^{\sigma}\text{-}\alg}) \hookrightarrow J_B(\Pi_{\infty}(\xi,\tau)^{R_{\infty}(\xi,\tau)\text{-}\an})$ where $\Pi_{\infty}(\xi,\tau)^{\wt(\delta_{\fR})^{\sigma}\text{-}\alg}$ is defined as in (\ref{Elamsigalg})). Therefore it is enough to show the inequality $\dim_E \Hom_{T(K)}(\delta_{\fR}, V_{\sigma,i}[ \fm_{\pi}^{\wp}+\fm_{\pi,\wp}])\leq m$.\bigskip

Let $\widetilde{\mathfrak{X}}_{\overline{r}}^{\textbf{h}\text{-}\mathrm{cr}}$ be the refined framed weight $\textbf{h}$ crystalline deformation space of $\overline{r}$ constructed as in \cite[\S~2.2]{BHS172} (``weight $\textbf{h}$'' means Hodge-Tate weights $h_{0,\sigma}>h_{1,\sigma}>\cdots>h_{n-1,\sigma}$ for each $\sigma\in \Sigma$ and we drop the $\square$ of the framing in the notation). Recall $\widetilde{\mathfrak{X}}_{\overline{r}}^{\textbf{h}\text{-}\mathrm{cr}}$ parametrizes weight $\textbf{h}$ framed crystalline $\Gal(\overline K/K)$-deformations $r'$ of $\overline{r}$ together with an ordering $\underline{\varphi}':=(\varphi_0', \dots, \varphi_{n-1}')$ of the eigenvalues of $\varphi^{f}$ on $D_{\cris}(r')_{\sigma}$ for one (or equivalently any) $\sigma\in \Sigma$. As in \cite[(2.9)]{BHS172}, there is a natural closed immersion (with the notation of (\ref{eq:dominant}))
\begin{equation*}
\iota_{\textbf{h}}:\widetilde{\mathfrak{X}}_{\overline{r}}^{\textbf{h}\text{-}\mathrm{cr}} \hooklongrightarrow X_{\tri}(\overline{r}), \ (r', \underline{\varphi}') \mapsto (r', \unr(\underline{\varphi}')t^{\textbf{h}}).
\end{equation*}
Let $\mathcal{U}$ be an open smooth neighbourhood of $x_{\fR}$ in $\cE_{\infty}(\xi,\tau)_{\sigma,i}$ such that $\cM_{\infty}(\xi,\tau)_{\sigma,i}|_{\mathcal{U}}$ is free, and $\mathcal{V}$ an open neighbourhood of $x_{\fR}$ in $\cE_{\infty}(\xi,\tau)$ such that $\mathcal {V} \cap \cE_{\infty}(\xi,\tau)_{\sigma,i}\subset \mathcal{U}$. As $(\Spf R_{\infty}^{\wp}(\xi,\tau))^{\rig}$ is smooth at $\fm_{\pi}^{\wp}$ (see the proof of Corollary \ref{C: smooth}), and $X_{\tri}(\overline{r})$ is normal (hence irreducible) at $y_{\fR}$ (cf.~\cite[Thm.~1.5]{BHS19}), by Proposition \ref{prop:pembd} and shrinking $\mathcal{V}$ if needed, we can and do assume $\mathcal{V}$ has the form $\mathcal{V}^{\wp} \times \iota_p^{-1}(\mathcal{V}_{\wp})$ where $\mathcal{V}^{\wp}$ (resp.~$\mathcal{V}_{\wp}$) is an open subset of $(\Spf R_{\infty}^{\wp}(\xi,\tau))^{\rig}$ (resp.~of $X_{\tri}(\overline{r})$). By \cite[Lemma 2.4]{BHS172} (and its proof), there exists $(r',\underline{\varphi}')\in \widetilde{\mathfrak{X}}_{\overline{r}}^{\textbf{h}\text{-}\mathrm{cr}}$ such that $\varphi_j'(\varphi_k')^{-1}\notin \{1, p^f\}$ for $j\neq k$, the refinement $\underline{\varphi}'$ of $r'$ is non-critical, $(r',\underline{\varphi}')$ lies on the same irreducible component of $\widetilde{\mathfrak{X}}_{\overline{r}}^{\textbf{h}\text{-}\mathrm{cr}}$ as $(r,\underline{\varphi})$, and $\iota_{\textbf{h}}((r',\underline{\varphi}'))\in \mathcal{V}_{\wp}$. Let $\fm_{r'}\subset R_{\overline{r}}[1/p]$ be the maximal ideal corresponding to the deformation $r'$ and $\delta':=\unr(\underline{\varphi}')t^{\textbf{h}}\delta_B (\boxtimes_{j=0}^{n-1} \varepsilon^j)\in \widehat T$, we then obtain a non-critical point $x:=(\fm_{\pi}^{\wp}, \fm_{r'}, \delta')\in \mathcal{V}$. We have, noting that $\wt(\delta')=\wt(\delta_{\fR})$ (see the notation above (\ref{eq:W})) and that $\fm_{\pi}^{\wp}+\fm_{r'}$ is a maximal ideal of $R_{\infty}(\xi,\tau)[1/p]$:
\small{
\begin{multline}\label{eq:adx'}
\Hom_{T(K)}\big(\delta', J_B( \Pi_{\infty}(\xi,\tau)^{R_{\infty}(\xi,\tau)\text{-}\an}[\fm_{\pi}^{\wp}+\fm_{r'}])\big)\\
\xlongrightarrow{\sim} \Hom_{\GL_n(K)}\Big(\cF_{B^-}^{\GL_n}((\text{U}(\fg_{\Sigma})\otimes_{\text{U}(\fb_{\Sigma}^-)}(-\wt(\delta_{\fR})))^{\vee}, (\delta')^{\sm}\delta_B^{-1}), \Pi_{\infty}(\xi,\tau)^{R_{\infty}(\xi,\tau)\text{-}\an}[\fm_{\pi}^{\wp}+\fm_{r'}]\Big)\\ \xlongleftarrow{\sim} \Hom_{\GL_n(K)}\big(\cF_{B^-}^{\GL_n}(L(\wt(\delta_{\fR}))^{\vee}, (\delta')^{\sm}\delta_B^{-1}), \Pi_{\infty}(\xi,\tau)^{R_{\infty}(\xi,\tau)\text{-}\an}[\fm_{\pi}^{\wp}+\fm_{r'}]\big)
\end{multline}}
\!\!where the first isomorphism follows from \cite[Thm.~4.3]{Br15} (with $(\delta')^{\sm}:=\unr(\underline{\varphi}')\delta_B(\boxtimes_{j=0}^{n-1} \vert \cdot\vert_K^{j})$) and the second follows from the only if part of \cite[Thm.~5.3.3]{BHS19} with the non-criticality of $x$ (note that $\cF_{B^-}^{\GL_n}(L(\wt(\delta_{\fR}))^{\vee}, (\delta')^{\sm}\delta_B^{-1})$ is locally $\Qp$-algebraic and see also Remark \ref{rem:setting}). Let $R_{\overline{r}}^{\textbf{h}-\rm{cr}}$ be the quotient of $R_{\overline{r}}$ parametrizing weight $\textbf{h}$ framed crystalline deformations of $\overline r$ as in \cite[\S~2.2]{BHS172}. By an easy variation of \cite[Prop.~4.34]{CEGGPS16} applied to the points
\[(\fm_{\pi}^{\wp},\fm_r), (\fm_{\pi}^{\wp},\fm_{r'}) \in \Spec \big((R^{\wp}_{\infty}(\xi,\tau) \widehat{\otimes}_{\cO_E} R_{\overline{r}}^{\textbf{h}-\rm{cr}})[1/p]\big)\]
(note that $\fm_r$ is denoted $\fm_{\pi,\wp}$ in (\ref{EptxR}) and that the right hand side is the analogue in our case of the ring $\Spec (R_{\infty}(\lambda)'[1/p])$ of \emph{loc.~cit.}), we deduce
\begin{multline}\label{eq:algfree}
\dim_E \Hom_{\GL_n(K)}\big(\cF_{B^-}^{\GL_n}(L(\wt(\delta_{\fR}))^{\vee}, (\delta')^{\sm}\delta_B^{-1}), \Pi_{\infty}(\xi,\tau)^{R_{\infty}(\xi,\tau)\text{-}\an}[\fm_{\pi}^{\wp}+\fm_{r'}]\big)\\
=\dim_E \Hom_{\GL_n(K)}\big(\cF_{B^-}^{\GL_n}(L(\wt(\delta_{\fR}))^{\vee}, \delta_{\fR}^{\sm}\delta_B^{-1}), \Pi_{\infty}(\xi,\tau)^{R_{\infty}(\xi,\tau)\text{-}\an}[ \fm_{\pi}^{\wp}+\fm_{\pi,\wp}]\big)=m.
\end{multline}
Indeed, these two points are smooth on $\Spec ((R^{\wp}_{\infty}(\xi,\tau) \widehat{\otimes}_{\cO_E} R_{\overline{r}}^{\textbf{h}-\rm{cr}})[1/p])$ (see the argument in the proof of Corollary \ref{C: smooth}), lie on the same irreducible component (using \cite[Rk.~2.6(i)]{BHS172} and \cite[Thm.~2.3.1]{Co99}), and are automorphic. Here ``automorphic" means that these points lie in the support of the following patched module (which easily follows from (\ref{Einjlalg}) and the fact the $E$-vector space in (\ref{eq:adx'}) is non-zero):
\[\Big(\Hom_{\cO_E[[\GL_n(\cO_K)]]}^{\rm{cont}}\big(M_{\infty}(\xi,\tau), (L(\wt(\delta_{\fR}))^0)^{\vee}\big)\Big)^{\vee}\]
 where $M_{\infty}(\xi,\tau)$ is defined in (\ref{eq:patched}), $L(\wt(\delta_{\fR}))^0$ is a $\GL_n(\cO_K)$-invariant lattice in $L(\wt(\delta_\fR))$ and $\vee$ denotes the Schikoff dual (see \cite[\S~4.28]{CEGGPS16} for details).\bigskip
 
We have an isomorphism of non-zero $E$-vector spaces using (\ref{eq:algfree}) and since unramified principal series have $1$-dimensional $\GL_n(\cO_K)$-invariants (recall $\delta_{\sm}'$ is unramified): 
\begin{multline*}
\Hom_{\GL_n(K)}\Big(\cF_{B^-}^{\GL_n}(L(\wt(\delta_{\fR}))^{\vee}, (\delta')^{\sm}\delta_B^{-1}), \Pi_{\infty}(\xi,\tau)^{R_{\infty}(\xi,\tau)\text{-}\an}[\fm_{\pi}^{\wp}+\fm_{r'}]\Big)\\
\xlongrightarrow{\sim} \Hom_{\GL_n(\cO_K)}\big(L(\wt(\delta_{\fR})), \Pi_{\infty}(\xi,\tau)^{R_{\infty}(\xi,\tau)\text{-}\an}[\fm_{\pi}^{\wp}+\fm_{r'}]^{\Qp\text{-}\alg}\big).
\end{multline*}
This implies $x\in \cE_{\infty}(\xi,\tau)_{\sigma,i}$ by the definition of $V_{\sigma,i}$ in (\ref{Eparabolic}), and hence $x\in \mathcal{U}$. Moreover we have (as $\cM_{\infty}(\xi,\tau)_{\sigma,i}$ is a quotient of $\cM_{\infty}(\xi,\tau)$)
\[\dim_E x^* \cM_{\infty}(\xi,\tau)_{\sigma,i} \leq \dim_E x^* \cM_{\infty}(\xi,\tau)=m\]
where the last equality follows from (\ref{eq:adx'}) and (\ref{eq:algfree}).
Since $\cM_{\infty}(\xi,\tau)_{\sigma,i}|_{\mathcal{U}}$ is free and $x_{\fR},x\in \mathcal{U}$, we deduce $\dim_E \Hom_{T(K)}(\delta_{\fR}, V_{\sigma,i}[ \fm_{\pi}^{\wp}+\fm_{\pi,\wp}])=\dim_E x_{\fR}^*\cM_{\infty}(\xi,\tau)_{\sigma,i}=\dim_E x^*\cM_{\infty}(\xi,\tau)_{\sigma,i}\leq m$. This gives the required upper bound.
\end{proof}

Note that, as a consequence of the first assertion of Lemma \ref{lem:free}, the first injection in (\ref{mult:later}) is an isomorphism.\bigskip

Let $\widehat{\cO}_{x_{\fR}}$ be the completion of $\cE_{\infty}(\xi,\tau)_{\sigma,i}$ at the point $x_{\fR}$. The morphism of $E$-algebras
\[R_{\infty}(\xi,\tau)[1/p]\longrightarrow \Gamma\big(\cE_{\infty}(\xi,\tau), \cO_{\cE_{\infty}(\xi,\tau)}\big)\longrightarrow \Gamma\big(\cE_{\infty}(\xi,\tau)_{\sigma,i}, \cO_{\cE_{\infty}(\xi,\tau)_{\sigma,i}}\big)\longrightarrow \widehat{\cO}_{x_{\fR}}\]
induces a morphism of local complete $E$-algebras
\begin{equation}\label{eq:surjRr}
R_r \longrightarrow R_{\infty}(\xi,\tau)[1/p]/\fm_{\pi}^{\wp} \longrightarrow \widehat{\cO}_{x_{\fR}}/\fm_{\pi}^{\wp}.
\end{equation}
Let $\widehat{\cO}_{x_{\fR},\wp}:=\widehat{\cO}_{x_{\fR}}/\fm_{\pi}^{\wp}$, it follows from (\ref{eq:inclWu}) that we have isomorphisms $(\fm_{\widehat{\cO}_{x_{\fR}, \wp}}/\fm^2_{\widehat{\cO}_{x_{\fR}, \wp}})^{\vee}\cong X_{r,\fR, \sigma,i}^{w_0}(E[\epsilon]/\epsilon^2)$ and $\widehat{\cO}_{x_{\fR},\wp}/\fm^2_{\widehat{\cO}_{x_{\fR}, \wp}}\cong R_{r,\fR, \sigma,i}^{w_0}/\fm^2_{R_{r,\fR, \sigma,i}^{w_0}}$ (see the proof of Proposition \ref{prop:Psmoothbis} for $R_{r,\fR, \sigma,i}^{w_0}$). Since $R_{r,\fR, \sigma,i}^{w_0}$ is a quotient of $R_{r,\fR}^{w_0}$ which is a quotient of $R_r$ (see above (\ref{modeltri})), (\ref{eq:surjRr}) induces a surjection $R_r/\fm_{R_r}^2\twoheadrightarrow \widehat{\cO}_{x_{\fR},\wp}/\fm^2_{\widehat{\cO}_{x_{\fR}, \wp}}$, and hence (\ref{eq:surjRr}) is surjective (as both local rings are complete). Moreover, it follows from the discussion above Proposition \ref{prop:Psmoothbis} that $X_{r,\fR, \sigma,i}^{w_0}(E[\epsilon]/\epsilon^2)$ is the preimage of $\ol{\Ext}^1_{\fR,w_0,\sigma,i}(\cM(D), \cM(D))\subset \ol{\Ext}^1_{(\varphi, \Gamma)}(\cM(D), \cM(D))$ via (\ref{Esurj1}). Using (\ref{eq:ADSRi}), we deduce that the splitting (\ref{eq:aDsplit}) induces a splitting:
\begin{equation*}
\fa_D/\fm_{R_r}^2\ \bigoplus \ \fm_{A_{D,\fR, \sigma, i}} \buildrel\sim\over\longrightarrow \fm_{\widehat{\cO}_{x_{\fR}, \wp}}/\fm^2_{\widehat{\cO}_{x_{\fR}, \wp}},
\end{equation*}
from which we deduce an isomorphism of local Artinian $E$-algebras
\begin{equation}\label{EtangE}
A_{D,\fR, \sigma, i} \xlongrightarrow{\sim} \widehat{\cO}_{x_{\fR}, \wp}/\fa_D
\end{equation}
(still denoting $\fa_D$ the image of the ideal $\fa_D\subset R_r$ in $\widehat{\cO}_{x_{\fR}, \wp}$).\bigskip

\noindent \textbf{Step 1: Non-critical case.}\\
We first assume that $s_{i,\sigma}$ does not appear in some (equivalently any) reduced expression of $w_{\fR,\sigma} w_{0,\sigma}$. Recall the tangent space of $\widehat{T}$ at the point $\delta_{\fR}$ is isomorphic to $\Hom(T(K),E)$. We let $\fa_{\sigma,i}\supset \fm_{\cO_{\widehat{T}, \delta_{\fR}}}^2$ to be the ideal of $\cO_{\widehat{T},\delta_{\fR}}$ associated to the subspace \[\Hom_{\sigma,i}(T(K),E) \subset \Hom(T(K),E)\cong \big(\fm_{\cO_{\widehat{T}, \delta_{\fR}}}/ \fm_{\cO_{\widehat{T}, \delta_{\fR}}}^2\big)^{\vee},\]
that is, we have $\Hom_{\sigma,i}(T(K),E)\cong (\fm_{\cO_{\widehat{T}, \delta_{\fR}}}/ \fa_{\sigma,i})^{\vee}$. It follows from Corollary \ref{C: smooth} (in particular the last statement) with (\ref{Eisom11}) that the natural map of noetherian local complete $E$-algebras $\widehat{\cO}_{\widehat{T}, \delta_{\fR}}\!\rightarrow \widehat{\cO}_{x_{\fR}, \wp}/\fa_D$ factors through $\widehat{\cO}_{\widehat{T}, \delta_{\fR}}/\fa_{\sigma,i} \!\rightarrow \widehat{\cO}_{x_{\fR}, \wp}/\fa_D$ which by (\ref{EtangE}) and (\ref{eq:ADSRi}) is an isomorphism. In particular $\widehat{\cO}_{x_{\fR}, \wp}/\fa_D$ is a $\widehat{\cO}_{\widehat{T}, \delta_{\fR}}/\fa_{\sigma,i}$-module and the natural map $T(K) \rightarrow \widehat{\cO}_{\widehat{T}, \delta_{\fR}}/\fa_{\sigma,i}$ then endows $\widehat{\cO}_{x_{\fR}, \wp}/\fa_D$ with a $T(K)$-action. As $\widehat{\cO}_{\widehat{T}, \delta_{\fR}}$ is the universal deformation ring of $\delta_{\fR}$, it follows from the definition of $\widetilde{\delta}_{\fR, \sigma, i}$ (see above (\ref{eq:jpcan})) that we obtain a $T(K)$-equivariant isomorphism $\big(\widehat{\cO}_{x_{\fR}, \wp}/\fa_D\big)^{\vee}\buildrel\sim\over\rightarrow \widetilde{\delta}_{\fR, \sigma, i} $. Moreover, using the statements below (\ref{diag: Esigmai}) and a similar discussion as above \cite[Lemma 4.3]{Di25}, unwinding the actions we can check that this isomorphism is $A_{D,\fR, \sigma,i}$-equivariant, where $A_{D,\fR, \sigma,i}$ acts on the left via (\ref{EtangE}) (and its natural action on $A_{D,\fR, \sigma,i}^\vee$) and on the right as in the discussion above (\ref{eq:jpcan}).\bigskip

From the definition of $\widehat{\cO}_{x_{\fR}}$, we have a closed immersion of rigid spaces (recall $\fa_{\pi}=\fa_D+\fm_{\pi}^{\wp}$ and $\widehat{\cO}_{x_{\fR},\wp}/\fa_D$ is finite dimensional)
\begin{equation}\label{Etildex}
\widetilde x_{\fR}:\Spec \big(\widehat{\cO}_{x_{\fR},\wp}/\fa_D\big)\cong \Spec \big(\widehat{\cO}_{x_{\fR}}/\fa_{\pi}\big) \hooklongrightarrow \cE_{\infty}(\xi,\tau)_{\sigma,i}
\end{equation}
and we define $\cM_{\fR,\sigma,i}:=\widetilde x_{\fR}^*\cM_{\infty}(\xi,\tau)_{\sigma,i}$. By Lemma \ref{lem:free} and the previous paragraph we deduce a $T(K)\times A_D$-equivariant isomorphism
\begin{equation*}
(\cM_{\fR,\sigma,i})^{\vee} \cong \Big(\big(\widehat{\cO}_{x_{\fR}}/\fa_{\pi}\big)^\vee\Big)^{\oplus m} \cong \ \widetilde{\delta}_{\fR,\sigma,i}^{\oplus m}.
\end{equation*}
Moreover we have $T(K) \times A_D$-equivariant injections
\begin{multline}\label{injJac}
\widetilde{\delta}_{\fR,\sigma,i}^{\oplus m} \cong (\cM_{\fR,\sigma,i})^{\vee} \hooklongrightarrow (V_{\sigma,i}\otimes_E\varepsilon^n)[\fa_{\pi}] \hooklongrightarrow J_B\big(\Pi_{\infty}(\xi,\tau)^{\wt(\delta_{\fR})^{\sigma}\text{-}\alg}[\fa_{\pi}]\big) \\
\hooklongrightarrow J_B\big(\Pi_{\infty}(\xi,\tau)^{R_{\infty}(\xi,\tau)\text{-}\an}[\fa_{\pi}]\big)
\end{multline}
where the first injection follows from (\ref{eq:sections}) (compare with \cite[(5.17)]{BHS19} for instance) and the two others follow from (\ref{Eparabolic}) (and the discussion after (\ref{mult:later})).\bigskip

Consider first the minimal closed $L_{P_i}(K)$-subrepresentation containing $\widetilde{\delta}_{\fR,\sigma,i}$ (see (\ref{eq:jpcan}))
\[I_{B^-\cap L_{P_i}}^{L_{P_i}}(\widetilde{\delta}_{\fR,\sigma,i} \delta_{B\cap L_{P_i}}^{-1})\subset \big(\Ind_{B^-(K)\cap L_{P_i}(K)}^{L_{P_i}(K)} \widetilde{\delta}_{\fR,\sigma,i} \delta_{B\cap L_{P_i}}^{-1}\big)^{\Qp\text{-}\an}.\]
From the definitions of $\widetilde{\delta}_{\fR,\sigma,i}$ above (\ref{eq:jpcan}) and $\Hom_{\sigma,i}(T(K),E)$ in (\ref{eqtisigma}), similarly as in (\ref{isoIPG2}) the representation $I_{B^-\cap L_{P_i}}^{L_{P_i}}(\widetilde{\delta}_{\fR,\sigma,i} \delta_{B\cap L_{P_i}}^{-1})$ is isomorphic to the closed subrepresentation of $(\Ind_{B^-(K)\cap L_{P_i}(K)}^{L_{P_i}(K)} \widetilde{\delta}_{\fR,\sigma,i} \delta_{B\cap L_{P_i}}^{-1})^{\Qp\text{-}\an}$ of locally $\Qp$-algebraic vectors up to twist. By \cite[Prop.~2.14]{Di18}, which generalizes to the case where the representation $\pi \otimes_E \cL_1(\lambda)$ of \emph{loc.~cit.}~is of finite length and locally $\Qp$-algebraic up to twist, we have an isomorphism (using the first isomorphism in (\ref{Eparabolic}))
{\scriptsize
\begin{multline*}
\Hom_{T(K)}\big(\widetilde{\delta}_{\fR,\sigma,i}, (V_{\sigma,i}\otimes_E\varepsilon^n)[\fa_{\pi}] \big) \cong \Hom_{L_{P_i}(K)} \bigg(I_{B^-\cap L_{P_i}}^{L_{P_i}}(\widetilde{\delta}_{\fR,\sigma,i} \delta_{B\cap L_{P_i}}^{-1}), \\
\Big(J_{P_i}\Big(\big(\Pi_{\infty}(\xi,\tau)^{R_{\infty}(\xi,\tau)\text{-}\an}[\fa_{\pi}] \otimes_E (\otimes_{\tau\neq \sigma} L(\lambda_{\tau})^{\vee})\big)^{\sigma\text{-}\an}\Big) \otimes_E L_i(\lambda_{\sigma})^{\vee}\Big)^{\fl_{P_i}'}\otimes_E (\otimes_{\tau\in \Sigma} L_{i}(\lambda_{\tau}))\otimes_E\varepsilon^n\bigg).
\end{multline*}}
\!\!Indeed, a key ingredient of \cite[Prop.~2.14]{Di18} is that the $L_{P_2}$-representation $V_{\lambda_0}$ of \emph{loc.~cit.}~consists of locally algebraic vectors up to twist, which stays true above. In particular the first injection in (\ref{injJac}) induces an $L_{P_i}(K) \times A_D$-equivariant injection
{\small
\begin{multline}\label{injJacPi}
I_{B^-\cap L_{P_i}}^{L_{P_i}}(\widetilde{\delta}_{\fR,\sigma,i} \delta_{B\cap L_{P_i}}^{-1})\hooklongrightarrow \\
\bigg(J_{P_i}\Big(\big(\Pi_{\infty}(\xi,\tau)^{R_{\infty}(\xi,\tau)\text{-}\an}[\fa_{\pi}]\otimes_E (\otimes_{\tau\neq \sigma} L(\lambda_{\tau})^{\vee})\big)^{\sigma\text{-}\an}\Big)\otimes_E L_i(\lambda_{\sigma})^{\vee}\bigg)^{\fl_{P_i}'}\otimes_E (\otimes_{\tau\in \Sigma} L_{i}(\lambda_{\tau}))\otimes_E\varepsilon^n\\
\hooklongrightarrow J_{P_i}\big(\Pi_{\infty}(\xi,\tau)^{\wt(\delta_{\fR})^{\sigma}\text{-}\alg}[\fa_{\pi}]\big)
\end{multline}}
\!\!where the second injection in (\ref{injJacPi}) follows from the injection (\ref{EPiclass}) with the isomorphism (\ref{eq:Ding}). Note that the first map in (\ref{injJacPi}) is indeed injective (and not just non-zero) because taking $J_{B\cap L_{P_i}}(-)$ one reobtains from it the first \emph{injection} in (\ref{injJac}), and this easily implies that the first map in (\ref{injJacPi}) also has to be injective.

\begin{lem}\label{lem:balanced}
The composition (\ref{injJacPi}) is balanced in the sense of \cite[Def.~0.8]{Em07}.
\end{lem}
\begin{proof}
We use the equivalent definition \cite[Def.~5.17]{Em072}. For simplicity we write $W:=I_{B^-\cap L_{P_i}}^{L_{P_i}}(\widetilde{\delta}_{\fR,\sigma,i} \delta_{B\cap L_{P_i}}^{-1})$. As $W$ is locally $\Qp$-algebraic up to twist, it is isomorphic to extensions of the (irreducible) locally $\Qp$-algebraic representation $W_0:=I_{B^-\cap L_{P_i}}^{L_{P_i}}(\delta_{\fR} \delta_{B\cap L_{P_i}}^{-1})$ by itself. Moreover one checks $W_0\cong L_i(\wt(\delta_{\fR})) \otimes_E W_0^{\sm}$ where $W_0^{\sm}=(\Ind_{B^-(K)\cap L_{P_i}(K)}^{L_{P_i}(K)}\delta_{\fR}^{\sm} \delta_{B\cap L_{P_i}}^{-1})^{\infty}$ (recall $\delta_{\fR}^{\sm}=\unr(\underline{\varphi})\delta_B$) and $L_i(\wt(\delta_{\fR}))$ is the irreducible algebraic representation of $(L_{P_i})_{\Sigma}$ over $E$ of highest weight $\wt(\delta_{\fR})$ with respect to the upper Borel (see the notation above (\ref{eq:W})). We have a natural $(\fg_{\Sigma}, P_i(K))$-equivariant morphism (we refer to \cite[\S~5]{Em072} for the actions)
\begin{equation}\label{EtensorG}
\text{U}(\fg_{\Sigma}) \otimes_{\text{U}((\fp_i)_{\Sigma})} \cC^{\sm}_c(N_{P_i}(K),W\otimes_E \delta_{P_i}^{-1}) \longrightarrow I_{P_i^-}^{\GL_n} (W\otimes_E \delta_{P_i}^{-1})
\end{equation}
where $\cC^{\sm}_c$ means as usual locally constant functions with compact support. By similar arguments as in the proof of \cite[Lemma 4.11]{Di191}, one can show that the kernel of (\ref{EtensorG}) admits a $(\fg_{\Sigma},P_i(K))$-equivariant filtration with graded pieces of the form
\[L(w \cdot \wt(\delta_{\fR})) \otimes_E \cC^{\sm}_c(N_{P_i}(K), W_0^{\sm} \otimes_E \delta_{P_i}^{-1})\]
such that $L(w\cdot \wt(\delta_{\fR}))$ is an irreducible constituent of the generalized Verma module $\text{U}(\fg_{\Sigma}) \otimes_{\text{U}((\fp_i)_{\Sigma})} L_i(\wt(\delta_{\fR}))$ with $w=(w_{\tau})_{\tau\in \Sigma}\in S_n^{\Sigma}\setminus\{1\}$ (the dot action is as in (\ref{eq:dotaction}) for each $\tau\in \Sigma$). Indeed, to generalize the arguments of \textit{loc.~cit.}~(which concerns $(\fg_{\Sigma}, B(K))$-modules) to our case, we only need to show that any $(\fg_{\Sigma}, P_i(K))$-submodule of
\[\text{U}(\fg_{\Sigma}) \otimes_{\text{U}((\fp_i)_{\Sigma})} (W_0 \otimes_E \delta_{P_i}^{-1}) \cong \big(\text{U}(\fg_{\Sigma}) \otimes_{\text{U}((\fp_i)_{\Sigma})} L_i(\wt(\delta_{\fR}))\big)\otimes_E (W_0^{\sm} \otimes_E \delta_{P_i}^{-1})\]
admits a $(\fg_{\Sigma},P_i(K))$-equivariant filtration whose graded pieces are of the form
\[L(w\cdot \wt(\delta_{\fR})) \otimes_E (W_0^{\sm} \otimes_E \delta_{P_i}^{-1}).\]
This is clear for $\text{U}(\fg_{\Sigma}) \otimes_{\text{U}((\fp_i)_{\Sigma})} (W_0 \otimes_E \delta_{P_i}^{-1}) $ itself. Using the easy fact that $L(w\cdot \wt(\delta_{\fR})) \otimes_E (W_0^{\sm} \otimes_E \delta_{P_i}^{-1})$ is irreducible as a $(\fg_{\Sigma}, P_i(K))$-module for any $w$, the filtration on $\text{U}(\fg_{\Sigma}) \otimes_{\text{U}((\fp_i)_{\Sigma})} (W_0 \otimes_E \delta_{P_i}^{-1})$ then induces a filtration of the same form on any of its $(\fg_{\Sigma}, P_i(K))$-submodules. By \cite[Def.~5.17]{Em072}, to prove the lemma it suffices to show 
\begin{equation}\label{Ehomzero}
\Hom_{(\fg_{\Sigma}, P_i(K))}\Big(L(w \cdot \wt(\delta_{\fR})) \otimes_E \cC^{\sm}_c(N_{P_i}(K), W_0^{\sm} \otimes_E \delta_{P_i}^{-1}),\Pi_{\infty}(\xi,\tau)^{\wt(\delta_{\fR})^{\sigma}\text{-}\alg}[\fa_{\pi}]\Big)=0
\end{equation}
for all such $w$. By \cite[Prop.~4.2]{Br15}, the Hom on the left hand side of (\ref{Ehomzero}) is isomorphic to
\begin{equation}\label{Eadjpre}
\Hom_{\GL_n(K)}\Big(\cF_{P_i^-}^{\GL_n}\big(L^-(-w \cdot \wt(\delta_{\fR})), W_0^{\sm} \otimes_E \delta_{P_i}^{-1}\big), \Pi_{\infty}(\xi,\tau)^{\wt(\delta_{\fR})^{\sigma}\text{-}\alg}[\fa_{\pi}]\Big).
\end{equation}
Assume first that there is $\tau\ne \sigma$ such that $w_\tau\ne 1$, then by (\ref{Elamsigalg}) and comparing the $\fg_\tau$-actions on both sides of the Hom in (\ref{Ehomzero}), we see that each copy of $L(w \cdot \wt(\delta_{\fR}))$ maps to $0$, hence (\ref{Ehomzero}) holds for such $w$ (and (\ref{Eadjpre}) is not needed). Assume now $w_\tau= 1$ for $\tau\ne \sigma$, then we have $w_\sigma\neq 1$. Since $L(w\cdot \wt(\delta_{\fR}))$ is a constituent of $\text{U}(\fg_{\Sigma}) \otimes_{\text{U}((\fp_i)_{\Sigma})} L_i(\wt(\delta_{\fR}))$, it follows from \cite[\S~5.1]{Hu08} with \cite[Thm.~9.4(b)]{Hu08} that we have $s_{i,\sigma}\leq w$, or equivalently $s_{i,\sigma}\leq w_\sigma$. However, as $s_{i,\sigma}\nleq w_{\fR,\sigma}w_{0,\sigma}$ by assumption, this implies $w\nleq w_{\fR}w_{0}$ or equivalently $w_{\fR}\nleq ww_0$ and by \cite[Thm.~5.3.3]{BHS19} (with Remark \ref{rem:setting}) we have
\[\Hom_{\GL_n(K)}\Big(\cF_{P_i^-}^{\GL_n}\big(L^-(-w \cdot \wt(\delta_{\fR})), W_0^{\sm} \otimes_E \delta_{P_i}^{-1}\big), \Pi_{\infty}(\xi,\tau)^{R_{\infty}(\xi,\tau)\text{-}\an}[ \fm_{\pi}^{\wp}+\fm_{\pi,\wp}]\Big)=0.\]
Since the action of $\fm_{A_D}$ on $\Pi_{\infty}(\xi,\tau)^{R_{\infty}(\xi,\tau)\text{-}\an}[\fa_{\pi}]$ is $\GL_n(K)$-equivariant and nilpotent, it follows \ that \ $\cF_{P_i^-}^{\GL_n}\big(L^-(-w \cdot \wt(\delta_{\fR})), W_0^{\sm} \otimes_E \delta_{P_i}^{-1}\big)$ \ also \ cannot \ be \ a \ subrepresentation \ of $\Pi_{\infty}(\xi,\tau)^{R_{\infty}(\xi,\tau)\text{-}\an}[\fa_{\pi}]$ (use that a nilpotent endomorphism on a non-zero vector space always has a non-zero kernel). We then deduce from (\ref{Elamsigalg}) that (\ref{Eadjpre}) is zero and thus (\ref{Ehomzero}) again holds.
\end{proof}

By \cite[Thm.~0.13]{Em07}, the composition (\ref{injJacPi}) induces a $\GL_n(K) \times A_D$-equivariant morphism 
\begin{multline}\label{eq:adnc}
I_{B^-}^{\GL_n} (\widetilde{\delta}_{\fR,\sigma, i}\delta_B^{-1}) ^{\oplus m} \cong I_{P_i^-}^{\GL_n} \big( (I_{B^-\cap L_{P_i}}^{L_{P_i}}(\widetilde{\delta}_{\fR,\sigma,i} \delta_{B\cap L_{P_i}}^{-1})) \otimes_E\delta_{P_i}^{-1}\big)^{\oplus m}\\
\longrightarrow \Pi_{\infty}(\xi,\tau)^{\wt(\delta_{\fR})^{\sigma}\text{-}\alg}[\fa_{\pi}]\hookrightarrow \Pi_{\infty}(\xi,\tau)^{R_{\infty}(\xi,\tau)\text{-}\an}[\fa_{\pi}]
\end{multline}
where the first isomorphism follows by definition and from the transitivity of parabolic induction (both representations there coincide with the minimal closed subrepresentation of $\big((\Ind_{B^-(K)}^{\GL_n(K)}(\widetilde{\delta}_{\fR,\sigma, i}\delta_B^{-1}))^{\Q_p\text{-}\an}\big)^{\!\oplus m}$ generated by $\widetilde{\delta}_{\fR,\sigma, i}^{\oplus m}$ via (\ref{eq:jpcan})). Note that the composition in (or equivalently the first map in) (\ref{eq:adnc}) is moreover injective since its restriction to
\[\soc_{\GL_n(K)}\big(I_{B^-}^{\GL_n} (\widetilde{\delta}_{\fR,\sigma, i}\delta_B^{-1}) ^{\oplus m}\big)	\cong I_{B^-}^{\GL_n} (\delta_{\fR,\sigma, i}\delta_B^{-1}) ^{\oplus m} \cong (\pi_{\alg}(D) \otimes_E \varepsilon^{n-1})^{\oplus m}\]
(see Lemma \ref{Luniv1} with the definition of $\widetilde{\pi}_{\sigma,I}(D)$ above (\ref{Einjuniv})) coincides with (\ref{Einjlalg}). Indeed, the restriction of the composition in (\ref{injJac}) to the subspace $(x_{\fR}^* \cM_{\infty}(\xi,\tau)_{\sigma,i})^{\vee}$ (the dual of the fiber of $\cM_{\infty}(\xi,\tau)_{\sigma,i}$ at $x_{\fR}$) coincides with (\ref{ExRfibre}), and by the above argument applied to (\ref{ExRfibre}) instead of (\ref{injJacPi}) we recover the injection (\ref{Einjlalg}). Now, by Lemma \ref{Luniv1} again we finally deduce from (\ref{eq:adnc}) a $\GL_n(K) \times A_D$-equivariant injection extending (\ref{Einjlalg})
\begin{equation}\label{Eiogtasi}
\iota_{\sigma,I}:	(\widetilde{\pi}_{\sigma,I}(D)\otimes_E \varepsilon^{n-1})^{\oplus m} \hooklongrightarrow \Pi_{\infty}(\xi,\tau)^{R_{\infty}(\xi,\tau)\text{-}\an}[\fa_{\pi}].
\end{equation}

\noindent \textbf{Step 2: Critical case.}\\
We now assume that $s_{i,\sigma}$ appears with multiplicity $1$ in some reduced expression of $w_{\fR,\sigma} w_{0,\sigma}$. By the same reasoning as in the beginning of Step 1 replacing (\ref{Eisom11}) by (\ref{Esurj12}) composed with the injection $\Hom_{\sigma,0}(T(K),E)\hookrightarrow \Hom_{\sigma,i}(T(K),E)$, the natural map $\widehat{\cO}_{\widehat{T}, \delta_{\fR}} \rightarrow \widehat{\cO}_{x_{\fR}, \wp}/\fa_D$ still factors through $\widehat{\cO}_{\widehat{T}, \delta_{\fR}}/\fa_{\sigma,i} \rightarrow \widehat{\cO}_{x_{\fR}, \wp}/\fa_D$. However, this morphism now is not an isomorphism. Indeed, consider the composition (see (\ref{eq:BDSRi}) for $B_{D,\fR,\sigma,i}$ and (\ref{eq:ADSRi}) for $A_{D,\fR,\sigma,i}$)
\begin{equation}\label{EcriT}
\widehat{\cO}_{\widehat{T}, \delta_{\fR}}/\fa_{\sigma,i} \longrightarrow \widehat{\cO}_{x_{\fR}, \wp}/\fa_D {\buildrel (\ref{EtangE}) \over \cong} A_{D,\fR,\sigma,i} \twoheadlongrightarrow B_{D,\fR,\sigma,i}.
\end{equation}
The induced map $(\fm_{B_{D,\fR,\sigma,i}})^{\vee} \rightarrow (\fm_{\cO_{\widehat{T}, \delta_{\fR}}}/ \fa_{\sigma,i})^{\vee}\cong \Hom_{\sigma,i}(T(K),E)$ is zero by definition of $B_{D,\fR,\sigma,i}$, hence (\ref{EcriT}) factors through an injection
\begin{equation*}
\widehat{\cO}_{\widehat{T}, \delta_{\fR}}/\fm_{\cO_{\widehat{T}, \delta_{\fR}}} \hooklongrightarrow B_{D,\fR,\sigma,i}.
\end{equation*}
It follows that the $T(K)$-action on $B_{D,\fR,\sigma,i}$ induced by the natural map $T(K) \rightarrow \widehat{\cO}_{\widehat{T}, \delta_{\fR}}$ and the $\widehat{\cO}_{\widehat{T}, \delta_{\fR}}$-module structure of $B_{D,\fR,\sigma,i}$ is just the multiplication by the character $\delta_{\fR}$. Since $\dim_E\fm_{B_{D,\fR,\sigma,i}}=\dim_E E_{\fR,\sigma,i}=1$ (see (\ref{eq:BDSRi})), we deduce a $T(K)$-equivariant isomorphism $(B_{D,\fR,\sigma,i})^{\vee} \cong \delta_{\fR}^{\oplus 2}$. We fix a ($T(K)$-equivariant) injection $\jmath: \delta_{\fR}\hookrightarrow (B_{D,\fR,\sigma,i})^{\vee}$ such that we have an isomorphism (of $T(K)$-representations):
\begin{equation}\label{EdecoB}
 (B_{D,\fR,\sigma,i})^{\vee}\cong (B_{D,\fR,\sigma,i}/\fm_{B_{D,\fR,\sigma,i}})^{\vee} \bigoplus \Ima(\jmath).
\end{equation}
Note that a non-zero element in $\fm_{B_{D,\fR,\sigma,i}}$ cancels $(B_{D,\fR,\sigma,i}/\fm_{B_{D,\fR,\sigma,i}})^{\vee}$ and sends $\Ima(\jmath)$ onto $(B_{D,\fR,\sigma,i}/\fm_{B_{D,\fR,\sigma,i}})^{\vee}$ (in particular $\Ima(\jmath)$ is not stabilized by $B_{D,\fR,\sigma,i}$).\bigskip

The point $x_{\fR}: \Spec E \hookrightarrow\cE_{\infty}(\xi,\tau)_{\sigma,i}$ factors as (see (\ref{Etildex}) for $\widetilde{x}_{\fR}$)
\begin{multline*}
x_{\fR}:\Spec E \cong \Spec \big(B_{D, \fR, \sigma, i}/\fm_{B_{D,\fR,\sigma,i}}\big)\hooklongrightarrow \Spec B_{D, \fR, \sigma, i}\\
\hooklongrightarrow \Spec A_{D,\fR,\sigma,i}\buildrel (\ref{EtangE})\over \cong \Spec \big( \widehat{\cO}_{x_{\fR}, \wp}/\fa_{D}\big) \buildrel {\widetilde{x}_{\fR}} \over \hooklongrightarrow \cE_{\infty}(\xi,\tau)_{\sigma,i}.
\end{multline*}
Let $\widetilde{x}_{\fR,1}:\Spec B_{D, \fR, \sigma, i} \hooklongrightarrow \cE_{\infty}(\xi,\tau)_{\sigma,i}$ be the composition of the last two maps, then $\widetilde{x}_{\fR,1}^* \cM_{\infty}(\xi,\tau)_{\sigma,i}\cong (B_{D,\fR,\sigma,i})^{\oplus m}$ as $\cM_{\infty}(\xi,\tau)_{\sigma,i}$ is locally free of rank $m$ at $x_{\fR}$ (Lemma \ref{lem:free}). Similarly as in (\ref{injJac}) and by the discussion in the previous paragraph, we have $T(K)\times A_D$-equivariant injections
\begin{multline}\label{EthickB}
(x_{\fR}^*\cM_{\infty}(\xi,\tau)_{\sigma,i})^{\vee} \big(\cong \delta_{\fR}^{\oplus m}\big)\hooklongrightarrow (\widetilde{x}_{\fR,1}^* \cM_{\infty}(\xi,\tau)_{\sigma,i})^{\vee} \big(\cong (B_{D,\fR,\sigma,i})^{\vee, \oplus m}\big) \\
\hooklongrightarrow (V_{\sigma,i}\otimes_E\varepsilon^n)[\fa_{\pi}] \hooklongrightarrow J_B\big(\Pi_{\infty}(\xi,\tau)^{\wt(\delta_{\fR})^{\sigma}\text{-}\alg}[\fa_{\pi}]\big)\\
\hooklongrightarrow J_B(\Pi_{\infty}(\xi,\tau)^{R_{\infty}(\xi,\tau)\text{-}\an}[\fa_{\pi}])
\end{multline}
where the composition has image in $J_B(\Pi_{\infty}(\xi,\tau)^{R_{\infty}(\xi,\tau)\text{-}\an}[ \fm_{\pi}^{\wp}+\fm_{\pi,\wp}])$ and coincides with (\ref{ExRfibre}). \ Now \ choose \ a \ direct \ summand \ $\iota : (B_{D,\fR,\sigma,i})^{\vee} \hookrightarrow (\widetilde{x}_{\fR,1}^* \cM_{\infty}(\xi,\tau)_{\sigma,i})^{\vee}$ \ of $(\widetilde{x}_{\fR,1}^* \cM_{\infty}(\xi,\tau)_{\sigma,i})^{\vee}\cong (B_{D,\fR,\sigma,i})^{\vee, \oplus m}$. Restricting the second injection in (\ref{EthickB}) to $\Ima(\iota)$ and using (\ref{EdecoB}), we obtain $T(K) \times A_D$-equivariant injections
\begin{multline}\label{Einj3}
(B_{D,\fR,\sigma,i}/\fm_{B_{D,\fR,\sigma,i}})^{\vee}\bigoplus \Ima(\jmath) \cong (B_{D,\fR,\sigma,i})^{\vee} \hooklongrightarrow (V_{\sigma,i}\otimes_E\varepsilon^n)[\fa_{\pi}] \\
\hooklongrightarrow J_B\big(\Pi_{\infty}(\xi,\tau)^{\wt(\delta_{\fR})^{\sigma}\text{-}\alg}[\fa_{\pi}]\big) \hooklongrightarrow J_B\big(\Pi_{\infty}(\xi,\tau)^{R_{\infty}(\xi,\tau)\text{-}\an}[\fa_{\pi}]\big).
\end{multline}
Note that the restriction of the composition (\ref{Einj3}) to $(B_{D,\fR,\sigma,i}/\fm_{B_{D,\fR,\sigma,i}})^{\vee}$ factors through $ (x_{\fR}^*\cM_{\infty}(\xi,\tau)_{\sigma,i})^{\vee}$ hence corresponds to an injection (see the comment after the proof of Lemma \ref{lem:free})
\begin{equation}\label{E1copy}
\pi_{\alg}(D) \otimes_E \varepsilon^{n-1} \hooklongrightarrow \Pi_{\infty}(\xi,\tau)^{R_{\infty}(\xi,\tau)\text{-}\an}[ \fm_{\pi}^{\wp}+\fm_{\pi,\wp}].
\end{equation}
Applying \cite[Prop.~5.5]{Wu24} to the injective composition induced by (\ref{Einj3})
\begin{equation}\label{eq:je}
\Ima(\jmath) \hooklongrightarrow (V_{\sigma,i}\otimes_E\varepsilon^n)[\fa_{\pi}] \hooklongrightarrow J_B(\Pi_{\infty}(\xi,\tau)^{R_{\infty}(\xi,\tau)\text{-}\an}[\fa_{\pi}])
\end{equation}
we obtain a non-zero $\GL_n(K)$-equivariant map
\begin{equation}\label{Eadjun1}
\cF_{B^-}^{\GL_n}\big(M^-_{\sigma,i}(-\wt(\delta_{\fR}))^{\vee}, \delta_{\fR}^{\sm} \delta_B^{-1}\big) \longrightarrow \Pi_{\infty}(\xi,\tau)^{R_{\infty}(\xi,\tau)\text{-}\an}[\fa_{\pi}]
\end{equation}
where $M^-_{\sigma,i}(-\wt(\delta_{\fR}))^{\vee}$ is the dual in the sense of \cite[\S~3.2]{Hu08} of 
\[M^-_{\sigma,i}(-\wt(\delta_{\fR})):=\big(\text{U}(\fg_{\sigma}) \otimes_{\text{U}((\fp_i^-)_{\sigma})}L_i^-(-\wt(\delta_{\fR})_{\sigma})\big) \otimes_E \big(\otimes_{\tau \neq \sigma} L^-(-\wt(\delta_{\fR})_{\tau})\big)\]
(recall $L^-(-\wt(\delta_{\fR})_\tau)$ is the finite dimensional simple $\text{U}(\fg_{\tau})$-module over $E$ of highest weight $-\wt(\delta_{\fR})_\tau$ with respect to the lower Borel $\fb_\tau^-$, and likewise with $(\fl_{P_i})_{\sigma}$ instead of $\fg_{\tau}$ for $L_i^-(-\wt(\delta_{\fR})_{\sigma})$). We give a quick explanation on how we get (\ref{Eadjun1}). By \cite[Prop.~2.14]{Di18}, (the first injection of) (\ref{eq:je}) induces an $L_{P_i}(K)$-equivariant map
{\small
\begin{multline*}
\big(\Ind_{B^-(K)\cap L_{P_i}(K)}^{L_{P_i}(K)}\delta_{\fR}^{\sm} \delta_{B\cap L_{P_i}}^{-1}\big)^{\infty}\otimes_E L_i(\wt(\delta_{\fR}))\hooklongrightarrow \\
\Big(J_{P_i}\Big(\big(\Pi_{\infty}(\xi,\tau)^{R_{\infty}(\xi,\tau)\text{-}\an}[\fa_{\pi}]\otimes_E (\otimes_{\tau\neq \sigma} L(\lambda_{\tau})^{\vee})\big)^{\sigma\text{-}\an}\Big)\otimes_E L_i(\lambda_{\sigma})^{\vee}\Big)^{\fl_{P_i}'}\otimes_E (\otimes_{\tau\in \Sigma} L_{i}(\lambda_{\tau}))\otimes_E\varepsilon^n\\
\hooklongrightarrow J_{P_i}\big(\Pi_{\infty}(\xi,\tau)^{\wt(\delta_{\fR})^{\sigma}\text{-}\alg}[\fa_{\pi}]\big).
\end{multline*}}
\!\!(recall $(V_{\sigma,i}\otimes_E\varepsilon^n)[\fa_{\pi}]\cong V_{\sigma,i}[\fa_{\pi}]\otimes_E\varepsilon^n=J_{B\cap L_{P_i}}(\text{second representation above})$ using the first isomorphism in (\ref{Eparabolic})). By \cite[Thm.~4.3]{Br15}, this corresponds to a non-zero map 
{\small
\begin{multline*}
\cF_{B^-}^{\GL_n}\big(\big(\text{U}(\fg_{\Sigma}) \otimes_{\text{U}((\fp_i)_{\Sigma}^-)} L_i^-(-\wt(\delta_{\fR}))\big)^{\vee}, \delta_{\fR}^{\sm}\delta_{B}^{-1}\big) \cong \\
\cF_{P_i^-}^{\GL_n}\Big(\big(\text{U}(\fg_{\Sigma}) \otimes_{\text{U}((\fp_i)_{\Sigma}^-)} L_i^-(-\wt(\delta_{\fR}))\big)^{\vee}, \big(\Ind_{B^-(K)\cap L_{P_i}(K)}^{L_{P_i}(K)}\delta_{\fR}^{\sm} \delta_{B\cap L_{P_i}}^{-1}\big)^{\infty} \otimes_E \delta_{P_i}^{-1}\Big)\\
 \longrightarrow \Pi_{\infty}(\xi,\tau)^{\wt(\delta_{\fR})^{\sigma}\text{-}\alg}[\fa_{\pi}].
\end{multline*}}
\!\!However, by definition of $\Pi_{\infty}(\xi,\tau)^{\wt(\delta_{\fR})^{\sigma}\text{-}\alg}$ (see (\ref{Elamsigalg})) and comparing the $\fg_{\tau}$-actions for $\tau\neq \sigma$, this map has to factor through the representation $\cF_{B^-}^{\GL_n}\big(M^-_{\sigma,i}(-\wt(\delta_{\fR}))^{\vee}, \delta_{\fR}^{\sm} \delta_B^{-1}\big)$ as in (\ref{Eadjun1}).

\begin{lem}\label{lem:splitX}
The map (\ref{Eadjun1}) factors through a $\GL_n(K)$-equivariant injection (see above (\ref{eq:W}) for $X_{\sigma,i}^-(-\wt(\delta_{\fR}))$)
\begin{equation}\label{Eadjun2}
W_{\sigma,I}\otimes_E\varepsilon^{n-1}\cong \cF_{B^-}^{\GL_n}(X_{\sigma,i}^-(-\wt(\delta_{\fR}))^{\vee}, \delta_{\fR}^{\sm} \delta_B^{-1}) \hooklongrightarrow \Pi_{\infty}(\xi,\tau)^{R_{\infty}(\xi,\tau)\text{-}\an}[\fa_{\pi}].
\end{equation}
\end{lem}
\begin{proof}
An unravelling of \cite[Thm.~8.4(iii)]{ES87} applied to the maximal parabolic subgroup $P_i$ in the case of a Hermitian symmetric pair of type HS1 (in the notation of \emph{loc.~cit.}), which requires a bit of work but which is elementary, shows that, if \ $L^-(-w \cdot \wt(\delta_{\fR}))$ is an irreducible constituent of the kernel of $M^-_{\sigma,i}(-\wt(\delta_{\fR})) \twoheadrightarrow X_{\sigma,i}^-(-\wt(\delta_{\fR}))$, then we must have $w_\sigma\geq s_{i,\sigma} s_{i+1,\sigma} s_{i-1,\sigma} s_{i,\sigma}$ (note that $M^-_{\sigma,i}(-\wt(\delta_{\fR})) \xrightarrow{\sim}X_{\sigma,i}^-(-\wt(\delta_{\fR}))$ if $i\in \{1,n-1\}$). In particular $s_{i,\sigma}$ has multiplicity \emph{at least} $2$ in any reduced expression of $w_\sigma$. By assumption, we therefore have $w_\sigma\nleq w_{\fR,\sigma}w_{0,\sigma}$. By \cite[Thm.~5.3.3]{BHS19} (with Remark \ref{rem:setting}) we have
\[\Hom_{\GL_n(K)}\Big(\cF_{B^-}^{\GL_n}\big(L^-(-w\cdot \wt(\delta_{\fR})), \delta_{\fR}^{\sm} \delta_B^{-1}\big),\Pi_{\infty}(\xi,\tau)^{R_{\infty}(\xi,\tau)\text{-}\an}[ \fm_{\pi}^{\wp}+\fm_{\pi,\wp}]\Big)=0.\]
By \ the \ same \ argument \ as \ at \ the \ end \ of \ the \ proof \ of \ Lemma \ref{lem:balanced}, \ we \ deduce \ that $\cF_{B^-}^{\GL_n}\big(L^-(-w\cdot \wt(\delta_{\fR})), \delta_{\fR}^{\sm} \delta_B^{-1}\big)$ also cannot be a subrepresentation of $\Pi_{\infty}(\xi,\tau)^{R_{\infty}(\xi,\tau)\text{-}\an}[\fa_{\pi}]$. It follows that (\ref{Eadjun1}) must factor through a non-zero map as in (\ref{Eadjun2}). If this map is not injective, hence factors through the quotient $\pi_{\alg}(D) \otimes_E \varepsilon^{n-1}$, then the image of the composition (\ref{eq:je}) has to be contained in $(x_{\fR}^*\cM_{\infty}(\xi,\tau)_{\sigma,i})^{\vee}$ via (\ref{ExRfibre}), a contradiction with the choice of $\jmath: \delta_{\fR}\hookrightarrow (B_{D,\fR,\sigma,i})^{\vee}$ above (\ref{EdecoB}).
\end{proof}

With Lemma \ref{lem:splitX} we see that from the composition (\ref{Einj3}) we obtain a $\GL_n(K)$-equiva\-riant injection
\begin{equation}\label{eq:criinj}
\big(	\pi_{\alg}(D) \otimes_E \varepsilon^{n-1} \big)\bigoplus W_{\sigma,I}\otimes_E\varepsilon^{n-1} \hooklongrightarrow \Pi_{\infty}(\xi,\tau)^{R_{\infty}(\xi,\tau)\text{-}\an}[\fa_{\pi}].
\end{equation}
Now we take the $B_{D,\fR,\sigma,i}$-action into consideration. By \cite[Thm.~4.3]{Br15}, we have an isomorphism
{\small
\begin{multline}\label{Eadj00}
\Hom_{T(K)}\big(\delta_{\fR}, J_B(\Pi_{\infty}(\xi,\tau)^{R_{\infty}(\xi,\tau)\text{-}\an}[\fa_{\pi}]) \big)\\
\xlongrightarrow{\sim} \Hom_{\GL_n(K)}\Big(\cF_{B^-}^{\GL_n}\big(\big(\text{U}(\fg_{\Sigma}) \otimes_{\text{U}(\fb_{\Sigma}^-)} (-\wt(\delta_{\fR}))\big)^{\vee}, \delta_{\fR}^{\sm}\delta_{B}^{-1}\big), \Pi_{\infty}(\xi,\tau)^{R_{\infty}(\xi,\tau)\text{-}\an}[\fa_{\pi}]\Big).
\end{multline}}
\!\!and this isomorphism is functorial in $\Pi_{\infty}(\xi,\tau)^{R_{\infty}(\xi,\tau)\text{-}\an}[\fa_{\pi}]$, i.e.~if we have a $\GL_n(K)$-equivariant morphism $\Pi_{\infty}(\xi,\tau)^{R_{\infty}(\xi,\tau)\text{-}\an}[\fa_{\pi}] \rightarrow \Pi_{\infty}(\xi,\tau)^{R_{\infty}(\xi,\tau)\text{-}\an}[\fa_{\pi}]$ there is an abvious commutative diagram. Note also that (\ref{Eadj00}) sends the composition (\ref{eq:je}) to the map (\ref{Eadjun2}) (by Lemma \ref{lem:splitX}) and the restriction of the composition (\ref{Einj3}) to $(B_{D,\fR,\sigma,i}/\fm_{B_{D,\fR,\sigma,i}})^{\vee}$ to the map (\ref{E1copy}). Now, let $0 \neq x\in \fm_{B_{D,\fR,\sigma,i}}$ and $\widetilde x$ an arbitrary preimage of $x$ in $\fm_{A_D}$. Then the restriction of (\ref{Einj3}) to $\Ima(\jmath)$ induces a $T(K)$-equivariant commutative diagram
\begin{equation}\label{EdiagTx}
\begin{tikzcd}
\Ima(\jmath) \arrow[r, hook, "(\ref{eq:je})"] \arrow[d, "\wr"] &J_B(\Pi_{\infty}(\xi,\tau)^{R_{\infty}(\xi,\tau)\text{-}\an}[\fa_{\pi}]\arrow[d, "\widetilde x"] \\
(B_{D,\fR,\sigma,i}/\fm_{B_{D,\fR,\sigma,i}})^{\vee} \arrow[r, hook, "(\ref{Einj3})"] & J_B(\Pi_{\infty}(\xi,\tau)^{R_{\infty}(\xi,\tau)\text{-}\an}[\fa_{\pi}]
\end{tikzcd}
\end{equation}
where the isomorphism in the left vertical map follows from the sentence below (\ref{EdecoB}). By the discussion below (\ref{Eadj00}), the isomorphism (\ref{Eadj00}) sends $\widetilde x \circ (\ref{eq:je})$ to $\widetilde x\circ (\ref{Eadjun2})$. However, by (\ref{EdiagTx}), $\widetilde x \circ (\ref{eq:je})$ is equal to the restriction of (\ref{Einj3}) to $(B_{D,\fR,\sigma,i}/\fm_{B_{D,\fR,\sigma,i}})^{\vee}$ up to a non-zero scalar, which hence is sent to (\ref{E1copy}) by (\ref{Eadj00}). From the functoriality discussed below (\ref{Eadj00}) we see that there exists a surjection of $\GL_n(K)$-representations (only depending on $x$):
\[\kappa_x: W_{\sigma,I}\otimes_E\varepsilon^{n-1}\cong \cF_{B^-}^{\GL_n}\big(X_{\sigma,i}^-(-\wt(\delta_{\fR}))^{\vee},\delta_{\fR}^{\sm} \delta_B^{-1}\big)\twoheadlongrightarrow \pi_{\alg}(D) \otimes_E \varepsilon^{n-1}\]
such that the following diagram commutes
\begin{equation}\label{Ediagx}
\begin{tikzcd}
W_{\sigma,I}\otimes_E\varepsilon^{n-1} \arrow[r, hook, "(\ref{Eadjun2})"] \arrow[d, two heads, "\kappa_x"] &\Pi_{\infty}(\xi,\tau)^{R_{\infty}(\xi,\tau)\text{-}\an}[\fa_{\pi}]\arrow[d, "\widetilde x"] \\
\pi_{\alg}(D) \otimes_E \varepsilon^{n-1} \arrow[r, hook, "(\ref{E1copy})"] & \Pi_{\infty}(\xi,\tau)^{R_{\infty}(\xi,\tau)\text{-}\an}[\fa_{\pi}].
\end{tikzcd}
\end{equation}
We let $x$ act on the left hand side of (\ref{eq:criinj}) via 
\begin{multline*}
\big(	\pi_{\alg}(D) \otimes_E \varepsilon^{n-1} \big)\bigoplus W_{\sigma,I}\otimes_E\varepsilon^{n-1} \twoheadlongrightarrow W_{\sigma,I}\otimes_E\varepsilon^{n-1} \xlongrightarrow{\kappa_x} \pi_{\alg}(D) \otimes_E \varepsilon^{n-1} \\
\hooklongrightarrow	\big(	\pi_{\alg}(D) \otimes_E \varepsilon^{n-1} \big)\bigoplus W_{\sigma,I}\otimes_E\varepsilon^{n-1}.
\end{multline*}
This determines a unique $B_{D,\fR,\sigma,i}$-action (hence an $A_D$-action via $A_D \twoheadrightarrow B_{D,\fR,\sigma,i}$) on the left hand side of (\ref{eq:criinj}), and (\ref{Ediagx}) shows that the map (\ref{eq:criinj}) is $A_D$-equivariant. Finally using (\ref{eq:kappax2}) for $\kappa=\kappa_x$ with Lemma \ref{lem: OStildepi} and letting $\iota$ vary (see below (\ref{EthickB})), we finally deduce a $\GL_n(K)\times A_D$-equivariant injection extending (\ref{Einjlalg}) (see (\ref{eq:piI1}) for $\widetilde{\pi}_{\sigma,I,1}(D)$):
\begin{equation}\label{Ecrisi}
(\widetilde{\pi}_{\sigma,I,1}(D) \otimes_E \varepsilon^{n-1} )^{\oplus m} \hooklongrightarrow \Pi_{\infty}(\xi,\tau)^{R_{\infty}(\xi,\tau)\text{-}\an}[\fa_{\pi}].
\end{equation}

\noindent \textbf{Step 3: Amalgam of the maps.}\\
We amalgamate the maps in (\ref{Eiogtasi}) and in (\ref{Ecrisi}). Recall the representation $\widetilde{\pi}_{\alg,\sigma}(D)$ for $\sigma\in \Sigma$ is defined above (\ref{Einjuniv}).

\begin{lem}\label{lem:imageI}
Let $(\sigma,I)\in S^{\rm{nc}}(D)$ and $\iota_{\sigma,I}$ as in (\ref{Eiogtasi}) for a choice of refinement compatible with $I$. The subrepresentation $\iota_{\sigma,I}((\widetilde{\pi}_{\alg,\sigma}(D) \otimes_E \varepsilon^{n-1})^{\oplus m})$ of $\Pi_{\infty}(\xi,\tau)^{R_{\infty}(\xi,\tau)\text{-}\an}[\fa_{\pi}]$ does not depend on $I$ or on the refinement compatible with $I$.
\end{lem}
\begin{proof}
By \ the \ same \ argument \ as \ in \ \cite[Lemma 4.2(2)]{Di25} \ we \ have \ (using \ (\ref{Elalg}) \ and $\widehat{S}_{\xi,\tau}(U^{\wp},E)^{\Qp\text{-}\an}[\fm_{\pi}]\cong \Pi_{\infty}(\xi,\tau)^{R_{\infty}(\xi,\tau)\text{-}\an}[ \fm_{\pi}^{\wp}+\fm_{\pi,\wp}]$)
\begin{equation}\label{Elalgapi}
\dim_E \Hom_{\GL_n(K)}\big(\pi_{\alg}(D) \otimes_E \varepsilon^{n-1}, \Pi_{\infty}(\xi,\tau)^{R_{\infty}(\xi,\tau)\text{-}\an}[\fa_{\pi}]\big)=m.
\end{equation}
Suppose $\iota_{\sigma,I}((\widetilde{\pi}_{\alg,\sigma}(D) \otimes_E \varepsilon^{n-1})^{\oplus m}) \neq \iota_{\sigma,J}((\widetilde{\pi}_{\alg,\sigma}(D) \otimes_E \varepsilon^{n-1})^{\oplus m})$ for some non-critical $I$, $J$ (and choices of compatible refinements), and let $W$ be the closed subrepresentation of $\Pi_{\infty}(\xi,\tau)^{R_{\infty}(\xi,\tau)\text{-}\an}[\fa_{\pi}]$ generated by these two subrepresentations. Using the fact $\iota_{\sigma,I}$, $\iota_{\sigma,J}$ both extend (\ref{Einjlalg}) and that $\widetilde{\pi}_{\alg,\sigma}(D)$ is by construction maximal as an extension of finitely many $\pi_{\alg}(D)$ by (a single) $\pi_{\alg}(D)$ which is locally algebraic up to twist by locally $\sigma$-analytic characters, we must have $\dim_E \Hom_{\GL_n(K)}(\pi_{\alg}(D) \otimes_E \varepsilon^{n-1}, W)>m$, a contradiction.\end{proof}

Let $(\sigma,I)\in S^{\rm{nc}}(D)$, then the map $\iota_{\sigma,I}|_{(\widetilde{\pi}_{\alg,\sigma}(D) \otimes_E \varepsilon^{n-1})^{\oplus m} }$ is $A_{D,\sigma,0}$-equivariant (see below (\ref{eq:ADSRi})) but may depend on $I$. Fix an arbitrary $\GL_n(K)\times A_{D,\sigma, 0}$-equivariant injection extending (\ref{Einjlalg})
\[\iota_{\sigma,0}: (\widetilde{\pi}_{\alg,\sigma}(D) \otimes_E \varepsilon^{n-1})^{\oplus m} \hooklongrightarrow \Pi_{\infty}(\xi,\tau)^{R_{\infty}(\xi,\tau)\text{-}\an}[\fa_{\pi}].\]
As in the discussion of Remark \ref{Rnuni}, we have an isomorphism $A_{D,\sigma, 0} \buildrel\sim\over \rightarrow \End_{\GL_n(K)}(\widetilde{\pi}_{\alg,\sigma}(D))$ and thus an isomorphism $\End_{\GL_n(K)}(\widetilde{\pi}_{\alg,\sigma}(D)^{\oplus m})\cong M_m(A_{D,\sigma,0})$. Using Lemma \ref{lem:imageI}, we deduce that there exists a matrix $M_I\in M_m(A_{D,\sigma,0})$ such that $M_I\equiv \id \mod{\fm_{A_{D,\sigma,0}}}$ and 
\begin{equation}\label{eq:MI}
\iota_{\sigma,0}=(\iota_{\sigma,I}|_{(\widetilde{\pi}_{\alg,\sigma}(D) \otimes_E \varepsilon^{n-1})^{\oplus m}}) \circ M_I.
\end{equation}
Let $\widetilde M_I$ be a lift of $M_I$ in $M_m(A_{D,\fR,\sigma,i})$, which corresponds to an automorphism of $(\widetilde{\pi}_{I,\sigma} \otimes_E \varepsilon^{n-1})^{\oplus m}$ \ (see \ Remark \ \ref{Rnuni}). \ Replacing \ $\iota_{\sigma,I}$ \ by \ $\iota_{\sigma,I}\circ \widetilde M_I$, \ by \ (\ref{eq:MI}) \ we \ can \ assume $\iota_{\sigma,I}\vert_{(\widetilde{\pi}_{\alg,\sigma}(D) \otimes_E \varepsilon^{n-1})^{\oplus m}}=\iota_{\sigma,0}$. We can now amalgamate these (new) $\iota_{\sigma,I}$ into a $\GL_n(K) \times A_D$-equivariant injection extending (\ref{Einjlalg})
\begin{equation}\label{eq:amalg0}
\Big(\bigoplus_{\text{$I$ non-critical for $\sigma$}, \ \!\widetilde{\pi}_{\alg, \sigma}(D)} \widetilde{\pi}_{\sigma,I} \Big)^{\oplus m}\otimes_E \varepsilon^{n-1} \hooklongrightarrow \Pi_{\infty}(\xi,\tau)^{R_{\infty}(\xi,\tau)\text{-}\an}[\fa_{\pi}].
\end{equation}
Now, one easily checks that as above the restriction map induces a surjection $A_{D,\sigma, 0}\cong \End_{\GL_n(K)}(\widetilde{\pi}_{\alg,\sigma}(D))\twoheadrightarrow \End_{\GL_n(K)}(\widetilde{\pi}_{\alg}(D))$ for $\sigma\in \Sigma$, and as above one can modify the injections $\iota_{\sigma,0}$ so that $\iota_{\sigma,0}\vert_{(\widetilde{\pi}_{\alg}(D) \otimes_E \varepsilon^{n-1})^{\oplus m}}$ does not depend on $\sigma$. Then one can amalgamate (\ref{eq:amalg0}) for $\sigma\in \Sigma$ into a $\GL_n(K) \times A_D$-equivariant injection extending (\ref{Einjlalg}).
\begin{equation*}
\Big(\bigoplus_{\sigma\in \Sigma, \widetilde{\pi}_{\alg}(D)}\big(\bigoplus_{\text{$I$ non-critical for $\sigma$}, \ \!\widetilde{\pi}_{\alg, \sigma}} \widetilde{\pi}_{\sigma,I} \big)\Big)^{\oplus m} \otimes_E \varepsilon^{n-1}\hooklongrightarrow \Pi_{\infty}(\xi,\tau)^{R_{\infty}(\xi,\tau)\text{-}\an}[\fa_{\pi}].
\end{equation*}
Finally, amalgamating over $(\pi_{\alg}(D) \otimes_E \varepsilon^{n-1})^{\oplus m}$ with $(\widetilde{\pi}_{\sigma,I,1}(D) \otimes_E \varepsilon^{n-1} )^{\oplus m}$ for $(\sigma,I)\in S^\flat(D)\setminus S^{\rm{nc}}(D)$ and using (\ref{Ecrisi}), by (\ref{eq:tildepiI}) and (\ref{eq:decopflat}) we obtain a $\GL_n(K) \times A_D$-equivariant injection as in (\ref{injuniv}) (we use here that for each $\sigma\in \Sigma$ there is at least one $I$ such that $(\sigma,I)\in S^{\rm{nc}}(D)$). This finishes the proof of Proposition \ref{prop:T: lg}.\bigskip

We now prove Proposition \ref{prop:nasty}. Note first that, by (\ref{Elalg}), the filtered $\varphi$-module $D'$ must have same Hodge-Tate weights and same Frobenius eigenvalues as the filtered $\varphi$-module $D$ since these data can be read from $\pi_{\alg}(D$). We need the following result:

\begin{prop}\label{prop:multiplicity}
For $\sigma\in \Sigma$ and $I\subset \{\varphi_0, \dots, \varphi_{n-1}\}$, we have
\begin{multline}\label{Emulti}
\dim_E\Hom_{\GL_n(K)}\big(\widetilde{C}(I,s_{i,\sigma}) \otimes_E \varepsilon^{n-1}, \Pi_{\infty}(\xi,\tau)^{R_{\infty}(\xi,\tau)\text{-}\an}[ \fm_{\pi}^{\wp}+\fm_{\pi,\wp}]\big)\\
=
\begin{cases} 
m & \text{\!\!\!if \ } (\sigma,I) \in S^{\flat}(D)\setminus S^{\mathrm{nc}}(D)\\
0 & \text{\!\!\!if \ } (\sigma, I)\in S^{\mathrm{nc}}(D).
\end{cases}
\end{multline}
\end{prop}

\noindent
The proof of Proposition \ref{prop:multiplicity} is somewhat independent of the rest of this section and is given in Appendix \ref{sec:multiplicity}. Assume we have an injection as in the statement of Proposition \ref{prop:nasty}
\begin{equation}\label{EinjD'}
\pi(D')^{\flat}\otimes_E \varepsilon^{n-1} \hooklongrightarrow \Pi_{\infty}(\xi,\tau)^{R_{\infty}(\xi,\tau)\text{-}\an}[ \fm_{\pi}^{\wp}+\fm_{\pi,\wp}].
\end{equation}
As $S^{\flat}(D)=S^{\flat}(D')$ and $S^{\mathrm{nc}}(D)=S^{\mathrm{nc}}(D')$, from (\ref{eq:amalgcrit}) and (\ref{eq:pi(D)_flat}) we have $\pi_{\flat}(D') \cong \pi_{\flat}(D)$. It follows from (\ref{Elalg}) and (\ref{Emulti}) with the dimension $1$ assertion of \cite[Lemma 3.5(1)]{Di25} that the composition
\begin{equation}\label{EinjD'2}
\pi_{\flat}(D')\otimes_E \varepsilon^{n-1} \hooklongrightarrow \pi(D')^{\flat} \otimes_E \varepsilon^{n-1} \buildrel (\ref{EinjD'})\over \hooklongrightarrow \Pi_{\infty}(\xi,\tau)^{R_{\infty}(\xi,\tau)\text{-}\an}[ \fm_{\pi}^{\wp}+\fm_{\pi,\wp}]
\end{equation}
must factor through (\ref{Elg}). Indeed, otherwise the socle of the closed subrepresentation of $\Pi_{\infty}(\xi,\tau)^{R_{\infty}(\xi,\tau)\text{-}\an}[ \fm_{\pi}^{\wp}+\fm_{\pi,\wp}]$ generated by $\pi_{\flat}(D')\otimes_E \varepsilon^{n-1}$ and the image of $(\ref{Elg})$ contains either an extra copy of $\pi_{\alg}(D) \otimes_E \varepsilon^{n-1}$ or an extra copy of $\widetilde{C}(I,s_{i,\sigma})\otimes_E \varepsilon^{n-1}$ for some $(\sigma,I) \in S^{\flat}(D)$, a contradiction. Since $\pi(D)^{\flat}\cong \widetilde{\pi}_\flat(D)[\fm_{A_D}]$, the image of (\ref{EinjD'2}) lies in $(\widetilde{\pi}_\flat(D) \otimes_E \varepsilon^{n-1})^{\oplus m}\buildrel (\ref{injuniv}) \over \hooklongrightarrow \Pi_{\infty}(\xi,\tau)^{R_{\infty}(\xi,\tau)\text{-}\an}[\fa_{\pi}]$. Then we deduce that the injection (\ref{EinjD'}) also has to factor through 
\begin{equation}\label{eq:D'}
\pi(D')^{\flat}\otimes_E \varepsilon^{n-1}\hooklongrightarrow(\widetilde{\pi}_\flat(D) \otimes_E \varepsilon^{n-1})^{\oplus m} \buildrel (\ref{injuniv})\over\hooklongrightarrow \Pi_{\infty}(\xi,\tau)^{R_{\infty}(\xi,\tau)\text{-}\an}[\fa_{\pi}].
\end{equation}
Indeed, otherwise by the universality of $\widetilde{\pi}_\flat(D)$ (see below (\ref{Esurj1})) the closed subrepresentation of $\Pi_{\infty}(\xi,\tau)^{R_{\infty}(\xi,\tau)\text{-}\an}[\fa_{\pi}]$ generated by $\pi(D')^{\flat} \otimes_E \varepsilon^{n-1}$ and $(\widetilde{\pi}_\flat(D) \otimes_E \varepsilon^{n-1})^{\oplus m}$ contains more than $m$ copies of $\pi_{\alg}(D)\otimes_E \varepsilon^{n-1}$ in its socle, which contradicts (\ref{Elalgapi}). As the image of the composition (\ref{eq:D'}) lies in $\Pi_{\infty}(\xi,\tau)^{R_{\infty}(\xi,\tau)\text{-}\an}[ \fm_{\pi}^{\wp}+\fm_{\pi,\wp}]$, we deduce it must factor through an injection
\begin{multline}\label{Einjfake}
\pi(D')^{\flat} \otimes_E \varepsilon^{n-1} \hooklongrightarrow \big(\pi(D)^{\flat} \otimes_E \varepsilon^{n-1}\big)^{\oplus m}=(\widetilde{\pi}_\flat(D) \otimes_E \varepsilon^{n-1})^{\oplus m}[\fm_{A_{D}}]\\
=(\widetilde{\pi}_\flat(D) \otimes_E \varepsilon^{n-1})^{\oplus m} \cap \Pi_{\infty}(\xi,\tau)^{R_{\infty}(\xi,\tau)\text{-}\an}[ \fm_{\pi}^{\wp}+\fm_{\pi,\wp}]
\end{multline}
where the intersection is taken in $\Pi_{\infty}(\xi,\tau)^{R_{\infty}(\xi,\tau)\text{-}\an}[\fa_{\pi}]$ (and the equalities follow from the definitions of $\pi(D)^{\flat}$, $A_{D}$ and $\fm_{\pi}$). For $i=1, \dots, m$, let $f_i$ be the composition of (\ref{Einjfake}) with the projection $\mathrm{pr}_i$ to the $i$-th copy $\pi(D)^{\flat} \otimes_E \varepsilon^{n-1}$. As (\ref{Einjfake}) is injective, for each irreducible closed subrepresentation $W\subset \pi(D')^{\flat} \otimes_E \varepsilon^{n-1}$, there exists at least one $i$ such that $f_i(W)\neq 0$. Therefore the set of $(\lambda_1,\dots,\lambda_r)\in E^{\oplus m}={\mathbb A}_E^{r}(E)$ such that $\sum_{i=1}^m\lambda_i f_i(W)\ne 0$ is the set of $E$-points of a non-empty Zariski-open subset of ${\mathbb A}_E^{r}$. Since $\soc_{\GL_n(K)}(\pi(D')^{\flat} \otimes_E \varepsilon^{n-1})$ has finite length and is multiplicity free, there exists $(\lambda_1,\dots,\lambda_r)\in {\mathbb A}_E^{r}(E)$ such that $\sum_{i=1}^m\lambda_i f_i:\pi(D')^{\flat} \otimes_E \varepsilon^{n-1} \rightarrow \pi(D)^{\flat} \otimes_E \varepsilon^{n-1}$ is injective on $\soc_{\GL_n(K)}(\pi(D')^{\flat} \otimes_E \varepsilon^{n-1})$, hence is injective, hence is bijective since both representations have the same length. Thus $\pi(D')^{\flat}\cong \pi(D)^{\flat}$ and by (ii) of Proposition \ref{prop:flat} we deduce isomorphisms of filtered $\varphi^f$-modules $D'_{\sigma}\cong D_{\sigma}$ for all $\sigma\in \Sigma$.\bigskip

Note that the above proof of Proposition \ref{prop:nasty} also shows that \emph{all} $\GL_n(K)$-equivariant injections $(\pi(D)^{\flat} \otimes_E \varepsilon^{n-1})^{\oplus m} \hookrightarrow \widehat{S}_{\xi,\tau}(U^{\wp},E)^{\Qp\text{-}\an}[\fm_{\pi}]$ have same image, and in particular satisfy the property below (\ref{Elg}).

\begin{rem}\label{rem:extra}\ 
\begin{enumerate}[label=(\roman*)]
\item
When $n=3$ and $r=\rho_{\pi,\widetilde \wp}$ is split (i.e.~is the direct sum of $3$ characters), the injection (\ref{Elg}) was first proved in \cite[Rk.~7.31]{HHS25} (note that when $n=3$ we always have $\pi(D)^{\flat}\buildrel\sim\over\rightarrow \pi(D)$). This was the first discovered case of copies of $\pi_{\alg}(D)\otimes_E \varepsilon^{n-1}$ in $\widehat{S}_{\xi,\tau}(U^{\wp},E)[\fm_{\pi}]^{\Q_p\text{-}\an}$ which are not in the socle.
\item
When all refinements are non-critical for all $\sigma\in \Sigma$, Theorem \ref{T: lg} and Proposition \ref{prop:nasty} were proved in \cite[Thm.~4.18]{Di25} and \cite[Cor.~4.21]{Di25}. But even in this non-critical case, combining Theorem \ref{thm:fil} with (the proofs of) Theorem \ref{T: lg} and Proposition \ref{prop:nasty} allow to read out finer information on $D$ in the $\GL_n(K)$-representation $\widehat{S}_{\xi,\tau}(U^{\wp},E)[\fm_{\pi}]^{\Q_p\text{-}\an}$. Fix $\sigma\in \Sigma$ and Hodge-Tate weights as in \S~\ref{sec:prel}. For a filtered $\varphi$-module $D'$ as in \S~\ref{sec:prel} let $S^{\flat}_{\sigma}(D'):=\{I \ |\ (\sigma, I)\in S^{\flat}(D')\}$ and $S^{\rm{nc}}_{\sigma}(D'):=\{I \ | \ (\sigma, I)\in S^{\rm{nc}}(D')\}$. For $S\subseteq R$ denote by $D'_{\sigma, S}$ the filtered $\varphi^f$-module endowed with the (partial) filtration $(\Fil^{-h_{i,\sigma}}(D_{\sigma}'), s_i\in S)$. Assume $S^{\flat}_{\sigma}(D')=S^{\flat}_{\sigma}(D)$ and $S^{\rm{nc}}_{\sigma}(D')=S^{\rm{nc}}_{\sigma}(D)$. By the same argument as in the proof of Proposition \ref{prop:nasty}, we can show there is an injection
\begin{equation}\label{eq:injS}
\pi(D'_{\sigma})(S)^{\flat} \otimes_E (\otimes_{\tau\neq \sigma} L(\lambda_{\tau})) \otimes_E \varepsilon^{n-1}\hooklongrightarrow \widehat{S}_{\xi,\tau}(U^{\wp},E)[\fm_{\pi}]^{\Q_p\text{-}\an}
\end{equation}
if and only if $D_{\sigma,S}'\cong D_{\sigma, S}$ (see (\ref{eq:flatS}) for $\pi(D'_{\sigma})(S)^{\flat}$). Indeed, assume $D_{\sigma,S}'\cong D_{\sigma, S}$, then we have $\pi(D'_{\sigma})(S)^{\flat}\cong \pi(D_{\sigma})(S)^{\flat}$ by the discussion below (\ref{eq:flatS}) and hence (\ref{eq:injS}) holds by Theorem \ref{T: lg}. Assume (\ref{eq:injS}), then by the same argument as in the proof of Proposition \ref{prop:nasty} the map (\ref{eq:injS}) must factor through an injection
\[\pi(D'_{\sigma})(S)^{\flat}\otimes_E (\otimes_{\tau\neq \sigma} L(\lambda_{\tau})) \otimes_E \varepsilon^{n-1}\hooklongrightarrow \pi(D_{\sigma})^{\flat}\otimes_E (\otimes_{\tau\neq \sigma} L(\lambda_{\tau})) \otimes_E \varepsilon^{n-1}.\]
Comparing irreducible constituents, this implies $\pi(D'_{\sigma})(S)^{\flat}\cong \pi(D_{\sigma})(S)^{\flat}$ (note that $\pi_{\alg}(D)\otimes_E \varepsilon^{n-1}$ has the same multiplicity in both representations using the right hand side of Lemma \ref{lem:multS} with (i) of Proposition \ref{prop:flat}). We deduce $D_{\sigma,S}'\cong D_{\sigma,S}$ by the discussion below (\ref{eq:flatS}).
\item
Let $\sigma\in \Sigma$, $I\subset \{\varphi_0, \dots, \varphi_{n-1}\}$ (of cardinality $\in \{1,\dots,n-1\}$) and $\fR$ a refinement compatible with $I$. The same argument as in the proof of Proposition \ref{prop:nasty} also implies that if we have an injection $W_{\sigma,I}\otimes_E\varepsilon^{n-1}\hookrightarrow \widehat{S}_{\xi,\tau}(U^{\wp},E)[\fm_{\pi}]^{\Q_p\text{-}\an}$, then $I$ is very critical for $\sigma$. Indeed, if we had $(\sigma,I)\in S^{\flat}(D)$, then by the equalities in (\ref{Einjfake}) we would get an injection
\[W_{\sigma,I}\otimes_E\varepsilon^{n-1}\hooklongrightarrow \big(\pi(D)^{\flat} \otimes_E \varepsilon^{n-1}\big)^{\oplus m}\]
which from the definition of $\pi(D)^{\flat}$ would contradict Proposition \ref{prop:mul2}.
\end{enumerate}
\end{rem}

\subsection{Towards local-global compatibility for \texorpdfstring{$\pi(D)$}{piD}}\label{sec:loc-glob2}

Although that we cannot prove that the representation $(\pi(D)\otimes_E\varepsilon^{n-1})^{\oplus m}$ (see (\ref{eq:pi(D)}) or (\ref{eq:pi(D)W})) embeds into $\widehat{S}_{\xi,\tau}(U^{\wp},E)[\fm_{\pi}]^{\Qp\text{-}\an}$ when there are very critical $I$, using a result of Z.~Wu (Theorem \ref{theoreminequalityext1}) we prove that $\widehat{S}_{\xi,\tau}(U^{\wp},E)[\fm_{\pi}]^{\Qp\text{-}\an}$ at least contains in that case a representation strictly larger than $(\pi(D)^\flat\otimes_E\varepsilon^{n-1})^{\oplus m}$ with extra copies of $\pi_{\alg}(D)\otimes_E\varepsilon^{n-1}$ (Theorem \ref{T: lg2}).\bigskip

As usual, we keep all previous notation.

\begin{lem}\label{lem:dimext}
Let $D$ be a regular filtered $\varphi$-module satisfying (\ref{eq:phi}) as in \S~\ref{sec:prel}, we have
\begin{equation}\label{eq:dim+}
\dim_E \Ext^1_{\GL_n(K)}(\pi_{\alg}(D), \pi(D)^{\flat})=n+\frac{n(n+1)}{2}[K:\Qp].
\end{equation}
\end{lem}
\begin{proof}
From the definition of $\widetilde{\pi}_\flat(D)$ below (\ref{Esurj1}) we have $\Ext^1_{\GL_n(K)}(\pi_{\alg}(D), \widetilde{\pi}_\flat(D))=0$. Since $\Hom_{\GL_n(K)}(\pi_{\alg}(D), \pi(D)^{\flat})\buildrel\sim\over\rightarrow \Hom_{\GL_n(K)}(\pi_{\alg}(D), \widetilde{\pi}_\flat(D))$ ($\cong E$), from the short exact sequence $0\rightarrow \pi(D)^{\flat}\rightarrow \widetilde{\pi}_\flat(D) \rightarrow \widetilde{\pi}_\flat(D)/\pi(D)^{\flat} \rightarrow 0$ we deduce
\begin{equation}\label{eq:homext}
\Hom_{\GL_n(K)}\big(\pi_{\alg}(D), \widetilde{\pi}_\flat(D)/\pi(D)^{\flat}\big)\buildrel\sim\over\longrightarrow \Ext^1_{\GL_n(K)}(\pi_{\alg}(D), \pi(D)^{\flat}).
\end{equation}
For \ \ $\sigma\in \Sigma$ \ \ denote \ \ by \ \ $\widetilde{\pi}_\flat(D_\sigma)$ \ \ the \ \ tautological \ \ extension \ \ of \ \ $\pi_{\alg}(D_\sigma) \otimes_E \Ext^1_{\GL_n(K),\sigma}(\pi_{\alg}(D_\sigma), \pi_{\flat}(D_\sigma))$ by $\pi_{\flat}(D_\sigma)$. Then using (\ref{eq:isoI}) and the definition of $\widetilde{\pi}_{\sigma,I}(D)$ above it, similarly to (\ref{eq:decopflat}) we have a $\GL_n(K)$-equivariant isomorphism
\[\bigoplus_{\text{$I$ \!not \!v.~\!c.~\!for \!$\sigma$}, \ \!\widetilde{\pi}_{\alg,\sigma}(D)} \!\!\!\!\widetilde{\pi}_{\sigma,I}(D)\ \ \cong \ \ \widetilde{\pi}_\flat(D_\sigma)\otimes_E (\otimes_{\tau\neq \sigma} L(\lambda_{\tau}))\]
(here we leave the details as an easy exercise to the reader). From (\ref{eq:decopflat}) we obtain
\[\widetilde{\pi}_\flat(D)\cong \bigoplus_{\sigma\in \Sigma, \ \!\widetilde{\pi}_{\alg}(D)} \!\!\!\big(\widetilde{\pi}_\flat(D_\sigma)\otimes_E (\otimes_{\tau\neq \sigma} L(\lambda_{\tau}))\big)\]
and with (\ref{eq:pi(D)flat}) we deduce
\begin{equation}\label{eq:homtilde}
\widetilde{\pi}_\flat(D)/\pi(D)^{\flat}\ \cong \ \bigoplus_{\sigma\in \Sigma, \ \!\widetilde{\pi}_{\alg}(D)/\pi_{\alg}(D)} \!\!\!\!\big(\widetilde{\pi}_\flat(D_\sigma)/\pi(D_\sigma)^{\flat}\otimes_E (\otimes_{\tau\neq \sigma} L(\lambda_{\tau}))\big).
\end{equation}
It follows from (\ref{eq:homtilde}) with (\ref{eq:algsigma}) and (\ref{eq:alg}) that we have
{\small
\begin{multline}\label{mult:dimamalg}
\dim_E\Hom_{\GL_n(K)}\big(\pi_{\alg}(D), \widetilde{\pi}_\flat(D)/\pi(D)^{\flat}\big) = \Big(\sum_{\sigma\in \Sigma}\dim_E\Hom_{\GL_n(K)}\big(\pi_{\alg}(D_\sigma), \widetilde{\pi}_\flat(D_\sigma)/\pi(D_\sigma)^{\flat}\big)\Big) \\
- ([K:\Qp]-1)\dim_E\Hom_{\GL_n(K)}\big(\pi_{\alg}(D), \widetilde{\pi}_{\alg}(D)/\pi_{\alg}(D)\big).
\end{multline}}
Now, exactly as in (\ref{eq:homext}) we have
\begin{equation*}
\Hom_{\GL_n(K)}\big(\pi_{\alg}(D_\sigma), \widetilde{\pi}_\flat(D_\sigma)/\pi(D_\sigma)^{\flat}\big)\buildrel\sim\over\longrightarrow \Ext^1_{\GL_n(K),\sigma}(\pi_{\alg}(D_\sigma), \pi(D_\sigma)^{\flat}),
\end{equation*}
hence from Corollary \ref{cor:isoflat} (both parts) we deduce
\begin{equation}\label{eq:dim1}
\dim_E\Big(\sum_{\sigma\in \Sigma}\dim_E\Hom_{\GL_n(K)}\big(\pi_{\alg}(D_\sigma), \widetilde{\pi}_\flat(D_\sigma)/\pi(D_\sigma)^{\flat}\big)\Big)=[K:\Qp]\big(n+\frac{n(n+1)}{2}\big).
\end{equation}
Moreover from the definition of $\widetilde{\pi}_{\alg}(D)$ above (\ref{Einjuniv}) with (\ref{eq: twiso}) (for $*=\alg$) and Lemma \ref{lem:isoextalg} we have
\begin{equation}\label{eq:dim2}
\dim_E\Hom_{\GL_n(K)}\big(\pi_{\alg}(D), \widetilde{\pi}_{\alg}(D)/\pi_{\alg}(D)\big)=n.
\end{equation}
Then (\ref{eq:dim+}) follows from (\ref{eq:homext}) and (\ref{mult:dimamalg}) with (\ref{eq:dim1}) and (\ref{eq:dim2}).
\end{proof}

The following proposition crucially uses the main result of Appendix \ref{Wu} by Z.~Wu.

\begin{prop}\label{lem:isoWu}
Keep the setting of Theorem \ref{T: lg}. Then any $\GL_n(K)$-equivariant injection $(\pi(D)^{\flat} \otimes_E \varepsilon^{n-1})^{\oplus m} \hookrightarrow \widehat{S}_{\xi,\tau}(U^{\wp},E)[\fm_{\pi}]^{\Qp\text{-}\an}$ induces an isomorphism of finite dimensional $E$-vector spaces
\begin{multline*}
\Ext^1_{\GL_n(K)}\big(\pi_{\alg}(D)\otimes_E \varepsilon^{n-1},(\pi(D)^{\flat} \otimes_E \varepsilon^{n-1})^{\oplus m}\big)\\\buildrel\sim\over\longrightarrow \Ext^1_{\GL_n(K)}\big(\pi_{\alg}(D) \otimes_E \varepsilon^{n-1}, \widehat{S}_{\xi,\tau}(U^{\wp},E)[\fm_{\pi}]^{\Qp\text{-}\an}\big).
\end{multline*}
\end{prop}
\begin{proof}
Any injection gives a short exact sequence
\[0\longrightarrow (\pi(D)^{\flat} \otimes_E \varepsilon^{n-1})^{\oplus m} \longrightarrow \widehat{S}_{\xi,\tau}(U^{\wp},E)[\fm_{\pi}]^{\Qp\text{-}\an} \longrightarrow X\longrightarrow 0\]
(denoting by $X$ the cokernel)
which induces an exact sequence
\begin{multline*}
0\longrightarrow \Hom_{\GL_n(K)}\big(\pi_{\alg}(D)\otimes_E \varepsilon^{n-1}, X\big)
\longrightarrow \Ext^1_{\GL_n(K)}\big(\pi_{\alg}(D) \otimes_E \varepsilon^{n-1},(\pi(D)^{\flat} \otimes_E \varepsilon^{n-1})^{\oplus m}\big)\\
\longrightarrow \Ext^1_{\GL_n(K)}\big(\pi_{\alg}(D) \otimes_E \varepsilon^{n-1}, \widehat{S}_{\xi,\tau}(U^{\wp},E)[\fm_{\pi}]^{\Qp\text{-}\an}\big)
\end{multline*}
where we have used (see (\ref{Elalg}))
\begin{multline*}
\Hom_{\GL_n(K)}\big(\pi_{\alg}(D)\otimes_E \varepsilon^{n-1}, (\pi(D)^{\flat} \otimes_E \varepsilon^{n-1})^{\oplus m}\big)\\\buildrel\sim\over\longrightarrow \Hom_{\GL_n(K)}\big(\pi_{\alg}(D)\otimes_E \varepsilon^{n-1},\widehat{S}_{\xi,\tau}(U^{\wp},E)[\fm_{\pi}]^{\Qp\text{-}\an}\big).
\end{multline*}
Since $\Hom_{\GL_n(K)}(\pi_{\alg}(D)\otimes_E \varepsilon^{n-1}, X)=0$ by the paragraph just above Remark \ref{rem:extra}, it then follows with Lemma \ref{lem:dimext} that we have the lower bound
\[\dim_E\Ext^1_{\GL_n(K)}\big(\pi_{\alg}(D) \otimes_E \varepsilon^{n-1}, \widehat{S}_{\xi,\tau}(U^{\wp},E)[\fm_{\pi}]^{\Qp\text{-}\an}\big)\geq m\Big(n+\frac{n(n+1)}{2}[K:\Qp]\Big).\]
But Theorem \ref{theoreminequalityext1} shows this is also an upper bound. Hence this is an equality and the lemma follows.
\end{proof}

For $\sigma\in \Sigma$ and $I\subset \{\varphi_0, \dots, \varphi_{n-1}\}$ of cardinality $\in \{1,\dots,n-1\}$ which is critical for~$\sigma$,~let
\[m_{\sigma,I}:=\dim_E\Hom_{\GL_n(K)}\big(\widetilde{C}(I,s_{\vert I\vert,\sigma}) \otimes_E \varepsilon^{n-1}, \widehat{S}_{\xi,\tau}(U^{\wp},E)[\fm_{\pi}]^{\Qp\text{-}\an}\big).\]
It follows from (i) of Corollary \ref{cor:multiplicity} with the discussion above Corollary \ref{cor:smooth} that we have $m_{\sigma,I}\geq m$. We fix an arbitrary $\GL_n(K)$-equivariant injection (using Theorem \ref{T: lg})
\begin{multline}\label{eq:injbig}
f:\bigg((\pi(D)^{\flat})^{\oplus m} \bigoplus \Big(\bigoplus_{\sigma \in \Sigma, \ \!\text{$I$ \!very \!critical \!for \!$\sigma$}}\!\!\!\!\widetilde{C}(I,s_{\vert I\vert,\sigma})^{\oplus m_{\sigma,I}}\Big)\bigg)\otimes_E \varepsilon^{n-1}\\
\hookrightarrow \widehat{S}_{\xi,\tau}(U^{\wp},E)[\fm_{\pi}]^{\Qp\text{-}\an}
\end{multline}
and we note that the image of $f$ does not depend on the choice of such an injection (see the comment above Remark \ref{rem:extra}). Recall that Conjecture \ref{conj:main} predicts that one has a $\GL_n(K)$-equivariant injection $(\pi(D) \otimes_E \varepsilon^{n-1})^{\oplus m} \hookrightarrow \widehat{S}_{\xi,\tau}(U^{\wp},E)[\fm_{\pi}]^{\Qp\text{-}\an}$. The following theorem can be seen as evidence towards this prediction (recall $W_{\sigma,I}$ is the unique non-split extension of $\pi_{\alg}(D)$ by $\widetilde{C}(I,s_{i,\sigma})$, see above (\ref{eq:pi(D)W})):

\begin{thm}\label{T: lg2}
Keep the setting of Theorem \ref{T: lg}. There exists a (possibly split) extension of the form
\begin{equation}\label{eq:posplit}
(\pi(D)^{\flat})^{\oplus m}\!\begin{xy} (30,0)*+{}="a"; (38,0)*+{}="b"; {\ar@{-}"a";"b"}\end{xy}\!\Big(\bigoplus_{\sigma \in \Sigma, \ \!\text{$I$ \!very \!critical \!for \!$\sigma$}}(W_{\sigma,I})^{\oplus m_{\sigma,I}}\Big)
\end{equation}
containing the left hand side of (\ref{eq:injbig}) and a $\GL_n(K)$-equivariant injection
\begin{equation}\label{eq:injBig}
\bigg((\pi(D)^{\flat})^{\oplus m}\!\begin{xy} (30,0)*+{}="a"; (38,0)*+{}="b"; {\ar@{-}"a";"b"}\end{xy}\!\Big(\bigoplus_{\sigma \in \Sigma, \ \!\text{$I$ \!very \!critical \!for \!$\sigma$}}(W_{\sigma,I})^{\oplus m_{\sigma,I}}\Big)\bigg) \otimes_E \varepsilon^{n-1}\hooklongrightarrow \widehat{S}_{\xi,\tau}(U^{\wp},E)[\fm_{\pi}]^{\Qp\text{-}\an}
\end{equation}
extending the injection (\ref{eq:injbig}) and such that
\[\Hom_{\GL_n(K)}\big(\pi_{\alg}(D)\otimes_E \varepsilon^{n-1}, \widehat{S}_{\xi,\tau}(U^{\wp},E)[\fm_{\pi}]^{\Qp\text{-}\an}/Y\big)=0\]
where $Y$ denotes the image of (\ref{eq:injBig}).
\end{thm} 
\begin{proof}
Denote by $X$ the image of $f$ in $\widehat{S}_{\xi,\tau}(U^{\wp},E)[\fm_{\pi}]^{\Qp\text{-}\an}$, since
\[\Hom_{\GL_n(K)}\big(\pi_{\alg}(D)\otimes_E \varepsilon^{n-1}, X\big)\buildrel\sim\over\longrightarrow \Hom_{\GL_n(K)}\big(\pi_{\alg}(D)\otimes_E \varepsilon^{n-1},\widehat{S}_{\xi,\tau}(U^{\wp},E)[\fm_{\pi}]^{\Qp\text{-}\an}\big)\]
as in the proof of Proposition \ref{lem:isoWu} we have an exact sequence of finite dimensional $E$-vector spaces
\begin{multline*}
0\longrightarrow \Hom_{\GL_n(K)}\big(\pi_{\alg}(D)\otimes_E \varepsilon^{n-1}, \widehat{S}_{\xi,\tau}(U^{\wp},E)[\fm_{\pi}]^{\Qp\text{-}\an}/X\big)\\
\longrightarrow \Ext^1_{\GL_n(K)}\big(\pi_{\alg}(D) \otimes_E \varepsilon^{n-1},X\big)\\
\longrightarrow \Ext^1_{\GL_n(K)}\big(\pi_{\alg}(D) \otimes_E \varepsilon^{n-1}, \widehat{S}_{\xi,\tau}(U^{\wp},E)[\fm_{\pi}]^{\Qp\text{-}\an}\big).
\end{multline*}
Since $(\pi(D)^{\flat}\otimes_E \varepsilon^{n-1})^{\oplus m}$ is a direct summand of $X$ by (\ref{eq:injbig}), it follows from Proposition \ref{lem:isoWu} that the last map is surjective and that its kernel has dimension 
\[\dim_E\Ext^1_{\GL_n(K)}\Big(\pi_{\alg}(D)\otimes_E \varepsilon^{n-1}, \Big(\bigoplus_{\sigma \in \Sigma, \ \!\text{$I$ \!very \!critical \!for \!$\sigma$}}\!\!\!\!\widetilde{C}(I,s_{\vert I\vert,\sigma})^{\oplus m_{\sigma,I}}\Big)\otimes_E \varepsilon^{n-1}\Big)=\sum m_{\sigma,I}\]
where the last equality follows from \cite[Lemma 3.5(1)]{Di25}. Hence we obtain
\[\dim_E\Hom_{\GL_n(K)}\big(\pi_{\alg}(D)\otimes_E \varepsilon^{n-1}, \widehat{S}_{\xi,\tau}(U^{\wp},E)[\fm_{\pi}]^{\Qp\text{-}\an}/X\big)=\sum m_{\sigma,I}\]
(the sum being over those $(\sigma,I)$ such that $I$ is very critical for $\sigma$). Using the last equality in Theorem \ref{T: lg} together with $\dim_E\Ext^1_{\GL_n(K)}(\pi_{\alg}(D),\widetilde{C}(I,s_{\vert I\vert,\sigma}))=1$ (\cite[Lemma 3.5(1)]{Di25}), it is not difficult to deduce the statement.
\end{proof}

Note that we do not know if the representation (\ref{eq:posplit}) is local, i.e.~only depends on the filtered $\varphi$-module $D$ (as its definition is global). The following conjecture implies the first part of Conjecture \ref{conj:main} under the Taylor-Wiles assumptions (Hypothesis \ref{TayWil0}):

\begin{conj}\label{conj:bis}
We have $m_{\sigma,I}=m$ for every $(\sigma,I)$ such that $I$ is very critical for $\sigma$, and the extension in (\ref{eq:posplit}) is split (hence equal to $\pi(D)^{\oplus m}$).
\end{conj}

We finish this article by some indirect evidence towards Conjecture \ref{conj:bis} via the Bezru\-kavnikov functor as defined in \cite[\S~7.2]{HHS25}.\bigskip

Fix an arbitrary refinement $\fR$ of $D=D_{\cris}(r)=D_{\cris}(\rho_{\pi,\widetilde \wp})$. Let $R^{\wp}_{\infty,\pi}$ be the completed local ring of the rigid variety $(\Spf R_{\infty}^{\wp}(\xi,\tau))^{\rig}$ at the point associated to the maximal ideal $\fm_{\pi}^{\wp}$ and define the noetherian local complete $E$-algebra
\[R_{\infty,\pi,\fR}:=R^{\wp}_{\infty,\pi}\widehat \otimes_E R_{r,\fR}\]
(see above (\ref{modeltri}) for $R_{r,\fR}$ which pro-represents the groupoid $X_{r,\fR}$ denoted $X_{r,\cM_\bullet}$ in \cite[\S~3.6]{BHS19}). Let $\cO_{\wt(\delta_{\fR})}$ be the block in the BGG category $\cO$ for $\fg_\Sigma$ with respect to the upper Borel $\fb_\Sigma$ containing the finite dimensional simple module $L(\wt(\delta_{\fR}))$. In \cite{HHS25} Hellmann, Hernandez and Schraen define two exact covariant functors $\cM_{\infty,\pi,\fR}$ and $\cB_{\infty,\pi,\fR}$ from $\cO_{\wt(\delta_{\fR})}$ to the category of finite type $R_{\infty,\pi,\fR}$-modules (strictly speaking their global setting is the one of \cite[\S~5]{BHS19} but their construction will also work in our setting, see Remark \ref{rem:setting}). The first functor $\cM_{\infty,\pi,\fR}$, called the patching functor, has a global and highly non-canonical construction (as it uses the patching), see \cite[\S\S~6.1,6.2]{HHS25}. The second functor $\cB_{\infty,\pi,\fR}$, called the Bezrukavnikov functor, is defined as the pull-back via $\Spf R_{\infty,\pi,\fR}\rightarrow \Spf R_{r, \fR}$ (and the local model of $\Spf R_{r, \fR}$) of a canonical functor due to Bezrukavnikov from $\cO_{\wt(\delta_{\fR})}$ to the category of coherent sheaves on a completion of the variety $X_\Sigma$ defined at the beginning of \S~\ref{sec:model}, see \cite[Cor.~7.7]{HHS25}. Most importantly in \cite[Rk.~1.5]{HHS25} it is conjectured that $\cM_{\infty,\pi,\fR}=(\cB_{\infty,\pi,\fR})^{\oplus m}$ with $m$ as in (\ref{Elalg}) (in particular $\cM_{\infty,\pi,\fR}$ should essentially be local and canonical). This is known for $\GL_2$ and $\GL_3$ (\cite[Cor.~7.17]{HHS25}).\bigskip

We only need here the following important property of $\cM_{\infty,\pi,\fR}$:

\begin{lem}\label{lem:dualbreuil}
For any $M$ in $\cO_{\wt(\delta_{\fR})}$ we have
\[\dim_E\Hom_{\GL_n(K)}\Big(\cF_{B^-}^{\GL_n}\big(M^*, \delta_{\fR}^{\sm} \delta_B^{-1}\big),\widehat{S}_{\xi,\tau}(U^{\wp},E)[\fm_{\pi}]^{\Qp\text{-}\an}\Big) = \dim_E (\cM_{\infty,\pi,\fR}(M)/\fm_{R_{\infty,\pi,\fR}})\]
where $M^*$ is the dual of $M$ defined in \cite[\S~3]{Br15}.
\end{lem}
\begin{proof}
This follows from \cite[Lemma 7.27]{HHS25} with \cite[Lemma 5.2.1]{BHS19}.
\end{proof}

For $\sigma\in \Sigma$ and $i\in \{1, \dots, n-1\}$ (arbitrary), let $L( s_{i,\sigma}\!\cdot\!\wt(\delta_{\fR}))$ be the simple $\text{U}(\fg_{\Sigma})$-module in $\cO_{\wt(\delta_{\fR})}$ of highest weight $ s_{i,\sigma}\!\cdot\!\wt(\delta_{\fR})$ (see above (\ref{eq:W}) for $ s_{i,\sigma}\!\cdot\!\wt(\delta_{\fR})$). Since for any weight $\lambda$ we have $(\text{U}(\fg_{\Sigma})\otimes_{\text{U}(\fb_{\Sigma})}\lambda)^*\cong (\text{U}(\fg_{\Sigma})\otimes_{\text{U}(\fb_{\Sigma}^-)}(-\lambda))^{\vee}$ where the latter is the dual in the sense of \cite[\S~3.2]{Hu08} (see the proof of \cite[Thm.~4.3]{Br15}), we deduce $L( s_{i,\sigma}\cdot\wt(\delta_{\fR}))^*\cong L^-(- s_{i,\sigma}\!\cdot\!\wt(\delta_{\fR}))$. Thus by Lemma \ref{lem:dualbreuil} with the discussion below (\ref{eq:W}) we have for the unique subset $I$ of cardinality $i$ such that $\fR$ is compatible with $I$ for $\sigma$:
\begin{multline*}
\dim_E (\cM_{\infty,\pi,\fR}(L( s_{i,\sigma}\!\cdot\!\wt(\delta_{\fR})))/\fm_{R_{\infty,\pi,\fR}})\\
=\dim_E\Hom_{\GL_n(K)}\big(\widetilde{C}(I,s_{i,\sigma}) \otimes_E \varepsilon^{n-1}, \widehat{S}_{\xi,\tau}(U^{\wp},E)[\fm_{\pi}]^{\Qp\text{-}\an}\big).
\end{multline*}
The following result was recently proved by Bezrukavnikov:

\begin{thm}[\cite{Be25}]\label{thm:bez}
The $R_{\infty,\pi,\fR}$-module $\cB_{\infty,\pi,\fR}(L( s_{i,\sigma}\!\cdot\!\wt(\delta_{\fR})))$ is free of rank $1$. In particular we have $\dim_E (\cB_{\infty,\pi,\fR}(L( s_{i,\sigma}\!\cdot\!\wt(\delta_{\fR})))/\fm_{R_{\infty,\pi,\fR}})=1$.
\end{thm}

Hoping for $\cM_{\infty,\pi,\fR}=(\cB_{\infty,\pi,\fR})^{\oplus m}$, we can therefore see Theorem \ref{thm:bez} as an indirect piece of evidence for the first statement of Conjecture \ref{conj:bis}.\bigskip

We finally give an indirect piece of evidence for the second statement of Conjecture \ref{conj:bis} in the case of $\GL_4(\Qp)$. We assume $K=\Qp$ (so we can forget about $\sigma$) and $r=\bigoplus_{j=0}^3\varepsilon^j\unr(p^{j}\varphi_j)$ so that by our conventions $(h_0,h_1,h_2,h_3)=(3,2,1,0)$ and the $\varphi_j$ are the Frobenius eigenvalues on $D_{\cris}(r)$ (note that $\val(\varphi_j)=-j$). Then $\pi_{\alg}(D)=\pi_p$ (see (\ref{eq:algsigma})) and one can check that $\{\varphi_0,\varphi_1\}$ is the only very critical subset, so that (with the notation of (i) of Proposition \ref{prop:flat}):
\[\pi(D)=\pi(D)^\flat \ \ \bigoplus \ \ C(\{\varphi_0,\varphi_1\}, s_2)\!\!\begin{xy} (30,0)*+{}="a"; (38,0)*+{}="b"; {\ar@{-}"a";"b"}\end{xy}\!\!\pi_p.\]
Let $\fR:=(\varphi_0,\varphi_1,\varphi_2,\varphi_3)$ and $M$ the unique non-split extension of $L(\wt(\delta_{\fR}))$ by $L(s_{2} \cdot\wt(\delta_{\fR}))$ in $\cO_{\wt(\delta_{\fR})}$. Then $M^*\cong X_{2}^-(-\wt(\delta_{\fR}))^{\vee}$ (as defined above (\ref{eq:W}))) and $\cF_{B^-}^{\GL_4}\big(M^*, \delta_{\fR}^{\sm} \delta_B^{-1}\big)\cong \big(C(\{\varphi_0,\varphi_1\}, s_2)\!\!\begin{xy} (30,0)*+{}="a"; (38,0)*+{}="b"; {\ar@{-}"a";"b"}\end{xy}\!\!\pi_p\big)\otimes_E\varepsilon^3$. Thus by Theorem \ref{T: lg2} and $\dim_E\Hom_{\GL_4(\Qp)}(\pi_p,\pi(D)^\flat)=1$ it is easy to see that Conjecture \ref{conj:bis} is true in that case if and only if $m_{\{\varphi_0,\varphi_1\}}=m$ and
\[\dim_E\Hom_{\GL_4(\Qp)}\Big(\cF_{B^-}^{\GL_4}\big(M^*, \delta_{\fR}^{\sm} \delta_B^{-1}\big),\widehat{S}_{\xi,\tau}(U^{\wp},E)[\fm_{\pi}]^{\Qp\text{-}\an}\Big)=2m.\]
If moreover $\cM_{\infty,\pi,\fR}=(\cB_{\infty,\pi,\fR})^{\oplus m}$ is true, then from Lemma \ref{lem:dualbreuil} we must have
\begin{equation}\label{eq:dim=2}
\dim_E (\cB_{\infty,\pi,\fR}(M)/\fm_{R_{\infty,\pi,\fR}})=2.
\end{equation}
The latter (or rather its variant with Bezrukavnikov's original functor) was implemented on a computer by Hernandez and Schraen who could check that we do have (\ref{eq:dim=2}).

\appendix

\section{On multiplicities of the companion constituents}\label{sec:multiplicity}

Building on the proof of \cite[Thm.~5.3.3]{BHS19}, we show that the multiplicities of the companion constituents are always at least the multiplicity of the locally algebraic contituent, which slightly strengthens \cite[Thm.~5.3.3]{BHS19}. We use this to prove Proposition \ref{prop:multiplicity}.\bigskip

We first use without comment the setting and notation of \cite[\S~5]{BHS19} (our setting is slightly different but this will not affect the proofs, see Remark \ref{rem:setting} or below). We \emph{do not} recall the notation and assumptions of \cite[\S~5]{BHS19} as this would be too tedious. Instead we refer the (motivated) reader to \emph{loc.~cit.} We define
\[m:=\dim_L\Hom_{G_p}\big(\cF_{\overline B_p}^{G_p}(\overline L(-\lambda),\underline\delta_{\fR,\sm}\delta_{B_p}^{-1}),\widehat S(U^p, L)^{\an}_{\fm^S}[\fm_{\rho}]\big)\in \Z_{\geq 1}\]
($L$ in \emph{loc.~cit.}~is the field denoted $E$ here). As in (\ref{Elalg}) this is the multiplicity of the locally algebraic vectors.

\begin{prop0}\label{prop:bound}
We keep the assumptions of \cite[Thm.~5.3.3]{BHS19}. We let $y$ ($=y_{w_0}$), $w_y$ and $y_w$ for $w_y\leq w$ as in Step $3$ of the proof of \cite[Thm.~5.3.3]{BHS19}. We let $\cL_{ww_0\cdot \mu}$ be the coherent sheaf over $X_p(\rhobar)$ as defined in \cite[(5.29)]{BHS19}. Then for all $w_y\leq w$ we have
\[\dim_{k(y_w)}\cL_{ww_0\cdot \mu}\otimes_{\cO_{X_p(\rhobar)}}k(y_w)\geq m\]
where $k(y_w)$ is the residue field of the point $y_w\in X_p(\rhobar)$.
\end{prop0}
\begin{proof}
Let $\cM_\infty=J_{B_p}(\Pi_{\infty}^{R\infty\text{-}\an})^\vee$ be the coherent sheaf over $X_p(\rhobar)$ as in Step 4 of the proof of \cite[Thm.~5.3.3]{BHS19}. By \cite[Remark 5.3.4]{BHS19} there is an integer $m_y\geq 1$ such that $\cM_\infty$ is locally free of rank $m_y$ in a sufficiently small smooth neighbourhood of $y_w$ for any $w\geq w_y$ (eventhough $y_w$ itself is not necessarily smooth). Taking $w=w_0$ and arguing as for the proof of (\ref{eq:algfree}) above in a neighbourhood of the crystalline dominant point $y_{w_0}$, we deduce that we must have $m_y=m$.\bigskip

We now claim that we can add the extra condition $\dim_{k(y_w)}\cL_{ww_0\cdot \mu}\otimes_{\cO_{X_p(\rhobar)}}k(y_w)\geq m$ to the induction hypothesis $\cH_\ell$ in Step 6 of the proof of \cite[Thm.~5.3.3]{BHS19}. One has to check that this condition is satisfied all along the rest of the proof. This is clear in Step 7 of the proof of \cite[Thm.~5.3.3]{BHS19} using the equality of non-zero cycles $[\cL(w_yw_0\cdot \mu)]=m_y{\mathfrak C}_{w_y}=m{\mathfrak C}_{w_y}$. This is also clear for the same reason in Step 10 of the proof of \cite[Thm.~5.3.3]{BHS19}. The only issue is to check that this condition is still satisfied in Step 9 of the proof of \cite[Thm.~5.3.3]{BHS19}, more precisely in the end of \emph{Ad}(i). But it is indeed satisfied at all points of the rigid space $Z_p(\rhobar)^{{\mathfrak U}^p}_{ww_0\cdot \mu}$ coming from the Zariski-closure ${\mathfrak U}^p\times \overline{{\widetilde W}_{\rhobar_p,w}^{\mu^{{\text HT}\text{-cr},{\mathfrak X}^p\text{-aut}}}}\times {\mathbb U}^g$ (with the notation of \emph{loc.~cit.}) by the (new) induction hypothesis $\cH_\ell$ together with the upper continuity of the rank of a coherent sheaf $\cL$ on a rigid analytic space $X$ (which says that the set of points $x\in X$ such that $\dim_{k(x)}\cL\otimes_{\cO_X}k(x)\geq d$ is Zariski-closed for any integer $d$). And the proof of \emph{loc.~cit.}~can proceed.
\end{proof}

Proposition \ref{prop:bound} has the following nice consequence:

\begin{cor0}\label{cor:multiplicity}
We use the notation of \cite[Conj.~5.3.1]{BHS19} and the assumptions of \cite[Thm.~5.3.3]{BHS19}.
\begin{enumerate}[label=(\roman*)]
\item
For all $w_{\fR}\leq w$ we have
\[\dim_L\Hom_{G_p}\big(\cF_{\overline B_p}^{G_p}\big(\overline L(-ww_0\cdot\lambda),\underline\delta_{\fR,\sm}\delta_{B_p}^{-1}\big),\widehat S(U^p, L)^{\an}_{\fm^S}[\fm_{\rho}]\big)\geq m.\]
\item
If $X_p(\rhobar)$ is smooth at the (companion) point $x_{\fR,w}$ defined in \cite[\S~5.3]{BHS19}, we have
\[\dim_L\Hom_{G_p}\big(\cF_{\overline B_p}^{G_p}\big(\overline L(-ww_0\cdot\lambda),\underline\delta_{\fR,\sm}\delta_{B_p}^{-1}\big),\widehat S(U^p, L)^{\an}_{\fm^S}[\fm_{\rho}]\big)=m.\]
\end{enumerate}
\end{cor0}
\begin{proof}
Part (i) follows from the definition of $\cL_{ww_0\cdot \mu}$ in \cite[(5.29)]{BHS19} with Proposition \ref{prop:bound} and \cite[Prop.~5.2.2]{BHS19} applied with $s=1$. We prove (ii). Since $x_{\fR,w}$ is a smooth point of $X_p(\rhobar)$, the coherent sheaf $\cM_\infty=J_{B_p}(\Pi_{\infty}^{R\infty\text{-}\an})^\vee$ (see the proof of Proposition \ref{prop:bound}) is locally free of rank $m$ at $x_{\fR,w}$ by (the proof of) the induction hypothesis $\cH_\ell$ in Step 6 of the proof of \cite[Thm.~5.3.3]{BHS19}. Then by \cite[Prop.~5.2.3]{BHS19} (applied with $s=1$) we deduce 
\[\dim_L\Hom_{G_p}\big(\cF_{\overline B_p}^{G_p}\big((\text{U}(\fg)\otimes_{\text{U}(\overline\fb)}(-ww_0\cdot \lambda))^\vee, \underline\delta_{\fR,\sm}\delta_{B_p}^{-1}\big), \widehat S(U^p, L)^{\an}_{\fm^S}[\fm_{\rho}]\big)=m.\]
But since $\cF_{\overline B_p}^{G_p}(\overline L(-ww_0\cdot\lambda),\underline\delta_{\fR,\sm}\delta_{B_p}^{-1})$ is a quotient of $\cF_{\overline B_p}^{G_p}((\text{U}(\fg)\otimes_{\text{U}(\overline\fb)}(-ww_0\cdot \lambda))^\vee, \underline\delta_{\fR,\sm}\delta_{B_p}^{-1})$, we obtain
\[\dim_L\Hom_{G_p}\big(\cF_{\overline B_p}^{G_p}\big(\overline L(-ww_0\cdot\lambda),\underline\delta_{\fR,\sm}\delta_{B_p}^{-1}\big),\widehat S(U^p, L)^{\an}_{\fm^S}[\fm_{\rho}]\big)\leq m\]
which together with (i) gives the equality.
\end{proof}

We now go back to the notation of the present paper and to the setting of \S~\ref{sec:global0}. As mentioned in Remark \ref{rem:setting}, using that $(\Spf \widehat{\otimes}_{v\in S_p\setminus \{\wp\}} R_{\overline{\rho}_{\widetilde{v}}}(\xi_v,\tau_v))^{\rig}$ is smooth, all the arguments of \cite[\S~5]{BHS19}, and thus also Corollary \ref{cor:multiplicity}, extend replacing the ``prime-to-$p$ part" by the ``prime-to-${\wp}$ part'' since the only property used in \emph{loc.~cit.}~on the ``prime-to-$p$'' factor of $X_p(\rhobar)$ is that it is smooth at the points we consider. For instance we have the following corollary of (ii) of Corollary \ref{cor:multiplicity}:

\begin{cor0}\label{cor:smooth}
Let $\sigma\in \Sigma$, $I\subset \{\varphi_j,\ 0\leq j \leq n-1\}$ critical for $\sigma$ of cardinality $i\in \{1,\dots,n-1\}$ and $\fR$ a refinement compatible with $I$ for $\sigma$ (recall that critical mans that $s_{i,\sigma}$ appears with multiplicity $\geq 1$ in some (equivalently any) reduced expression of $w_{\fR,\sigma}w_{0,\sigma}$ where $w_{\fR,\sigma}\in S_n$ is defined just above Proposition \ref{prop:Psmooth}). Assume the rigid analytic space $\cE_{\infty}(\xi,\tau)$ (see the beginning of \S~\ref{sec:patched}) is smooth at the point
\begin{equation}\label{EptxRi}
x_{\fR,s_{i,\sigma}w_{0,\sigma}}:=\big(\fm_{\pi}^{\wp}, \fm_{\pi,\wp}, \delta_{\fR,\sigma,i}\big)
\end{equation}
where $\delta_{\fR,\sigma,i}:= (\unr(\underline{\varphi})t^{-s_{i,\sigma}\cdot \wt(\delta_{\fR})} \delta_B(\boxtimes_{j=0}^{n-1} \vert \cdot\vert_K^{j}))\in \widehat{T}$ (see below (\ref{eq:kappax1}) for the notation and note that $x_{\fR,s_{i,\sigma}w_{0,\sigma}}$ is a companion point of the dominant point $x_{\fR}$ in (\ref{EptxR}) which lies in $\cE_{\infty}(\xi,\tau)$ since $I$ is critical). Then we have
\[\dim_E\Hom_{\GL_n(K)}\big(\widetilde{C}(I,s_{i,\sigma}) \otimes_E \varepsilon^{n-1}, \Pi_{\infty}(\xi,\tau)^{R_{\infty}(\xi,\tau)\text{-}\an}[\fm_{\pi}]\big)=m.\]
\end{cor0}

We now prove Proposition \ref{prop:multiplicity}. We first need a combinatorial lemma. For $w\in S_n$ we let $D_L(w)\subseteq R=\{s_1,\dots,s_{n-1}\}$ be the subset $\{s_j\in R,\ s_jw<w\}$.

\begin{lem0}\label{lem:multfree}
Let $w\in S_n$ such that $s_i$ appears with multiplicity $1$ in some reduced expression of $w$. Then there is $w'\in S_n$ satisfying the two properties:
\begin{enumerate}[label=(\roman*)]
\item
$s_i$ does not appear in some (equivalently any) reduced expression of $w'$;
\item
$w'w$ is multiplicity free of one of the following $4$ forms
\begin{equation}\label{eq:weil}
\left\{\begin{array}{ccl}
w'w&=&s_i\\
w'w&=&s_is_{i-1}\cdots s_{i-\delta^-}\text{ for some $\delta^->0$}\\
w'w&=&s_is_{i+1}\cdots s_{i+\delta^+}\text{ for some $\delta^+>0$}\\
w'w&=&s_is_{i-1}\cdots s_{i-\delta^-}s_{i+1}\cdots s_{i+\delta^+}\text{ for some $\delta^-,\delta^+>0$}.
\end{array}\right.
\end{equation}
\end{enumerate}
\end{lem0}
\begin{proof}
Replacing $w$ by $w'w$ for some $w'$ satisfying (i), we can assume $w=s_iw_1$ where $s_i$ does not appear in any reduced expression of $w_1$, $\lg(w_1)=\lg(w)-1$ and $D_L(w_1)\subseteq \{s_{i-1}, s_{i+1}\}$ (deleting $s_j$ if $j\notin \{i-1,i+1\}$). Assume $s_{i-1}\in D_L(w_1)$ (the case $s_{i+1}\in D_L(w_1)$ is similar), then we have $w=s_is_{i-1}w_2$ where $w_1=s_{i-1}w_2$, $s_i$ does not appear in any reduced expression of $w_2$, $\lg(w_2)=\lg(w_1)-1$ and $D_L(w_2)\subseteq \{s_{i-2}, s_{i+1}\}$. If $s_{i-2}\in D_L(w_2)$, then we have $w=s_is_{i-1}s_{i-2}w_3$ where $w_2=s_{i-2}w_3$, $s_i$ does not appear in any reduced expression of $w_3$, $\lg(w_3)=\lg(w_2)-1$ and $D_L(w_3)\subseteq \{s_{i-3}, s_{i-1}, s_{i+1}\}$. But we cannot have $s_{i-1}\in D_L(w_3)$ since one easily checks that this implies $s_{i-2}\in D_L(w_1)$ which contradicts $D_L(w_1)\subseteq \{s_{i-1}, s_{i+1}\}$. Hence $D_L(w_3)\subseteq \{s_{i-3}, s_{i+1}\}$. If $s_{i+1}\in D_L(w_2)$, then we have $w=s_is_{i-1}s_{i+1}w_3$ where $w_2=s_{i+1}w_3$, $s_i$ does not appear in any reduced expression of $w_3$, $\lg(w_3)=\lg(w_2)-1$ and $D_L(w_3)\subseteq \{s_{i-2}, s_{i+2}\}$. Iterating this process, we see that we have $w=s_i$, or $w=s_is_{i-1}\cdots s_{i-\delta^-}$ for some $\delta^->0$, or $w=s_is_{i+1}\cdots s_{i+\delta^+}$ for some $\delta^+>0$, or $w=s_is_{i-1}\cdots s_{i-\delta^-}s_{i+1}\cdots s_{i+\delta^+}$ for some $\delta^-,\delta^+>0$.
\end{proof}

Lemma \ref{lem:multfree} has the direct consequence:

\begin{prop0}\label{prop:case}
Let $I\subset \{\varphi_j,\ 0\leq j \leq n-1\}$ of cardinality $i\in \{1,\dots,n-1\}$ and $\fR$ a refinement compatible with $I$ for $\sigma$. If $s_{i,\sigma}$ appears with multiplicity $1$ in some reduced expression of $w_{\fR,\sigma}w_{0,\sigma}$, then there is a refinement $\fR'$ compatible with $I$ for $\sigma$ such that $w_{\fR',\sigma}w_{0,\sigma}$ has one of the $4$ forms in (\ref{eq:weil}) (with $s_{j,\sigma}$ instead~of~$s_j$).
\end{prop0}
\begin{proof}
One easily checks that, for any $w_\sigma\in S_n$ such that $s_{i,\sigma}$ does not appear in some (equivalently any) reduced expression of $w_\sigma$, there exists a refinement $\fR'$ compatible with $I$ such that $w_{\fR',\sigma}=w_\sigma w_{\fR,\sigma}$. By Lemma \ref{lem:multfree} applied to $w=w_{\fR,\sigma}w_{0,\sigma}$, there exists such a $w_\sigma$ with $w_\sigma w_{\fR,\sigma}w_{0,\sigma}$ as in (\ref{eq:weil}), and we take a corresponding $\fR'$.
\end{proof}

We then have the following smoothness result in the spirit of Corollary \ref{C: smooth}:

\begin{prop0}\label{C: smoothbis}
Let $I\subset \{\varphi_j,\ 0\leq j \leq n-1\}$ of cardinality $i\in \{1,\dots,n-1\}$ and $\fR$ a refinement compatible with $I$ for $\sigma$ such that $w_{\fR,\sigma}w_{0,\sigma}$ has one of the $4$ forms in (\ref{eq:weil}). Then the rigid analytic space $\cE_{\infty}(\xi,\tau)$ is smooth at the companion point $x_{\fR,s_{i,\sigma}w_{0,\sigma}}$ in (\ref{EptxRi}).
\end{prop0}
\begin{proof}
Since $\fm_{\pi}^{\wp}$ defines a smooth point on $(\Spf R_{\infty}^{\wp}(\xi,\tau))^{\rig}$ (see the proof of Corollary \ref{C: smooth}), it is enough to prove that $\iota_p(\fm_{\pi,\wp}, \delta_{\fR,\sigma,i})$ defines a smooth point on $X_{\tri}(\overline{r})$ (see (\ref{eq:iotap})). By (ii) of \cite[Prop.~4.1.5]{BHS19} applied with $w_x=w_{\fR}$ and $w=s_iw_0$ (forgetting the index $\sigma$ in the notation) it is enough to prove that the Schubert variety $\overline{Bs_iw_0B/B}$ is smooth at the point $w_{\fR}B$, and that we have $d_{s_iw_0w_{\fR}^{-1}}=\lg(s_iw_0)-\lg(w_{\fR})$ where $d_w=n-\dim_E\ft^w$ for $w\in S_n$. By assumption we have $w_{\fR}=w w_0$ where $w$ has one of the $4$ forms in (the right hand side of) (\ref{eq:weil}), hence this is equivalent to $\overline{Bs_iw_0B/B}$ smooth at $ww_0B$ and $n-\dim_E\ft^{s_iw^{-1}}=\lg(w)-1$. The second equality is a direct explicit check on the $4$ forms of $w$ in (\ref{eq:weil}) that we leave to the reader. For the smoothness assertion, by \cite[Thm.~1]{LS84} we need to check that there is exactly one (not necessarily simple) reflection $s_\alpha\in S_n$ such that $s_i$ does not appear in some (equivalently any) reduced expression of $s_\alpha w$. Using that $s_i$ appears only once in some reduced expression of $w$, this is an easy exercise (note that here we do not need (\ref{eq:weil})).
\end{proof}

\begin{cor0}
Proposition \ref{prop:multiplicity} holds.
\end{cor0}
\begin{proof}
The case when $I$ is non-critical (for $\sigma$) is clear. Assume that $s_{i,\sigma}$ appears with multiplicity $1$ in some reduced expression of $w_{\fR,\sigma}w_{0,\sigma}$. Changing the refinement $\fR$ by Proposition \ref{prop:case} if necessary, we can assume that $w_{\fR,\sigma}w_{0,\sigma}$ has one of the $4$ forms in the right hand side of (\ref{eq:weil}) (recall $\widetilde{C}(I,s_{i,\sigma})$ does not see which refinement compatible with $I$ for $\sigma$ is chosen). Then the result follows from Proposition \ref{C: smoothbis} and Corollary \ref{cor:smooth}.
\end{proof}

\setcounter{footnote}{0}
\renewcommand{\thefootnote}{\fnsymbol{footnote}}

\section[An estimate for certain extension groups by Zhixiang Wu]{An estimate for certain extension groups by Zhixiang Wu\footnote{School of Mathematical Sciences, University of Science and Technology of China, 96 Jinzhai Road, 230026 Hefei, China}}\label{Wu}

We establish an upper bound for the dimension of certain extension groups between locally algebraic representations and the Hecke eigenspaces of the completed cohomology (Theorem \ref{theoreminequalityext1}). We work with the patched completed cohomology/homology $\Pi_{\infty}(\xi,\tau)$, $M_{\infty}(\xi,\tau)$ and the patched Galois deformation ring $R_{\infty}(\xi,\tau)$ introduced in \S~\ref{sec:global} and follow the notation in that section.\bigskip

We fix a maximal ideal $\fm=\fm_{\pi}\subset R_{\overline{\rho},\cS}(\xi,\tau)[\frac{1}{p}]$ with the residue field $E$ associated to an automorphic representation $\pi$. We view $\fm$ as a maximal ideal of $R_{\infty}(\xi,\tau)[\frac{1}{p}]$ via the quotient map $R_{\infty}(\xi,\tau)\rightarrow R_{\overline{\rho},\cS}(\xi,\tau)$. We also write $\fm$ for its intersection with $R_{\infty}(\xi,\tau)$ by abuse of notation. Let
\[\Pi_{\infty}(\xi,\tau)[\fm]^{\Qp\text{-}\an}=\widehat{S}_{\xi,\tau}(U^{\wp},E)[\fm]^{\Qp\text{-}\an}\]
be the subspace of locally $\Q_p$-analytic vectors of the corresponding Hecke eigenspace, which are representations of $\GL_n(K)=G(F_{\wp})$. Let $\rho_{\pi}$ be the Galois representation associated to $\pi$ and let $\rho_{\pi,\widetilde{\wp}}:=\rho_{\pi}|_{\Gal(\overline{F_{\widetilde{\wp}}}/F_{\widetilde{\wp}})}$. We assume that $\rho_{\pi,\widetilde{\wp}}$ is crystalline with regular Hodge-Tate weights 
\begin{align}\label{HTweights}
\{h_{j,\sigma}=\lambda_{j,\sigma}+n-1+j\}_{j=0,\dots,n-1,\sigma\in\Sigma}.
\end{align} 
We also assume that $\rho_{\pi,\widetilde{\wp}}$ is generic in the sense that the eigenvalues of $\varphi^f$, where $\varphi$ is the crystalline Frobenius on $D_{\cris}(\rho_{\pi,\widetilde{\wp}})$ and $f$ is the degree of the residue field of $K=F_{\widetilde{\wp}}$ over $\bbF_p$, given by 
\begin{align}\label{phifeigenvalues}
\{\varphi_0,\dots,\varphi_{n-1}\}
\end{align}
satisfy that $\varphi_i\varphi_j^{-1}\neq 1,p^{f}$ for all $i\neq j$. By (\ref{Elalg}), $\Pi_{\infty}(\xi,\tau)[\fm]^{\Qp\text{-}\alg}=\pi_{\alg}^{\oplus m}$ where we write 
\[\pi_{\alg}:=\pi_{\alg}(D)\otimes_E\varepsilon^{n-1}\] for short.\bigskip

Let $D(\GL_n(K),E)$ be the locally $\Q_p$-analytic distribution algebra of $\GL_n(K)$. The strong duals $(\Pi_{\infty}(\xi,\tau)[\fm]^{\Qp\text{-}\an})^{\vee},\pi_{\alg}^{\vee}$ of $\Pi_{\infty}(\xi,\tau)[\fm]^{\Qp\text{-}\an}$ and $\pi_{\alg}$ are coadmissible $D(\GL_n(K),E)$-modules. For two admissible locally analytic representations $V_1,V_2$ of $\GL_n(K)$, we set 
\[\Ext^i_{\GL_n(K)}(V_1,V_2):=\Ext^i_{D(\GL_n(K),E)}(V_2^{\vee},V_1^{\vee})\] 
where the latter is calculated in the derived category of abstract $D(\GL_n(K),E)$-modules. The goal of this appendix is to establish the following upper bound.

\begin{thm0}\label{theoreminequalityext1}
There is an inequality
\[\dim_E\Ext^1_{\GL_n(K)}(\pi_{\alg},\widehat{S}_{\xi,\tau}(U^{\wp},E)[\fm]^{\Qp\text{-}\an})\leq m(n+\frac{n(n+1)}{2}[K:\Q_p]).\]
Moreover, the inequality is an equality if $M_{\infty}(\xi,\tau)[\frac{1}{p}]$ is flat over $R_{\infty}(\xi,\tau)[\frac{1}{p}]$ at $\fm$.
\end{thm0}
\begin{proof}
This follows from Lemma \ref{lemmainequality} and Proposition \ref{propositionderivedfiber} below.
\end{proof} Let $\cO_E[[\GL_n(K)]]:=\cO_E[\GL_n(K)]\otimes_{\cO_E[\GL_n(\cO_K)]}\cO_E[[\GL_n(\cO_K)]]$ be the Iwasawa algebra of $\GL_n(K)$ and let $E[[\GL_n(K)]]:=\cO_E[[\GL_n(K)]][\frac{1}{p}]$. We consider the derived tensor product 
\begin{align*}
M_{\infty}(\xi,\tau)\otimes_{\cO_E[[\GL_n(K)]]}^L\pi_{\alg}=M_{\infty}(\xi,\tau)[\frac{1}{p}]\otimes_{E[[\GL_n(K)]]}^L\pi_{\alg}
\end{align*}
of abstract $\cO_E[[\GL_n(K)]]$-modules. Our convention is that the left $\cO_E[[\GL_n(K)]]$-module $M_{\infty}(\xi,\tau)$ is viewed as a right $\cO_E[[\GL_n(K)]]$-module via the involution of $\cO_E[[\GL_n(K)]]$ induced \ by \ $g\mapsto g^{-1}$ \ for \ $g\in \GL_n(K)$. \ Since \ $M_{\infty}(\xi,\tau)$ \ is \ an \ $R_{\infty}(\xi,\tau)$-module, \ $M_{\infty}(\xi,\tau)\otimes_{\cO_E[[\GL_n(K)]]}^L\pi_{\alg}$ is an object in the derived category of $R_{\infty}(\xi,\tau)$-modules (and $R_{\infty}(\xi,\tau)[\frac{1}{p}]$-modules).

\begin{lem0}\label{lemmainequality}
There is an inequality
\begin{align*}
&\dim_E\Ext^1_{\GL_n(K)}(\pi_{\alg},\widehat{S}_{\xi,\tau}(U^{\wp},E)[\fm]^{\Qp\text{-}\an})\\
\leq &\dim_EH^{-1}((R_{\infty}(\xi,\tau)/\fm)\otimes_{R_{\infty}(\xi,\tau)}^L(M_{\infty}(\xi,\tau)\otimes_{\cO_E[[\GL_n(K)]]}^L\pi_{\alg}))
\end{align*}
where $H^{-1}$ denotes the cohomology group in the cohomological degree $-1$. Moreover, the inequality is an equality if $M_{\infty}(\xi,\tau)[\frac{1}{p}]$ is flat over $R_{\infty}(\xi,\tau)[\frac{1}{p}]$ at $\fm$.
\end{lem0}
\begin{proof}
We first show that 
{\small
\begin{align}
\dim_E\Ext^1_{\GL_n(K)}(\pi_{\alg},\Pi_{\infty}(\xi,\tau)[\fm]^{\Qp\text{-}\an})=\dim_EH^{-1}((M_{\infty}(\xi,\tau)/\fm)\otimes_{\cO_E[[\GL_n(K)]]}^L\pi_{\alg}).\label{equationlemmainequality1}
\end{align}}
\!\!By \cite[Thm.~7.1 (iii)]{ST03}
\[(\Pi_{\infty}(\xi,\tau)[\fm]^{\Qp\text{-}\an})^{\vee}=D(\GL_n(\cO_K),E)\otimes_{E[[\GL_n(\cO_K)]]}(M_{\infty}(\xi,\tau)[\frac{1}{p}]/\fm).\]
Using \ the \ flatness \ of \ $D(\GL_n(\cO_K),E)$ \ over \ $E[[\GL_n(\cO_K)]]$ \ \cite[Thm.~4.11]{ST03} \ and \ that $D(\GL_n(K),E)\simeq D(\GL_n(\cO_K),E)\otimes_{E[[\GL_n(\cO_K)]]}E[[\GL_n(K)]]$, we have
{\small
\begin{align*}
(\Pi_{\infty}(\xi,\tau)[\fm]^{\Qp\text{-}\an})^{\vee}&=D(\GL_n(\cO_K),E)\otimes_{E[[\GL_n(\cO_K)]]}(M_{\infty}(\xi,\tau)[\frac{1}{p}]/\fm),\\
&=D(\GL_n(\cO_K),E)\otimes_{E[[\GL_n(\cO_K)]]}E[[\GL_n(K)]]\otimes_{E[[\GL_n(K)]]}(M_{\infty}(\xi,\tau)[\frac{1}{p}]/\fm),\\
&=D(\GL_n(\cO_K),E)\otimes_{E[[\GL_n(\cO_K)]]}^LE[[\GL_n(K)]]\otimes_{E[[\GL_n(K)]]}^L(M_{\infty}(\xi,\tau)[\frac{1}{p}]/\fm),\\
&=D(\GL_n(K),E)\otimes_{E[[\GL_n(K)]]}^L(M_{\infty}(\xi,\tau)[\frac{1}{p}]/\fm)
\end{align*}}
\!\!(for the last equality, we used the associativity of derived tensor products \cite[Example 10.8.1]{We94}). By definition of the extension groups and the above equality, we get
\begin{align}
&\Ext^1_{\GL_n(K)}(\pi_{\alg},\Pi_{\infty}(\xi,\tau)[\fm]^{\Qp\text{-}\an})\nonumber\\
=&\Ext^1_{D(\GL_n(K),E)}((\Pi_{\infty}(\xi,\tau)[\fm]^{\Qp\text{-}\an})^{\vee},\pi_{\alg}^{\vee})\nonumber \\
=&H^1(R\Hom_{D(\GL_n(K),E)}(D(\GL_n(K),E)\otimes_{E[[\GL_n(K)]]}^L(M_{\infty}(\xi,\tau)[\frac{1}{p}]/\fm),\pi_{\alg}^{\vee}))\nonumber \\
=&H^1(R\Hom_{E[[\GL_n(K)]]}(M_{\infty}(\xi,\tau)[\frac{1}{p}]/\fm,\pi_{\alg}^{\vee})).\label{equationlemmainequality2} \end{align}
 
The strong dual $\pi_{\alg}^{\vee}=\Hom_E^{\rm cont}(\pi_{\alg},E)$ is the space of continuous linear functions on $\pi_{\alg}$. Since $\pi_{\alg}$ is equipped with the finest locally convex topology (cf.~\cite[\S~3, p.119]{ST01}), any linear function on $\pi_{\alg}$ is continuous. We see (using that any $E$-vector space is injective in the category of $E$-vector spaces for the last equality)
\[\pi_{\alg}^{\vee}=\Hom_E^{\rm cont}(\pi_{\alg},E)=\Hom_E(\pi_{\alg},E)=R\Hom_E(\pi_{\alg},E).\]
Hence by the tensor-Hom adjunction (replacing the $E[[\GL_n(K)]]$-module $M_{\infty}(\xi,\tau)[\frac{1}{p}]/\fm$ by a projective resolution to calculate $R\Hom_{E[[\GL_n(K)]]}$ \cite[Thm.~10.7.4]{We94} and then applying the adjunction between the functors $-\otimes_{E[[\GL_n(K)]]}\pi_{\alg}$ and $\Hom_{E}(\pi_{\alg},-)$ \cite[II.4.1]{Bo98}), we have
\begin{align*}
R\Hom_{E[[\GL_n(K)]]}(M_{\infty}(\xi,\tau)[\frac{1}{p}]/\fm,\pi_{\alg}^{\vee}) =&R\Hom_{E[[\GL_n(K)]]}(M_{\infty}(\xi,\tau)[\frac{1}{p}]/\fm,R\Hom_E(\pi_{\alg},E))\\
=&R\Hom_{E}((M_{\infty}(\xi,\tau)[\frac{1}{p}]/\fm)\otimes_{E[[\GL_n(K)]]}^L\pi_{\alg},E). 
\end{align*}The equality (\ref{equationlemmainequality1}) follows from taking $H^1$ and $E$-dual of the above and (\ref{equationlemmainequality2}).\bigskip

Next, we show that
\begin{align}
&\dim_EH^{-1}((M_{\infty}(\xi,\tau)/\fm)\otimes_{\cO_E[[\GL_n(K)]]}^L\pi_{\alg})\nonumber \\
\leq &\dim_EH^{-1}((R_{\infty}(\xi,\tau)/\fm)\otimes_{R_{\infty}(\xi,\tau)}^L(M_{\infty}(\xi,\tau)\otimes_{\cO_E[[\GL_n(K)]]}^L\pi_{\alg})).\label{equationlemmainequality3}
\end{align}
Write 
\begin{align*}
M_x&:=(R_{\infty}(\xi,\tau)/\fm)\otimes_{R_{\infty}(\xi,\tau)}^LM_{\infty}(\xi,\tau),\\
M_{x,0}&:=H^0(M_x)=M_{\infty}(\xi,\tau)/\fm
\end{align*} 
for short. We have an exact triangle in the derived category of $\cO_E[[\GL_n(K)]]$-modules:
\begin{align*}
\tau_{\leq -1}M_x\rightarrow M_x\rightarrow \tau_{\geq 0} M_x\rightarrow 
\end{align*}
where $\tau_{\leq -1},\tau_{\geq 0}$ denotes the canonical truncations for the cohomological complexes (see \cite[\href{https://stacks.math.columbia.edu/tag/08J5}{Tag 08J5}]{St25}). Then $\tau_{\geq 0}M_x\simeq M_{x,0}$. Applying the functor $-\otimes_{\cO_E[[\GL_n(K)]]}^L\pi_{\alg}$, we obtain an exact triangle
\[(\tau_{\leq -1}M_x)\otimes_{\cO_E[[\GL_n(K)]]}^L\pi_{\alg}\rightarrow M_x\otimes_{\cO_E[[\GL_n(K)]]}^L\pi_{\alg}\rightarrow M_{x,0}\otimes_{\cO_E[[\GL_n(K)]]}^L\pi_{\alg}\rightarrow.\]
Taking cohomology groups, we get a long exact sequence
\begin{align*}
\cdots\rightarrow H^{-1}(M_x\otimes_{\cO_E[[\GL_n(K)]]}^L\pi_{\alg})\rightarrow H^{-1}(M_{x,0}\otimes_{\cO_E[[\GL_n(K)]]}^L\pi_{\alg})\\
\rightarrow H^0((\tau_{\leq -1}M_x)\otimes_{\cO_E[[\GL_n(K)]]}^L\pi_{\alg})\rightarrow \cdots.
\end{align*}
Notice that $H^0((\tau_{\leq -1}M_x)\otimes_{\cO_E[[\GL_n(K)]]}^L\pi_{\alg})=0$ since $\tau_{\leq -1}M_x$ concentrates in degrees $\leq -1$ and $-\otimes_{\cO_E[[\GL_n(K)]]}^L\pi_{\alg}$ is right exact. Hence 
\[\dim_E H^{-1}(M_{x,0}\otimes_{\cO_E[[\GL_n(K)]]}^L\pi_{\alg})\leq \dim_E H^{-1}(M_x\otimes_{\cO_E[[\GL_n(K)]]}^L\pi_{\alg})\] 
by the above exact sequence. This is exactly the desired (\ref{equationlemmainequality3}).\bigskip

Finally, \ the \ inequality \ in \ the \ lemma \ follows \ from \ combining \ (\ref{equationlemmainequality1}) \ and \ (\ref{equationlemmainequality3}). \ If $M_{\infty}(\xi,\tau)[\frac{1}{p}]$ is flat over $R_{\infty}(\xi,\tau)[\frac{1}{p}]$ at $\fm$, then $(R_{\infty}(\xi,\tau)/\fm)\otimes_{R_{\infty}(\xi,\tau)}^LM_{\infty}(\xi,\tau)[\frac{1}{p}]=(M_{\infty}(\xi,\tau)/\fm)[\frac{1}{p}]$ and hence (\ref{equationlemmainequality3}) is an equality.
\end{proof}

We will study the (derived) $R_{\infty}(\xi,\tau)$-module $M_{\infty}(\xi,\tau)\otimes_{\cO_E[[\GL_n(K)]]}^L\pi_{\alg}$ and its derived specialization 
\[R_{\infty}(\xi,\tau)/\fm\otimes^L_{R_{\infty}(\xi,\tau)}(M_{\infty}(\xi,\tau)\otimes_{\cO_E[[\GL_n(K)]]}^L\pi_{\alg})\]
at $\fm$. Recall that $R_{\infty}(\xi,\tau)=((\widehat{\otimes}_{v\in S^p\cup\{\wp\}}R_{\overline{\rho}_{\widetilde{v}}})\widehat{\otimes}_{\cO_E}(\widehat{\otimes}_{v\in S_p\setminus\{\wp\}}R_{\overline{\rho}_{\widetilde{v}}}(\xi_v,\tau_v)))[[x_1,\dots,x_{g}]]$. Let $R_{\overline{\rho}_{\widetilde{\wp}}}^{\mathrm{cris},\lambda}$ be the quotient of $R_{\overline{\rho}_{\widetilde{\wp}}}$ constructed by Kisin \cite{Ki08} parametrizing framed crystalline deformations of $\overline{\rho}_{\widetilde{\wp}}$ with Hodge-Tate weights $\{h_{j,\sigma}\}$ (\ref{HTweights}). Let $\rho_{\widetilde{\wp}}^{\mathrm{cris},\lambda}:\Gal(\overline{F_{\widetilde{\wp}}}/F_{\widetilde{\wp}})\rightarrow \GL_n(R_{\overline{\rho}_{\widetilde{\wp}}}^{\mathrm{cris},\lambda})$ be the universal framed crystalline deformation. There exists a universal $\varphi$-module $D_{\rm cris}(\rho_{\widetilde{\wp}}^{\mathrm{cris},\lambda}[\frac{1}{p}])$ over $R_{\overline{\rho}_{\widetilde{\wp}}}^{\mathrm{cris},\lambda}[\frac{1}{p}]\otimes_{\Q_p}\Q_{p^f}$ attached to $\rho_{\widetilde{\wp}}^{\mathrm{cris},\lambda}[\frac{1}{p}]$ as in \cite[Thm.~2.5.5]{Ki08}. Let $T\subset \GL_n$ be the subgroup of diagonal matrices and let $W$ be the Weyl group for $\GL_n$. After fixing an embedding $K_0=\Q_{p^f}\hookrightarrow E$, the coefficients of the characteristic polynomial of $\varphi^f$ on $D_{\rm cris}(\rho_{\widetilde{\wp}}^{\mathrm{cris},\lambda}[\frac{1}{p}])\otimes_{E\otimes_{\Q_p}\Q_{p^f}}E$ induce a map
\begin{align}\label{equationmapgit}
\Spec(R_{\overline{\rho}_{\widetilde{\wp}}}^{\mathrm{cris},\lambda}[\frac{1}{p}])\rightarrow \GL_n/\!/\GL_n
\end{align}
where $\GL_n/\!/\GL_n\simeq T/W\simeq \bbA_E^{n-1}\times \bbG_{m,E}$ is the GIT quotient for the adjoint action of $\GL_n$ on itself. The coefficients of the polynomial $\prod_{j=0}^{n-1}(X-\varphi_j)$ (see (\ref{phifeigenvalues}) for $\varphi_j$) define an $E$-point $\underline{\varphi}$ in $\GL_n/\!/\GL_n$. We let $\Spec(R_{\overline{\rho}_{\widetilde{\wp}}}^{\mathrm{cris},\lambda}[\frac{1}{p}])_{\underline{\varphi}}$ be the fiber over $\underline\varphi$ of the map (\ref{equationmapgit}), a closed subscheme of $\Spec(R_{\overline{\rho}_{\widetilde{\wp}}}^{\mathrm{cris},\lambda}[\frac{1}{p}])$.

\begin{lem0}\label{lemmalcidim}
The following statements hold:
\begin{enumerate}[label=(\roman*)]
\item
There exists an open neighborhood $U\subset \GL_n/\!/\GL_n$ of $\underline{\varphi}$ such that the restriction of the map (\ref{equationmapgit}) to the inverse image of $U$ is flat.
\item
The closed embedding $\Spec(R_{\overline{\rho}_{\widetilde{\wp}}}^{\mathrm{cris},\lambda}[\frac{1}{p}])_{\underline{\varphi}}\hookrightarrow \Spec(R_{\overline{\rho}_{\widetilde{\wp}}}[\frac{1}{p}])$ is a regular immersion of codimension $n+\frac{n(n+1)}{2}[K:\Q_p]$.
\end{enumerate}
\end{lem0}
\begin{proof}
(i) Let $\Spec(R_{\overline{\rho}_{\widetilde{\wp}}}^{\mathrm{cris},\lambda}[\frac{1}{p}])^{\square}\rightarrow \Spec(R_{\overline{\rho}_{\widetilde{\wp}}}^{\mathrm{cris},\lambda}[\frac{1}{p}])$ be the $\GL_{n,E}$-torsor trivializing the universal rank $n$ bundle $D_{\rm cris}(\rho_{\widetilde{\wp}}^{\mathrm{cris},\lambda}[\frac{1}{p}])\otimes_{E\otimes_{\Q_p}\Q_{p^f}}E$. By the proof of \cite[Thm.~3.3.8]{Ki08} (and the equivalence between the category of $\varphi$-modules over $E\otimes_{\Q_p}\Q_{p^f}$ and $\varphi^f$-modules over $E$ as in \cite[\S~4]{BS07}), the map 
\begin{align}\label{equationlemmaflatproof}
\Spec(R_{\overline{\rho}_{\widetilde{\wp}}}^{\mathrm{cris},\lambda}[\frac{1}{p}])^{\square}\rightarrow \GL_{n} 
\end{align}
induced by the matrix of $\varphi^f$ is formally smooth at any closed points of $\Spec(R_{\overline{\rho}_{\widetilde{\wp}}}^{\mathrm{cris},\lambda}[\frac{1}{p}])^{\square}$. This means that the maps between the complete local rings at closed points induced by (\ref{equationlemmaflatproof}) are formally smooth and, thus, flat by \cite[\href{https://stacks.math.columbia.edu/tag/07PM}{Tag 07PM}]{St25}. Since flatness can be checked after completion \cite[\href{https://stacks.math.columbia.edu/tag/0C4G}{Tag 0C4G}]{St25}, the induced maps between local rings at closed points are flat. Consequently, the map (\ref{equationlemmaflatproof}) itself is flat by \cite[\href{https://stacks.math.columbia.edu/tag/00HT}{Tag 00HT}]{St25}. The GIT quotient map $\GL_{n}\rightarrow \GL_{n}/\!/\GL_n$ is flat on the (open \cite[2.14]{St65}) regular semisimple locus of $\GL_n$ (the fibers in this locus have constant dimensions, see \cite[Thm.~6.11, Rem. 6.15]{St65}, and one can apply the miracle flatness theorem). The composition $\Spec(R_{\overline{\rho}_{\widetilde{\wp}}}^{\mathrm{cris},\lambda}[\frac{1}{p}])^{\square}\rightarrow \GL_{n}\rightarrow \GL_{n}/\!/\GL_n$ factors through (\ref{equationmapgit}). As $\varphi_i\neq\varphi_j$ for $i\neq j$, we get that $ \Spec(R_{\overline{\rho}_{\widetilde{\wp}}}^{\mathrm{cris},\lambda}[\frac{1}{p}])^{\square}\rightarrow \GL_n/\!/\GL_n$ is flat over an open neighborhood of $ \underline{\varphi}$. Since $\Spec(R_{\overline{\rho}_{\widetilde{\wp}}}^{\mathrm{cris},\lambda}[\frac{1}{p}])^{\square}\rightarrow \Spec(R_{\overline{\rho}_{\widetilde{\wp}}}^{\mathrm{cris},\lambda}[\frac{1}{p}])$ is flat and surjective, (\ref{equationmapgit}) is also flat over an open neighborhood of $ \underline{\varphi}$ by \cite[\href{https://stacks.math.columbia.edu/tag/02JZ}{Tag 02JZ}]{St25}.\bigskip

(ii) By (i) the map (\ref{equationmapgit}) is flat over the inverse image of an open neighborhood $U$ of $\underline{\varphi}$. Denote this inverse image by $V$. Since the closed embedding $ \underline{\varphi}\hookrightarrow U\subset \bbA^{n-1}\times\bbG_m$ is a regular embedding of codimension $n$, its flat base change $\Spec(R_{\overline{\rho}_{\widetilde{\wp}}}^{\mathrm{cris},\lambda}[\frac{1}{p}])_{\underline{\varphi}}\hookrightarrow V\subset \Spec(R_{\overline{\rho}_{\widetilde{\wp}}}^{\mathrm{cris},\lambda}[\frac{1}{p}])$ is also regular by \cite[\href{https://stacks.math.columbia.edu/tag/067P}{Tag 067P}]{St25} and of codimension $n$. Hence, it suffices to show that the closed embedding $\Spec(R_{\overline{\rho}_{\widetilde{\wp}}}^{\mathrm{cris},\lambda}[\frac{1}{p}])\hookrightarrow \Spec(R_{\overline{\rho}_{\widetilde{\wp}}}[\frac{1}{p}])$ is regular at points in $\Spec(R_{\overline{\rho}_{\widetilde{\wp}}}^{\mathrm{cris},\lambda}[\frac{1}{p}])_{\underline{\varphi}}$ (i.e.~for any point $x\in \Spec(R_{\overline{\rho}_{\widetilde{\wp}}}^{\mathrm{cris},\lambda}[\frac{1}{p}])_{\underline{\varphi}}$, there exists an open affine neighborhood $V_x\subset \Spec(R_{\overline{\rho}_{\widetilde{\wp}}}[\frac{1}{p}])$ such that the immersion $V_x\cap \Spec(R_{\overline{\rho}_{\widetilde{\wp}}}^{\mathrm{cris},\lambda}[\frac{1}{p}])\hookrightarrow V_x$ is regular) using \cite[\href{https://stacks.math.columbia.edu/tag/067Q}{Tag 067Q}]{St25}. Since $\Spec(R_{\overline{\rho}_{\widetilde{\wp}}}^{\mathrm{cris},\lambda}[\frac{1}{p}])$ is regular of dimension $n^2+[K:\Q_p]n^2-[K:\Q_p]\frac{n(n+1)}{2}$ (\cite[Thm.~3.3.8]{Ki08}), by \cite[Prop.~19.1.1]{Gr67}, we only need to show that $\Spec(R_{\overline{\rho}_{\widetilde{\wp}}}[\frac{1}{p}])$ is regular at any point in $\Spec(R_{\overline{\rho}_{\widetilde{\wp}}}^{\mathrm{cris},\lambda}[\frac{1}{p}])_{\underline{\varphi}}$. For any closed point $x\in \Spec(R_{\overline{\rho}_{\widetilde{\wp}}}^{\mathrm{cris},\lambda}[\frac{1}{p}])_{\underline{\varphi}}$, the complete local ring of $\Spec(R_{\overline{\rho}_{\widetilde{\wp}}}[\frac{1}{p}])$ at $x$ is the framed deformation ring of the Galois representation $\rho_x$ at $x$ associated with $x$ over local Artinian $E_x$-algebras where $E_x$ denotes the residue field at $x$ (cf.~\cite[Prop.~2.3.5]{Ki09}). Since $\rho_x$ is crystalline and generic by our assumption ($\varphi_i\varphi_j^{-1}\neq p^f$ for all $i\neq j$), we have $H^2(\Gal(\overline{F_{\widetilde{\wp}}}/F_{\widetilde{\wp}}),\mathrm{End}(\rho_x))\simeq H^0(\Gal(\overline{F_{\widetilde{\wp}}}/F_{\widetilde{\wp}}),\mathrm{End}(\rho_x)\otimes \varepsilon)=0$ (using the local Tate duality). Hence the framed deformation problem for $\rho_x$ is unobstructed and the complete local ring at $x$ is a formal power series ring of dimension $n^2+[K:\Q_p]n^2$. We see $\Spec(R_{\overline{\rho}_{\widetilde{\wp}}}[\frac{1}{p}])$ is regular along $\Spec(R_{\overline{\rho}_{\widetilde{\wp}}}^{\mathrm{cris},\lambda}[\frac{1}{p}])_{\underline{\varphi}}$. This concludes the proof.
\end{proof}

Let $\cH=E[T]^W$ be the coordinate ring of $\GL_n/\!/\GL_n=T/W$. Let $\chi_{\underline{\varphi}}:\cH\rightarrow E$ be the character associated to the point $\underline{\varphi}$. Then $R_{\overline{\rho}_{\widetilde{\wp}}}^{\cris,\lambda}[\frac{1}{p}]\otimes_{\cH}\chi_{\underline{\varphi}}$ is the coordinate ring of the fiber $\Spec(R_{\overline{\rho}_{\widetilde{\wp}}}^{\mathrm{cris},\lambda}[\frac{1}{p}])_{\underline{\varphi}}$. Set 
\[R_{\infty}^{\mathrm{cris},\lambda}(\xi,\tau):=R_{\infty}(\xi,\tau)\otimes_{R_{\overline{\rho}_{\widetilde{\wp}}}}R_{\overline{\rho}_{\widetilde{\wp}}}^{\mathrm{cris},\lambda}.\] 
We get a map $\cH\rightarrow R_{\overline{\rho}_{\widetilde{\wp}}}^{\cris,\lambda}[\frac{1}{p}]\rightarrow R_{\infty}^{\mathrm{cris},\lambda}(\xi,\tau)[\frac{1}{p}]$ induced by (\ref{equationmapgit}). By our assumption on $\rho_{\pi,\widetilde{\wp}}$, the maximal ideal $\fm\subset R_{\infty}(\xi,\tau)[\frac{1}{p}]$ corresponds to a maximal ideal of its quotient $R_{\infty}^{\mathrm{cris},\lambda}(\xi,\tau)[\frac{1}{p}]\otimes_{\cH}\chi_{\underline{\varphi}}$.

\begin{lem0}\label{lemmavectorbundle}
The complex $M_{\infty}(\xi,\tau)\otimes_{\cO_E[[\GL_n(K)]]}^L\pi_{\alg}$ of $R_{\infty}(\xi,\tau)[\frac{1}{p}]$-modules is quasi-iso\-morphic to a complex of finite $R_{\infty}^{\mathrm{cris},\lambda}(\xi,\tau)[\frac{1}{p}]\otimes_{\cH}\chi_{\underline{\varphi}}$-modules (seen as $R_{\infty}(\xi,\tau)[\frac{1}{p}]$-modules via the surjection $R_{\infty}(\xi,\tau)[\frac{1}{p}]\rightarrow R_{\infty}^{\mathrm{cris},\lambda}(\xi,\tau)[\frac{1}{p}]\otimes_{\cH}\chi_{\underline{\varphi}}$). Moreover, its localization at $\fm$,
\begin{align*}
&(M_{\infty}(\xi,\tau)\otimes_{\cO_E[[\GL_n(K)]]}^L\pi_{\alg})_{\fm}\\
=&(M_{\infty}(\xi,\tau)\otimes_{\cO_E[[\GL_n(K)]]}^L\pi_{\alg})\otimes_{R_{\infty}^{\mathrm{cris},\lambda}(\xi,\tau)[\frac{1}{p}]\otimes_{\cH}\chi_{\underline{\varphi}}}^L(R_{\infty}^{\mathrm{cris},\lambda}(\xi,\tau)[\frac{1}{p}]\otimes_{\cH}\chi_{\underline{\varphi}})_{\fm},
\end{align*}
concentrates in degree $0$ and is free of rank $m$ over $(R_{\infty}^{\mathrm{cris},\lambda}(\xi,\tau)[\frac{1}{p}]\otimes_{\cH}\chi_{\underline{\varphi}})_{\fm}$.
\end{lem0}
\begin{proof}
The representation $\pi_{\alg}$ has the form $\pi_{\alg}=\pi_{\rm sm}\otimes_E\sigma_{\alg}$ where $\pi_{\rm sm}$ is an irreducible unramified principal series representation of $\GL_n(K)$ and $\sigma_{\alg}$ is an irreducible algebraic representation. The $E[\GL_n(K)]$-module $E[\GL_n(K)]\otimes_{E[\GL_n(\cO_K)]}E$ coincides with the compact induction of the trivial representation $E$ of $\GL_n(\cO_K)$ as a representation of $\GL_n(K)$. Then 
\[\cH\simeq\End_{\GL_n(K)}(E[\GL_n(K)]\otimes_{E[\GL_n(\cO_K)]}E)\] 
is isomorphic to the usual spherical Hecke algebra via the Satake isomorphism, and $\chi_{\underline{\varphi}}:\cH\rightarrow E$ is the Satake parameter associated with $\pi_{\rm sm}$. There is an isomorphism of $\GL_n(K)$-representations \ (see \ \cite[Thm.~1.2]{Mo21} \ for \ the \ first \ isomorphism \ and \ the \ flatness \ of $E[\GL_n(K)]\otimes_{E[\GL_n(\cO_K)]}E$ over $\cH$ in \textit{loc.~cit.}~for the second isomorphism):
\[\pi_{\rm sm}\simeq (E[\GL_n(K)]\otimes_{E[\GL_n(\cO_K)]}E)\otimes_{\cH}\chi_{\underline{\varphi}}\simeq(E[\GL_n(K)]\otimes_{E[\GL_n(\cO_K)]}E)\otimes_{\cH}^L\chi_{\underline{\varphi}}.\]
Tensoring $\sigma_{\alg}$ over $E$ induces an isomorphism $\cH\simeq \End_{\GL_n(K)}(E[\GL_n(K)]\otimes_{E[\GL_n(\cO_K)]}\sigma_{\alg})$ and also an isomorphism (cf.~the proof of \cite[Prop.~3.16]{BHS171})
\[\pi_{\alg}\simeq (E[\GL_n(K)]\otimes_{E[\GL_n(\cO_K)]}\sigma_{\alg})\otimes_{\cH}\chi_{\underline{\varphi}}\simeq (E[\GL_n(K)]\otimes_{E[\GL_n(\cO_K)]}\sigma_{\alg})\otimes_{\cH}^L\chi_{\underline{\varphi}}.\]
Hence, we have
\begin{align*}
M_{\infty}(\xi,\tau)\otimes_{\cO_E[[\GL_n(K)]]}^L\pi_{\alg}&=M_{\infty}(\xi,\tau)\otimes_{\cO_E[[\GL_n(K)]]}^L(E[\GL_n(K)]\otimes_{E[\GL_n(\cO_K)]}\sigma_{\alg})\otimes_{\cH}^L\chi_{\underline{\varphi}}\\
&= (M_{\infty}(\xi,\tau)\otimes_{\cO_E[[\GL_n(\cO_K)]]}^L\sigma_{\alg})\otimes_{\cH}^L\chi_{\underline{\varphi}}.
\end{align*}
Here, $\cH$ acts on $M_{\infty}(\xi,\tau)\otimes_{\cO_E[[\GL_n(\cO_K)]]}^L\sigma_{\alg}$ by acting on the the second factor of the right-hand side of the following isomorphism
\[M_{\infty}(\xi,\tau)\otimes_{\cO_E[[\GL_n(\cO_K)]]}^L\sigma_{\alg}\simeq M_{\infty}(\xi,\tau)\otimes_{\cO_E[[\GL_n(K)]]}^L(E[\GL_n(K)]\otimes_{E[\GL_n(\cO_K)]}\sigma_{\alg}).\] 
Since $M_{\infty}(\xi,\tau)$ is finite projective over $S_{\infty}[[\GL_n(\cO_K)]]$, we see 
\begin{align}\label{equationvectorbundle}
M_{\infty}(\xi,\tau)\otimes_{\cO_E[[\GL_n(\cO_K)]]}^L\sigma_{\alg}=M_{\infty}(\xi,\tau)\otimes_{\cO_E[[\GL_n(\cO_K)]]}\sigma_{\alg} 
\end{align}
concentrates in degree $0$ and is in fact a variant of the patched module $M_{\infty}(\sigma^{\circ})[\frac{1}{p}]$ in \cite[Lemma 4.14]{CEGGPS16} if we take $\sigma^{\circ}$ a $\GL_n(\cO_K)$-stable $\cO_E$-lattice in $\sigma:=\sigma_{\alg}|_{\GL_n(\cO_K)}$. As in \cite[p.257]{CEGGPS16}, \ the \ continuous \ $E$-dual \ of \ $M_{\infty}(\xi,\tau)\otimes_{\cO_E[[\GL_n(\cO_K)]]}\sigma_{\alg}$ \ is $\Hom_{\GL_n(\cO_K)}(\sigma_{\alg},\Pi_{\infty}(\xi,\tau))$ and the transpose of the previous action of $\cH$ on the Hom space coincides with the usual Hecke action via the Frobenius reciprocity: 
\[\Hom_{\GL_n(\cO_K)}(\sigma_{\alg},\Pi_{\infty}(\xi,\tau))=\Hom_{\GL_n(K)}(E[\GL_n(K)]\otimes_{E[\GL_n(\cO_K)]}\sigma_{\alg},\Pi_{\infty}(\xi,\tau)).\] 
By (the same proof of) \cite[Lemma 4.17]{CEGGPS16} using the classical local-global compatibility at $\widetilde{\wp}$, the action of $R_{\overline{\rho}_{\widetilde{\wp}}}$ on $M_{\infty}(\xi,\tau)\otimes_{\cO_E[[\GL_n(\cO_K)]]}\sigma_{\alg}$ factors through $ R_{\overline{\rho}_{\widetilde{\wp}}}^{\cris,\lambda}[\frac{1}{p}]$, and the Hecke action of $\cH$ on it factors through the map 
\begin{align}\label{equationHecketoGalois}
\cH\rightarrow R_{\overline{\rho}_{\widetilde{\wp}}}^{\cris,\lambda}[\frac{1}{p}] 
\end{align}
in \cite[Thm.~4.1]{CEGGPS16}.\bigskip
 
We show that the map (\ref{equationHecketoGalois}) induces the ring map (\ref{equationmapgit}) after taking the spectra. The map (\ref{equationHecketoGalois}) interpolates the classical unramified local Langlands correspondence at $\widetilde{\wp}$: for any maximal ideal $x\in \Spec(R_{\overline{\rho}_{\widetilde{\wp}}}^{\cris,\lambda}[\frac{1}{p}])$ with the residue field $E_x$, the composition $\cH\rightarrow R_{\overline{\rho}_{\widetilde{\wp}}}^{\cris,\lambda}[\frac{1}{p}]\rightarrow E_x$ is the Satake parameter of the smooth representation of $\GL_n(K)$ associated with the Weil-Deligne representation attached to the crystalline representation $\rho_x$ associated with $x$, cf.~\cite[Prop.~4.2]{CEGGPS16}. Since the Satake parameters are exactly given by the characteristic polynomials of the $f$-power of the crystalline Frobenius, we see that the map (\ref{equationHecketoGalois}) coincides with the ring map inducing (\ref{equationmapgit}) after modulo an arbitrary maximal ideal of $R_{\overline{\rho}_{\widetilde{\wp}}}^{\cris,\lambda}[\frac{1}{p}]$. Since the ring $R_{\overline{\rho}_{\widetilde{\wp}}}^{\cris,\lambda}[\frac{1}{p}]$ is Jacobson (see for instance \cite[\S~2]{Co09}), the map $R_{\overline{\rho}_{\widetilde{\wp}}}^{\cris,\lambda}[\frac{1}{p}]\rightarrow \prod_{x}E_x$, where $x$ runs through all closed points of $\Spec(R_{\overline{\rho}_{\widetilde{\wp}}}^{\cris,\lambda}[\frac{1}{p}])$, is an injection. We conclude that (\ref{equationHecketoGalois}) is indeed the ring map inducing (\ref{equationmapgit}).\bigskip

Moreover, \ by \ the \ same \ proof \ as \ for \ \cite[Lemma 4.18]{CEGGPS16}, \ the \ module $M_{\infty}(\xi,\tau)\otimes_{\cO_E[[\GL_n(\cO_K)]]}\sigma_{\alg}$ is finite over $R_{\infty}(\xi,\tau)[\frac{1}{p}]$, and is Cohen-Macaulay over its support, which must be a union of components of $\Spec(R_{\infty}^{\cris,\lambda}(\xi,\tau)[\frac{1}{p}])\subset \Spec(R_{\infty}(\xi,\tau)[\frac{1}{p}])$. The argument on \cite[p.1633]{BHS171} and the fact that the rings $R_{\overline{\rho}_{\widetilde{\wp}}}^{\cris,\lambda}[\frac{1}{p}],R_{\overline{\rho}_{\widetilde{v}}}(\xi_v,\tau_v)[\frac{1}{p}],v\in S_p\setminus\{\wp\}$ are regular \cite[Thm.~3.3.8]{Ki08} imply that the ring $R_{\infty}^{\cris,\lambda}(\xi,\tau)[\frac{1}{p}]$ is regular at the maximal ideal $\fm$. Hence $M_{\infty}(\xi,\tau)\otimes_{\cO_E[[\GL_n(\cO_K)]]}\sigma_{\alg}$ is locally free over $R_{\infty}^{\cris,\lambda}(\xi,\tau)[\frac{1}{p}]$ at $\fm$ by the arguments in the proof of \cite[Lemma 4.18]{CEGGPS16}: the difference from \textit{loc.~cit.}~is that the rank of $M_{\infty}(\xi,\tau)\otimes_{\cO_E[[\GL_n(\cO_K)]]}\sigma_{\alg}$ at $\fm$ is $m$, the multiplicity of $\sigma_{\alg}$ in $\Pi_{\infty}(\xi,\tau)[\fm]^{\Q_p\mathrm{-alg}}|_{\GL_n(\cO_K)}$.\bigskip

Hence, using (\ref{equationvectorbundle}), the flatness of the map $\cH\rightarrow R_{\overline{\rho}_{\widetilde{\wp}}}^{\cris,\lambda}[\frac{1}{p}]$ at $\underline{\varphi}$ in Lemma \ref{lemmalcidim} (1), the flatness of $R_{\infty}^{\cris,\lambda}(\xi,\tau)[\frac{1}{p}]$ over $R_{\overline{\rho}_{\widetilde{\wp}}}^{\cris,\lambda}[\frac{1}{p}]$, and the associativity of the derived tensor product \cite[Example 10.8.1]{We94}, the complex
\begin{align*}
&(M_{\infty}(\xi,\tau)\otimes_{\cO_E[[\GL_n(\cO_K)]]}^L\sigma_{\alg})\otimes_{\cH}^L\chi_{\underline{\varphi}}\\
\simeq &(M_{\infty}(\xi,\tau)\otimes^L_{\cO_E[[\GL_n(\cO_K)]]}\sigma_{\alg})\otimes_{R_{\infty}^{\cris,\lambda}(\xi,\tau)[\frac{1}{p}]}^L(R_{\infty}^{\cris,\lambda}(\xi,\tau)[\frac{1}{p}]\otimes_{\cH}^L\chi_{\underline{\varphi}})\\
\simeq &(M_{\infty}(\xi,\tau)\otimes_{\cO_E[[\GL_n(\cO_K)]]}\sigma_{\alg})\otimes_{R_{\infty}^{\cris,\lambda}(\xi,\tau)[\frac{1}{p}]}^L(R_{\infty}^{\cris,\lambda}(\xi,\tau)[\frac{1}{p}]\otimes_{\cH}\chi_{\underline{\varphi}})
\end{align*}
calculating \ $M_{\infty}(\xi,\tau)\otimes_{\cO_E[[\GL_n(K)]]}^L\pi_{\alg} $ \ is \ quasi-isomorphic \ to \ a \ complex \ of \ finite $R_{\infty}^{\cris,\lambda}(\xi,\tau)[\frac{1}{p}]\otimes_{\cH}\chi_{\underline{\varphi}}$-modules. (If we take a projective resolution of $M_{\infty}(\xi,\tau)\otimes_{\cO_E[[\GL_n(\cO_K)]]}\sigma_{\alg}$ by finite free $R_{\infty}^{\cris,\lambda}(\xi,\tau)[\frac{1}{p}]$-modules, we see $M_{\infty}(\xi,\tau)\otimes_{\cO_E[[\GL_n(K)]]}^L\pi_{\alg}$ is also quasi-isomorphic to a complex of finite free $R_{\infty}^{\cris,\lambda}(\xi,\tau)[\frac{1}{p}]\otimes_{\cH}\chi_{\underline{\varphi}}$-modules.) Its localization at $\fm$ is \ free \ of \ rank \ $m$ \ over \ $(R_{\infty}^{\cris,\lambda}(\xi,\tau)[\frac{1}{p}]\otimes_{\cH}\chi_{\underline{\varphi}})_{\fm}$ \ since \ the \ $R_{\infty}^{\cris,\lambda}(\xi,\tau)[\frac{1}{p}]$-module $M_{\infty}(\xi,\tau)\otimes_{\cO_E[[\GL_n(\cO_K)]]}\sigma_{\alg}$ is locally free of rank $m$ at $\fm$ by the previous discussions. 
\end{proof}\bigskip

\begin{prop0}\label{propositionderivedfiber}
We have
\[\dim_E H^{-1}((R_{\infty}(\xi,\tau)/\fm)\otimes_{R_{\infty}(\xi,\tau)}^L(M_{\infty}(\xi,\tau)\otimes_{\cO_E[[\GL_n(K)]]}^L\pi_{\alg}))=m(n+\frac{n(n+1)}{2}[K:\Q_p]). \]
\end{prop0}
\begin{proof}
By Lemma \ref{lemmavectorbundle}, we have
\begin{align*}
(M_{\infty}(\xi,\tau)\otimes_{\cO_E[[\GL_n(K)]]}^L\pi_{\alg})_{\fm}\simeq ((R_{\infty}^{\mathrm{cris},\lambda}(\xi,\tau)[\frac{1}{p}]\otimes_{\cH}\chi_{\underline{\varphi}})_{\fm})^{\oplus m}
\end{align*}
as $ R_{\infty}(\xi,\tau)[\frac{1}{p}]_{\fm}$-modules. Hence
\begin{align}\label{equationderivedfiber}
&(R_{\infty}(\xi,\tau)/\fm)\otimes_{R_{\infty}(\xi,\tau)}^L(M_{\infty}(\xi,\tau)\otimes_{\cO_E[[\GL_n(K)]]}^L\pi_{\alg})\nonumber\\
\simeq &\left((R_{\infty}(\xi,\tau)[\frac{1}{p}]_{\fm}/\fm)\otimes_{R_{\infty}(\xi,\tau)[\frac{1}{p}]_{\fm}}^L(R_{\infty}^{\mathrm{cris},\lambda}(\xi,\tau)[\frac{1}{p}]\otimes_{\cH}\chi_{\underline{\varphi}})_{\fm}\right)^{\oplus m}.
\end{align}
Write \ for \ short \ $R=R_{\infty}(\xi,\tau)[\frac{1}{p}]$ \ and \ let \ $I$ \ be \ the \ kernel \ of \ the \ surjection $R\twoheadrightarrow R_{\infty}^{\mathrm{cris},\lambda}(\xi,\tau)[\frac{1}{p}]\otimes_{\cH}\chi_{\underline{\varphi}}$. By (2) of Lemma \ref{lemmalcidim} and the flat base change along $R_{\overline{\rho}_{\widetilde{\wp}}}[\frac{1}{p}]\rightarrow R$, the kernel $I_{\fm}$ of the map $R_{\fm}\rightarrow (R/I)_{\fm}$ is generated by a regular sequence $f_1,\dots,f_d$ in $R_{\fm}$ of length $d:=n+\frac{n(n+1)}{2}[K:\Q_p]$. The sequence $f_1,\dots,f_d$ is thus a Koszul regular sequence in $R_{\fm}$ (\cite[\href{https://stacks.math.columbia.edu/tag/062F}{Tag 062F}]{St25}): the Koszul complex \cite[\href{https://stacks.math.columbia.edu/tag/0623}{Tag 0623}]{St25}
\begin{equation}\label{Koszulcomplex}
0\rightarrow\wedge^{d}R_{\fm}^d\rightarrow \cdots\rightarrow R_{\fm}^d\rightarrow R_{\fm}
\end{equation}
is a projective resolution of the $R_{\fm}$-module $R_{\fm}/I_{\fm}$ where, for a basis $e_1,\dots,e_d$ of $R_{\fm}^d$, the differential $\wedge^k R_{\fm}^d\rightarrow \wedge^{k-1} R_{\fm}^d$ is given by
\[e_{s_1}\wedge \cdots \wedge e_{s_k}\rightarrow \sum_{i=1,\dots,k}(-1)^{i+1}f_{s_i} e_{s_1}\wedge \cdots \wedge \widehat{e_{s_i}}\wedge \cdots\wedge e_{s_k}.\]
The derived tensor product $R_{\fm}/\fm\otimes_{R_{\fm}}^LR_{\fm}/I_{\fm}$ is calculated by the base change of the complex (\ref{Koszulcomplex}) from $R_{\fm}$ to $R_{\fm}/\fm$:
\begin{equation}\label{Koszulcomplexbasechange}
0\rightarrow\wedge^{d}(R_{\fm}/\fm)^d\rightarrow \cdots\rightarrow (R_{\fm}/\fm)^d\rightarrow R_{\fm}/\fm
\end{equation}
where the differentials still send the generators $ e_{s_1}\wedge \cdots \wedge e_{s_k}$ of $\wedge^{k}(R_{\fm}/\fm)^d$ to the image of $\sum_{i=1,\dots,k}(-1)^{i+1}f_{s_i} e_{s_1}\wedge \cdots \wedge \widehat{e_{s_i}}\wedge \cdots\wedge e_{s_k}$ in $\wedge^{k-1}(R_{\fm}/\fm)^d$. Since $f_1,\dots,f_d\in I_{\fm}\subset \fm R_{\fm}$ act by zero on $\wedge^{k-1}(R_{\fm}/\fm)^d$ for all $k\geq 1$, the differentials in the complex (\ref{Koszulcomplexbasechange}) are all zero. In particular, we have
\[H^{-1}((R_{\fm}/\fm)\otimes_{R_{\fm}}^LR_{\fm}/I_{\fm})\simeq \wedge^{1}(R_{\fm}/\fm)^d=(R_{\fm}/\fm)^d\]
has \ dimension \ $d=n+\frac{n(n+1)}{2}[K:\Q_p]$ \ over \ $E=R_{\fm}/\fm$. \ By \ (\ref{equationderivedfiber}), \ we \ see \ that $H^{-1}((R_{\infty}(\xi,\tau)/\fm)\otimes_{R_{\infty}(\xi,\tau)}^L(M_{\infty}(\xi,\tau)\otimes_{\cO_E[[\GL_n(K)]]}^L\pi_{\alg}))$ has dimension $m(n+\frac{n(n+1)}{2}[K:\Q_p])$ over $E$. 
\end{proof}

\end{document}